\documentclass[a4paper,twoside]{amsart}
\usepackage{xypic,amscd,amssymb,latexsym}
\usepackage{qtree}
\makeatletter
\def\@setcopyright{}
\makeatother

\newcommand{\GL}{\operatorname{GL}\nolimits}
\newcommand{\GF}{\operatorname{GF}\nolimits}

\newcommand{\U}{\operatorname{U}\nolimits}
\newcommand{\Sp}{\operatorname{Sp}\nolimits}

\newcommand{\Fi}{\operatorname{Fi}\nolimits}
\newcommand{\Co}{\operatorname{Co}\nolimits}

\newcommand{\M}{\operatorname{\mathcal M}\nolimits}

\newtheorem{lemma}{Lemma}[section]
\newtheorem{proposition}[lemma]{Proposition}
\newtheorem{corollary}[lemma]{Corollary}
\newtheorem{theorem}[lemma]{Theorem}
\newtheorem{remark}[lemma]{Remark}
\newtheorem{definition}[lemma]{Definition}
\newtheorem{example}[lemma]{Example}

\newtheorem{algorithm}[lemma]{Algorithm}
\newtheorem{cclass}[lemma]{}
\newtheorem{ctab}[lemma]{}

\newtheorem{implementation}[lemma]{Implementation}
\newtheorem{strategy}[lemma]{Strategy}

\renewcommand{\arraystretch}{1.5}

\begin{document}

\thispagestyle{empty}
\title{Representation Theoretic EXISTENCE PROOF FOR Fischer GROUP $\Fi_{23}$}.
\author{\bf
Hyun Kyu Kim}
\address{Department of Mathematics\\Cornell University\\Ithaca, N.Y. 14853\\
USA}

{\sc \normalsize \begin{center} Cornell University Mathematics Department Senior Thesis \end{center}}

\vspace{20mm}

\maketitle

\vspace{3mm}

{\Large \begin{center} May 2008 \end{center}}


\vspace{15mm}

{\sc \normalsize \begin{center} A thesis presented in partial fulfillment \\ of criteria for Honors in Mathematics \end{center}}

\vspace{30mm}





{\sc \Large \begin{center} Bachelor of Arts, Cornell University \end{center}}

\vspace{40mm}

{\sc \Large \begin{center} Thesis Advisor(s) \end{center}}

\vspace{3mm}

{\Large \begin{center} Gerhard O. Michler \\ Department of Mathematics \end{center}}

\newpage

\vspace{10mm}

{\bf \LARGE \begin{center} ABSTRACT \end{center}}

\vspace{10mm}

In the first section of this senior thesis the author provides some new
efficient algorithms for calculating with finite permutation
groups. They cannot be found in the computer algebra system \textsc{Magma},
but they can be implemented there.  For any finite group G with a
given set of generators, the algorithms calculate generators of a
fixed subgroup of G as short words in terms of original
generators. Another new algorithm provides such a short word for a
given element of G. These algorithms are very useful for
documentation and performing demanding experiments in
computational group theory.

In the later sections, the author gives a
self-contained existence proof for Fischer's sporadic simple group
$\Fi_{23}$ of order $2^{18}\cdot 3^{13}\cdot 5^2\cdot 7\cdot
11\cdot 13\cdot 17\cdot 23$ using G. Michler's Algorithm
\cite{michler1} constructing finite simple groups from irreducible
subgroups of $\GL_n(2)$. This sporadic group was originally
discovered by B. Fischer in \cite{fischer1} by investigating
$3$-transposition groups, see also \cite{fischer}. This thesis
gives a representation theoretic and algorithmic existence proof
for his group. 
The author constructs the three non-isomorphic
extenstions $E_i$ by the two $11$-dimensional non-isomorphic
simple modules of the Mathieu group $\M_{23}$ over $F=\GF(2)$. In two cases Michler's
Algorithm fails. In the third case the author constructs the centralizer
$H=C_G(z)$ of a $2$-central involution $z$ of $E_i$ in any target
simple group $G$. Then the author proves that all conditions of Michler's Algorithm are satisfied. This allows the author to construct $G$ inside
$\GL_{782}(17)$. Its four generating matrices are too large to be
printed in this thesis, but they can be downloaded from the author's
website \cite{kim1}. Furthermore, its character table and
representatives for conjugacy classes are computed. It follows that $G$ and
$\Fi_{23}$ have the same character table.

\newpage

\vspace{10mm}

{\bf \LARGE \begin{center} ACKNOWLEDGEMENTS \end{center}}

\vspace{10mm}

The author would like to thank professor Gerhard O. Michler for
all his mathematical support and encouragement, and also professor R. Keith Dennis
for his guidance in the Fall term of 2007.

The author also acknowledges financial support by the Hunter R.
Rawlings III Cornell Presidential Research Scholars Program for
participation in this research project. 

This work has also been
supported by the grant NSF/SCREMS DMS-0532/06.

\newpage

\tableofcontents

\newpage 

\setcounter{section}{-1}
\section{Introduction}\label{sec. intro}

A simple group $G$ is called sporadic if it is not isomorphic to
any alternating group $A_n$ or any finite group of Lie type, see
R. W. Carter \cite{carter}. Until recently there was no uniform construction
method for the known twenty six simple sporadic groups. In \cite{michler1}
a uniform construction method is given for constructing finite
simple groups from irreducible subgroups of $\GL_n(2)$. For
technical reasons it cannot construct the two largest known sporadic
simple groups because there is no computer which can hold all
their elements. But the other twenty four known sporadic simple groups can
be constructed by Michler's Algorithm. In this thesis it is
applied to provide a new self-contained existence proof for Fischer's sporadic
simple group $\Fi_{23}$.


In $1971$ B. Fischer \cite{fischer} found three sporadic groups by
characterizing all finite groups $G$ that can be generated by a
conjugacy class $D=z^G$ of $3$-transpositions, which means that
the product of two elements of $D$ has order $1,2$, or $3$. He
proved that besides the symmetric groups $S_n$, the symplectic
groups $\Sp_n(2)$, the projective unitary groups $\U_n(2)$ over
the field with $4$ elements and certain orthogonal groups, his two
sporadic simple groups $\Fi_{22}$ and $\Fi_{23}$, and the
automorphism group $\Fi_{24}$ of the simple group $\Fi'_{24}$
describe all $3$-transposition groups, see \cite{fischer1}. For
each $3$-transposiiton group $G=\langle D\rangle$ Fischer
constructs a graph $\mathcal{G}$ and an action on it. As its
vertices he takes the $3$-transpositions $x$ of $D$. Two distinct
elements $x, y \in D$ are called to be \textit{connected} and
joined by an edge $(x,y)$ in $\mathcal{G}$ if they commute in $G$.
He showed that each of the groups considered in his theorem has a
natural representation as an automorphism group of its graph
$\mathcal{G}$. Unfortunately, Fischer's proofs are only published
in his set of lecture notes of the University of Warwick
\cite{fischer1}. See also \cite{aschbacher}, for a coherent
account on Fischer's theorem.

In \cite{fischer1} Fischer gave the first existence proof for
$\Fi_{22}$, $\Fi_{23}$, and $\Fi_{24}$, by constructing the three graphs on which the groups act
as automorphisms. Eighteen years later, M. Aschbacher proves in
\cite{aschbacher} the existence of $\Fi_{24}$ and hence
also $\Fi_{22}$ and $\Fi_{23}$, using a quotient of the normalizer
of a cyclic subgroup of order $3$ in the Monster simple group $M$,
see \cite{aschbacher}, p. 5. However, both approaches do not
allow specific calculations with elements in these groups nor do
their methods generalize to arbitrary finite simple groups.

The purpose of this thesis is to provide a new existence proof for
$\Fi_{23}$, which has two advantages over the previous proofs.
Fischer's proof doesn't generalize to all simple groups, because
not all simple groups can be described by $3$-transposition
groups. Aschbacher's proof obviously doesn't generalize nor does
it provide access to explicit computation with elements in the
resulting groups, because the Monster group $M$ doesn't have
a permutation representation or a matrix representation of small
enough degree; the smallest known faithful permutation representation of $M$
wouldn't fit into any modern super-computer,
and the smallest degree of matrix representation of $M$ is about $183,000$. Dealing with dense
$183,000$ by $183,000$ matrices is practically impossible for currently existing computers.

The new proof uses representation theoretic and algorithmic
methods, mainly based on the Algorithm 2.5 of \cite{michler1}, which is also stated in section \ref{sec. Michler. Algorithm}.
The second part of Algorithm 2.5 of \cite{michler1} is not repeated in this thesis,
because it is identical to Algorithm 7.4.8 of \cite{michler}.
Using this algorithm, the author constructs
$\Fi_{23}$ as a subgroup of $\GL_{782}(17)$. From the $782$-dimensional
matrix representation, the author also constructs a
faithful permutation representation of $\Fi_{23}$ of degree
$31671$, by which we can actually compute with elements and
subgroups and check the performed calculations using the high
performance computer algebra system \textsc{Magma}. The character table and
representatives for conjugacy classes of $\Fi_{23}$ are also
computed by means of this permutation representation.

The other strong point of this algorithmic proof is that this
method generalizes to all simple groups (not having Sylow
$2$-subgroups which are cyclic, dihedral or semi-dihedral), see
\cite{michler}. In \cite{kim} the author and G. Michler construct
$\Fi_{22}$ and Conway's sporadic group $\Co_{2}$ simultaneously.
The author and G. Michler are also working on the simultaneous contructions
of $\Co_1$, Janko's sporadic group $J_4$, and $\Fi'_{24}$.
In \cite{michler} and recent work to which the author's joint article \cite{kim} with
Michler and also the further work belongs, G. Michler and coauthors construct all known sporadic
groups except the Baby Monster and the Monster. These methods
allow a systematic search for simple groups.


Here is the summary of each section. 
Section \ref{sec. Alg} contains the 
description and {\sc Magma} implementation of the author's
short-word algorithms.
In section \ref{sec. Michler. Algorithm}, the author states Michler's Algorithm 2.5 of \cite{michler1},
which is a main tool of the construction of $\Fi_{23}$ in this thesis.

Section \ref{sec. M23-extensions} shows the author's construction of
the extensions of the Mathieu group $\M_{23}$ by its two non-isomorphic irreducible representations $V_1$ and $V_2$ of
dimension $11$ over $\GF(2)$. In this thesis
it is shown that there is a uniquely determined non-split
extension $E$ of $\M_{23}$ by $V_1$.
Micher's Algorithm 2.5 in \cite{michler1} is applied to $E$ for the construction of $\Fi_{23}$ in the next sections.
However, its application to the split extensions $E_1$ and $E_2$ of $\M_{23}$ by $V_1$ and $V_2$
does not lead to any result.

Section \ref{sec. H(Fi_23)} contains the author's construction of the centralizer $H$ (unique up to isomorphism) of a $2$-central
involution $z_1$ of $E$ in any target simple group $G$. In particular, it is shown that $Z(H) = \langle z_1 \rangle$
and $H/Z(H) \cong \Fi_{22}$. 
Taking a $2$-central involution $z_1$ in $E$ and calculating $D = C_E(z_1)$ the author
finds a suitable normal subgroup $Q$ in $D$ which enables him to
construct a group $H$ with center $Z(H) = \langle z_1 \rangle$ of
order $2$ such that $H/Z(H) \cong \Fi_{22}$. It is shown that $E,
D=C_E(z_1), H$ satisfy all conditions of Algorithm 2.5 of
\cite{michler1}, stated in Algorithm \ref{alg. simple} in section \ref{sec. Michler. Algorithm}.

In order to construct $H \cong 2\Fi_{22}$, the author quotes a
result on Fischer's sporadic simple group $\Fi_{22}$ from his
joint article \cite{kim} with G. Michler. This was necessary
because the implementation of Holt's Algorithm into \textsc{Magma} was not
able to construct a central extension of $\Fi_{22}$ by a cyclic
group of order $2$. Thus, the author constructs another
amalgam $H_2 \leftarrow D_2 \rightarrow D$, where $D_2 = C_D(t)$
for some involution $t$ and $H_2 = 2H(\Fi_{22})$, see Propositions \ref{prop. H(Fi_23)_1} and \ref{prop. H(Fi_23)_2} and Theorem \ref{thm. H(Fi_23)}. The
free product $H_2 *_{D_2} D$ with amalgamated subgroup $D_2$ has a
$352$-dimensional faithful irreducible representation over
$\GF(17)$, whose corresponding matrix group $\mathfrak{H}$ is
proved to be isomorphic to the $2$-fold cover $2\Fi_{22}$ of
$\Fi_{22}$, see Theorem \ref{thm. H(Fi_23)_embed}.

In section \ref{sec. Fi_23} the author finally constructs the simple target group $\mathfrak{G}$ as a matrix group inside $\GL_{782}(17)$,
by applying Algorithm 7.4.8 of \cite{michler} to the amalgam $H\leftarrow D\rightarrow E$. In particular, it has been shown that $\mathfrak{G}$ has a same character
table as $\Fi_{23}$ as stated in \cite{atlas}.
The amalgam $E \leftarrow D \rightarrow H$ has a unique compatible
pair of degree $782$ over $\GF(17)$ which is not multiplicity-free
at the $D$-level, see Theorem \ref{thm. Fi_23}. 
So the author uses Thompson's Theorem
7.2.2 of \cite{michler} in the application of Step 5(c) of
Algorithm 7.4.8 of \cite{michler}. It is shown that the free
product $H*_D E$ with amalgamated subgroup $D$ has exactly one (unique up to isomorphism)
irreducible representation of degree $782$ over $\GF(17)$
satisfying the Sylow $2$-subgroup test, see Theorem \ref{thm. Fi_23}. The corresponding
matrix group $\mathfrak{G}$ in $\GL_{782}(17)$ is proved to be a
simple group of order $2^{18} \cdot 3^{13} \cdot 5^2 \cdot 7 \cdot
11 \cdot 13 \cdot 17 \cdot 23$ which has a $2$-central involution
$\mathfrak{z}$ such that $C_{\mathfrak{G}}(\mathfrak{z}) \cong H$,
see Theorem \ref{thm. Fi_23}. Furthermore, the author constructs a faithful permutation
representation of $\mathfrak{G}$, and then the character table of
$\mathfrak{G}$. It agrees with the one of $\Fi_{23}$ as stated in
the Atlas \cite{atlas}.

The following diagram summarizes the author's construction of $\Fi_{23}$ given in sections \ref{sec. H(Fi_23)} and \ref{sec. Fi_23}.
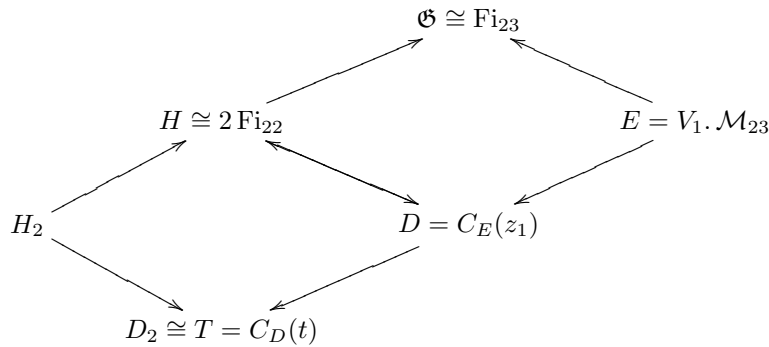
$$
\diagram
& & \mathfrak{G} \cong \Fi_{23} & \\
& H \cong 2\Fi_{22} \urto \drto & & E = V_1.\M_{23} \ulto \dlto \\
H_2 \urto \drto & & D = C_E(z_1) \ulto \dlto & \\
& D_2 \cong T = C_D(t)& &
\enddiagram
$$


All the performed demanding calculations for the given existence
proof were only possible because of the implementations of the author's new
algorithms described in the first section. They are very general and can
be used in the course of mathematical research in computational
group theory. The author's code works well in the computer algebra system
\textsc{Magma}. For documentation of the performed calculations it is often
crucial to get short-word generators for a certain subgroup of a
group G with a given set of generators:
$
G = \langle g_1, g_2, \ldots, g_n \rangle.
$
Let $S$ be a subgroup $G$. We want to get a generating set for $S$
as short words in terms of the original generators $g_1, g_2,
\ldots, g_n$.

For example,
$
S = \langle g_1 g_3 g_6 , \mbox{ } g_2, \mbox{ } g_4 g_5\rangle.
$
We want relatively small number of generators, and the lengths of words to be short.

\textsc{Magma}'s inverse word map function (often) provides a generating
set in terms of given generators, but unfortunately consisting of
terribly lengthy words. Reiner Staszewski, who was a former
research assistant of Professor Michler, developed a stand-alone algorithm for finding short-word generators. Paul K. Young, a
current graduate student of the Mathematics Department of Cornell
University, polished the idea and implemented the algorithm in
\textsc{Magma}. Since Young's implementation had some problems when dealing
with groups of large order (or permutation groups of large
degree), the author modified and added several new ideas. Thus the author
produced a relatively efficient implementation of the resulting
algorithm.

Besides this short-word-generator algorithm, another algorithm
for getting a short word for an element of a group with given
generators is also presented in the first section of this thesis. For
each algorithm, the description and its \textsc{Magma} implementation are
given.

These new algorithms are indeed crucial in the author's construction of
$\Fi_{23}$; they are not just for documentation.They provide a successful 
method for constructing the $2$-fold cover $2\Fi_{22}$ where Holt's Algorithm failed.



For readers who would like to find more background materials for
this thesis, they should see Holt's book \cite{holt} for
computational group theory, and Michler's book \cite{michler} for
algorithmic representation theory of finite simple groups.

\newpage 

\section{Algorithms}\label{sec. Alg}


\newcommand{\Wu}{\operatorname{W}\nolimits}
\newcommand{\stend}{\operatorname{st\_end}\nolimits}
\newcommand{\nG}{\operatorname{n_G}\nolimits}
\newcommand{\ai}{\operatorname{i}\nolimits}
\newcommand{\jai}{\operatorname{j}\nolimits}

\newcommand{\Wuone}{\operatorname{W\_1}\nolimits}
\newcommand{\stendone}{\operatorname{st\_end\_1}\nolimits}

\newcommand{\Gi}{\operatorname{G}\nolimits}
\newcommand{\wu}{\operatorname{w}\nolimits}
\newcommand{\elt}{\operatorname{elt}\nolimits}

\newcommand{\Tar}{\operatorname{T}\nolimits}
\newcommand{\Exc}{\operatorname{S_{exc}}\nolimits}

\newcommand{\gens}{\operatorname{gens}\nolimits}
\newcommand{\Spart}{\operatorname{S_{part}}\nolimits}
\newcommand{\sofar}{\operatorname{S_{sofar}}\nolimits}
\newcommand{\iter}{\operatorname{iter}\nolimits}
\newcommand{\iterlimit}{\operatorname{N}\nolimits}
\newcommand{\pow}{\operatorname{p}\nolimits}
\newcommand{\completed}{\operatorname{completed}\nolimits}
\newcommand{\sgens}{\operatorname{short-gens}\nolimits}

In this section, a finite group $G$ is always assumed to be realized as a permutation group or a matrix group. This allows us to compute the orders of elements and subgroups of $G$, and check equalities.

Before going into the actual algorithms, it is natural to have definitions of some vocabulary: `word-tree', `length of a word', `short-word', `lexicographic order of words', and so on. The following definition, as stated in Definition 5.3.9 of \cite{michler}, is due to M. Kratzer.

\begin{definition}[Word Tree, Word Length]\label{def. WordTree}
Let $G$ be a finite group. Let $\{g_1$, $g_2$, $\ldots$, $g_k\}$ be a fixed set of generators for $G$. The infinitely deep $k$-nary tree $\mathcal{C}(G)$ in which the root vertex is marked by the identity $1_G \in G$ and the $k$ successors of each vertex are marked successively by $g_1, g_2, \ldots, g_k$ is called the ``complete word tree of $G$":

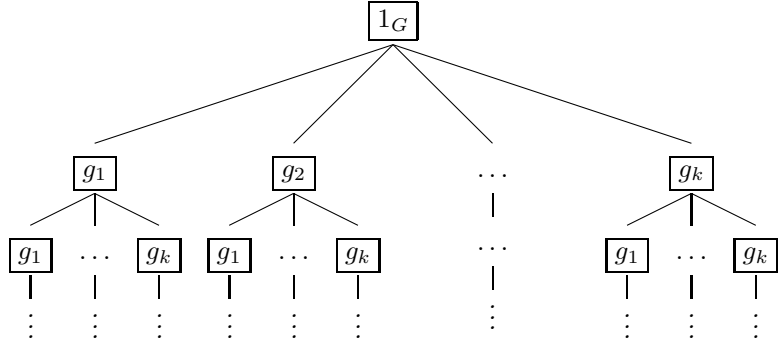
\Tree [.\fbox{$1_G$} [.\fbox{$g_1$} [.\fbox{$g_1$} $\vdots$ ] [.$\ldots$ $\vdots$ ] [.\fbox{$g_k$} $\vdots$ ] ] [.\fbox{$g_2$} [.\fbox{$g_1$} $\vdots$ ] [.$\ldots$ $\vdots$ ] [.\fbox{$g_k$} $\vdots$ ] ] [.$\ldots$ [.$\ldots$ $\vdots$ ] ] [.\fbox{$g_k$} [.\fbox{$g_1$} $\vdots$ ] [.$\ldots$ $\vdots$ ] [.\fbox{$g_k$} $\vdots$ ] ] ]

There is a canonical one-to-one correspondence between vertices in $\mathcal{C}(G)$ and words in generators of $G$: For any vertex $v$ of $\mathcal{C}(G)$ let $w_v$ denote the incremental product of the vertex markers occurring along the direct path from the root vertex to $v$ in $\mathcal{C}(G)$. Conversely, starting from the root vertex and reading a given word $w=w(g_1,g_2,\ldots,g_k)$ in generators of $G$ like a sequence of directions guides one to the unique vertex $v_w$ in $\mathcal{C}(G)$ such that $w_{v_w} = w$.

The ``length" of a word $w$ is the depth of the unique vertex $v_w$ corresponding to $w$ in the word tree $\mathcal{C}(G)$, i.e. the number of steps needed to reach the vertex $v_w$ from root vertex. For example, the word $g_1^2 g_3 g_2$ has length $4$.
\end{definition}

Now, a ``short-word" refers to a word of short length, though notion of ``short" might not be consistent. To simplify arguments, we may look only at the indices of the generators. That is, $g_1^2 g_3 g_2$ can be identified with $[1,1,3,2]$. Precise definition of this idea is as follows:

\begin{definition}[Numerical Word]\label{def. NumericalWord}
A ``numerical word of $k$ generators" is a finite sequence (can be empty sequence) of integers in $\{1,2,\ldots, k\}$. For example, $[5,1,2,2,2,3,4,4,5]$ is an example of a numerical word of $5$ generators, as well as a numerical word of $6$ generators, but not a numerical word of $4$ generators.
\end{definition}

Let $\{g_1, g_2, \ldots, g_k\}$ be a fixed set of generators of a finite group $G$. Then, if we identify each $g_i$ with $i$, there is a natural one-to-one correspondence between the set of all numerical word of $k$ generators and $\mathcal{C}(G)$. For example, the word $g_1g_4^3g_2g_3^2$ corresponds to the numerical word $[1,4,4,4,2,3,3]$.

Now we can define a total ordering on $\mathcal{C}(G)$, i.e. on the set of all finite words in the given (ordered) generators $g_1, g_2, \ldots, g_k$ of $G$.

\begin{definition}[Lexicographical Order]\label{def. LexicographicalOrder}
Let $\{g_1, g_2, \ldots, g_k\}$ be a fixed set of generators of a finite group $G$. Let $v_u$ and $v_w$ be the vertices of $\mathcal{C}(G)$, corresponding to distinct words $u$ and $w$ in $g_1, g_2, \ldots, g_n$, respectively. Let $N_u$ and $N_w$ be the numerical words of $k$ generators corresponding to $v_u$ and $v_w$, respectively (each $g_i$ is identified with $i$). Then, $u<w$ if and only if one of the followings hold (this is called ``lexicographic order"):
\begin{enumerate}
\item length of $u$ $<$ length of $w$

\item length of $u$ $=$ length of $w$, and $N_u[n] < N_w[n]$ holds, where $n$ is the smallest number such that $N_u[n] \neq N_w[n]$.
\end{enumerate}
\end{definition}

\begin{remark}
If $u$ and $w$ are distinct words of $G$ (in given generators), then exactly one of $u<w$ or $w<u$ holds.
\end{remark}

The main algorithm is named $\verb"GetShortGens"$; for any finite group $G$ with given set of generators, for any subgroup $S$ of $G$, this algorithm returns a short-word generating set for $S$ in terms of the given generators. It needs three other small programs, named $\verb"EnumWords"$, $\verb"ReduceGensForGroup"$ and $\verb"Word2Elt"$. All four programs are described below, with description of the algorithms and their implementation in \textsc{Magma}. The first algorithm $\verb"EnumWords"$ enables us to descend one level deeper in the word tree:

\begin{algorithm}[EnumWords]\label{alg. EnumWords}
Let $k$ be the number of given generators for the finite group of interest. Let $W$ be a sequence of pre-built numerical words in $k$ generators. Let $s,e$ be integers s.t. $1\le s \le e \le \# W$ ($s$ for ``start", $e$ for ``end"). Then, we can get a new sequence $W_1$ of numerical words by adding new words to $W$, where the new words are obtained by appending $1,2,\ldots,k$ at the end of the words $W[s], W[s+1], \ldots, W[e]$.
\end{algorithm}

\begin{implementation}[EnumWords]\label{imp. EnumWords}

{\em \footnotesize\begin{flushleft}
\quad \\
function EnumWords(W, $\stend$, nGens)\\
\quad \\
    local NewW, New$\stend$;\\
\quad \\
    NewW := W;\\
    if \#NewW eq 0 then\\
\quad
        NewW := [[i] : i in [1..nGens]];\\
\quad
        New$\stend$ := [1, nGens];\\
    else\\
\quad
        for i:=$\stend$[1] to $\stend$[2] do\\
\quad\quad
            for j:=1 to nGens do\\
\quad\quad\quad
                Append($\sim$NewW, Append(W[i],j));\\
\quad\quad
            end for;\\
\quad
        end for;\\
\quad
        New$\stend$ := [$\stend$[2]+1, \#NewW];\\
    end if;\\
    return NewW, New$\stend$;\\
\quad \\
end function;\\


\end{flushleft}}

\end{implementation}

\begin{example}[EnumWords]\label{ex. EnumWords}
\textsc{Magma} code \\
$
\verb"> W:=[ [1,3,4], [3,3,2], [2,1], [4] ];"
$ \\
$
\verb"> s:=2; e:=3;"
$ \\
$
\verb"> W:=EnumWords(W, [s,e], 4);"
$ \\
$
\verb"> W;"
$ \\
$
\verb"[ [1,3,4], [3,3,2], [2,1], [4], [3,3,2,1], [3,3,2,2], [3,3,2,3],"
$ \\
$
\verb" [3,3,2,4], [2,1,1], [2,1,2], [2,1,3], [2,1,4] ] "
$
\end{example}

The next algorithm $\verb"ReduceGensForGroup"$ gives a method to reduce the given set of generators as much as possible, while our subgroup of interest still should lie inside the subgroup generated by the reduced set of generators:

\begin{algorithm}[ReduceGensForGroup]\label{alg. ReduceGensForGroup}
Let $G$ be a finite group generated by the given generators $g_1, g_2, \ldots, g_k$. For any subgroup $S$ of $G$, a subset of $\{g_1, g_2, \ldots, g_k\}$ which generates a subgroup of $G$ containing $S$ is obtained by the following method:


[Reducing original generators] If there is some $i \in \{1,2,\ldots,k\}$ such that the set $\{g_1,\ldots, g_{i-1},g_{i+1},\ldots,g_k\}$ generates a subgroup of $G$ containing $S$, pick the largest such $i$. Now, call this function $\verb"ReduceGensForGroup"$ recursively, with same $G$ and $S$, with reduced set of generators $g_1,\ldots, g_{i-1},g_{i+1},\ldots,g_k$ (with corresponding names). If there is no such $i$, then return the original set $g_1, g_2, \ldots, g_k$, with their names.

\end{algorithm}

\begin{implementation}[ReduceGensForGroup]\label{imp. ReduceGensForGroup}

{\em \footnotesize\begin{flushleft}
\quad \\
function ReduceGensForGroup(G, Target : wordgens:=[], gencollection:=[1..\#Generators(G)], exclude:=sub$\langle$G$\mid$$\rangle$,CoverGroup:=G)\\
\quad \\
local reducedlist;\\
\quad \\
        if \#wordgens eq 0 or \#Generators(G) ne \#wordgens then\\
\quad
        wordgens:=[(``\$." cat Sprint(i)) :i in [1..\#Generators(G)]];\\
        end if;\\
    for i:=1 to \#Generators(G) do\\
\quad
        reducedlist := Exclude([1..\#Generators(G)],\#Generators(G)+1-i);\\
\quad
        if (Target meet sub$\langle$CoverGroup$\mid$[G.j : j in reducedlist],exclude$\rangle$) eq Target then\\
\quad\quad
            return ReduceGensForGroup(sub$\langle$G$\mid$[G.j : j in reducedlist]$\rangle$, Target : wordgens:=[wordgens[j] : j in reducedlist], gencollection:=[gencollection[j]: j in reducedlist],exclude:=exclude,CoverGroup:=CoverGroup);\\
\quad
        end if;\\
    end for;\\
    return G, wordgens, gencollection;\\
\quad \\
end function;\\


\end{flushleft}}

\end{implementation}

The following algorithm $\verb"Word2Elt"$ converts a numerical word to the corresponding actual element of the group:


\vspace{3mm}

\begin{algorithm}[Word2Elt]\label{alg. Word2Elt}
Let $G$ be a finite group generated by the given generators $g_1, g_2, \ldots, g_k$. For any numerical word $w$ in $k$ generators, return the element in $G$ corresponding to the word $w$ by the following steps:

{\bf Step 1} Let $e=1$, the identity element of $G$.

{\bf Step 2} If $w$ is an empty word, return $e$. If not, and if $w$ can be written as $w=[a_1, a_2, \ldots, a_m]$ where $a_i \in \{1,2,\ldots,k\}$, then let $e:=e\cdot g_{a_1}$.

{\bf Step 3} Let $w:=[a_2,a_3,\ldots,a_m]$, and go to {\bf Step 2}.

\end{algorithm}

\begin{implementation}[Word2Elt]\label{imp. Word2Elt}

{\em \footnotesize\begin{flushleft}

\quad \\
function Word2Elt(G, word)\\
\quad \\
    local elt;\\
\quad \\
    elt := Id(G);\\
    for i:=1 to \#word do\\
\quad
        elt := elt * G.word[i];\\
    end for;\\
    return elt;\\
\quad \\
end function;\\


\end{flushleft}}

\end{implementation}

\begin{example}[Word2Elt]\label{ex. Word2Elt}
\textsc{Magma} code (Note: \textsc{Magma} composes two permutations from left, not from right) \\
$
\verb"> G:=sub<Sym(3)| Sym(3)!(1,2), Sym(3)!(1,2,3)>;"
$ \\
$
\verb"> Word2Elt(G, [1,2,2]);"
$ \\
$
\verb"(2,3)"
$
\end{example}

Finally, the next algorithm $\verb"GetShortGens"$ enables us to obtain a short-word generating set for a subgroup of a group with given generators:

\begin{algorithm}[GetShortGens]\label{alg. GetShortGens}
Let $G$ be a finite group generated by the given generators $g_1, g_2, \ldots, g_k$. For any subgroup $S$ of $G$, a generating set of $S$ consisting of short-word elements of $G$ in terms of $g_1, g_2, \ldots, g_k$ is obtained by the following steps:


{\bf Step 1} [Reducing original generators] If desired, try to get a subset of $\{g_1$, $g_2$, $\ldots$, $g_k\}$ which generates a subgroup of $G$ containing $S$, using the command\\ $\verb"ReduceGensForGroup"$. For convenience, suppose that the set of generators $g_1$, $g_2$, $\ldots$, $g_k$ is already a result of this reducing process.

{\bf Step 2} [Building word list, and finding generators]

\begin{enumerate}
\item Set $F = \langle 1 \rangle$ (trivial subgroup).

\item Let $W = [ [1], [2], \ldots, [k] ]$, the set of words, initially set to have only simplest numerical words of $k$ generators of length $1$.

\item\label{getshortgens_check} Take the word $w$ in $W$ of lowest lexicographic order which is not checked yet, and let $x = \verb"Word2Elt(G,w)"$, the element of $G$ corresponding to the word $w$. If there is some $m \in \{1,2,\ldots,Order(w)-1\}$ such that $w^m \in S$ and $w^m \notin F$, then enlarge $F$ by $F:= \langle F, w^m\rangle$.

\item If $F = S$, then proceed to {\bf Step 3}. If not, enlarge $W$ by appending $1,2,\ldots, k$ to all words of $W$ of longest length, using the command $\verb"EnumWords"$ (this is same as descending one level deeper in the word tree). Now, go to (\ref{getshortgens_check}).

\end{enumerate}

{\bf Step 3} [Printing] Print the words obtained.

\end{algorithm}


\begin{implementation}[GetShortGens]\label{imp. GetShortGens}

{\em \footnotesize\begin{flushleft}

\quad \\
function GetShortGens(G, Target :
exclude := sub$\langle$G$\mid\rangle$,
limit:=0,
wordgens:=\texttt{[]},
Words:=\texttt{[]},
startpoint:=1,
powers:=:=\texttt{[}1\texttt{]},
Hard:=true,
OrderRestriction:=\texttt{[]},
CoverGroup:=G,
generatecheck:=true,
auto:=true,
EltReturn:=false,
ReduceMore:=true)\\
\quad \\
local gens, iter, $\stend$, tempelt, tempord, temppow, WordsSoFar, WordsNumSoFar, gencollection, SubgroupSoFar, generatingset;\\
\quad \\
gens := Generators(G);\\
\quad \\
function GetGeneratingSet(WordsForGenerators)\\
\quad
    return [Word2Elt(G,WordsForGenerators[i][1])\^{ }WordsForGenerators[i][2] :\\
\quad
    i in [1..\#WordsForGenerators]];\\
end function;\\
\quad \\
if generatecheck and Target meet sub$\langle$CoverGroup$\mid$G,exclude$\rangle$ ne Target then\\
\quad
    print ``can't generate subgroup";\\
\quad
    return ``";\\
end if;\\
\quad \\
if \#wordgens eq 0 or \#Generators(G) ne \#wordgens then\\
\quad
    if \#wordgens ne 0 and \#Generators(G) ne \#wordgens then\\
\quad\quad
        print ``the number of generator names you provided is incompatible\\
\quad\quad
        with the number of generators of the group, so re-building names";\\
\quad
    end if;\\
\quad
    wordgens:=[(``\$." cat Sprint(i)) :i in [1..\#Generators(G)]];\\
end if;\\
\quad \\
if Hard then\\
\quad
    print "Reducing Generators...";\\
\quad
    time G, wordgens, gencollection := ReduceGensForGroup(G, Target : wordgens:=wordgens,exclude:=exclude,CoverGroup:=CoverGroup);\\
\quad
    printf ``Using only \%o generators \%o, out of \%o ", \#Generators(G),wordgens,\#gens;\\
\quad
    gens := Generators(G);\\
else\\
\quad
    gencollection := [1..\#Generators(G)];\\
end if;\\
\quad \\
Include($\sim$powers,1);\\
Sort($\sim$powers);\\
\quad \\
WordsSoFar:=[];\\
WordsNumSoFar:=[];\\
SubgroupSoFar := Target meet exclude;\\
iter:=0;\\
\quad \\
while limit eq 0 or iter lt limit do\\
\quad
    if \#gens eq 1 and iter ge Order(G.1) then\\
\quad\quad
        print ``use other method";\\
\quad\quad
        return ``",[];\\
\quad
end if;\\
\quad \\
\quad
    iter:=iter+1;\\
\quad
    printf ``\%o-th iteration$\backslash$n",iter;
\quad \\
\quad
    if \#Words eq 0 then\\
\quad\quad
        Words, $\stend$ := EnumWords([],[1,1],\#gens);\\
\quad\quad
        startpoint:=1;\\
\quad
    end if;\\
\quad \\
\quad
    for j:=startpoint to \#Words do\\
\quad\quad
        if \#OrderRestriction eq 0 then\\
\quad\quad\quad
            tempelt := Word2Elt(G, Words[j]);\\
\quad\quad\quad
            if auto then powers:=[1..(Order(tempelt)-1)]; end if;\\
\quad\quad\quad
            if exists(temppow)\{x : x in powers$\mid$ tempelt\^{ }x notin SubgroupSoFar\\
\quad\quad\quad
            and tempelt\^{ }x in Target\} then\\
\quad\quad\quad\quad
                Append($\sim$WordsSoFar, WordPrint(Words[j], wordgens:power:=temppow));\\
\quad\quad\quad\quad
                Append($\sim$WordsNumSoFar, $\langle$[gencollection[Words[j][k]]:k in [1..\#Words[j]]], temppow$\rangle$);\\
\quad\quad\quad\quad
                printf ``Got a new elt: \%o$\backslash$n", WordsSoFar[\#WordsSoFar];\\
\quad\quad\quad\quad
                SubgroupSoFar := sub$\langle$Target$\mid$SubgroupSoFar, Word2Elt(G,Words[j])\^{ }temppow$\rangle$;\\
\quad\quad\quad\quad
                if SubgroupSoFar eq Target then\\
\quad\quad\quad\quad\quad
                    printf ``Got a generating set.$\backslash$n";\\
\quad\quad\quad\quad\quad
                    if ReduceMore then\\
\quad\quad\quad\quad\quad\quad
                        print ``Getting a smaller generating set...";\\
\quad\quad\quad\quad\quad\quad
                        time G, wordgens, gencollection := ReduceGensForGroup(sub$\langle$G$\mid$GetGeneratingSet(WordsNumSoFar)$\rangle$, Target : wordgens:=[],exclude:=exclude,CoverGroup:=CoverGroup);\\
\quad\quad\quad\quad\quad\quad
                        printf ``Resulting set has \%o generators, out of \%o$\backslash$n", \#wordgens,\#WordsSoFar;\\
\quad\quad\quad\quad\quad\quad
                        printf ``subcollection indices:\%o$\backslash$n",gencollection;\\
\quad\quad\quad\quad\quad\quad
                        gens := Generators(G);\\
\quad\quad\quad\quad\quad
                    end if;\\
\quad \\
\quad\quad\quad\quad\quad
                    WordsSoFar := [WordsSoFar[gencollection[i]] : i in [1..\#gencollection]];\\
\quad\quad\quad\quad\quad
                    WordsNumSoFar := [WordsNumSoFar[gencollection[i]] : i in [1..\#gencollection]];\\
\quad\quad\quad\quad\quad
                    if not EltReturn then\\
\quad\quad\quad\quad\quad\quad
                        return WordsSoFar, WordsNumSoFar;\\
\quad\quad\quad\quad\quad
                    else\\
\quad\quad\quad\quad\quad\quad
                        return WordsSoFar, WordsNumSoFar, GetGeneratingSet(WordsNumSoFar);\\
\quad\quad\quad\quad\quad
                    end if;\\
\quad\quad\quad\quad
                end if;\\
\quad\quad\quad
            end if;\\
\quad\quad
        else \quad\quad // if there is some restriction on orders\\
\quad\quad\quad
            tempelt := Word2Elt(G, Words[j]);\\
\quad\quad\quad
            tempord := Order(tempelt);\\
\quad\quad\quad
            for ord in OrderRestriction do\\
\quad\quad\quad\quad
                if tempord mod ord eq 0 then\\
\quad\quad\quad\quad\quad
                    temppow := Integers()!(tempord/ord);\\
\quad\quad\quad\quad\quad
                    if tempelt\^{ }temppow notin SubgroupSoFar and tempelt\^{ }temppow in Target then\\
\quad\quad\quad\quad\quad\quad
                        Append($\sim$WordsSoFar, WordPrint(Words[j],wordgens:power:=temppow));\\
\quad\quad\quad\quad\quad\quad
                        Append($\sim$WordsNumSoFar, $\langle$[gencollection[Words[j][k]]:k in [1..\#Words[j]]],temppow$\rangle$);\\
\quad\quad\quad\quad\quad\quad
                        printf ``Got a new elt of order \%o: \%o$\backslash$n",ord,WordsSoFar[\#WordsSoFar];\\
\quad\quad\quad\quad\quad\quad
                        SubgroupSoFar := sub$\langle$Target$\mid$SubgroupSoFar, Word2Elt(G,Words[j])\^{ }temppow$\rangle$;\\
\quad\quad\quad\quad\quad\quad
                        if SubgroupSoFar eq Target then\\
\quad\quad\quad\quad\quad\quad\quad
                            printf ``Done.$\backslash$n";\\
\quad\quad\quad\quad\quad\quad\quad
                            if not EltReturn then\\
\quad\quad\quad\quad\quad\quad\quad\quad
                                return WordsSoFar, WordsNumSoFar;\\
\quad\quad\quad\quad\quad\quad\quad
                            else\\
\quad\quad\quad\quad\quad\quad\quad\quad
                                return WordsSoFar, WordsNumSoFar, Generators(SubgroupSoFar);\\
\quad\quad\quad\quad\quad\quad\quad
                            end if;\\
\quad\quad\quad\quad\quad\quad
                        end if;\\
\quad\quad\quad\quad\quad
                    end if;\\
\quad\quad\quad\quad
                end if;\\
\quad\quad\quad
            end for;\\
\quad\quad
        end if;\\
\quad
    end for;\\
\quad
    Words, $\stend$ := EnumWords(Words, $\stend$, \#gens);\\
\quad
    startpoint := $\stend$[1];\\
end while;\\
\quad \\
print ``Couldn't generatate group";\\
return WordsSoFar;\\
end function;\\

\end{flushleft}}

\end{implementation}

\begin{example}[GetShortGens]\label{ex. GetShortGens}
\textsc{Magma} code \\
$
\verb"> g1:=Sym(8)!(1,2); g2:=Sym(8)!(1,2,3,4,5,6,7,8);"
$ \\
$
\verb"> G:=sub<Sym(8)|g1,g2>;"
$ \\
$
\verb"> S := sub<Sym(8)| Sym(8)!(1,3,6)(2,4), Sym(8)!(1,7,8)(2,5)>;"
$ \\
$
\verb"> res := GetShortGens(G,S : wordgens:=[``g1'',``g2'']);"
$ \\
$
\verb"> res;"
$ \\
$
\verb"[ (g2*g1*g2^4)^5, (g1*g2*g1*g2^4)^3, (g1*g2^3*g1*g2*g1)^2, "
$ \\
$
\verb"(g1*g2*g1*g2*g1*g2^3*g1)^3 ]"
$
\end{example}

For a given element of a finite group with a given generating set, it is often important to obtain a short word for the element in terms of given generating set. This element-version (as opposed to subgroup-version: $\verb"GetShortGens"$) of short-word program is named $\verb"LookupWord"$, and it needs a different version of $\verb"ReduceGensForGroup"$ which is called $\verb"ReduceGensForElt"$. It is designed for finding a word of an element; it reduces the given set of generators as much as possible, while the element of interest still should lie inside the subgroup generated by the reduced set of generators:

\begin{algorithm}[ReduceGensForElt]\label{alg. ReduceGensForElt}
Let $G$ be a finite group generated by the given generators $g_1, g_2, \ldots, g_k$. For any element $x$ of $G$, a subset of $\{g_1, g_2, \ldots, g_k\}$ which generates a subgroup of $G$ containing $S$ is obtained by the following method:


[Reducing original generators] If there is some $i \in \{1,2,\ldots,k\}$ such that the set $\{g_1,\ldots, g_{i-1},g_{i+1},\ldots,g_k\}$ generates a subgroup of $G$ containing $S$, pick the largest such $i$. Now, call this function $\verb"ReduceGensForElt"$ recursively, with same $G$ and $S$, with reduced set of generators $g_1,\ldots, g_{i-1},g_{i+1},\ldots,g_k$ (with corresponding names). If there is no such $i$, then return the original set $g_1, g_2, \ldots, g_k$, with their names.

\end{algorithm}

\begin{implementation}[ReduceGensForElt]\label{imp. ReduceGensForElt}

{\em \footnotesize\begin{flushleft}

\quad \\
function ReduceGensForElt(G, TargetElt : wordgens:=[], gencollection:=[1..\#Generators(G)])\\
\quad \\
local reducedlist;\\
\quad \\
for i:=1 to \#Generators(G) do\\
\quad
    reducedlist := Exclude([1..\#Generators(G)],\#Generators(G)+1-i);\\
\quad
    if TargetElt in sub$\langle$G$\mid$[G.j : j in reducedlist]$\rangle$ then\\
\quad\quad
        return ReduceGensForElt(sub$\langle$G$\mid$[G.j : j in reducedlist]$\rangle$, TargetElt : wordgens:=reducedlist], gencollection:=[gencollection[j]:j in reducedlist]);\\
\quad
    end if;\\
end for;\\
return G, wordgens, gencollection;\\
end function;\\

\end{flushleft}}

\end{implementation}

The algorithm $\verb"LookupWord"$ enables us to obtain a short word for an element of a group with given generators:

\begin{algorithm}[LookupWord]\label{alg. LookupWord}
Let $G$ be a finite group generated by the given generators $g_1, g_2, \ldots, g_k$. For any element $x$ of $G$, a short word for $x$ in terms of $g_1, g_2, \ldots, g_k$ is obtained by the following steps:


{\bf Step 1} [Reducing original generators] If desired, try to get a subset of $\{g_1$, $g_2$, $\ldots$, $g_k\}$ which generates a subgroup of $G$ containing $x$, using the command $\verb"ReduceGensForElt"$. For convenience, suppose that the set of generators $g_1$, $g_2$, $\ldots$, $g_k$ is already a result of this reducing process.

{\bf Step 2} [Building word list, and finding generators]

\begin{enumerate}
\item Let $W = [ [1], [2], \ldots, [k] ]$, the set of words, initially set to have only simplest numerical words of $k$ generators of length $1$.

\item\label{lookupword_check} Take the word $w$ in $W$ of lowest lexicographic order which is not checked yet, and let $y = \verb"Word2Elt(G,w)"$, the element of $G$ corresponding to the word $w$.

\item Let $r=Order(y)$ and $s=Order(x)$. If $s\nmid r$, go to \eqref{lookupword_more}.

\item If there is some $m \in \{t \in \mathbb{Z} | 1\le t < s, gcd(t,Order(x)) =1 \}$ such that $w^{rm/s} = x$, then proceed to {\bf Step 3}.

\item\label{lookupword_more} Enlarge $W$ by appending $1,2,\ldots, k$ to all words of $W$ of longest length, using the command $\verb"EnumWords"$ (this is same as descending one level deeper in the word tree). Now, go to (\ref{lookupword_check}).

\end{enumerate}

{\bf Step 3} [Printing] Print the words obtained.

{\bf Alternative option :} The above algorithm finds a short word which is equal to the given element. A similar algorithm can be used to find a short word which is conjugate to the given element: in {\bf Step 2}$(4)$, we look for $m$ such that $w^{rm/s}$ is conjugate to $x$ (the equality test is replaced by the conjugacy test). This option is incorporated in the following implementation, as the hidden parameter $\verb"ConjugateCheck"$; if we set $\verb"ConjugateCheck:=true"$, then the following program finds a short word which is conjugate to the given element.

\end{algorithm}

\begin{implementation}[LookupWord]\label{imp. LookupWord}

{\em \footnotesize\begin{flushleft}

\quad \\
function LookupWord(G, TargetElt : limit:=0, wordgens:=[], Words:=[], $\stend$:=[], startpoint:=1, Hard:=true, ConjugateCheck:=false, CoverGroup:=G, containcheck:=true, InfoLevel:=2)\\
\quad \\
    local Ord, gens, iter;\\
    local tempelt, tempord, temppow;\\
\quad \\
    Ord := Order(TargetElt);\\
    gens := Generators(G);\\
\quad \\
    if containcheck and TargetElt notin G then\\
\quad
        print ``Coudln't find a word for it";\\
\quad
        return ``";\\
    end if;\\
\quad \\
    if TargetElt eq Id(G) then\\
\quad
        if InfoLevel gt 1 then print ``it is the identity!"; end if;\\
\quad
        return ``Id(\$)", $\langle$[],1$\rangle$;\\
    end if;\\
\quad \\
    if \#wordgens eq 0 or \#gens ne \#wordgens then\\
\quad
        for i:=1 to \#gens do\\
\quad\quad
            Append($\sim$wordgens, ``\$." cat IntegerToString(i));\\
\quad
        end for;\\
    end if;\\
\quad \\

        if Hard and not ConjugateCheck then\\
\quad
                print ``Reducing Generators...";\\
\quad
                time G, wordgens, gencollection := ReduceGensForElt(G, TargetElt : wordgens:=wordgens);\\
\quad
                printf ``Using only \%o generators \%o, out of \%o ", \#Generators(G),wordgens,\#gens;\\
\quad
                gens := Generators(G);\\
        else\\
\quad
                gencollection := [1..\#Generators(G)];\\
\quad
                if \#wordgens eq 0 or \#Generators(G) ne \#wordgens then\\
\quad\quad
                        wordgens:=[(``\$." cat Sprint(i)) :i in [1..\#Generators(G)]];\\
\quad
                end if;\\
        end if;\\
\quad \\
    iter:=0;\\

     while limit eq 0 or iter lt limit do\\
\quad
        iter:=iter+1;\\
\quad
        if InfoLevel gt 1 then\\
\quad\quad
            printf ``\%o-th iteration$\backslash$n",iter;\\
\quad
        end if;\\

\quad
        if \#Words eq 0 or \#$\stend$ eq 0 then\\
\quad\quad
            Words, $\stend$ := EnumWords([],[1,1],\#gens);\\
\quad\quad
            startpoint:=1;\\
\quad
        end if;\\
\quad \\
\quad
        for j:=startpoint to \#Words do\\
\quad\quad
            tempelt := Word2Elt(G, Words[j]);\\
\quad\quad
            tempord := Order(tempelt);\\
\quad\quad
            if tempord mod Ord eq 0 then\\
\quad\quad\quad
                temppow := Integers()!(tempord/Ord);\\
\quad\quad\quad
                  for addi in [a:a in [1..Ord]$\mid$GCD(a,Ord) eq 1] do\\
\quad\quad\quad\quad
                if tempelt\^{ }(temppow*addi) eq TargetElt then\\
\quad\quad\quad\quad\quad
                    if InfoLevel gt 1 then print ``Got a word (exact)"; end if;\\
\quad\quad\quad\quad\quad
                    return WordPrint(Words[j], wordgens:power:=(temppow*addi)), $\langle$[gencollection[Words[j][k]]:k in [1..\#Words[j]]], temppow*addi$\rangle$;\\
\quad\quad\quad\quad
                          elif ConjugateCheck and IsConjugate(CoverGroup,tempelt\^{ }(temppow*addi),TargetElt) then\\
\quad\quad\quad\quad\quad
                                            if InfoLevel gt 1 then print ``Got a conjugate word (new ver)"; end if;\\
\quad\quad\quad\quad\quad
                                         return WordPrint(Words[j], wordgens:power:=(temppow*addi)), $\langle$[Words[j][k]:k in [1..\#Words[j]]], temppow*addi$\rangle$;\\
\quad\quad\quad\quad
                                end if;\\
\quad\quad\quad
                  end for;\\
\quad\quad
            end if;\\
\quad
        end for;\\
\quad \\
\quad
        Words, $\stend$ := EnumWords(Words, $\stend$, \#gens);\\
\quad
        startpoint := $\stend$[1];\\
    end while;\\
\quad \\
    print ``Couldn't find a word";\\
    return ``";\\
\quad \\
end function;\\

\end{flushleft}}

\end{implementation}

\begin{example}[LookupWord]\label{ex. LookupWord}
\textsc{Magma} code \\
$
\verb"> g1:=Sym(8)!(1,2); g2:=Sym(8)!(1,2,3,4,5,6,7,8);"
$ \\
$
\verb"> G:=sub<Sym(8)|g1,g2>;"
$ \\
$
\verb"> x := Sym(8)!(2, 8, 7, 6, 4, 3);"
$ \\
$
\verb"> res := LookupWord(G,x : wordgens:=[``g1'',``g2'']);"
$ \\
$
\verb"> res;"
$ \\
$
\verb"g1*g2^4*g1*g2^3"
$ \\
\end{example}

Sometimes, the subgroup or the element that we want to get short words of lies too deep inside the group (whose generators are given), so $\verb"GetShortGens"$ or $\verb"LookupWord"$ doesn't work well. For example, if an element $x$ can't be represented by a word of length$\le 24$ in terms of given generators, then $\verb"LookupWord"$ either takes too long a time to get the word for $x$, or causes memory overflow due to too much required space for all the words built up so far. The author found a strategy to overcome this issue, described as follows. 

The strategy needs a proper subgroup $T$ of $G$, which contains the target subgroup $S$ or the target element $x$. A standard trick to get such $T$ is to take the normalizer of $S$ in $G$, the centralizer of $x$ in $G$, or the normalizer of $\langle x \rangle$ of $G$.


\begin{strategy}[two-step GetShortGens]\label{str. GetShortGens}
Let $G$ be a finite group generated by the given generators $g_1, g_2, \ldots, g_k$. For any subgroups $S$ and $T$ of $G$ such that $S\lneqq T \lneqq G$, a short-word generating set of $S$ in terms of $g_1, g_2, \ldots, g_k$ is obtained by the following steps:

{\bf Step 1} [Generators for $T$] Get short-word generators $t_1, t_2, \ldots, t_n$ for $T$ in terms of $g_1, g_2, \ldots, g_k$ using $\verb"GetShortGens"$.

{\bf Step 2} [Generators for $S$] Get short-word generators for $S$ in terms of $t_1$, $t_2$, $\ldots$, $t_n$ using $\verb"GetShortGens"$.
\end{strategy}

\begin{strategy}[two-step LookupWord]\label{str. LookupWord}
Let $G$ be a finite group generated by the given generators $g_1, g_2, \ldots, g_k$. For any element $x\in G$ and a subgroup $T$ of $G$ such that $x\in T \lneqq G$, a short word for $x$ in terms of $g_1, g_2, \ldots, g_k$ is obtained by the following steps:

{\bf Step 1} [Generators for $T$] Get short-word generators $t_1, t_2, \ldots, t_n$ for $T$ in terms of $g_1, g_2, \ldots, g_k$ using $\verb"GetShortGens"$.

{\bf Step 2} [Short Words for $x$] Get short word for $x$ in terms of $t_1, t_2, \ldots, t_n$ using $\verb"LookupWord"$.
\end{strategy}

The above two strategies are used throughout this thesis and therefore constantly referred to. Note that sometimes it can be helpful to iterate the strategies many times. For example, for a subgroup $S$ of $G$, it can be a good idea to try to find subgroups $T_1$ and $T_2$ such that $S\lneqq T_1 \lneqq T_2 \lneqq G$ and then find short-word generators for $T_2$, $T_1$, and finally $S$ in turn.

\newpage 

\section{Michler's Algorithm}\label{sec. Michler. Algorithm}

In this section, G. Michler's Algorithm 2.5 of \cite{michler1} is presented. This algorithm
gives a uniform method to construct all finite simple groups (not having Sylow
$2$-subgroups which are cyclic, dihedral or semi-dihedral), from indecomposable subgroups of $\GL_n(2)$.
The author's algorithms described in the previous section are used to
implement this algorithm of Michler. In particular, in this thesis, $\Fi_{23}$ is constructed by the following algorithm.

\begin{algorithm}\label{alg. simple} Let $T$ be an indecomposable subgroup of
$\GL_n(2)$ acting on $V = F^n$ by matrix multiplication.

\begin{itemize}
\item \underline{\em Step 1:} Calculate a faithful permutation
representation $PT$ of $T$ and a finite presentation $T = \langle
t_i\,|\, 1 \le i \le r\rangle$ with set $\mathcal R{(T)}$ of defining
relations.

\item \underline{\em Step 2:} Compute all extension groups $E$ of
$T$ by $V$ by means of Holt's Algorithm \cite{holt}. Determine a
complete set $\mathfrak S$ of non isomorphic extension groups $E$
by means of the Cannon-Holt Algorithm \cite{cannon}.

\item \underline{\em Step 3:} Let $E \in \mathfrak S$. From the
given presentation of $E$ determine a faithful permutation
representation $PE$ of $E$. Using it and Kratzer's Algorithm
5.3.18 of \cite{michler} calculate a complete system of
representatives of all the conjugacy classes of $E$.

\item \underline{\em Step 4:} Let $z \neq 1$ be a $2$-central
involution of $E$. Calculate $D = C_E(z)$ and fix a Sylow
$2$-subgroup $S$ of $D$. Check that the elementary abelian normal
subgroup $V$ of $E$ is a maximal elementary abelian normal
subgroup of $S$.

\noindent If it is not maximal, then the algorithm terminates.

\item \underline{\em Step 5:} Construct a group $H > D$ with the
following properties:
\begin{enumerate}
\item[\rm(a)] $z$ belongs to the center $Z(H)$ of $H$.

\item[\rm(b)] The index $|H : D|$ is odd.

\item[\rm(c)] The normalizer $N_H(V) = D = C_E(z)$.
\end{enumerate}
If no such $H$ exists the algorithm terminates.

\smallskip
Otherwise, apply for each constructed group $H$ the following
steps of Algorithm 7.4.8 of \cite{michler}. By step $5 (c)$ it may
be assumed from now on that $D = H \cap E$.

\end{itemize}
\end{algorithm}

After the \underline{Step 5}, the remaining steps are identical to Algorithm 7.4.8 of \cite{michler}, so omitted here.

\newpage

\section{Extensions of Mathieu group $\M_{23}$}\label{sec. M23-extensions}

The Mathieu group $M_{23}$ is defined in Definition 8.2.1 of
\cite{michler} by means of generators and relations. This
beautiful presentation is due to J.A. Todd. The irreducible
$2$-modular representations of the Mathieu group $M_{23}$ were
determined by G. James \cite{james}. Here only the $2$ non
isomorphic simple modules $V_i$, $i = 1,2$, of dimension $11$ over
$F = \GF(2)$ will be considered. Todd's permutation representations
of the Mathieu groups are stated in Lemma 8.2.2 of \cite{michler}.
Therefore all conditions of Holt's Algorithm \cite{holt}
implemented in \textsc{Magma} are satisfied. It is applied here. Thus it is
shown in this section that for the simple module $V_1$ there are exactly two
extensions of $M_{23}$ by $V_1$, the split extension $E_1$ and the
non-split extension $E$ and that $\M_{23}$ has only the split
extension $E_2$ by $V_2$. It will be shown that the
applications of Algorithm \ref{alg. simple} to $E_1$ and $E_2$
fail. Therefore only the constructed presentation of $E$ is given
in Lemma \ref{l. M23-extensions}.

\begin{lemma}\label{l. M23-extensions}
Let $\M_{23} = \langle a,b,c,d,t,g,h,i,j \rangle$ be the finitely
presented group with set of defining relations $\mathcal
R(\M_{23})$ given in Definition 8.2.1 of \cite{michler}. Then the
following statements hold:

\begin{enumerate}
\item[\rm(a)] A faithful permutation representation of degree $23$
of $\M_{23}$ is stated in Lemma 8.2.2 of \cite{michler}.

\item[\rm(b)] The first irreducible representation $(\rho_1, V_1)$ of
$\M_{23}$ is described by the following matrices:
{\renewcommand{\arraystretch}{0.5}

\scriptsize
$$
\rho_1(a) = \left( \begin{array}{*{11}{c@{\,}}c}
0 & 1 & 0 & 0 & 1 & 1 & 1 & 1 & 1 & 0 & 0\\
0 & 1 & 1 & 0 & 0 & 0 & 1 & 1 & 1 & 1 & 1\\
0 & 0 & 0 & 0 & 0 & 0 & 1 & 0 & 0 & 0 & 0\\
0 & 0 & 1 & 1 & 0 & 1 & 1 & 1 & 0 & 1 & 0\\
1 & 1 & 0 & 0 & 0 & 1 & 1 & 0 & 0 & 1 & 1\\
0 & 0 & 0 & 0 & 0 & 1 & 0 & 0 & 0 & 0 & 0\\
0 & 0 & 1 & 0 & 0 & 0 & 0 & 0 & 0 & 0 & 0\\
0 & 0 & 0 & 0 & 0 & 0 & 0 & 1 & 0 & 0 & 0\\
0 & 0 & 0 & 0 & 0 & 0 & 0 & 0 & 1 & 0 & 0\\
0 & 0 & 0 & 0 & 0 & 0 & 0 & 0 & 0 & 1 & 0\\
0 & 0 & 0 & 0 & 0 & 0 & 0 & 0 & 0 & 0 & 1
\end{array} \right),\quad
\rho_1(b) = \left( \begin{array}{*{11}{c@{\,}}c}
0 & 1 & 0 & 0 & 0 & 0 & 0 & 0 & 0 & 0 & 0\\
1 & 0 & 0 & 0 & 0 & 0 & 0 & 0 & 0 & 0 & 0\\
0 & 0 & 0 & 0 & 1 & 0 & 0 & 0 & 0 & 0 & 0\\
1 & 1 & 0 & 1 & 0 & 1 & 0 & 1 & 1 & 1 & 0\\
0 & 0 & 1 & 0 & 0 & 0 & 0 & 0 & 0 & 0 & 0\\
0 & 0 & 0 & 0 & 0 & 1 & 0 & 0 & 0 & 0 & 0\\
1 & 1 & 0 & 0 & 0 & 1 & 1 & 0 & 0 & 1 & 1\\
0 & 0 & 0 & 0 & 0 & 0 & 0 & 1 & 0 & 0 & 0\\
0 & 0 & 0 & 0 & 0 & 0 & 0 & 0 & 1 & 0 & 0\\
0 & 0 & 0 & 0 & 0 & 0 & 0 & 0 & 0 & 1 & 0\\
0 & 0 & 0 & 0 & 0 & 0 & 0 & 0 & 0 & 0 & 1
\end{array} \right),
$$
}

{\renewcommand{\arraystretch}{0.5}

\scriptsize
$$
\rho_1(c) = \left( \begin{array}{*{11}{c@{\,}}c}
0 & 0 & 0 & 0 & 1 & 0 & 0 & 0 & 0 & 0 & 0\\
0 & 0 & 1 & 0 & 0 & 0 & 0 & 0 & 0 & 0 & 0\\
0 & 1 & 0 & 0 & 0 & 0 & 0 & 0 & 0 & 0 & 0\\
1 & 0 & 0 & 1 & 1 & 1 & 0 & 1 & 0 & 0 & 1\\
1 & 0 & 0 & 0 & 0 & 0 & 0 & 0 & 0 & 0 & 0\\
0 & 0 & 0 & 0 & 0 & 1 & 0 & 0 & 0 & 0 & 0\\
0 & 1 & 1 & 0 & 0 & 0 & 1 & 1 & 1 & 1 & 1\\
0 & 0 & 0 & 0 & 0 & 0 & 0 & 1 & 0 & 0 & 0\\
0 & 0 & 0 & 0 & 0 & 0 & 0 & 0 & 1 & 0 & 0\\
0 & 0 & 0 & 0 & 0 & 0 & 0 & 0 & 0 & 1 & 0\\
0 & 0 & 0 & 0 & 0 & 0 & 0 & 0 & 0 & 0 & 1
\end{array} \right),\quad
\rho_1(d) = \left( \begin{array}{*{11}{c@{\,}}c}
0 & 1 & 1 & 1 & 1 & 0 & 1 & 0 & 0 & 0 & 1\\
1 & 0 & 1 & 1 & 1 & 0 & 1 & 1 & 1 & 0 & 0\\
0 & 0 & 1 & 1 & 0 & 1 & 1 & 1 & 0 & 1 & 0\\
0 & 0 & 0 & 0 & 0 & 0 & 1 & 0 & 0 & 0 & 0\\
0 & 0 & 0 & 1 & 1 & 1 & 1 & 0 & 1 & 1 & 1\\
0 & 0 & 0 & 0 & 0 & 1 & 0 & 0 & 0 & 0 & 0\\
0 & 0 & 0 & 1 & 0 & 0 & 0 & 0 & 0 & 0 & 0\\
0 & 0 & 0 & 0 & 0 & 0 & 0 & 1 & 0 & 0 & 0\\
0 & 0 & 0 & 0 & 0 & 0 & 0 & 0 & 1 & 0 & 0\\
0 & 0 & 0 & 0 & 0 & 0 & 0 & 0 & 0 & 1 & 0\\
0 & 0 & 0 & 0 & 0 & 0 & 0 & 0 & 0 & 0 & 1
\end{array} \right),
$$
}

{\renewcommand{\arraystretch}{0.5}

\scriptsize
$$
\rho_1(t) = \left( \begin{array}{*{11}{c@{\,}}c}
0 & 1 & 0 & 0 & 0 & 0 & 0 & 0 & 1 & 0 & 0\\
1 & 0 & 0 & 1 & 1 & 1 & 0 & 1 & 1 & 0 & 1\\
0 & 1 & 0 & 0 & 1 & 1 & 1 & 1 & 0 & 0 & 0\\
1 & 0 & 1 & 1 & 1 & 0 & 1 & 1 & 0 & 0 & 0\\
0 & 1 & 1 & 1 & 1 & 0 & 1 & 0 & 1 & 0 & 1\\
0 & 0 & 0 & 0 & 0 & 0 & 0 & 0 & 1 & 0 & 0\\
1 & 1 & 0 & 1 & 0 & 1 & 0 & 1 & 0 & 1 & 0\\
0 & 0 & 0 & 0 & 0 & 0 & 0 & 1 & 1 & 0 & 0\\
0 & 0 & 0 & 0 & 0 & 1 & 0 & 0 & 1 & 0 & 0\\
0 & 0 & 0 & 0 & 0 & 0 & 0 & 0 & 1 & 1 & 0\\
0 & 0 & 0 & 0 & 0 & 0 & 0 & 0 & 1 & 0 & 1
\end{array} \right),\quad
\rho_1(g) = \left( \begin{array}{*{11}{c@{\,}}c}
0 & 1 & 0 & 1 & 0 & 0 & 0 & 0 & 0 & 0 & 0\\
1 & 0 & 0 & 1 & 0 & 0 & 0 & 0 & 0 & 0 & 0\\
0 & 1 & 0 & 1 & 1 & 1 & 1 & 1 & 1 & 0 & 0\\
0 & 0 & 0 & 1 & 0 & 0 & 0 & 0 & 0 & 0 & 0\\
1 & 1 & 0 & 1 & 0 & 1 & 1 & 0 & 0 & 1 & 1\\
1 & 0 & 1 & 0 & 1 & 0 & 1 & 1 & 1 & 0 & 0\\
0 & 1 & 1 & 1 & 0 & 0 & 1 & 1 & 1 & 1 & 1\\
1 & 0 & 1 & 0 & 0 & 0 & 0 & 0 & 1 & 1 & 1\\
0 & 0 & 0 & 0 & 1 & 1 & 1 & 0 & 1 & 1 & 1\\
0 & 0 & 0 & 1 & 0 & 0 & 0 & 0 & 0 & 1 & 0\\
0 & 0 & 0 & 1 & 0 & 0 & 0 & 0 & 0 & 0 & 1
\end{array} \right),
$$
}

{\renewcommand{\arraystretch}{0.5}

\scriptsize
$$
\rho_1(h) = \left( \begin{array}{*{11}{c@{\,}}c}
1 & 0 & 0 & 0 & 0 & 1 & 0 & 0 & 0 & 0 & 0\\
1 & 0 & 0 & 1 & 1 & 0 & 0 & 1 & 0 & 0 & 1\\
0 & 0 & 1 & 1 & 0 & 0 & 1 & 1 & 0 & 1 & 0\\
0 & 0 & 0 & 0 & 0 & 1 & 1 & 0 & 0 & 0 & 0\\
1 & 1 & 0 & 0 & 0 & 0 & 1 & 0 & 0 & 1 & 1\\
0 & 0 & 0 & 0 & 0 & 1 & 0 & 0 & 0 & 0 & 0\\
0 & 0 & 0 & 1 & 0 & 1 & 0 & 0 & 0 & 0 & 0\\
0 & 0 & 0 & 0 & 0 & 1 & 0 & 0 & 0 & 1 & 0\\
0 & 0 & 0 & 0 & 0 & 1 & 0 & 0 & 1 & 0 & 0\\
0 & 0 & 0 & 0 & 0 & 1 & 0 & 1 & 0 & 0 & 0\\
0 & 0 & 0 & 0 & 0 & 1 & 0 & 0 & 0 & 0 & 1
\end{array} \right),\quad
\rho_1(i) = \left( \begin{array}{*{11}{c@{\,}}c}
1 & 0 & 0 & 1 & 1 & 0 & 0 & 1 & 0 & 0 & 1\\
0 & 1 & 0 & 0 & 0 & 1 & 0 & 0 & 0 & 0 & 0\\
1 & 1 & 0 & 1 & 0 & 0 & 0 & 1 & 1 & 1 & 0\\
0 & 0 & 0 & 1 & 0 & 1 & 0 & 0 & 0 & 0 & 0\\
0 & 0 & 0 & 0 & 1 & 1 & 0 & 0 & 0 & 0 & 0\\
0 & 0 & 0 & 0 & 0 & 1 & 0 & 0 & 0 & 0 & 0\\
0 & 1 & 0 & 0 & 1 & 0 & 1 & 1 & 1 & 0 & 0\\
0 & 0 & 0 & 0 & 0 & 1 & 0 & 1 & 0 & 0 & 0\\
0 & 0 & 0 & 0 & 0 & 1 & 0 & 0 & 1 & 0 & 0\\
1 & 1 & 1 & 0 & 1 & 0 & 0 & 0 & 1 & 0 & 1\\
0 & 0 & 0 & 0 & 0 & 1 & 0 & 0 & 0 & 0 & 1
\end{array} \right),
$$
} {\renewcommand{\arraystretch}{0.5}

\scriptsize
$$
\rho_1(j) = \left( \begin{array}{*{11}{c@{\,}}c}
0 & 1 & 0 & 0 & 0 & 0 & 0 & 0 & 0 & 0 & 0\\
1 & 0 & 0 & 0 & 0 & 0 & 0 & 0 & 0 & 0 & 0\\
0 & 0 & 0 & 0 & 1 & 0 & 0 & 0 & 0 & 0 & 0\\
0 & 0 & 0 & 1 & 0 & 0 & 0 & 0 & 0 & 0 & 0\\
0 & 0 & 1 & 0 & 0 & 0 & 0 & 0 & 0 & 0 & 0\\
0 & 0 & 0 & 0 & 0 & 0 & 0 & 0 & 1 & 0 & 0\\
0 & 1 & 1 & 0 & 0 & 0 & 1 & 1 & 1 & 1 & 1\\
0 & 0 & 0 & 0 & 0 & 0 & 0 & 1 & 0 & 0 & 0\\
0 & 0 & 0 & 0 & 0 & 1 & 0 & 0 & 0 & 0 & 0\\
0 & 0 & 0 & 0 & 0 & 0 & 0 & 0 & 0 & 1 & 0\\
1 & 1 & 1 & 0 & 1 & 1 & 0 & 0 & 1 & 0 & 1
\end{array} \right).
$$
}

\item[\rm(c)] The second irreducible representation $(\rho_2, V_2)$ of
$\M_{23}$ is described by the transpose inverse matrices of the
generating matrices of $\M_{23}$ defining \\ $(\rho_1, V_1)$:
\begin{align*}
\rho_2(a) = [\rho_1(a)^{-1}]^{T}, \mbox{ }
\rho_2(b) = [\rho_1(b)^{-1}]^{T},\mbox{ }
\rho_2(c) = [\rho_1(c)^{-1}]^{T}, \\
\rho_2(d) = [\rho_1(d)^{-1}]^{T}, \mbox{ }
\rho_2(t) = [\rho_1(t)^{-1}]^{T},\mbox{ }
\rho_2(g) = [\rho_1(g)^{-1}]^{T}, \\
\rho_2(h) = [\rho_1(h)^{-1}]^{T}, \mbox{ }
\rho_2(i) = [\rho_1(i)^{-1}]^{T}, \mbox{ and }
\rho_2(j) = [\rho_1(j)^{-1}]^{T}.
\end{align*}

\item[\rm(d)] $dim_F[H^2(\M_{23},V_1)] = 1$ and
$dim_F[H^2(\M_{23},V_2)] = 0$.

\item[\rm(e)] The presentations of the split extensions $E_1$ and
$E_2$ of $\M_{23}$ by $V_1$ and $V_2$, respectively, can be
constructed from the matrices of (b) and (c) using \textsc{Magma}.

\item[\rm(f)] The unique non-split extension $E$ of $\M_{23}$ by
$V_1$ has the presentation

$$E = \langle a_1,b_1,c_1,d_1,t_1,g_1,h_1,i_1,j_1,v_i \mid 1 \le i \le 11 \rangle$$

with set ${\mathcal R}(E)$ of defining relations consisting of the
following set of relations:
\begin{eqnarray*}
&&a_1^2 = b_1^2 = c_1^2 = d_1^2 = t_1^3 = g_1^4 = h_1^4 = j_1^2 = 1,\\
&&(b_1, a_1) = (c_1, a_1) = (d_1, a_1) = (h_1, a_1) = (c_1, b_1) = (d_1, b_1) = (d_1, c_1) = 1,\\
&&(j_1, b_1) = (j_1, c_1) = (j_1, g_1) = 1,\\
&&v_r^2 = 1 \quad \mbox{for} \quad 1 \le r \le 11,\\
&&(v_r,v_s) = 1 \quad \mbox{for} \quad 1 \le r, s \le 11,\\
&&a_1^{-1}v_1a_1v_2^{-1}v_5^{-1}v_6^{-1}v_7^{-1}v_8^{-1}v_9^{-1} =
a_1^{-1}v_2a_1v_2^{-1}v_3^{-1}v_7^{-1}v_8^{-1}v_9^{-1}v_{10}^{-1}v_{11}^{-1} = 1,\\
&&a_1^{-1}v_3a_1v_7^{-1} =  a_1^{-1}v_7a_1v_3^{-1} =
a_1^{-1}v_4a_1v_3^{-1}v_4^{-1}v_6^{-1}v_7^{-1}v_8^{-1}v_{10}^{-1} = 1,\\
&& a_1^{-1}v_5a_1v_1^{-1}v_2^{-1}v_6^{-1}v_7^{-1}v_{10}^{-1}v_{11}^{-1} = 1,\\
&&(a_1,  v_6) = (a_1,  v_8) = (a_1,  v_9) = (a_1,   v_{10}) = (a_1,   v_{11}) = 1,\\
&&b_1^{-1}v_1b_1v_2^{-1} = b_1^{-1}v_2b_1v_1^{-1} = b_1^{-1}v_3b_1v_5^{-1} = b_1^{-1}v_5b_1v_3^{-1} = 1,\\
&&b_1^{-1}v_4b_1v_1^{-1}v_2^{-1}v_4^{-1}v_6^{-1}v_8^{-1}v_9^{-1}v_{10}^{-1} = 1,\\
&&(b_1,  v_6^{-1}) = (b_1,  v_8^{-1}) = (b_1,  v_9^{-1}) = (b_1,   v_{10}^{-1}) = (b_1,   v_{11}^{-1}) = 1,\\
&&b_1^{-1}v_7b_1v_1^{-1}v_2^{-1}v_6^{-1}v_7^{-1}v_{10}^{-1}v_{11}^{-1} =
c_1^{-1}v_1c_1 v_5^{-1} = c_1^{-1}v_2c_1v_3^{-1} = 1,\\
&& c_1^{-1}v_3c_1v_2^{-1} = c_1^{-1}v_5c_1v_1^{-1} =
c_1^{-1}v_4c_1v_1^{-1}v_4^{-1}v_5^{-1}v_6^{-1}v_8^{-1}v_{11}^{-1} = 1,\\
&&(c_1,  v_6^{-1}) = (c_1,  v_8^{-1}) = (c_1,  v_9^{-1}) = (c_1,   v_{10}^{-1}) = (c_1,   v_{11}^{-1}) = 1,\\
&&c_1^{-1}v_7c_1v_2^{-1}v_3^{-1}v_7^{-1}v_8^{-1}v_9^{-1}v_{10}^{-1}v_{11}^{-1} =
d_1^{-1}v_1d_1v_2^{-1}v_3^{-1}v_4^{-1}v_5^{-1}v_7^{-1}v_{11}^{-1} = 1,\\
&&d_1^{-1}v_2d_1v_1^{-1}v_3^{-1}v_4^{-1}v_5^{-1}v_7^{-1}v_8^{-1}v_9^{-1} =
d_1^{-1}v_3d_1v_3^{-1}v_4^{-1}v_6^{-1}v_7^{-1}v_8^{-1}v_{10}^{-1} = 1,\\
&&d_1^{-1}v_4d_1v_7^{-1} =  d_1^{-1}v_7d_1v_4^{-1} =
d_1^{-1}v_5d_1v_4^{-1}v_5^{-1}v_6^{-1}v_7^{-1}v_9^{-1}v_{10}^{-1}v_{11}^{-1} = 1,\\
&&(d_1,  v_6^{-1}) = (d_1,  v_8^{-1}) = (d_1,  v_9^{-1}) = (d_1,   v_{10}^{-1}) = (d_1,   v_{11}^{-1}) = 1,\\
&&t_1^{-1}v_1t_1v_2^{-1}v_9^{-1} = t_1^{-1}v_6t_1v_9^{-1} =
t_1^{-1}v_2t_1v_1^{-1}v_4^{-1}v_5^{-1}v_6^{-1}v_8^{-1}v_9^{-1}v_{11}^{-1} = 1,\\
&&t_1^{-1}v_3t_1v_2^{-1}v_5^{-1}v_6^{-1}v_7^{-1}v_8^{-1} =
t_1^{-1}v_4t_1v_1^{-1}v_3^{-1}v_4^{-1}v_5^{-1}v_7^{-1}v_8^{-1} = 1,\\
&&t_1^{-1}v_5t_1v_2^{-1}v_3^{-1}v_4^{-1}v_5^{-1}v_7^{-1}v_9^{-1}v_{11}^{-1} =
t_1^{-1}v_7t_1v_1^{-1}v_2^{-1}v_4^{-1}v_6^{-1}v_8^{-1}v_{10}^{-1} = 1,\\
&&t_1^{-1}v_8t_1v_8^{-1}v_9^{-1} = t_1^{-1}v_9t_1v_6^{-1}v_9^{-1} =
t_1^{-1}v_{10}t_1v_9^{-1}v_{10}^{-1} = t_1^{-1}v_{11}t_1v_9^{-1}v_{11}^{-1} = 1,\\
&&g_1^{-1}v_1g_1v_2^{-1}v_4^{-1} = g_1^{-1}v_2g_1v_1^{-1}v_4^{-1} =
g_1^{-1}v_3g_1v_2^{-1}v_4^{-1}v_5^{-1}v_6^{-1}v_7^{-1}v_8^{-1}v_9^{-1} = 1,\\
\end{eqnarray*}
\begin{eqnarray*}
&& (g_1,  v_4) =
g_1^{-1}v_5g_1v_1^{-1}v_2^{-1}v_4^{-1}v_6^{-1}v_7^{-1}v_{10}^{-1}v_{11}^{-1} = 1,\\
&&g_1^{-1}v_6g_1v_1^{-1}v_3^{-1}v_5^{-1}v_7^{-1}v_8^{-1}v_9^{-1} =
g_1^{-1}v_7g_1v_2^{-1}v_3^{-1}v_4^{-1}v_7^{-1}v_8^{-1}v_9^{-1}v_{10}^{-1}v_{11}^{-1} = 1,\\
&&g_1^{-1}v_8g_1v_1^{-1}v_3^{-1}v_9^{-1}v_{10}^{-1}v_{11}^{-1} =
g_1^{-1}v_9g_1v_5^{-1}v_6^{-1}v_7^{-1}v_9^{-1}v_{10}^{-1}v_{11}^{-1} = 1,\\
&&g_1^{-1}v_{10}g_1v_4^{-1}v_{10}^{-1} = g_1^{-1}v_{11}g_1v_4^{-1}v_{11}^{-1} =
h_1^{-1}v_1h_1v_1^{-1}v_6^{-1} = 1,\\
&& h_1^{-1}v_4h_1v_6^{-1}v_7^{-1} = (h_1,  v_6) =
h_1^{-1}v_2h_1v_1^{-1}v_4^{-1}v_5^{-1}v_8^{-1}v_{11}^{-1} = 1,\\
&&h_1^{-1}v_3h_1v_3^{-1}v_4^{-1}v_7^{-1}v_8^{-1}v_{10}^{-1} =
h_1^{-1}v_5h_1v_1^{-1}v_2^{-1}v_7^{-1}v_{10}^{-1}v_{11}^{-1} = 1,\\
&&h_1^{-1}v_7h_1v_4^{-1}v_6^{-1} = h_1^{-1}v_8h_1v_6^{-1}v_{10}^{-1} =
h_1^{-1}v_9h_1v_6^{-1}v_9^{-1} = 1,\\
&& h_1^{-1}v_{10}h_1v_6^{-1}v_8^{-1} = h_1^{-1}v_{11}h_1v_6^{-1}v_{11}^{-1} =
i_1^{-1}v_1i_1v_1^{-1}v_4^{-1}v_5^{-1}v_8^{-1}v_{11}^{-1} = 1,\\
&&i_1^{-1}v_2i_1v_2^{-1}v_6^{-1} = i_1^{-1}v_4i_1v_4^{-1}v_6^{-1} = i_1^{-1}v_5i_1v_5^{-1}v_6^{-1} = 1,\\
&&i_1^{-1}v_3i_1v_1^{-1}v_2^{-1}v_4^{-1}v_8^{-1}v_9^{-1}v_{10}^{-1} = 1,\quad (i_1,  v_6) = 1,\\
&&i_1^{-1}v_7i_1v_2^{-1}v_5^{-1}v_7^{-1}v_8^{-1}v_9^{-1} = i_1^{-1}v_{11}i_1v_6^{-1}v_{11}^{-1} =
i_1^{-1}v_8i_1v_6^{-1}v_8^{-1} = 1,\\
&& i_1^{-1}v_9i_1v_6^{-1}v_9^{-1} =
i_1^{-1}v_{10}i_1v_1^{-1}v_2^{-1}v_3^{-1}v_5^{-1}v_9^{-1}v_{11}^{-1} = 1,\\
&&j_1^{-1}v_1j_1v_2^{-1} = j_1^{-1}v_2j_1v_1^{-1} = j_1^{-1}v_3j_1v_5^{-1} =
(j_1,  v_4) = (j_1,  v_8) = 1,\\
&& (j_1,   v_{10}) =
j_1^{-1}v_5j_1v_3^{-1} = j_1^{-1}v_6j_1v_9^{-1} = j_1^{-1}v_9j_1v_6^{-1} = 1,\\
&&j_1^{-1}v_7j_1v_2^{-1}v_3^{-1}v_7^{-1}v_8^{-1}v_9^{-1}v_{10}^{-1}v_{11}^{-1} =
j_1^{-1}v_{11}j_1v_1^{-1}v_2^{-1}v_3^{-1}v_5^{-1}v_6^{-1}v_9^{-1}v_{11}^{-1} = 1,\\
&&t_1^{-1}a_1t_1d_1^{-1}c_1^{-1}v_6^{-1}v_9^{-1} = t_1^{-1}b_1t_1d_1^{-1}a_1^{-1}v_6^{-1}v_9^{-1} = 1,\\
&&t_1^{-1}c_1t_1d_1^{-1}b_1^{-1}v_6^{-1}v_9^{-1} =
t_1^{-1}d_1t_1c_1^{-1}b_1^{-1}a_1^{-1}v_1^{-1}v_2^{-1}v_3^{-1}v_5^{-1}v_9^{-1} = 1,\\
&&g_1^2v_1^{-1}v_3^{-1}v_8^{-1}v_9^{-1}v_{10}^{-1}v_{11}^{-1} =
(g_1a_1)^3v_{10}^{-1}v_{11}^{-1} = 1,\\
&&(g_1b_1)^3v_1^{-1}v_2^{-1}v_3^{-1}v_5^{-1}v_6^{-1}v_9^{-1} =
(g_1c_1)^3 = 1,\quad (i_1j_1)^3 = 1\\
&&(g_1t_1)^2v_1^{-1}v_3^{-1}v_8^{-1}v_{10}^{-1}v_{11}^{-1} = 1,\\
&&h_1^2v_6^{-1} = 1,\quad h_1^{-1}b_1h_1d_1^{-1}b_1^{-1}a_1^{-1}v_1^{-1}v_2^{-1}v_3^{-1}v_5^{-1} = 1,\\
&&h_1^{-1}c_1h_1c_1^{-1}a_1^{-1}v_1^{-1}v_2^{-1}v_3^{-1}v_5^{-1}v_6^{-1}v_9^{-1} =
h_1^{-1}d_1h_1d_1^{-1}v_6^{-1} = 1,\\
&& h_1^{-1}t_1h_1t_1v_9^{-1} =
(g_1h_1)^3v_3^{-1}v_7^{-1}v_9^{-1}v_{10}^{-1}v_{11}^{-1} = 1,\\
&&i_1^{-1}a_1i_1d_1^{-1}c_1^{-1}v_1^{-1}v_2^{-1}v_3^{-1}v_5^{-1}v_6^{-1}v_9^{-1}v_{10}^{-1}v_{11}^{-1} = 1,\\
&&i_1^{-1}b_1i_1d_1^{-1}a_1^{-1}v_1^{-1}v_2^{-1}v_3^{-1}v_5^{-1}v_6^{-1}v_9^{-1}v_{10}^{-1}v_{11}^{-1} = 1,\\
&&i_1^{-1}c_1i_1d_1^{-1}c_1^{-1}b_1^{-1}a_1^{-1}v_1^{-1}v_2^{-1}v_3^{-1}v_5^{-1}v_6^{-1}v_9^{-1}v_{10}^{-1}v_{11}^{-1} = 1,\\
&&i_1^{-1}d_1i_1d_1^{-1}c_1^{-1}b_1^{-1} =  i_1^{-1}t_1i_1t_1 =
i_1^{-1}g_1i_1g_1^{-1}t_1^{-1}v_5^{-1}v_6^{-1}v_7^{-1}v_9^{-1}v_{10}^{-1}v_{11}^{-1} = 1,\\
&&(h_1i_1)^3v_1^{-1}v_2^{-1}v_3^{-1}v_5^{-1}v_6^{-1}v_8^{-1}v_{10}^{-1} =
j_1^{-1}a_1j_1c_1^{-1}b_1^{-1}a_1^{-1} = 1,\\
&&j_1^{-1}d_1j_1d_1^{-1}c_1^{-1}v_1^{-1}v_2^{-1}v_3^{-1}v_5^{-1}v_6^{-1} v_9^{-1} =
j_1^{-1}t_1j_1t_1 = 1,\quad j_1^{-1}h_1j_1h_1^{-1}t_1^{-1} = 1.\\
\end{eqnarray*}
\end{enumerate}
\end{lemma}

\begin{proof}
The $2$ irreducible $F\M_{23}$-modules $V_i$, $i =1,2$, occur as
composition factors with multiplicity $1$ in the permutation
module $(1_{\M_{22}})^{\M_{23}}$ of degree $23$ where $\M_{22} =
\langle a_1,b_1,c_1,d_1,t_1,g_1,h_1,i_1 \rangle$. They are dual to each other.
Using the faithful permutation representation of $\M_{23}$ stated
in (a) and the Meat-axe Algorithm implemented in \textsc{Magma} one obtains
the generating matrices of $\M_{23}$ stated in (b) defining $V_1$.
Their dual matrices define $V_2$. They are stated in (c).

(d) The cohomological dimensions $dim_F[H^2(\M_{23},V_i)]$,
$i =1,2$, have been calculated by means of \textsc{Magma} using Holt's
Algorithm 7.4.5 of \cite{michler}. Its hypothesis is satisfied by
the presentation of $\M_{23}$ stated in Definition 8.2.1 of
\cite{michler} and all the data of (a), (b) and (c). It follows
that $dim_F[H^2(\M_{23},V_1)] = 1$ and $dim_F[H^2(\M_{23},V_2)] = 0$.

(e) This statement is checked easily with \textsc{Magma}.

(g) The presentation of the non-split extension $E$ has been
obtained by means of the commands $\verb"ExtensionProcess"$ and $\verb"Extension"$ of
Holt's Algorithm 7.4.5 of \cite{michler} implemented in \textsc{Magma}
\cite{holt}. This completes the proof.

\end{proof}

The first statement of the following subsidiary lemma is mainly
due to Paul Young.

\begin{lemma}\label{l. E(Fi_23)classes} With the notations of Lemma \ref{l.
M23-extensions} the following statements hold:

\begin{enumerate}
\item[\rm(a)] The non-split extension $E$ of $\M_{23}$ by its
simple module $V_1$ has a faithful permutation representation $PE$ of
degree $1012$ with stabilizer $U$ generated by the two elements $(i_1^{-1}j_1^{-1}b_1g_1)^2$ and
$(i_1^{-1}b_1t_1h_1^{-1})^4$.

\item[\rm(b)] $E = \langle a_1,b_1,c_1,d_1,t_1,g_1,h_1,i_1,j_1 \rangle$.

\item[\rm(c)] $E$ has $3$ conjugacy classes of $2$-central
involutions. They are represented by $z_1 = (d_1g_1)^5$, $z_2 = h_1^2$ and
$z_3 = g_1^2$. Their centralizers have orders $|C_{E}(z_1)| =
2^{18}\cdot3^{2}\cdot5\cdot7\cdot11$, $|C_{E}(z_2)| =
2^{18}\cdot3^{2}\cdot5\cdot7$ and $|C_{E}(z_3)| =
2^{18}\cdot3^{2}\cdot5$.

\item[\rm(d)] $D = C_{E}(z_1) = \langle x_1, y_1 \rangle$ and $E =
\langle D, e_1\rangle$ where $x_1 = a_1$, $y_1 = b_1g_1h_1i_1$ and $e_1 = j_1$
have respective orders $2$, $14$ and $2$.

\item[\rm(e)] $E = \langle x_1, y_1, e_1 \rangle$ has $56$ conjugacy
classes. A system of representatives is given in Table
\ref{Fi_23cc E}.

\item[\rm(f)] $D = \langle x_1, y_1 \rangle$ has $69$ conjugacy
classes. A system of representatives is given in Table
\ref{Fi_23cc D}.

\item[\rm(g)] The character tables of $E$ and $D$ are given in Tables \ref{Fi_23ct_E} and \ref{Fi_23ct_D}, respectively.

\end{enumerate}

\end{lemma}

\begin{proof}
(a) The two generators of the stabilizer $U$ have been
found by means of a program due to P. Young. Using the
\textsc{Magma} command $\verb"CosetAction(E,U)"$ one obtains a faithful
permutation representation $PE$ of $E$ with stabilizer $U$ and
degree $1012$.

(b), (c) and (d) These statements are easily checked using the
permutation representation $PE$ and \textsc{Magma}.

(e) and (f) The faithful permutation
representation $PE$ of $E$, \textsc{Magma} and Kratzer's Algorithm 5.3.18
of \cite{michler} are employed to calculate a system of representatives of all
conjugacy classes of $D$ and $E$ in terms of their generators
given in (d). The results are stated in Tables \ref{Fi_23cc E} and
\ref{Fi_23cc D}.


The results of (g) were computed by means of  \textsc{Magma}.

\end{proof}



\newpage 

\section{Construction of the $2$-central involution centralizer of $\Fi_{23}$}\label{sec. H(Fi_23)}

Let $z_1$ denote the central involution of the extension $E$ of
$\M_{23}$ defined in Lemma \ref{l. M23-extensions} and let $D =
C_E(z_1)$. Then in the following subsidiary result it is shown that
$D_1 = D/\langle z_1 \rangle$ is isomorphic to the extension $E_2 = E(\Fi_{22})$
of Lemma \ref{l. M23-extensions}. As proved in \cite{kim}, the Fischer's simple group
$G_1 \cong \Fi_{22}$ can be contructed from the group $D_1$ by Algorithm 7.4.8 of \cite{michler}.
Unfortunately, (in 2007) \textsc{Magma} was not able to perform all steps of Holt's Algorithm
to establish the $2$-fold cover $H$ of $G_1$ from the presentation of the
given group $G_1$. Therefore the author constructs first
all the central extensions $H_2$ of the centralizer $H_1 = C_{G_1}(t)$
of a $2$-central involution $t$ of $G_1$ by a cyclic group of
order $2$. It will be shown that one of them has a Sylow
$2$-subgroup which is isomorphic to the ones of $D$. Thus $H_2$ which will be a subgroup of $H$ is
uniquely determined up to isomorphism.

The group $H$, with center $Z(H)=\langle z_1 \rangle$ such that $H/Z(H) \cong \Fi_{22}$, isomorphic to the centralizer of a $2$-central involution
of the target group $\mathfrak{G}$, is then constructed by means
of Algorithm \ref{alg. simple} as a matrix subgroup
$\mathfrak H$ of $\GL_{352}(17)$. In this way a
faithful permutation representation representation of degree $28160$
and a presentation of $H$ are built. All these results are described in this
section.



\begin{proposition}\label{prop. H(Fi_23)_1} Keep the notations of Lemma \ref{l. M23-extensions} and \ref{l. E(Fi_23)classes}, and let \\
$E = \langle a_1, b_1, c_1, d_1, t_1, g_1, h_1, i_1, j_1, v_i | 1\le i \le 11\rangle$ be the nonsplit extension of $\M_{23}$ by its simple module $V_1$ of dimension $11$ over $F=\GF(2)$, and $D=C_E(z_1)$, where $z_1 = (d_1g_1)^5$. Then the following statements hold:

\begin{enumerate}

\item[\rm(a)] $z_1 = (d_1g_1)^5$ is a $2$-central involution of $E$ with
centralizer $D = C_{E}(z_1)$ of order $2^{18}\cdot3^{2}\cdot5\cdot7\cdot11$.

\item[\rm(b)] $D = \langle x_1, y_1 \rangle$, $z_1 = (x_1y_1^2)^7$ and $E =
\langle D, e_1 \rangle$, where $x_1 = a_1$, $y_1 = b_1g_1h_1i_1$ and $e_1 = j_1$ have
respective orders $2$, $14$ and $2$. $D$ has a center of order $2$, generated by $z_1$.

\item[\rm(c)] $V$ is a unique normal subgroup of $D$ of order $2^{11}$. $V$ is elementary abelian, and has a basis $\mathcal{B} = \{z_1, v_i \mid 1 \le i \le 10 \}$, where
\begin{eqnarray*}
&&v_1 = y_1^7,\quad v_2 = (x_1y_1x_1)^7,\quad v_3 = (x_1y_1x_1y_1^2)^8,\quad v_4 = (x_1y_1^2x_1y_1)^8,\\
&&v_5 = (y_1x_1y_1^2x_1)^8,\quad v_6 = (x_1y_1x_1y_1^3)^6,\quad v_7 = (x_1y_1^3x_1y_1)^6,\\
&&v_8 = (x_1y_1^5)^7,\quad v_9 = (y_1x_1y_1x_1y_1^2)^6,\quad v_{10} = (y_1^2x_1y_1^3)^7.\\
\end{eqnarray*}

\item[\rm(d)] $V_1 = V/\langle z_1 \rangle$ has a complement $W_1$ in
$D_1 = D/\langle z_1 \rangle$ such that $W_1 \cong \M_{22}$.

In particular, $D_1 \cong E(\Fi_{22})$ defined in Proposition 3.3 in \cite{kim} (in \cite{kim}, $E(\Fi_{22})$ appears as $E_2$).

\item[\rm(e)] Let $z_1 = (x_1y_1^2)^7$ and let $v_i$ be the $10$ basis
elements of $\mathcal B$ given in (c). Then $D = \langle x_1, y_1
\rangle$ has the following set $\mathcal R(D)$ of defining
relations:
\begin{eqnarray*}
&&z_1^2 = 1, (x_1,z_1) = (y_1,z_1) = 1,\\
&&v_i^2 = 1,\quad\mbox{for all}\quad 1 \le i \le 10,\\
&&(v_j,v_k) = 1 \quad\mbox{for all}\quad 1 \le j<k \le 10,\\
&&x_1^2 =  y_1^{-7}v_1 =  (x_1y_1^{-1})^7 =  (y_1^{-1}x_1y_1x_1)^5 = (y_1^{-2}  x_1)^7z_1 = 1,\\
&&(y_1^{-2}x_1y_1^2x_1y_1^{-1}x_1)^3v_2 = (x_1y_1^{-2}x_1y_1)^5v_1v_2v_3v_4v_7v_8v_{10} = 1,\\
&&(y_1x_1y_1x_1y_1^2x_1y_1x_1y_1^{-1}x_1y_1^2)^2v_1v_3v_5v_6v_9 z_1 =1, \\
&&x_1y_1^{-1}x_1y_1x_1y_1^2x_1y_1^{-1}x_1y_1^{-1}x_1y_1^2x_1y_1x_1y_1^{-1}x_1y_1^2x_1y_1^2x_1y_1v_1v_3v_6v_8 z_1,\\
&&y_1x_1y_1^3x_1y_1x_1y_1^{-2}x_1y_1^{-3}x_1y_1x_1y_1^3x_1y_1^{-1}x_1y_1^{-3}x_1 v_4v_8v_{10} =1,\\
\end{eqnarray*}
\begin{eqnarray*}
&&x_1v_1x_1^{-1}v_2 = x_1v_2x_1^{-1}v_1 = x_1v_3x_1^{-1}v_5 = x_1v_4x_1^{-1}v_3v_4v_5 = x_1v_5x_1^{-1}v_3 =1,\\
&&x_1v_6x_1^{-1}v_1v_2v_6 = x_1v_7x_1^{-1}v_1v_2v_7 = x_1v_8x_1^{-1}v_4v_6v_9 = 1,\\
&&x_1v_9x_1^{-1}v_1v_2v_3v_4v_5v_6v_8 =1,\\
&&x_1v_{10}x_1^{-1}v_1v_3v_8v_9v_{10}z_1 = y_1v_1y^{-1}v_1  = y_1v_2y^{-1}v_3v_4v_5v_{10} =1,\\
&&y_1v_3y^{-1}v_1v_3v_4 z_1 = y_1v_4y^{-1}v_5 = y_1v_5y^{-1}v_2v_6v_8v_9z_1 = y_1v_6y^{-1}v_9 =1,\\
&&y_1v_7y^{-1}v_1v_2v_6 = y_1v_8y^{-1}v_1v_2v_3v_4z_1 = y_1v_9y^{-1}v_1v_6v_7v_9 z_1=1,\\
&&y_1v_{10}y^{-1}v_2v_4v_5v_7 z_1 =1.\\
\end{eqnarray*}

\item[\rm(f)] The involution $t = (x_1y_1x_1y_1^3)^6$ of $D$ has a
centralizer $T = C_D(t)$ of order $2^{18}\cdot 3\cdot 5$.

\end{enumerate}

\end{proposition}

\begin{proof}

(a) and (b) are restatements of Lemma \ref{l. E(Fi_23)classes}.

(c) It can be checked by means of the \textsc{Magma} command $$\verb"Subgroups(D : Al:=``Normal'')"$$ that $D$ has a unique normal subgroup $V$ of order $2^{11}$, and that $V$ is elementary abelian. Now, by means of $$\verb"GetShortGens(sub<D|x_1,y_1>, V : exclude:=sub<D|z_1>)"$$ we get the $10$ generators $v_i$ which together with $z_1$ generate $V$. Since $V$ is elementary elementary abelian, we can regard the generators as a basis.

(d) As $z_1 \in V$ it follows that $V_1 = V/\langle z_1 \rangle$ is
the unique normal subgroup of $D_1 = D/\langle z_1 \rangle$ of order
$2^{10}$. Clearly, $V_1$ is elementary abelian.

Applying the \textsc{Magma} command $\verb"CompositionFactors(D)"$ one sees
that $D_1/V_1 \cong D/V \cong \M_{22}$. Thus Lemma 2.4(d) of \cite{kim} asserts that $D_1$ splits over $V_1$.

Let $Mx$ and $My$ be the matrices of the generators $x_1$ and $y_1$ of
$D$ w.r.t. the basis $\mathcal B$ of $V$. Then
{\renewcommand{\arraystretch}{0.5}
\scriptsize
$$
Mx = \left( \begin{array}{*{11}{c@{\,}}c}
1 & 0 & 0 & 0 & 0 & 0 & 0 & 0 & 0 & 0 & 0\\
0 & 0 & 1 & 0 & 0 & 0 & 0 & 0 & 0 & 0 & 0\\
0 & 1 & 0 & 0 & 0 & 0 & 0 & 0 & 0 & 0 & 0\\
0 & 0 & 0 & 0 & 0 & 1 & 0 & 0 & 0 & 0 & 0\\
0 & 0 & 0 & 1 & 1 & 1 & 0 & 0 & 0 & 0 & 0\\
0 & 0 & 0 & 1 & 0 & 0 & 0 & 0 & 0 & 0 & 0\\
0 & 1 & 1 & 0 & 0 & 0 & 1 & 0 & 0 & 0 & 0\\
0 & 1 & 1 & 0 & 0 & 0 & 0 & 1 & 0 & 0 & 0\\
0 & 0 & 0 & 0 & 1 & 0 & 1 & 0 & 0 & 1 & 0\\
0 & 1 & 1 & 1 & 1 & 1 & 1 & 0 & 1 & 0 & 0\\
1 & 1 & 0 & 1 & 0 & 0 & 0 & 0 & 1 & 1 & 1
\end{array} \right),\quad
My = \left( \begin{array}{*{11}{c@{\,}}c}
1 & 0 & 0 & 0 & 0 & 0 & 0 & 0 & 0 & 0 & 0\\
0 & 1 & 0 & 0 & 0 & 0 & 0 & 0 & 0 & 0 & 0\\
0 & 0 & 0 & 1 & 0 & 0 & 0 & 0 & 1 & 0 & 0\\
1 & 1 & 0 & 1 & 1 & 0 & 1 & 1 & 0 & 1 & 1\\
0 & 0 & 0 & 0 & 1 & 0 & 1 & 1 & 0 & 1 & 1\\
0 & 0 & 0 & 0 & 1 & 0 & 0 & 0 & 0 & 0 & 0\\
0 & 1 & 0 & 1 & 0 & 0 & 0 & 1 & 1 & 0 & 0\\
1 & 0 & 0 & 1 & 0 & 0 & 1 & 1 & 1 & 1 & 0\\
1 & 1 & 0 & 0 & 0 & 1 & 1 & 1 & 0 & 0 & 0\\
0 & 0 & 0 & 0 & 0 & 0 & 1 & 0 & 0 & 0 & 0\\
1 & 1 & 1 & 1 & 1 & 0 & 0 & 0 & 0 & 0 & 0
\end{array} \right),
$$
}

Both matrices are blocked lower triangular matrices with upper
left diagonal blocks equal to $1$ and lower diagonal $10 \times
10$ blocks $Mx_1$ and $My_1$ in $\GL_{10}(2)$. Hence $D_1/V_1 \cong
W_1$, where $W_1 = \langle Mx_1, My_1 \rangle \le \GL_{10}(2)$.
Applying the \textsc{Magma} command $\verb"FPGroup(sub<GL(10,2)|Mx1,My1>)"$
to $W_1$ one obtains the following set $\mathcal R(W_1)$ of defining
relations of $W_1 = \langle x_1, y_1\rangle$:
\begin{eqnarray*}
&&x_1^2 = 1,\quad y_1^7 = 1,\\
&&(x_1y_1^{-1})^7 = 1,\quad (y_1^{-1}x_1y_1x_1)^5 = 1,\quad (y_1^{-2}x_1)^7 = 1,\\
&&(y_1^{-2}x_1y_1^2x_1y_1^{-1}x_1) = 1,\quad (x_1y_1^{-2}x_1y_1)^5 = 1,\\
&&(y_1x_1y_1x_1y_1^2x_1y_1x_1y_1^{-1}x_1y_1^2)^2 = 1,\\
&&x_1y_1^{-1}x_1y_1x_1y_1^2x_1y_1^{-1}x_1y_1^{-1}x_1y_1^2x_1y_1x_1y_1^{-1}x_1y_1^2x_1y_1^2x_1y_1 =1,\\
&&y_1x_1y_1^3x_1y_1x_1y_1^{-2}x_1y_1^{-3}x_1y_1x_1y_1^3x_1y_1^{-1}x_1y_1^{-3}x_1 = 1.
\end{eqnarray*}

Let $PD_1$ be the faithful permutation representation of $D_1$ of degree
$1024$ with stabilizer $W_1$. Let $E_2 = E(\Fi_{22})$ be the
finitely presented group of Lemma 2.4(f) of \cite{kim} with
faithful permutation representation $PE_2$ defined in Lemma
2.6(a) of \cite{kim}. By means of the \textsc{Magma} command
$\verb"IsIsomorphic(PE_2,PD_1)"$ it is verified that
$D_1$ is isomorphic to $E_2$.

(e) By means of \textsc{Magma} the semidirect product
$D_1 = \langle x_1, y_1, v_i \mid 1 \le i \le 10 \rangle$ of $W_1$ by $V_1 =
\langle v_i \mid 1 \le i \le 10 \rangle$ has a set of defining
relations $\mathcal R(D_1)$ consisting of $\mathcal R(W_1)$ and the
following set of relations:
\begin{eqnarray*}
&&v_i^2 = 1, (v_j,v_k) = 1 \quad \mbox{for all}\quad 1 \le i, j, k \le 10,\\
&&x_1v_1x_1^{-1}v_2 = x_1v_2x_1^{-1}v_1 = x_1v_3x_1^{-1}v_5 = 1,\\
&&x_1v_4x_1^{-1}v_3v_4v_5 = x_1v_5x_1^{-1}v_3 = 1,\\
&&x_1v_6x_1^{-1}v_1v_2v_6, x_1v_7x_1^{-1}v_1v_2v_7 = 1,\\
&&x_1v_8x_1^{-1}v_4v_6v_9, x_1v_9x_1^{-1}v_1v_2v_3v_4v_5v_6v_8 = 1,\\
&&x_1v_{10}x_1^{-1}v_1v_3v_8v_9v_{10} = y_1v_1y_1^{-1}v_1 = 1,\\
&&y_1v_2y_1^{-1}v_3v_4v_5v_{10} = y_1v_3y_1^{-1}v_1v_3v_4 = 1,\\
&&y_1v_4y_1^{-1}v_5 = y_1v_5y_1^{-1}v_2v_6v_8v_9 = 1,\\
&&y_1v_6y_1^{-1}v_9 = y_1v_7y_1^{-1}v_1v_2v_6 = 1,\\
&&y_1v_8y_1^{-1}v_1v_2v_3v_4 = y_1v_9y_1^{-1}v_1v_6v_7v_9 = y_1v_{10}y_1^{-1}v_2v_4v_5v_7 = 1.\\
\end{eqnarray*}

By (c) each $v_i$ is a word in the generators $x_1$ and $y_1$ of $D$.
Evaluating the relations of $\mathcal R(D_1)$ in $D$, we get that each relation has value equal to either $1$ or $z_1$. Appending $z_1$ to appropriate relations, we get the defining set of relations $\mathcal R(D)$, as written in the statement.

\end{proof}


\begin{proposition}\label{prop. H(Fi_23)_2} Keep the notations of Proposition \ref{prop. H(Fi_23)_1}. The following statements hold:

\begin{enumerate}

\item[\rm(a)] Let $H_1 = \langle r_i|1 \le i \le 14\rangle \cong
H(\Fi_{22})$ be the finitely presented group constructed in
Proposition 3.3(n) of \cite{kim} (for the notation, $h_i$'s in Proposition 3.3(n) of \cite{kim} are replaced by $r_i$'s). Then $H_1$ has a central
extension $H_2 = \langle k_j \mid 1 \le j \le 15 \rangle $ of
order $2^{18}\cdot3^4\cdot5$ having the following set of $\mathcal
R(H_2)$ of defining relations:
\begin{eqnarray*}
&&k_1^2 =  k_2^5 = k_3^3 = k_4^3 = k_5^2 = k_6^2 = k_7^2 = 1,\\
&&k_8^2 = k_9^2 = k_{10}^4 = k_{11}^2 = k_{12}^2 = k_{13}^2 = k_{14}^2 = k_{15}^2 = 1,\\
&&(k_i, k_{15}) = 1 \quad \mbox{for} \quad 1 \le i \le 14,\\
&&(k_i, k_{14}) = 1 \quad \mbox{for} \quad 1 \le i \le 13,\\
&&(k_i, k_{13}) = 1 \quad \mbox{for} \quad 2 \le i \le 12,\\
&&k_1^{-1}k_{13}k_1k_{13}^{-1}k_{14}^{-1} = 1,\\
&&(k_3, k_4) = 1,\quad k_3^{-1}k_2k_1k_2^{-1}k_3k_1 = 1,\\
&&(k_2^{-1}k_1)^4 = 1,\quad k_3^{-1}k_2k_3^{-1}k_1k_2k_3k_1k_2^{-1} = 1,\\
&&(k_2  k_3^{-1})^4 = 1,\quad (k_4, k_2^{-1}, k_4) = 1,\\
&&k_1k_4^{-1}k_1k_4^{-1}k_1k_4k_1k_4 = 1,\quad k_2^{-1}k_3^{-1}k_1k_2^{-2}k_3k_4k_1k_4^{-1} = 1,\\
&&k_2k_3^{-1}k_2k_4k_2^2k_3^{-1}k_2^{-1}k_3k_4 = 1,\quad k_1k_5k_1^{-1}k_5k_6k_{13}^{-1} = 1,\\
&&k_1k_6k_1^{-1}k_6k_{14}^{-1} = k_1k_7k_1^{-1}k_7k_8k_{14}^{-1} = k_1k_8k_1^{-1}k_8 = 1,\\
\end{eqnarray*}
\begin{eqnarray*}
&&k_1k_9k_1^{-1}k_{11}k_{13}^{-1}k_{14}^{-1}k_{15}^{-1} = 1,\\
&&k_1k_{10}k_1^{-1}k_6k_9k_{10}k_{11}k_{12}k_{13}^{-1}k_{15}^{-1} = 1,\\
&&k_1k_{11}k_1^{-1}k_9k_{13}^{-1}k_{15}^{-1} = k_1k_{12}k_1^{-1}k_{12}k_{14}^{-1} = 1,\\
&&k_2k_5k_2^{-1}k_5k_6k_7k_{14}^{-1}k_{15}^{-1} = k_2k_6k_2^{-1}k_7k_8k_{13}^{-1}k_{14}^{-1} = 1,\\
&&k_2k_7k_2^{-1}k_6k_7k_{14}^{-1}k_{15}^{-1} = k_2k_8k_2^{-1}k_5k_6k_7k_8 = 1,\\
&&k_2k_9k_2^{-1}k_5k_6k_{10}k_{11}k_{13}^{-1}k_{14}^{-1}k_{15}^{-1} = 1,\\
&&k_2k_{10}k_2^{-1}k_7k_8k_9k_{13}^{-1} = k_2k_{11}k_2^{-1}k_5k_{12}k_{14}^{-1}k_{15}^{-1} = 1,\\
&&k_2k_{12}k_2^{-1}k_6k_7k_8k_{10}k_{12} = k_3k_5k_3^{-1}k_6k_{14}^{-1} = 1,\\
&&k_3k_6k_3^{-1}k_5k_6k_{13}^{-1} = k_3k_7k_3^{-1}k_5k_8k_{15}^{-1} = 1,\\
&&k_3k_8k_3^{-1}k_5k_6k_7k_8 = k_3k_9k_3^{-1}k_5k_9k_{10}k_{14}^{-1} = 1,\\
&&k_3k_{10}k_3^{-1}k_6k_9k_{15}^{-1} = k_3k_{11}k_3^{-1}k_9k_{12}k_{14}^{-1}k_{15}^{-1} = 1,\\
&&k_3k_{12}k_3^{-1}k_5k_{10}k_{11}k_{12}k_{13}^{-1}k_{14}^{-1} = 1,\\
&&k_4k_5k_4^{-1}k_5k_6k_{13}^{-1}k_{14}^{-1} = 1,\\
&&k_4k_6k_4^{-1}k_5k_{14}^{-1} = k_4k_7k_4^{-1}k_6k_{10}k_{11}k_{12}k_{14}^{-1}k_{15}^{-1} = 1,\\
&&k_4k_8k_4^{-1}k_5k_6k_{10}k_{11}k_{13}^{-1}k_{14}^{-1}k_{15}^{-1} = 1,\\
&&k_4k_9k_4^{-1}k_5k_6k_{11}k_{12}k_{14}^{-1} = 1,\\
&&k_4k_{10}k_4^{-1}k_9k_{10}k_{12}k_{13}^{-1}k_{14}^{-1}k_{15}^{-1} = 1,\\
&&k_4k_{11}k_4^{-1}k_5k_6k_7k_8k_9k_{11}k_{13}^{-1}k_{15}^{-1} = 1,\\
&&k_4k_{12}k_4^{-1}k_5k_7k_{10}k_{11}k_{14}^{-1} = k_{10}^2k_{14}^{-1}k_{15}^{-1} = 1,\\
&&(k_5, k_6) = (k_5, k_7) = (k_5, k_8) = (k_5, k_9) = 1,\\
&&k_5^{-1}k_{10}^{-1}k_5k_{10}k_{14}^{-1}k_{15}^{-1} = k_5^{-1}k_{11}^{-1}k_5k_{11}k_{14}^{-1}k_{15}^{-1} = 1,\\
&&(k_5, k_{12}) = (k_6, k_7) = (k_6, k_8) = 1,\\
&&k_6^{-1}k_9^{-1}k_6k_9k_{14}^{-1}k_{15}^{-1} = k_6^{-1}k_{10}^{-1}k_6k_{10}k_{14}^{-1}k_{15}^{-1} = 1,\\
&&k_6^{-1}k_{11}^{-1}k_6k_{11}k_{14}^{-1}k_{15}^{-1} = 1,\\
&&(k_6, k_{12}) = (k_7, k_8) = (k_7, k_9) = (k_8, k_{10}) = (k_8, k_{12}) = 1,\\
&&k_7^{-1}k_{10}^{-1}k_7k_{10}k_{14}^{-1}k_{15}^{-1} = k_7^{-1}k_{11}^{-1}k_7k_{11}k_{14}^{-1}k_{15}^{-1} = 1,\\
&&k_7^{-1}k_{12}^{-1}k_7k_{12}k_{14}^{-1}k_{15}^{-1} = k_8^{-1}k_9^{-1}k_8k_9k_{14}^{-1}k_{15}^{-1} = 1,\\
&&k_8^{-1}k_{11}^{-1}k_8k_{11}k_{14}^{-1}k_{15}^{-1} = 1,\\
&&(k_9, k_{10}) = (k_9, k_{11}) = (k_9, k_{12}) = (k_{10}, k_{12}) = (k_{11}, k_{12}) = 1,\\
&&k_{10}^{-1}k_{11}^{-1}k_{10}k_{11}k_{14}^{-1}k_{15}^{-1} = 1.\\
\end{eqnarray*}

\item[\rm(b)]  $H_2$ has a faithful permutation representation of
degree $2048$ with stabilizer $U_2 = \langle k_1,k_2,k_3,k_4
\rangle$.


\item[\rm(c)] $H_2$ has a Sylow $2$-subgroup $S_2$ generated by the four involutions $m_1 = k_1$, $m_2 =
(k_2k_1k_2k_4k_5)^6$, $m_3 = (k_2^3k_1k_2)^2$ and $m_4 =
(k_4^2k_2k_4k_2k_4)^6$.

\item[\rm(d)] $A_2$ is the unique maximal elementary abelian normal subgroup of $S_2$ of order $2^{11}$, and is generated by the following $11$ elements:
\begin{align*}
m_4, (m_2 m_4)^2, (m_3 m_4)^2, (m_1 m_3 m_4)^4, (m_2 m_3 m_4)^4, \\
(m_1 m_2 m_3 m_4)^8, (m_1 m_3 m_1 m_4)^2, (m_2 m_3 m_2 m_4)^2, \\
(m_2 m_3 m_4 m_2)^2, (m_1 m_2 m_3 m_1 m_4)^4, \mbox{ and } (m_1 m_2 m_3 m_2 m_4)^4.
\end{align*}

\item[\rm(e)] $D_2 = N_{H_2}(A_2) = \langle p_2, q_2\rangle$ where $p_2 = k_1k_2$ and $q_2 = k_1(k_1k_5k_4k_2)^5$. Here $p_2$ and $q_2$ have orders $4$ and $6$, respectively.

\item[\rm(f)] $H_2 = \langle D_2, h_2\rangle = \langle p_2, q_2,
h_2\rangle$ where $h_2 = k_4$ has order $3$.


\item[\rm(g)] There is an isomorphism $\alpha : D_2 \rightarrow T$ s.t. $\alpha (p_2) = p_1 = (u_1 u_2 u_4 u_6 u_7)^3$ and $\alpha(q_2) = q_1 = (w_1 w_2 w_3 w_4 w_2 w_3 w_4)^5$, where
\begin{align*}
u_1 = (x_1 y_1^2)^7, \quad
u_2 = (x_1 y_1 x_1 y_1^3)^6, \quad
u_3 = (y_1^2 x_1 y_1^3)^7, \\
u_4 = (y_1^3 x_1 y_1^2)^7, \quad
u_5 = x_1 y_1 x_1 y_1 x_1 y_1 x_1 y_1^2 x_1 y_1 x_1 y_1, \\
u_6 = (y_1^3 x_1 y_1^2 x_1 y_1^5 x_1 y_1^3 x_1 y_1)^2, \quad
u_7 = (y_1^5 x_1 y_1^2 x_1 y_1 x_1 y_1 x_1 y_1 x_1 y_1^2 x_1)^4.
\end{align*}
and
\begin{align*}
w_1 = (s_1)^7,&\quad
w_2 = (s_1^4 s_2^2 s_1^2 s_2^2 s_1)^2,&\quad \\
w_3 = s_1^2 s_2 s_1^3 s_2^2 s_1^3,&\quad
w_4 = (s_1^4 s_2^2 s_1^2 s_2 s_1^3 s_2 s_1^3)^2, &\quad
\end{align*}
where $s_1 = x_1 y_1 x_1$ and $s_2 = (x_1 y_1 x_1 y_1^3)^2$.

\item[\rm(h)] So, $D = \langle p_1,  q_1, x_1, y_1 \rangle$, $T = C_D(t) = \langle p_1, q_1 \rangle$, $H_2 = \langle p_2, q_2, h_2 \rangle$, and $D_2 = N_{H_2}(A) = \langle p_2, q_2 \rangle$.



\item[\rm(i)] A system of representatives $r_i$ of the $189$
conjugacy classes of $H_2$ and the corresponding centralizers orders
$|C_{H_2}(r_i)|$ are given in Table \ref{Fi_23cc H_2}.

\item[\rm(j)] A system of representatives $d_i$ of the $151$
conjugacy classes of $D_2$ and the corresponding centralizers orders
$|C_{D_2}(d_i)|$ are given in Table \ref{Fi_23cc D_2}.

\item[\rm(k)] The character tables of $H_2$ and $D_2$ are given in Tables \ref{Fi_23ct_H_2} and \ref{Fi_23ct_D_2}, respectively.

\end{enumerate}

\end{proposition}

\begin{proof}

(a) By Table A.8 of \cite{kim} the simple group $\Fi_{22}$ has a
unique conjugacy class of $2$-central involutions, represented by an element $u$. Theorem
5.1 of \cite{kim} asserts that $H_1 = C_{\Fi_{22}}(u)$ is
isomorphic to the finitely presented group $H_1 = \langle r_i \mid
1 \le i \le 14\rangle$ constructed in Proposition 3.3(n) of \cite{kim} (with notation changed from $h_i$ in \cite{kim} to $r_i$ here). $H_1$ has a faithful permutation representation
$PH_1$ of degree $1024$ with stabilizer $\langle r_1,r_2,r_3,r_4
\rangle$ by Lemma 3.4 of \cite{kim}.

Let $F = \GF(2)$, and let $M$ be the trivial $FH_1$-module, which can be constructed in \textsc{Magma} by the commands $$\verb"FEalg := MatrixAlgebra< FiniteField(2), 1 | [[1] : i in [1..14]]>"$$ and $$\verb"M := GModule(PH_1, FEalg)".$$ Using the faithful permutation representation $PH_1$ and the \textsc{Magma}
command $$\verb"CohomologicalDimension(PH_1, M, 2)"$$ we get that this cohomological dimension is $3$. For each of the
eight $2$-cocycles $(a,b,c)$ with $a,b,c \in F$, \textsc{Magma} constructs a
finitely presented group $E_{(a,b,c)}$ by means of its commands
$$\verb"P := ExtensionProcess(PH_1,M,H_1)",$$
$$\verb"E_{a,b,c} := Extension(P, [a,b,c])".$$

Since $D$ is a non-split extension of $D_1$ by $\langle z_1 \rangle$
the split extension $E_{(0,0,0)}$ can be neglected.




It is checked by means of \textsc{Magma} using the faithful permutation
representation $PD$ of $D$ of degree $1012$ that $E_{(1,1,0)}$ is
a non-split extension of order $2^{18}\cdot3^4\cdot5$ of
$H_1$ having a Sylow $2$-subgroup which is isomorphic to any of
the Sylow $2$-subgroups of $D$. The presentation of $H_2 =
E_{(1,1,0)}$ is given in the statement.

(b) The \textsc{Magma} command $\verb"CosetAction(H_2,U_2)"$ for the subgroup $U_2
= \langle k_1,k_2,k_3,k_4 \rangle$ of the finitely presented group
$H_2$ gives the faithful permutation representation of $H_2$ of degree $2048$. We know it is faithful because the order of the resulting permutation group has the correct order $2^{18}\cdot 3^4\cdot 5$.

(c) A Sylow $2$-subgroup $S_2$ of $H_2$ can be obtained by the \textsc{Magma} command $$\verb"S_2 := SylowSubgroup(H_2,2)"$$ and its generators can be obtained by the command $$\verb"GetShortGens(H_2,S_2)",$$ which gave the four generators $m_1 = k_1$, $m_2 =(k_2k_1k_2k_4k_5)^6$, $m_3 = (k_2^3k_1k_2)^2$ and $m_4 = (k_4^2k_2k_4k_2k_4)^6$, as asserted. The four generators of $S_2$ all have order $2$.

(d) It can be verified by the \textsc{Magma} command $$\verb"Subgroups(S_2 : Al:=``Normal'', IsElementaryAbelian:=true)"$$ that $S_2$ has a unique maximal elementary abelian subgroup $A_2$ of order $2^{11}$. Now the generators for $A_2$ can be obtained by the command $$\verb"GetShortGens(sub<H_2|m_1,m_2,m_3,m_4>,A_2)",$$ which gave the $11$ generators as written in the statement.

(e) Letting $$\verb"D_2 := Normalizer(H_2,A_2)"$$ yields a subgroup $D_2$ of $H_2$ of order $2^{18}\cdot 3\cdot 5$, and its generators can be obtained by the command $$\verb"GetShortGens(H_2,D_2)",$$ which gave the two generators $p_2 = k_1k_2$ and $q_2 = k_1(k_1k_5k_4k_2)^5$, having orders $4$ and $6$, respectively.

(f) It can be verified with \textsc{Magma} that $H_2 = \langle D_2, k_4 \rangle = \langle p_2, q_2, k_4 \rangle$. We let $h_2 = k_4$, and it is easy to check with \textsc{Magma} that $h_2$ has order $3$.

(g) The isomorphism $\alpha : D_2 \rightarrow T$ can be obtained by the \textsc{Magma} command $$\verb"IsIsomorphic(D_2,T)".$$ Recall that $T$ is a subgroup of $D = \langle x_1, y_1 \rangle$. For documentation of this isomorphism $\alpha$ we need to get words for $\alpha(p_2)$ and $\alpha(q_2)$ in terms of $x_1$ and $y_1$. Now, for convenience, let $p_1 = \alpha(p_2)$ and $q_1 = \alpha(q_2)$.

In order to obtain short words for $p_1$ and $q_1$, the author employed Strategy \ref{str. LookupWord}. First, it is checked with \textsc{Magma} that $N_D(\langle p_1\rangle)$ has order $2^8$, and it is obtained by $\verb"GetShortGens"$ that $N_D(\langle p_1 \rangle) = \langle u_1, u_2, u_3, u_4 u_5, u_6, u_7 \rangle$, where the words $u_1$ through $u_7$ are as written in the statement. Then, the command \\
$\verb"LookupWord(sub<D|u_1,u_2,u_3,u_4,u_5,u_6,u_7>, p_1)"$ gave the answer $p_1 = (u_1 u_2 u_4 u_6 u_7)^3$.

Note that $q_1$ has order $6$, since so does $q_2$. It can be checked with \textsc{Magma} that the $N_D(\langle q_1 \rangle)$ has order $2^6 \cdot 3^2$ and that $C_D(q_1^3)$ has order $2^{17}\cdot 3^2 \cdot 5 \cdot 7$. It is easy to observe that $N_D(\langle q_1 \rangle) \le C_D(q_1^3)$. Employing Strategy \ref{str. GetShortGens}, we can find out that $C_D(q_1^3) = \langle s_1, s_2 \rangle$ and $N_D(\langle q_1 \rangle) = \langle w_1, w_2, w_3, w_4 \rangle$, where $s_1 = x_1 y_1 x_1$, $s_2 = (x_1 y_1 x_1 y_1^3)^2$, $w_1 = (s_1)^7$, $w_2 = (s_1^4 s_2^2 s_1^2 s_2^2 s_1)^2$, $w_3 = s_1^2 s_2 s_1^3 s_2^2 s_1^3$, and $w_4 = (s_1^4 s_2^2 s_1^2 s_2 s_1^3 s_2 s_1^3)^2$. Finally, the command $\verb"LookupWord(sub<D|w_1,w_2,w_3,w_4>, q_1)"$ successfully gave the word for $q_1$, namely $q_1 = (w_1 w_2 w_3 w_4 w_2 w_3 w_4)^5$.

(h) From (g) we know that $D_2 = \langle p_2, q_2 \rangle$ and $\alpha : D_2 \rightarrow T$ is an isomorphism, and also that $p_1 = \alpha(p_2)$ and $q_1 = \alpha(q_2)$. Therefore we have $T = \langle p_1, q_1 \rangle$. Since $D = \langle x_1, y_1 \rangle$, it is clear that $D = \langle p_1, q_1, x_1, y_1 \rangle$. It is verified in (f) that $H_2 = \langle D_2, k_4 \rangle = \langle p_2, q_2, k_4 \rangle$.


The results of (i) and (j) can be obtained by applying Kratzer's Algorithm 5.3.18 of \cite{michler} in \textsc{Magma} to the relevant groups, with generators as stated in (h). The results of (k) were computed by means of \textsc{Magma}.

\end{proof}

It will be shown in the rest of this section that the free product $H_2*_{D_2} D$ of $H_2$ and $D$ with amalgamated subgroup $D_2$ has an irreducible $352$-dimensional faithful representation over $\GF(17)$
which gives rise to the group $H$ such that $|Z(H)|=2$ and $H/Z(H)\cong \Fi_{22}$.

\begin{theorem}\label{thm. H(Fi_23)} Keep the notations in Proposition \ref{prop. H(Fi_23)_2}. Let $K = \GF(17)$. Using the notations of the character tables of \ref{Fi_23ct_H_2}, \ref{Fi_23ct_D_2}, and \ref{Fi_23ct_D} of $H_2 = \langle p_2, q_2, h_2 \rangle$, $D_2 = \langle p_2, q_2 \rangle$, $D=\langle p_1, q_1, x_1, y_1 \rangle$, the following statements hold:

\begin{enumerate}
\item[\rm(a)] The smallest degree of a non-trivial compatible pair $(\chi, \tau)\in mf \mbox{char}_{\mathbb{C}}(H_2) \times mf \mbox{char}_{\mathbb{C}}(D)$ is $352$.

\item[\rm(b)] There are exactly two compatible pairs $(\chi, \tau),(\chi', \tau')\in mf \mbox{char}_{\mathbb{C}}(H_2) \times mf \mbox{char}_{\mathbb{C}}(D)$ of degree $352$ of $H_2 = \langle D_2, h \rangle$ and $D = \langle \alpha(D_2), x,y \rangle$:
\begin{align*}
(\chi, \tau) = (\chi_{57} + \chi_{40} + \chi_{41}, \tau_{9} + \tau_{10})
\end{align*}
and
\begin{align*}
(\chi', \tau') = (\chi_{31} + \chi_{34} + \chi_{40} + \chi_{41}, \tau_{9} + \tau_{10})
\end{align*}
with common restriction
\begin{align*}
\chi_{|D_2} = \tau_{|\alpha(D_2)}
= \chi'_{|D_2} = \tau'_{|\alpha(D_2)}
= \psi_{26} + \psi_{27} + \psi_{59}
+ \psi_{63} + \psi_{70} + \psi_{73},
\end{align*}
where irreducible characters with bold face indices denote faithful irreducible characters.

\item[\rm(c)] Let $\mathfrak{V}$ and $\mathfrak{W}$ be the up to isomorphism uniquely determined faithful semi-simple multiplicity-free $352$-dimensional modules of $H_2$ and $D$ over $F=\GF(17)$ corresponding to the compatible pair $\chi, \tau$, respectively.

Let $\kappa_{\mathfrak{V}} : H_2 \rightarrow \GL_{352}(17)$ and $\kappa_{\mathfrak{W}} : D \rightarrow \GL_{352}(17)$ be the representations of $H_2$ and $D$ afforded by the modules $\mathfrak{V}$ and $\mathfrak{W}$, respectively.

Let $\mathfrak{p}_2 = \kappa_{\mathfrak{V}}(p_2)$, $\mathfrak{q}_2 = \kappa_{\mathfrak{V}}(q_2)$, $\mathfrak{h}_2 = \kappa_{\mathfrak{V}}(h_2)$ in $\kappa_{\mathfrak{V}}(H_2) \le \GL_{352}(17)$. Then the following assertions hold:

\begin{enumerate}
\item[\rm(1)] $\mathfrak{V}_{|D_2} \cong \mathfrak{W}_{|\alpha(D_2)}$, and there is a transformation matrix $\mathcal{T}_1 \in \GL_{352}(17)$ such that
\begin{align*}
\mathfrak{p}_2 = \mathcal{T}_1^{-1} \kappa_{\mathfrak{W}}(p_1) \mathcal{T}_1, \quad
\mathfrak{q}_2 = \mathcal{T}_1^{-1} \kappa_{\mathfrak{W}}(q_1) \mathcal{T}_1.
\end{align*}

\item[\rm(2)] There is a transformation matrix $\mathcal{S}_1 \in \GL_{352}(17)$ such that
\begin{align*}
\mathfrak{p}_2 = \mathcal{S}_1^{-1} \mathfrak{p}_2  \mathcal{S}_1, \quad
\mathfrak{q}_2 = \mathcal{S}_1^{-1} \mathfrak{q}_2  \mathcal{S}_1,
\end{align*}
and that if we let
\begin{align*}
\mathfrak{x}_2 & = (\mathcal{T}_1\mathcal{S}_1)^{-1} \kappa_{\mathfrak{W}}(x_1) \mathcal{T}_1\mathcal{S}_1 \in \GL_{352}(17), \quad \mbox{and} \\
\mathfrak{y}_2 & = (\mathcal{T}_1\mathcal{S}_1)^{-1} \kappa_{\mathfrak{W}}(y_1) \mathcal{T}_1\mathcal{S}_1 \in \GL_{352}(17),
\end{align*}
then $\mathfrak{H} = \langle \mathfrak{p}_2, \mathfrak{q}_2, \mathfrak{x}_2, \mathfrak{y}_2, \mathfrak{h}_2 \rangle = \langle \mathfrak{x}_2, \mathfrak{y}_2, \mathfrak{h}_2 \rangle$ satisfies the Sylow $2$-subgroup test of Algorithm 7.4.8 Step 5(c) of \cite{michler}. The proof showing that $\mathfrak{H}$ satisfies the Sylow $2$-subgroup test is split into two parts; first half is shown in the proof of this theorem (namely, the order of $\mathfrak{h}_2 \mathfrak{x}_2$ has to be the order of an element in $2\Fi_{22}$; it turned out to be $12$ in this case), and the other half in Theorem \ref{thm. H(Fi_23)_embed}.


\item[\rm(3)] The three generating matrices of $\mathfrak{H}$ are stated in \cite{kim1}.

\end{enumerate}

\item[\rm(d)] The construction shown above in {\rm (c)} can be applied to the compatible pair $\chi', \tau'$. However, there is no solution in this case which satisfies the Sylow $2$-subgroup test of Algorithm 7.4.8 Step 5(c) of \cite{michler}.

\end{enumerate}

\end{theorem}

\begin{proof}

(a) The character tables of the groups $H_2, D_2$, and $D$ are stated in the appendix. In the following we use their notations. Using \textsc{Magma} and the character tables of $H_2,D_2$, and $D$ and the fusion of the classes of $D_2$ in $H_2$ and $T=C_D(t) (\cong D_2)$ in $D$, an application of Kratzer's Algorithm 7.3.10 of \cite{michler} yields the compatible pair stated in assertion (a).

(b) The application of Kratzer's Algorithm 7.3.10 of \cite{michler} also shows that the pairs $(\chi, \tau)$ and $(\chi', \tau')$ of (b) are all the compatibles pairs of degree $352$ with respect to the fusion of the $D_2$-classes into the $H_2$- and into the $D$-classes.


(c) In order to construct the faithful irreducible representation
$\mathfrak V$ corresponding to the character $\chi = \chi_{57} + \chi_{40} + \chi_{41}$ of degree $352$, the \textsc{Magma} command
$\verb"LowIndexSubgroups(PH_2, 1500)"$ is applied, using the faithful
permutation presentation $PH_2$ of $H_2$ of degree $2048$. \textsc{Magma} found subgroups $U_1, U_2$, and $U_3$ such that the followings hold.

$U_1$ is of index $480$ in $H_2$, and $\chi_{57}$ (dimension $160$) is a constituent of the permutation character $(1_{U_1})^{H_2}$. The program $\verb"GetShortGens(PH_2, U_1)"$ gives a generating set of $U_1$, so we have $U_1 = \langle p_2 h_2 p_2, (p_2 q_2 h_2 q_2)^3, (q_2 h_2^2 p_2)^4, (h_2 p_2 q_2 h_2 q_2)^6 \rangle$. Using Meat-axe Algorithm to the permutation module $(1_{U_1})^{H_2}$, the author obtained the irreducible $KH_2$-module $\mathfrak{V}_{57}$ over $K$ corresponding to $\chi_{57}$. Here, the permutation module $(1_{U_1})^{H_2}$ is obtained by applying the \textsc{Magma} command $\verb"PermutationMatrix"$ to the generators of the permutation group obtained by $\verb"CosetAction(PH_2, U_1)"$.

$U_2$ is of index $512$ in $H_2$, and $\chi_{40}$ (dimension $96$) is a constituent of the permutation character $(1_{U_2})^{H_2}$. By $\verb"GetShortGens"$ we get $U_2 = \langle (q_2 p_2 h_2)^5, (q_2 h_2^2)^4,$ $(p_2 h_2 q_2 p_2)^4, (p_2^2 q_2^2 h_2)^5 \rangle$. By applying the Meat-axe Algorithm to the permutation module $(1_{U_2})^{H_2}$, the irreducible $KH_2$-module $\mathfrak{V}_{40}$ corresponding to $\chi_{40}$ is obtained.

$U_3$ is of index $512$ in $H_2$, and $\chi_{41}$ (dimension $96$) is a constituent of the permutation character $(1_{U_3})^{H_2}$. By $\verb"GetShortGens"$ we get $U_3 = \langle q_2^3, (h_2 q_2)^3, (q_2 p_2 h_2^2)^6,$ $(q_2 p_2 h_2 q_2^2 h_2)^5 \rangle$. By applying the Meat-axe Algorithm to the permutation module $(1_{U_3})^{H_2}$, the irreducible $KH_2$-module $\mathfrak{V}_{41}$ corresponding to $\chi_{41}$ is obtained.

In order to construct the faithful irreducible representation
$\mathfrak W$ corresponding to the character $\tau = \tau_{9} + \tau_{10}$ of degree $352$, the \textsc{Magma} command\\
$\verb"LowIndexSubgroups(PD, 1500)"$ is applied, using the faithful
permutation presentation $PD$ of $D$ of degree $1012$. \textsc{Magma} found subgroups $S_1$ and $S_2$ such that the followings hold.

$S_1$ is of index $352$ in $D$, and $\tau_{9}$ (dimension $176$) is a constituent of the permutation character $(1_{S_1})^{D}$. By $\verb"GetShortGens"$ we get $S_1 = \langle (q_1 p_1^3)^3$, $(x_1 q_1 p_1 q_1 x_1)^3$, $(x_1 q_1 x_1 p_1 q_1)^2 \rangle$. By applying the Meat-axe Algorithm to the permutation module $(1_{S_1})^{D}$, the irreducible $KD$-module $\mathfrak{W}_{9}$ corresponding to $\tau_{9}$ is obtained.

$S_2$ is of index $352$ in $D$, and $\tau_{10}$ (dimension $176$) is a constituent of the permutation character $(1_{S_2})^{D}$. By $\verb"GetShortGens"$ we get $S_2 = \langle (p_1 x_1 q_1)^4$, $x_1 q_1 x_1$, $(p_1 q_1 x_1 p_1)^2 \rangle$. By applying the Meat-axe Algorithm to the permutation module $(1_{S_2})^{D}$, the irreducible $KD$-module $\mathfrak{W}_{10}$ corresponding to $\tau_{10}$ is obtained.



Thus, we can obtain $\mathfrak V$ and $\mathfrak W$ by letting $\mathfrak{V} = \mathfrak{V}_{57} \oplus \mathfrak{V}_{40} \oplus \mathfrak{V}_{41}$ and $\mathfrak{W} = \mathfrak{W}_{9} \oplus \mathfrak{W}_{10}$, corresponding to the characters $\chi = \chi_{57}+\chi_{40}+\chi_{41}$ and $\tau = \tau_{9} + \tau_{10}$, respectively. Direct summation of modules here means block diagonal joining of matrices, in matrix sense. For example, $\mathfrak{p}_2 = \kappa_{\mathfrak{V}}(p_2)$ is obtained by
\begin{align*}
\kappa_{\mathfrak{V}}(p_2) =
\mbox{diag}(\kappa_{\mathfrak{V}_{57}}(p_2),
\kappa_{\mathfrak{V}_{40}}(p_2),\kappa_{\mathfrak{V}_{41}}(p_2)),
\end{align*}
where $\kappa_{\mathfrak{V}_{57}} : H_2 \rightarrow \GL_{160}(17)$,  $\kappa_{\mathfrak{V}_{57}} : H_2 \rightarrow \GL_{96}(17)$, and $\kappa_{\mathfrak{V}_{57}} : H_2 \rightarrow \GL_{96}(17)$ are the representations of $H_2$  afforded by the modules $\mathfrak{V}_{57}$, $\mathfrak{V}_{40}$, and $\mathfrak{V}_{41}$, respectively. The other matrices $\mathfrak{q}_2 = \kappa_{\mathfrak{V}}(q_2)$ and $\mathfrak{h}_2 = \kappa_{\mathfrak{V}}(h_2)$ are obtained by a similar manner, by block diagonal joining. Similar idea applies to the other side $\mathfrak{W}$, too.

$\chi_{|D_2} = \tau_{|\alpha(D_2)}$ means $\mathfrak{V}_{|D_2} \cong \mathfrak{W}_{|\alpha(D_2)} = \mathfrak{W}_{|T}$. Recall that two representations of a group being isomorphic means that there is some matrix $\mathcal{T}$ such that for every element of the group, the matrix for the element corresponding to one of the two representations is conjugate to the matrix for the same element corresponding to the other representation by $\mathcal{T}$. However, here we have two isomorphic groups $D_2$ and $T$ instead of identical groups. So, employing the isomorphism $\alpha : D_2 \rightarrow T$, now we can see that $\mathfrak{V}_{|D_2} \cong \mathfrak{W}_{|T}$ means that there is some $\mathcal{T}_1 \in \GL_{352}(17)$ such that ${\kappa_\mathfrak{V}}_{|D_2}(g) = \mathcal{T}_1^{-1} {\kappa_\mathfrak{W}}_{|T}(\alpha(g)) \mathcal{T}_1$ for all $g \in D_2$.

Assuming we have such $\mathcal{T}_1$, we get
\begin{align*}
\mathfrak{p}_2 & = {\kappa_\mathfrak{V}}_{|D_2} (p_2)
= \mathcal{T}_1^{-1} {\kappa_\mathfrak{W}}_{|T}(\alpha(p_2)) \mathcal{T}_1
= \mathcal{T}_1^{-1} \kappa_{\mathfrak{W}}(p_1) \mathcal{T}_1, \quad \mbox{and} \\
\mathfrak{q}_2 & = {\kappa_\mathfrak{V}}_{|D_2} (q_2)
= \mathcal{T}_1^{-1} {\kappa_\mathfrak{W}}_{|T}(\alpha(q_2)) \mathcal{T}_1
= \mathcal{T}_1^{-1} \kappa_{\mathfrak{W}}(q_1) \mathcal{T}_1,
\end{align*}
thus $\mathfrak{p}_2= \mathcal{T}_1^{-1} \kappa_{\mathfrak{W}}(p_1) \mathcal{T}_1$ and $\mathfrak{q}_2 = \mathcal{T}_1^{-1} \kappa_{\mathfrak{W}}(q_1) \mathcal{T}_1$, as desired in the statement (1). Knowing that such $\mathcal{T}_1$ should exist, we can apply the Parker's isomorphism test of Proposition 6.1.6 of
\cite{michler} by means of the \textsc{Magma} command
$$\verb"IsIsomorphic(GModule(sub<Y|W(p1),W(q1)>),GModule(sub<Y|V(p2),V(q2)>))",$$
which gives the boolean value, which is $\verb"true"$ in this case, and the desired transformation matrix $\mathcal{T}_1$. So (1) is done.

By assertion (a) and Corollary 7.2.4 of \cite{michler} this
transformation matrix ${\mathcal T}_1$ has to be multiplied by a
diagonal matrix ${\mathcal S}_1$ of $\GL_{352}(17)$. In order to
calculate its entries one has to get the composition factors of
the restrictions ${\chi_i}_{|D_2}$ and ${\tau_j}_{|T}$ to $D_2$
and $T$, respectively. From the fusion and the $3$ character
tables \ref{Fi_23ct_H_2}, \ref{Fi_23ct_D_2} and
\ref{Fi_23ct_D} follows that:
\begin{eqnarray*}
&&\chi_{40}|_{D_2} = \psi_{27} + \psi_{63}, \quad \chi_{41}|_{D_2} = \psi_{26} + \psi_{59},\\
&&\chi_{57}|_{D_2} = \psi_{70} + \psi_{73}, \quad \mbox{and}\\
&&\tau_{9}|_{T} = \psi_{26} + \psi_{63} + \psi_{73}, \quad \tau_{10}|_{T} = \psi_{27} + \psi_{59} + \psi_{70}.
\end{eqnarray*}

Thus, by applying the Meat-axe Algorithm to $\chi_{40}|_{D_2}, \chi_{41}|_{D_2}, \chi_{57}|_{D_2}$, we get the $KD_2$-modules $\mathfrak{U}_{26}, \mathfrak{U}_{27}, \mathfrak{U}_{59}, \mathfrak{U}_{63}, \mathfrak{U}_{70}$, and $\mathfrak{U}_{73}$ of $D_2 \cong T$, such that ${\mathfrak{V}_{40}}_{|D_2} \cong \mathfrak{U}_{27} \oplus \mathfrak{U}_{63}$, ${\mathfrak{V}_{41}}_{|D_2} \cong \mathfrak{U}_{26} \oplus \mathfrak{U}_{59}$, ${\mathfrak{V}_{57}}_{|D_2} \cong \mathfrak{U}_{70} \oplus \mathfrak{U}_{73}$, ${\mathfrak{W}_{9}}_{|T} \cong \mathfrak{U}_{26} \oplus \mathfrak{U}_{63} \oplus \mathfrak{U}_{73}$, and ${\mathfrak{W}_{9}}_{|T} \cong \mathfrak{U}_{26} \oplus \mathfrak{U}_{63} \oplus \mathfrak{U}_{73}$. Then, we know that
\begin{align*}
{\mathfrak{V}}_{|D_2}
\cong \mathfrak{U}_{26} \oplus \mathfrak{U}_{27} \oplus \mathfrak{U}_{59} \oplus \mathfrak{U}_{63} \oplus \mathfrak{U}_{70} \oplus \mathfrak{U}_{73}
\cong {\mathfrak{W}}_{|T}.
\end{align*}
Therefore, there is a transformation matrix $\mathcal{S}_0 \in \GL_{352}(17)$ such that
\begin{align*}
\mathcal{S}_0^{-1} \kappa_\mathfrak{V} (p_2) \mathcal{S}_0
& = \mbox{diag}(\kappa_{\mathfrak{U}_{26}} (p_2), \kappa_{\mathfrak{U}_{27}} (p_2),
\kappa_{\mathfrak{U}_{59}} (p_2),\kappa_{\mathfrak{U}_{63}} (p_2),
\kappa_{\mathfrak{U}_{70}} (p_2),\kappa_{\mathfrak{U}_{73}} (p_2)), \mbox{ and} \\
\mathcal{S}_0^{-1} \kappa_\mathfrak{V} (q_2) \mathcal{S}_0
& = \mbox{diag}(\kappa_{\mathfrak{U}_{26}} (q_2), \kappa_{\mathfrak{U}_{27}} (q_2),
\kappa_{\mathfrak{U}_{59}} (q_2),\kappa_{\mathfrak{U}_{63}} (q_2),
\kappa_{\mathfrak{U}_{70}} (q_2),\kappa_{\mathfrak{U}_{73}} (q_2)),
\end{align*}
where $\kappa_{\mathfrak{U}_{26}} : D_2 \rightarrow \GL_{16}(17)$, $\kappa_{\mathfrak{U}_{27}} : D_2 \rightarrow \GL_{16}(17)$, $\kappa_{\mathfrak{U}_{59}} : D_2 \rightarrow \GL_{80}(17)$,  $\kappa_{\mathfrak{U}_{63}} : D_2 \rightarrow \GL_{80}(17)$, $\kappa_{\mathfrak{U}_{70}} : D_2 \rightarrow \GL_{80}(17)$, and $\kappa_{\mathfrak{U}_{73}} : D_2 \rightarrow \GL_{80}(17)$ are the representations of $D_2$ afforded by the modules $\mathfrak{U}_{26}$, $\mathfrak{U}_{27}$, $\mathfrak{U}_{59}$, $\mathfrak{U}_{63}$, $\mathfrak{U}_{70}$, and $\mathfrak{U}_{73}$, respectively. This matrix $\mathcal{S}_0$ can be obtained by applying Parker's isomorphism test to the two modules $\mathfrak{V}$ and $\mathfrak{U}_{26} \oplus \mathfrak{U}_{27} \oplus \mathfrak{U}_{59} \oplus \mathfrak{U}_{63} \oplus \mathfrak{U}_{70} \oplus \mathfrak{U}_{73}$. We can assume that we started with $\mathcal{S}_0^{-1} \kappa_\mathfrak{V} (p_2) \mathcal{S}_0$, $\mathcal{S}_0^{-1} \kappa_\mathfrak{V} (q_2) \mathcal{S}_0$, and $\mathcal{S}_0^{-1} \kappa_\mathfrak{V} (h_2) \mathcal{S}_0$ as $\mathfrak{p}_2$, $\mathfrak{q}_2$, and $\mathfrak{h}_2$, at the beginning of the statement (c). Hence, $\mathfrak{p}_2$ and $\mathfrak{q}_2$ are now assumed to be in the block diagonal form as follows, from the beginning (then $\mathfrak{h}_2$ would be in a certain block form; although not block diagonal, it still carries the structure of $\mathfrak{V}_{57} \oplus \mathfrak{V}_{40} \oplus \mathfrak{V}_{41}$):
\begin{align*}
\mathfrak{p}_2
& = \mbox{diag}(\kappa_{\mathfrak{U}_{26}} (p_2), \kappa_{\mathfrak{U}_{27}} (p_2),
\kappa_{\mathfrak{U}_{59}} (p_2),\kappa_{\mathfrak{U}_{63}} (p_2),
\kappa_{\mathfrak{U}_{70}} (p_2),\kappa_{\mathfrak{U}_{73}} (p_2)), \mbox{ and} \\
\mathfrak{q}_2
& = \mbox{diag}(\kappa_{\mathfrak{U}_{26}} (q_2), \kappa_{\mathfrak{U}_{27}} (q_2),
\kappa_{\mathfrak{U}_{59}} (q_2),\kappa_{\mathfrak{U}_{63}} (q_2),
\kappa_{\mathfrak{U}_{70}} (q_2),\kappa_{\mathfrak{U}_{73}} (q_2)).
\end{align*}


Now, by Schur's Lemma (Lemma 2.1.8 of \cite{michler}) and the degrees of these characters appearing in the restriction pattern, the following
linear system in the variables $\sigma_a, \sigma_b, \sigma_c, \sigma_d, \sigma_e, \sigma_f \in K$ holds, where the variables $\sigma_a, \sigma_b, \sigma_c, \sigma_d, \sigma_e, \sigma_f$ correspond to $\psi_{26}, \psi_{27}, \psi_{59}, \psi_{63}, \psi_{70}, \psi_{73}$, respectively :
\begin{eqnarray*}
&&96 = 16\sigma_b + 80\sigma_d,\quad 96 = 16\sigma_a + 80\sigma_c,\quad 160 = 80\sigma_e + 80\sigma_f \quad \mbox{and}\\
&&176 = 16\sigma_a + 80\sigma_d + 80\sigma_f, \quad 176 = 16\sigma_b + 80\sigma_c + 80\sigma_e.\\
\end{eqnarray*}
This system of equations in $K$ has the solution: $\sigma_c = 8-7\sigma_a$, $\sigma_d = 8-7\sigma_b$, $\sigma_e = 7\sigma_a - 7\sigma_b + 1$, $\sigma_f = -7\sigma_a + 7\sigma_b + 1$, where $\sigma_a$ and $\sigma_b$ run through all nonzero elements of $K$. Hence the diagonal matrix ${\mathcal S}_1$ has the form
$$\mathcal S_{(\sigma_a,\sigma_b)} = diag(\sigma_a^{16},\sigma_b^{16},[8-7\sigma_a]^{80},[8-7\sigma_b]^{80},[7\sigma_a-7\sigma_b+1]^{80},[-7\sigma_a+7\sigma_b+1]^{80})$$
for suitable elements $\sigma_a,\sigma_b \in K$. Recall that $\mathfrak{p}_2$ and $\mathfrak{q}_2$ are in the block diagonal form as described above, where the sizes of the blocks are $16, 16, 80, 80, 80$, and $80$, in this order. Therefore, for any $\sigma_a$ and $\sigma_b$ in $K$, the diagonal matrix $\mathcal S_{(\sigma_a,\sigma_b)}$ centralizes the two matrices $\mathfrak{p}_2$ and $\mathfrak{q}_2$.

Let
\begin{align*}
\mathfrak x_{(\sigma_b,\sigma_d)} & = \left( \mathcal T_1\mathcal S_{(\sigma_b,\sigma_d)} \right)^{-1} \kappa_{\mathfrak W}(x_1) \mathcal T_1\mathcal S_{(\sigma_b,\sigma_d)}, \quad \mbox{and} \\
\mathfrak y_{(\sigma_b,\sigma_d)} & = \left( \mathcal T_1\mathcal S_{(\sigma_b,\sigma_d)} \right)^{-1} \kappa_{\mathfrak W}(y_1)\mathcal T_1\mathcal S_{(\sigma_b,\sigma_d)}.
\end{align*}
Then the field elements $\sigma_a$, $\sigma_b$ have to be chosen so that all the 
entries on the main diagonal of the matrix $\mathcal S_{(\sigma_a,\sigma_b)}$ are nonzero and
that the matrix group
$$\mathfrak H_{\sigma_a,\sigma_b} = \langle \mathfrak p_2, \mathfrak q_2, \mathfrak x_{\sigma_a,\sigma_b}, \mathfrak y_{\sigma_a,\sigma_b}, \mathfrak h_2\rangle$$
satisfies the Sylow $2$-subgroup test of Algorithm 7.4.8 Step 5(c)
of \cite{michler}. The test used here is the order of $\mathfrak x_{\sigma_a,\sigma_b} \mathfrak h_2$; the order has to be the order of an element in $2\Fi_{22}$. Running through all possible pairs $(\sigma_a,\sigma_b) \in F^2$ it follows that this test is only successful for the pair $(\sigma_a,\sigma_b) = (15,9)$, giving the order of $\mathfrak x_{\sigma_a,\sigma_b} \mathfrak h_2$ equal to $12$.

Now, let $\mathcal{S}_1 = \mathcal{S}_{(15,9)}$, $\mathfrak x_{2} = \left( \mathcal T_1\mathcal S_{1} \right)^{-1} \kappa_{\mathfrak W}(x_1) \mathcal T_1\mathcal S_{1}$, $\mathfrak y_{2} = \left( \mathcal T_1\mathcal S_{1} \right)^{-1} \kappa_{\mathfrak W}(y_1)\mathcal T_1\mathcal S_{1}$, and $\mathfrak H = \langle \mathfrak p_2, \mathfrak q_2, \mathfrak x_{2}, \mathfrak y_{2}, \mathfrak h_2\rangle$. Then we have $\mathfrak{p}_2 = \mathcal{S}_1^{-1} \mathfrak{p}_2  \mathcal{S}_1$, $\mathfrak{q}_2 = \mathcal{S}_1^{-1} \mathfrak{q}_2  \mathcal{S}_1$, since any $\mathcal S_{(\sigma_a,\sigma_b)}$ centralizes the two matrices $\mathfrak{p}_2$ and $\mathfrak{q}_2$. This matrix group $\mathfrak H$ is a candidate which can satisfy the Sylow $2$-subgroup test. In Theorem \ref{thm. H(Fi_23)_embed}, it is shown that $\mathfrak{H}$ indeed satisfies the Sylow $2$-subgroup test. So (2) is done.

Finally, since $\mathfrak{p}_2$ and $\mathfrak{q}_2$ can be expressed as words in terms of $\mathfrak{x}_2$ and $\mathfrak{y}_2$, we have $\mathfrak{H} = \langle \mathfrak{p}_2, \mathfrak{q}_2, \mathfrak{x}_2, \mathfrak{y}_2, \mathfrak{h}_2 \rangle = \langle \mathfrak{x}_2, \mathfrak{y}_2, \mathfrak{h}_2 \rangle$.

For (3), the matrices are stated in \cite{kim1}.

(d) Same idea and process as in (c) are applied to the compatible pair $\chi', \tau'$. As done in (2) of (c), a suitable transformation matrix has to be found. However, in this case, there is no such transformation matrix which makes the final matrix group to satisfy the Sylow $2$-subgroup test of Algorithm 7.4.8 of Step 5(c) of \cite{michler}. In other words, there was no $\mathcal{S}_1$ which makes the order of $\mathfrak{x}_2 \mathfrak{h}_2$ to be the order of an element in $2\Fi_{22}$.

\end{proof}


\begin{lemma}\label{l. H(Fi_23)_niceFi_22} Using a very nice presentation for Fischer's simple group $\Fi_{22}$ which is stated by Praeger and Soicher in \cite{praeger}, the finitely presented group $G_n$ $=$ $\langle$$a_n$,$b_n$,$c_n$,\\$d_n$,$e_n$,$f_n$,$g_n$,$h_n$,$i_n \rangle$ (here the subscript $n$ stands for ``nice") with set $\mathcal R(G_n)$ of
defining relations
\begin{eqnarray*}
&&a_n^2 = b_n^2 = c_n^2 = d_n^2 =  e_n^2 = f_n^2 = g_n^2 = h_n^2 = i_n^2 = 1,\\
&&(a_nb_n)^3 = 1, (b_nc_n)^3 = (c_nd_n)^3 = (d_ne_n)^3 = (e_nf_n)^3 = (f_ng_n)^3 = 1, \\
&&(a_nc_n)^2 = (a_nd_n)^2 = (a_ne_n)^2 = (a_nf_n)^2 = (a_ng_n)^2 = (a_nh_n)^2 = (a_ni_n)^2 = 1,\\
&&(b_nd_n)^2 = (b_ne_n)^2 = (b_nf_n)^2 = (b_ng_n)^2 = (b_nh_n)^2 = (b_ni_n)^2 = 1,\\
&&(c_ne_n)^2 = (c_nf_n)^2 = (c_ng_n)^2 = (c_nh_n)^2 = (c_ni_n)^2 = 1,\\
&&(d_nf_n)^2 = (d_ng_n)^2 = (e_ng_n)^2 = (e_nh_n)^2 = (e_ni_n)^2 = 1,\\
&& (d_nh_n)^3 = (h_ni_n)^3 = (d_ni_n)^2 = (f_nh_n)^2 = (f_ni_n)^2 = (g_nh_n)^2 = (g_ni_n)^2 = 1,\\
&&(d_nc_nb_nd_ne_nf_nd_nh_ni_n)^{10} = (a_nb_nc_nd_ne_nf_nh_n)^9 = (b_nc_nd_ne_nf_ng_nh_n)^9 = 1
\end{eqnarray*}
satisfies the following properties:

\begin{enumerate}

\item[\rm(a)] As in page 110 of \cite{praeger}, $G_n$ has a subgroup
\begin{align*}
U_n = \langle a_n,c_n,d_n,e_n,f_n,g_n,h_n,i_n,(a_nb_nc_nd_ne_nh_n)^5 \rangle
\end{align*}
which is isomorphic to $2.U_6(2)$. $G_n$ has a faithful permutation representation $PG_n$ of degree $3510$ with permutation stabilizer $U_n$.

\item[\rm(b)] As in page 110 of \cite{praeger}, $G_n$ has a subgroup
\begin{align*}
V_n = \langle b_n,c_n,d_n,e_n,f_n,g_n,h_n,i_n \rangle
\end{align*}
which is isomorphic to the (orthogonal) simple group $O_7(3)$. $G_n$ has a faithful permutation representation $(PG_n)'$ of degree $14080$ with permutation stabilizer $V_n$.

\item[\rm(c)] As in page 111 of \cite{praeger}, $G_n$ has a subgroup
\begin{align*}
E_n = \langle a_n,c_n,e_n,g_n,h_n,u_n,v_n,w_n,x_n,y_n,t_n \rangle
\end{align*}
which is isomorphic to $2^{10}:\M_{22}$, where
\begin{align*}
u_n = b_n a_n c_n b_n, \quad
v_n = d_n c_n e_n d_n, \quad
w_n = d_n e_n h_n d_n, \\
x_n = f_n e_n g_n f_n, \quad
y_n = d_n c_n h_n d_n, \quad
t_n = (c_n d_n e_n h_n i_n)^4.
\end{align*}
In particular, $E_n$ has order $2^{17}\cdot 3^{2} \cdot 5 \cdot 7 \cdot 11$.

\item[\rm(d)] $E_n$ has exactly one conjugacy class of $2$-central involutions, represented by $z_n = a_n c_n$.

\item[\rm(e)] The centralizer $H_n = C_{G_n}(z_n)$ of $z_n$ in $G_n$ has order $2^{17} \cdot 3^4 \cdot 5$, and $H_n = \langle f_n, g_n, i_n, (a_n b_n c_n)^2, (c_n d_n e_n h_n)^3 \rangle$.

\item[\rm(f)] The intersection $D_n = E_n \cap H_n$ of $E_n$ and $H_n$ has order $2^{17} \cdot 3 \cdot 5$.

\end{enumerate}

\end{lemma}

\begin{proof}

(a) In \cite{praeger}, the presentation is stated without the subscripts $n$. In page 110 of \cite{praeger}, it is stated that $G_n \cong \Fi_{22}$ and that $G_n$ has a subgroup $U_n \cong 2.U_6(2)$, with generators as written in the statement. By means of \textsc{Magma}'s command $\verb"CosetAction(G_n,U_n)"$, we can check that $G_n$ has a permutation representation $PG_n$ of degree $3510$ with stabilizer $U_n$, with order $|PG_n| = 2^{17} \cdot 3^9 \cdot 5^2 \cdot 7 \cdot 11 \cdot 13$. As proved in \cite{kim}, we have $G_n \cong \Fi_{22}$, and therefore $PG_n$ is a faithful permutation representation of $G_n$.

(b) In page 110 of \cite{praeger}, it is stated that $G_n \cong \Fi_{22}$ has a subgroup $V_n \cong O_7(3)$, with generators as written in the statement. By means of \textsc{Magma}'s command $\verb"CosetAction(G_n,V_n)"$, we can check that $G_n$ has a permutation representation $(PG_n)'$ of degree $14080$ with stabilizer $V_n$.

The statement (c) is as written in page 111 of \cite{praeger}.

(d) The \textsc{Magma} command $\verb"Classes(E_n)"$ tells us that $E_n$ has exactly one conjugacy class of $2$-central involutions, and it also gives a representative $z_n$ for the class. Now, the command $\verb"LookupWord(E_n, z_n : ConjugateCheck:=true)"$ in \textsc{Magma} gives an answer $a_n c_n$, which is conjugate to $z_n$. So, we can redefine $z_n$ to be $z_n = a_n c_n$, and use it as the representative for this class of $2$-central involutions.

(e) Let $z_n = a_n c_n$ as in (d). Then, the group $H_n = C_{G_n}(z_n)$ is verified by means of \textsc{Magma} to have order $2^{17} \cdot 3^4 \cdot 5$. Now, by applying $\verb"GetShortGens"$ the author obtained the five generators $f_n, g_n, i_n, (a_n b_n c_n)^2, (c_n d_n e_n h_n)^3$ for $H_n$.

(f) The statement is verified by means of \textsc{Magma}.

\end{proof}

In the following theorem the author obtains the presentation for $2\Fi_{22}$ as generators and relations, and establishes an isomorphism from $\mathfrak{H}$ to $2\Fi_{22}$.


\begin{theorem}\label{thm. H(Fi_23)_embed} Keep the notations of Lemma \ref{l. H(Fi_23)_niceFi_22} and Proposition \ref{prop. H(Fi_23)_1} and \ref{prop. H(Fi_23)_2}. The following statements hold:

\begin{enumerate}
\item[\rm(a)] $D/\langle z_1 \rangle = D_1 \cong E_n$ and $H_2 /\langle k_{15} \rangle = H_1 \cong H_n$.

\item[\rm(b)] Let $\phi_1 : D \rightarrow D/\langle z_1 \rangle = D_1$ and $\phi_2 : H_2 \rightarrow H_2/\langle k_{15} \rangle = H_1$ be the canonical epimorphisms. There exist an element $r_2 \in G_n$ and isomorphisms $\varphi_1 : D_1 \rightarrow E_n$ and $\varphi_2 : H_1 \rightarrow (H_n)^{r_2}$ such that $(\varphi_1 \circ \phi_1) (p_1) = p_3 = (\varphi_2 \circ \phi_2) (p_2) $ and $(\varphi_1 \circ \phi_1) (q_1) = q_3 = (\varphi_2 \circ \phi_2) (q_2)$, and $E_n \cap (H_n)^{r_2} = \langle p_3, q_3 \rangle$ has order $2^{17} \cdot 3 \cdot 5$. 

\item[\rm(c)] Let $x_3 = (\varphi_1 \circ \phi_1) (x_1)$, $y_3 = (\varphi_1 \circ \phi_1) (y_1)$, and $h_3 = (\varphi_2 \circ \phi_2) (h_2)$. Then $x_3,y_3$ and $h_3$ can be expressed by words in $a_n,b_n,c_n,d_n,e_n,f_n,g_n,h_n,i_n$, as follows:
\begin{align*}
x_3 & = \beta_1 (a_n,b_n,c_n,d_n,e_n,f_n,g_n,h_n,i_n) = (r_{n,1} r_{n,2} r_{n,3})^3,\\
y_3 & = \beta_2 (a_n,b_n,c_n,d_n,e_n,f_n,g_n,h_n,i_n) \\
& = (s_{n,1} s_{n,2} s_{n,4} s_{n,1} s_{n,4} s_{n,2} s_{n,4})^5
(t_{n,1} t_{n,3}^2 t_{n,1} t_{n,3} t_{n,1} t_{n,2} t_{n,1} t_{n,3})^{20}, \\
h_3 & = \beta_3 (a_n,b_n,c_n,d_n,e_n,f_n,g_n,h_n,i_n) = (v_{n,1} v_{n,2})^4,
\end{align*}
where
\begin{align*}
r_{n,1} & = (w_n u_n t_n u_n x_n t_n x_n t_n)^4, \quad
r_{n,2} = (w_n t_n w_n u_n w_n x_n t_n x_n)^4, \\
r_{n,3} & = (x_n w_n t_n u_n x_n w_n t_n x_n)^2, \quad
r_{n,4} = (w_n u_n t_n x_n w_n t_n u_n x_n t_n)^4, \\
s_{n,1} & = u_n w_n u_n, \quad
s_{n,2} = (u_n w_n x_n)^2, \quad
s_{n,3} = (u_n w_n t_n u_n)^2, \\
s_{n,4} & = (u_n t_n w_n x_n)^4, \quad
t_{n,1} = s_{n,2} s_{n,4} s_{n,2}^2, \\
t_{n,2} & = (s_{n,4} s_{n,2} s_{n,1} s_{n,3} s_{n,4})^2, \quad
t_{n,3} = (s_{n,4} s_{n,2} s_{n,4} s_{n,2} s_{n,1} s_{n,3})^2,\\
j_n & = (c_n d_n e_n h_n)^3, \quad
k_n = (c_n d_n e_n f_n g_n)^2, \quad
l_n = (a_n b_n c_n d_n e_n h_n)^5, \\
o_n & = (l_n b_n k_n j_n b_n i_n)^6, \quad
p_n = (j_n k_n j_n l_n b_n i_n j_n i_n)^5, \\
v_{n,1} & = (p_n o_n j_n o_n k_n)^4, \quad
v_{n,2} = (j_n k_n o_n k_n^2 p_n)^4, \quad
v_{n,3} = k_n j_n o_n j_n k_n^2 p_n.
\end{align*}

\item[\rm(d)] $G_n = \langle x_3, y_3, h_3 \rangle$.

\item[\rm(e)] $a_n,b_n,c_n,d_n,e_n,f_n,g_n,h_n,i_n$ can be expressed as words in $x_3,y_3,h_3$, denoted by
\begin{align*}
a_n = \omega_a (x_3, y_3, h_3), \quad
b_n = \omega_b(x_3, y_3, h_3), \quad
c_n = \omega_c (x_3, y_3, h_3), \\
d_n = \omega_d (x_3, y_3, h_3), \quad
e_n = \omega_e(x_3, y_3, h_3), \quad
f_n = \omega_f (x_3, y_3, h_3), \\
g_n = \omega_g (x_3, y_3, h_3), \quad
h_n = \omega_h(x_3, y_3, h_3), \quad
i_n = \omega_i (x_3, y_3, h_3).
\end{align*}
The explicit words are
\begin{align*}
\omega_a (x_3, y_3, h_3) & = (x_3 y_3 x_3)^7, \quad
\omega_b (x_3, y_3, h_3) = (h_3 y_3 h_3 y_3 h_3 y_3^3 h_3^2 y_3^3 h_3 y_3)^7, \\
\omega_c (x_3, y_3, h_3) & = (y_3^2 x_3 y_3 x_3 y_3^3)^5, \\
\omega_d (x_3, y_3, h_3) & = (h_3 y_3 h_3^2 y_3 h_3 y_3 h_3 y_3 h_3 y_3^2 h_3^2)^{15}, \\
\omega_e (x_3, y_3, h_3) & = (y_3 x_3 y_3^5 x_3)^5, \\
\omega_f (x_3, y_3, h_3) & = (y_3 h_3 y_3 h_3^2 y_3 h_3^2 y_3^2 h_3 y_3^4 h_3^2)^5, \\
\omega_g (x_3, y_3, h_3) & = (x_3 y_3^2 x_3 y_3^3 x_3)^7, \quad
\omega_h (x_3, y_3, h_3) = (y_3^5 x_3 y_3 x_3)^5, \\
\omega_i (x_3, y_3, h_3) & = (h_3^2 y_3^2 h_3 y_3 h_3^2)^7.
\end{align*}

\item[\rm(f)] The finitely presented group $G_\ell = \langle a_\ell,b_\ell,c_\ell,d_\ell,e_\ell,f_\ell,g_\ell,h_\ell,i_\ell,z_\ell \rangle$ (here $\ell$ stands for ``lifted") with set $\mathcal R(G_\ell)$ of defining relations
\begin{eqnarray*}
&&a_\ell^2 = b_\ell^2 = c_\ell^2 = d_\ell^2 =  e_\ell^2 = f_\ell^2 = g_\ell^2 = h_\ell^2 = i_\ell^2 = 1,\\
&&(a_\ell b_\ell)^3 = 1, (b_\ell c_\ell)^3 = z_\ell, (c_\ell d_\ell)^3 = (d_\ell e_\ell)^3 = 1, (e_\ell f_\ell)^3 = (f_\ell g_\ell)^3 = z_\ell, \\
&&(a_\ell c_\ell)^2 = (a_\ell d_\ell)^2 = (a_\ell e_\ell)^2 = (a_\ell f_\ell)^2 = (a_\ell g_\ell)^2 = (a_\ell h_\ell)^2 = (a_\ell i_\ell)^2 = 1,\\
&&(b_\ell d_\ell)^2 = (b_\ell e_\ell)^2 = (b_\ell f_\ell)^2 = (b_\ell g_\ell)^2 = (b_\ell h_\ell)^2 = (b_\ell i_\ell)^2 = 1,\\
&&(c_\ell e_\ell)^2 = (c_\ell f_\ell)^2 = (c_\ell g_\ell)^2 = (c_\ell h_\ell)^2 = (c_\ell i_\ell)^2 = 1,\\
&&(d_\ell f_\ell)^2 = (d_\ell g_\ell)^2 = (e_\ell g_\ell)^2 = (e_\ell h_\ell)^2 = (e_\ell i_\ell)^2 = 1,\\
&& (d_\ell h_\ell)^3 = (h_\ell i_\ell)^3 = (d_\ell i_\ell)^2 = (f_\ell h_\ell)^2 = (f_\ell i_\ell)^2 = (g_\ell h_\ell)^2 = (g_\ell i_\ell)^2 = 1,\\
&&(d_\ell c_\ell b_\ell d_\ell e_\ell f_\ell d_\ell h_\ell i_\ell)^{10} = (a_\ell b_\ell c_\ell d_\ell e_\ell f_\ell h_\ell)^9 = (b_\ell c_\ell d_\ell e_\ell f_\ell g_\ell h_\ell)^9 = 1,\\
&& z_\ell^2 = (z_\ell,a_\ell) = (z_\ell,b_\ell) = (z_\ell,c_\ell)= (z_\ell,d_\ell)=(z_\ell,e_\ell)=(z_\ell,f_\ell)= 1,\\
&& (z_\ell,g_\ell)=(z_\ell,h_\ell)=(z_\ell,i_\ell)=1
\end{eqnarray*}
has a faithful permutation representation $PG_\ell$ of degree $28160$, which is isomorphic to $2\Fi_{22}$, having stabilizer $\langle b_\ell z_\ell, c_\ell, d_\ell, e_\ell, f_\ell z_\ell, g_\ell,h_\ell,i_\ell \rangle$.

\item[\rm(g)] Let
\begin{align*}
x_0 = \beta_1 (a_\ell,b_\ell,c_\ell,d_\ell,e_\ell,f_\ell,g_\ell,h_\ell,i_\ell), \\
y_0 = \beta_2 (a_\ell,b_\ell,c_\ell,d_\ell,e_\ell,f_\ell,g_\ell,h_\ell,i_\ell), \\
h_0 = \beta_3 (a_\ell,b_\ell,c_\ell,d_\ell,e_\ell,f_\ell,g_\ell,h_\ell,i_\ell),
\end{align*}
where the three words $\beta_1, \beta_2$, and $\beta_3$ are the ones obtained in the statement (b). Then $G_\ell = \langle x_0, y_0, h_0 \rangle$.

\item[\rm(h)] There is an isomorphism $\vartheta : \mathfrak{H} \rightarrow G_\ell$ such that $\vartheta(\mathfrak{x}_2) = x_0, \vartheta(\mathfrak{y}_2) = y_0$, and $\vartheta(\mathfrak{h}_2) = h_0$, and $\mathfrak{H}$ is an irreducible subgroup of $\GL_{352}(17)$.

\end{enumerate}

\end{theorem}

\begin{proof}

(a) The presentation of $D_1 = D/\langle z_1\rangle$ can be obtained from the presentation of $D$, as stated in Proposition \ref{prop. H(Fi_23)_1}(e). Let $\phi_1 : D \rightarrow D/\langle z \rangle = D_1$  be the canonical epimorphism. 
Then, by trying some random short words, we get that a subgroup $\langle \phi_1(y_1)^4, (\phi_1(x_1)  \phi_1(y_1) \phi_1(v_1))^2 \rangle \cong \M_{22}$ works as a permutation stabilizer for the group $D_1$, and thus the \textsc{Magma} command $$\verb"CosetAction(D_1,<phi_1(y_1)^4,(phi_1(x_1) phi_1(y_1) phi_1(v_1))^2>"$$ gives the faithful permutation representation of $D_1$ of degree $1024$. We can now use this faithful permutation representation for all \textsc{Magma} computations in this theorem. From Lemma 3.4 of \cite{kim} we already have a presentation and faithful permutation representation of $H_1 = H_2 / \langle k_{15}\rangle$ of degree $1024$. Then, using the permutation representations, it is verified with \textsc{Magma} commands $\verb"IsIsomorphic(D_1,E_n)"$ and $\verb"IsIsomorphic(H_1,H_n)"$ that $D_1 \cong E_n$ and $H_1 \cong H_n$.

(b) Let $\phi_1 : D \rightarrow D/\langle z_1 \rangle = D_1$ and $\phi_2 : H_2 \rightarrow H_2/\langle k_{15}\rangle = H_1$ be the canonical epimorphisms. As proved in (a), there exist isomorphisms $\varphi_1 : D_1 \rightarrow E_n$ and $\varphi_0 : H_1 \rightarrow H_n$. Next step is to find some inner automorphism $\delta_0 \in \mbox{Aut}(G_n)$ of $G_n$ such that the new isomorphism $\varphi_2 = \delta_0 \circ \varphi_0 : H_1 \rightarrow \delta_0(H_n)$ satisfies $(\varphi_1 \circ \phi_1) (p_1) = (\varphi_2 \circ \phi_2) (p_2) $ and $(\varphi_1 \circ \phi_1) (q_1) = (\varphi_2 \circ \phi_2) (q_2)$. Let $\varphi_1$ and $\varphi_0$ be fixed. For convenience, let $p_3 = (\varphi_1 \circ \phi_1) (p_1)$, $q_3 = (\varphi_1 \circ \phi_1) (q_1)$, $p'_3 = (\varphi_0 \circ \phi_2) (p_2)$, and $q'_3 = (\varphi_0 \circ \phi_2) (q_2)$.

By means of \textsc{Magma} command $\verb"IsConjugate(G_n,p'_3,p_3)"$, we obtain an element $r_1 \in G_n$ such that $r_1^{-1} p'_3 r_1 =p_3$. Fix this $r_1$. By means of \textsc{Magma}, it is checked that $C_{G_n}(p'_3)$ has order $3072$, which is relatively small. Then it is verified by the \textsc{Magma} command $\verb"exists"$ that there is some element $r_3$ of $C_{G_n}(p'_3)$ such that $(r_3 r_1)^{-1} q'_3 (r_3 r_1) =q_3$.

Thus we also have $(r_3 r_1)^{-1} p'_3 (r_3 r_1) = r_1^{-1} (r_3^{-1} p'_3 r_3) r_1 = r_1^{-1} p'_3 r_1 = p_3$. Therefore, letting $r_2 = r_3r_1$ yields $r_2^{-1} p'_3 r_2 = p_3$ and $r_2^{-1} q'_3 r_2 = q_3$. Now, let $\delta_0 \in \mbox{Aut}(G_n)$ be the inner automorphism defined as conjugation by $r_2$. Then $\delta_0$ restricted to $H_n$ is an isomorphism ${\delta_0}_{|H_n} : H_n \rightarrow (H_n)^{r_2}$. Thus, $\varphi_2 : H_1 \rightarrow (H_n)^{r_2}$ defined by $\varphi_2 = {\delta_0}_{|H_n} \circ \varphi_0$ is an isomorphism, and we can observe that
\begin{align*}
(\varphi_2 \circ \phi_2)(p_2)
& = ({\delta_0}_{|H_n} \circ \varphi_0 \circ \phi_2)(p_2)
= {\delta_0}_{|H_n} (p'_3) = r_2^{-1} p'_3 r_2 = p_3, \quad \mbox{and} \\
(\varphi_2 \circ \phi_2)(q_2)
& = ({\delta_0}_{|H_n} \circ \varphi_0 \circ \phi_2)(q_2)
= {\delta_0}_{|H_n} (q'_3) = r_2^{-1} q'_3 r_2 = q_3,
\end{align*}
and therefore $(\varphi_1 \circ \phi_1) (p_1) = p_3 = (\varphi_2 \circ \phi_2) (p_2) $ and $(\varphi_1 \circ \phi_1) (q_1) = q_3 = (\varphi_2 \circ \phi_2) (q_2)$, in short. It is checked by means of \textsc{Magma} that $E_n \cap (H_n)^{r_2} = \langle p_3, q_3 \rangle$, and that $\langle p_3, q_3 \rangle$ is of order $2^{17} \cdot 3 \cdot 5$.

(c) Let $\phi_1, \phi_2, \varphi_1, \varphi_2$, and $r_2$ be as obtained in (b), so that the isomorphisms $\varphi_1 : D_1 \rightarrow E_n$ and $\varphi_2 : H_1 \rightarrow (H_n)^{r_2}$ satisfy $(\varphi_1 \circ \phi_1) (p_1) = p_3 = (\varphi_2 \circ \phi_2) (p_2) $ and $(\varphi_1 \circ \phi_1) (q_1) = q_3 = (\varphi_2 \circ \phi_2) (q_2)$ and $E_n \cap (H_n)^{r_2} = \langle p_3, q_3 \rangle$. Now, let $x_3 = (\varphi_1 \circ \phi_1) (x_1)$, $y_3 = (\varphi_1 \circ \phi_1) (y_1)$, and $h_3 = (\varphi_2 \circ \phi_2) (h_2)$. Then $x_3,y_3$ and $h_3$ can be expressed as words in $a_n,b_n,c_n,d_n,e_n,f_n,g_n,h_n$,and $i_n$. Thus there are three words by $\beta_1, \beta_2$, and $\beta_3$, such that
\begin{align*}
x_3 = \beta_1 (a_n,b_n,c_n,d_n,e_n,f_n,g_n,h_n,i_n), \\
y_3 = \beta_2 (a_n,b_n,c_n,d_n,e_n,f_n,g_n,h_n,i_n), \\
h_3 = \beta_3 (a_n,b_n,c_n,d_n,e_n,f_n,g_n,h_n,i_n).
\end{align*}
The three words $\beta_1, \beta_2$, and $\beta_3$ are obtained by Strategy \ref{str. LookupWord}, as follows.

It is checked by means of \textsc{Magma} that $C_{E_n}(x_3)$ of $x_3$ in $E_n$ has order $2^{13} \cdot 3 = 24576$, which is relatively small. Applying Strategy \ref{str. LookupWord} to this group $C_{E_n}(x_3)$, the author found a word for $x_3$; the results are $C_{E_n}(x_3) = \langle r_{n,1},r_{n,2},r_{n,3},r_{n,4} \rangle$ and $x_3 = (r_{n,1} r_{n,2} r_{n,3})^3$.

Note that $y_3$ is of order $14$ (since $y_1$ is). By means of \textsc{Magma}, it can be checked that the centralizer $C_{E_n}(y_3^7)$ of the involution $y_3^7$ in $E_n$ has order $2^{16} \cdot 3^2 \cdot 5 \cdot 7$. Applying Strategy \ref{str. LookupWord} to this group $C_{E_n}(y_3^7)$, the author found a word for $y_3^7$; the results are $C_{E_n}(y_3^7) = \langle s_{n,1},s_{n,2},s_{n,3},s_{n,4} \rangle$ and $y_3^7 = (s_{n,1} s_{n,2} s_{n,4} s_{n,1} s_{n,4} s_{n,2} s_{n,4})^5$. Now, by means of the \textsc{Magma} command $$\verb"Subgroups(sub<E_n|s_{n,1},s_{n,2},s_{n,3},s_{n,4}> :Al:=Normal")$$ it can be checked that $C_{E_n}(y_3^7) = \langle E_n \mid s_{n,1},s_{n,2},s_{n,3},s_{n,4}\rangle$ has a unique normal subgroup $V_n$ of order $2^{10}$. By means of the \textsc{Magma} command $$\verb"HasComplement(C_{E_n}(y_3^7),V_n)"$$ we can get a complement $C_n$ of $V_n$ in $C_{E_n}(y_3^7)$. Now, by means of the \textsc{Magma} command $$\verb"exists(r_n){r_n:r_n in C_{E_n}(y_3^7)| y_3^2 in (C_n)^{r_2}}"$$ we obtain some element $r_n \in C_{E_n}(y_3^7)$ such that $y_3^2 \in (C_n)^{r_2}$. It is also checked be means of \textsc{Magma} that $(C_n)^{r_2}$ has order $2^6 \cdot 3^2 \cdot 5 \cdot 7 = 20160$. Now, for such $C_n$ and $r_2$, the author applied Strategy \ref{str. LookupWord} for $(C_n)^{r_2}$ and found a word for $y_3^2$; the results are $(C_n)^{r_2} = \langle t_{n,1},t_{n,2},t_{n,3} \rangle$ and $y_3^2 = (t_{n,1} t_{n,3}^2 t_{n,1} t_{n,3} t_{n,1} t_{n,2} t_{n,1} t_{n,3})^5$. Finally, since $y_3$ is of order $14$, we get that
\begin{align*}
y_3 & = y_3^{15} = (y_3^{7})(y_3^{2})^{4} \\
& = (s_{n,1} s_{n,2} s_{n,4} s_{n,1} s_{n,4} s_{n,2} s_{n,4})^5
(t_{n,1} t_{n,3}^2 t_{n,1} t_{n,3} t_{n,1} t_{n,2} t_{n,1} t_{n,3})^{20},
\end{align*}
so we found a word $\beta_2$.

The application of Strategy \ref{str. LookupWord} was not so immediate. The command \\ $\verb"GetShortGens"$ for $(H_n)^{r_2}$ was stopped in the middle by the author, until it gave three elements $j_n, k_n$, and $l_n$ of $(H_n)^{r_2}$. Now, another application of $\verb"GetShortGens"$ using $j_n,k_n,l_n$ and the original generators for $G_n$ yielded $(H_n)^{r_2} = \langle j_n, k_n, o_n, p_n \rangle$. Note that $h_3$ is of order $3$ (since $h_2$ is). By means of \textsc{Magma}, it can be checked that the normalizer $N_{(H_n)^{r_2}}(\langle h_3 \rangle)$ of $h_3$ in $(H_n)^{r_2}$ has order $2^8 \cdot 3^4 = 20736$, which is relatively small. Using Strategy \ref{str. LookupWord} for $N_{(H_n)^{r_2}}(\langle h_3 \rangle)$ in $(H_n)^{r_2}$ yielded $N_{(H_n)^{r_2}}(\langle h_3 \rangle) = \langle v_{n,1}, v_{n,2}, v_{n,3} \rangle$ and finally $h_3 = (v_{n,1} v_{n,2})^4$.

(d) This can easily be checked by means of \textsc{Magma}.

(e) The $9$ words can be obtained by the command $\verb"LookupWord"$.

(f) The objective of this theorem is to verify that the matrix group $\mathfrak{H}$ is isomorphic to $2\Fi_{22}$, and also get a nice presentation for it. Thus $\mathfrak{H}$ is expected to have $\Fi_{22}$ as a quotient. In (c) and (e) we had $G_n = \langle a_n,b_n,c_n,d_n,e_n,f_n,g_n,h_n,i_n \rangle = \langle x_3, y_3, h_3 \rangle \cong \Fi_{22}$, with words for $x_3, y_3, h_3$ in terms of $a_n,b_n,c_n,d_n,e_n,f_n,g_n,h_n$, $i_n$, and vice versa (words for $a_n,b_n,\ldots,i_n$ in terms of $x_3, y_3,h_3$). There, the two quotient groups $D_1 = D/\langle z_1 \rangle$ and $H_1 = H_2/\langle k_{15} \rangle$ are embedded in $G_n \cong \Fi_{22}$ as $\langle p_3, q_3, x_3, y_3 \rangle = \langle x_3, y_3 \rangle$ and $\langle p_3, q_3, h_3 \rangle$, respectively. Recall that $D$ and $H_2$ are embedded in the matrix group $\mathfrak{H}$ as $\langle \mathfrak{p}_2, \mathfrak{q}_2, \mathfrak{x}_2, \mathfrak{y}_2 \rangle$ and $\langle \mathfrak{p}_2, \mathfrak{q}_2, \mathfrak{h}_2 \rangle$, respectively. 

Let
\begin{align*}
\mathfrak{a}_n = \omega_a (\mathfrak{x}_2, \mathfrak{y}_2, \mathfrak{h}_2), \quad
\mathfrak{b}_n = \omega_b (\mathfrak{x}_2, \mathfrak{y}_2, \mathfrak{h}_2), \quad
\mathfrak{c}_n = \omega_c (\mathfrak{x}_2, \mathfrak{y}_2, \mathfrak{h}_2), \\
\mathfrak{d}_n = \omega_d (\mathfrak{x}_2, \mathfrak{y}_2, \mathfrak{h}_2), \quad
\mathfrak{e}_n = \omega_e (\mathfrak{x}_2, \mathfrak{y}_2, \mathfrak{h}_2), \quad
\mathfrak{f}_n = \omega_f (\mathfrak{x}_2, \mathfrak{y}_2, \mathfrak{h}_2), \\
\mathfrak{g}_n = \omega_g (\mathfrak{x}_2, \mathfrak{y}_2, \mathfrak{h}_2), \quad
\mathfrak{h}_n = \omega_h (\mathfrak{x}_2, \mathfrak{y}_2, \mathfrak{h}_2), \quad
\mathfrak{i}_n = \omega_i (\mathfrak{x}_2, \mathfrak{y}_2, \mathfrak{h}_2),
\end{align*}
where the $9$ words $\omega_a, \omega_b, \omega_c, \ldots, \omega_i$ are the ones obtained in (e). Recall from Lemma \ref{l. H(Fi_23)_niceFi_22} that the generators $a_n,b_n,c_n,d_n,e_n,f_n,g_n,h_n$, and $i_n$ of $G_n$ satisfy the set $\mathcal R(G_n)$ of defining relations. Now, let's check if these nine matrices $\mathfrak{a}_n$, $\mathfrak{b}_n$, $\mathfrak{c}_n$, $\mathfrak{d}_n$, $\mathfrak{e}_n$, $\mathfrak{f}_n$, $\mathfrak{g}_n$, $\mathfrak{h}_n$, $\mathfrak{i}_n$ also satisfy the set of relations $\mathcal R(G_n)$. For example, we had $(b_nc_nd_ne_nf_ng_nh_n)^9 = 1$ as the very last relation in $\mathcal R(G_n)$, and thus we check if $(\mathfrak{b}_n \mathfrak{c}_n \mathfrak{d}_n \mathfrak{e}_n \mathfrak{f}_n \mathfrak{g}_n \mathfrak{h}_n)^9$ is the identity matrix in $\GL_{352}(17)$. By means of \textsc{Magma}, it is easy to check if each of these relations is satisfied by the nine matrices $\mathfrak{a}_n,\mathfrak{b}_n,\mathfrak{c}_n,\mathfrak{d}_n,\mathfrak{e}_n,\mathfrak{f}_n,\mathfrak{g}_n,\mathfrak{h}_n$, $\mathfrak{i}_n$.

It turned out that all relations are satisfied, except for the three relations $(b_nc_n)^3$, $(e_nf_n)^3$, and $(f_ng_n)^3$. By means of \textsc{Magma}, it is easy to check that $(\mathfrak{b}_n \mathfrak{c}_n)^3 = (\mathfrak{e}_n \mathfrak{f}_n)^3 = (\mathfrak{f}_n \mathfrak{g}_n)^3$. So define a matrix $\mathfrak{z}_n = (\mathfrak{b}_n \mathfrak{c}_n)^3$. Then, we can check that $\mathfrak{z}_n$ has order $2$, and commutes with all three matrices $\mathfrak{x}_2$, $\mathfrak{y}_2$, and $\mathfrak{h}_2$, and therefore also with the nine matrices $\mathfrak{a}_n,\mathfrak{b}_n,\ldots, \mathfrak{i}_n$ ; thus, $\mathfrak{z}_n$ is a central involution of $\mathfrak{H}$. Now, match the nine matrices $\mathfrak{a}_n,\mathfrak{b}_n,\ldots, \mathfrak{i}_n$ with nine new abstract variables $a_\ell, b_\ell, \ldots, i_\ell$, respectively, and $\mathfrak{z}_n$ with another new variable $z_\ell$. Then, the ten matrices $\mathfrak{a}_n$, $\mathfrak{b}_n$, $\ldots$, $\mathfrak{i}_n$, $\mathfrak{z}_n$ satisfy the set $\mathcal R(G_\ell)$ of relations, as written in the statement.

Now, let's verify that the finitely presented group $G_\ell$ $=$ $\langle a_\ell$, $b_\ell$, $c_\ell$, $d_\ell$, $e_\ell$, $f_\ell$, $g_\ell$, $h_\ell$, $i_\ell$, $z_\ell \rangle$ having $\mathcal R(G_\ell)$ as its defining relations is isomorphic to $2\Fi_{22}$. We can observe in the defining relations $\mathcal R(G_\ell)$ that $z_\ell$ is in the center of $G_\ell$, and that $z_\ell$ is of order $1$ or $2$, since $z_\ell^2=1$. It is then also easy to observe that $G_\ell/\langle z_\ell \rangle$ is isomorphic to $G_n$, having exactly same relations. Therefore, since we have $G_n \cong \Fi_{22}$, we now know $G_\ell/\langle z_\ell \rangle \cong \Fi_{22}$. Hence, if $Order(z_\ell)=2$ implies $G_\ell \cong 2\Fi_{22}$, and $Order(z_\ell)=1$ implies $G_\ell \cong \Fi_{22}$. So, if we prove that $G_\ell$ has a permutation representation of order $=|2\Fi_{22}|$, then we can deduce that $G_\ell$ is indeed isomorphic to $2\Fi_{22}$, and that this permutation representation is faithful.

As in Lemma \ref{l. H(Fi_23)_niceFi_22}(b), $G_n$ has a subgroup $V_n = \langle b_n,c_n,d_n,e_n,f_n,g_n,h_n,i_n \rangle $ which is isomorphic to the simple group $O(7,3)$, and that $G_n$ has a faithful permutation representation $(PG_n)'$ of degree $14080$ with permutation stabilizer $V_n$. 

The strategy is to lift $V_n$ to $V_\ell$ inside $G_\ell$ by $V_\ell$ $=$ $\langle b_\ell \xi_1$, $c_\ell \xi_2$, $d_\ell \xi_3$, $e_\ell \xi_4$, $f_\ell \xi_5$, $g_\ell \xi_6$, $h_\ell \xi_7$, $i_\ell \xi_8 \rangle$ so that $V_\ell \cong O(7,3)$, by picking suitable $\xi_1, \xi_2, \ldots, \xi_8$, where each of $\xi_1, \xi_2, \ldots, \xi_8$ is either $1$ or $z_\ell$. Since each of $\xi_1, \xi_2, \ldots, \xi_8$ is either $1$ or $z_\ell$, there are $2^8 = 256$ cases to check. For each choice, the author let $V_\ell = \langle b_\ell \xi_1,c_\ell \xi_2,d_\ell \xi_3,e_\ell \xi_4,f_\ell \xi_5,g_\ell \xi_6,h_\ell \xi_7,i_\ell \xi_8 \rangle$, and ran the \textsc{Magma} command $\verb"CosetAction(G_l, V_l)"$. Among the $256$ cases, only one case returned a permutation representation of degree $28160$, while all others returned that of degree $14080$. That single case, namely $\xi_1 = z_\ell, \xi_2 = \xi_3 = \xi_4 = 1, \xi_5 = z_\ell, \xi_6 = \xi_7 = \xi_8 = 1$, is exactly the lifting $V_\ell$ of $V_n$ which we want.

To summarize, let $V_\ell = \langle b_\ell z_\ell, c_\ell, d_\ell, e_\ell, f_\ell z_\ell, g_\ell, h_\ell, i_\ell \rangle$. Then, the \textsc{Magma} command $\verb"CosetAction(G_l, V_l)"$ gives a permutation representation $PG_\ell$ of degree $28160$, which is of order $2^{18} \cdot 3^9 \cdot 5^2 \cdot 7 \cdot 11 \cdot 13 = |2\Fi_{22}|$. Thus, as asserted, we have $G_\ell \cong 2\Fi_{22}$, and this permutation representation $PG_\ell$ of $G_\ell$ is faithful.

(g) Let
\begin{align*}
x_0 = \beta_1 (a_\ell,b_\ell,c_\ell,d_\ell,e_\ell,f_\ell,g_\ell,h_\ell,i_\ell), \\
y_0 = \beta_2 (a_\ell,b_\ell,c_\ell,d_\ell,e_\ell,f_\ell,g_\ell,h_\ell,i_\ell), \\
h_0 = \beta_3 (a_\ell,b_\ell,c_\ell,d_\ell,e_\ell,f_\ell,g_\ell,h_\ell,i_\ell),
\end{align*}
where the three words $\beta_1, \beta_2$, and $\beta_3$ are the ones we obtained in the statement (b). Then $G_\ell = \langle x_0, y_0, h_0 \rangle$ can be checked by means of \textsc{Magma}, using the faithful permutation representation $PG_\ell$ obtained in (f).

(h) Let the three words $\beta_1, \beta_2$, and $\beta_3$ be the ones we obtained in the statement (b). Then, it can be checked by means of \textsc{Magma} that
\begin{align*}
\mathfrak{x}_2 & = \beta_1 (\mathfrak{a}_n,\mathfrak{b}_n,\mathfrak{c}_n,\mathfrak{d}_n,\mathfrak{e}_n,\mathfrak{f}_n,\mathfrak{g}_n,\mathfrak{h}_n,\mathfrak{i}_n), \quad
\mathfrak{y}_2 = \beta_2 (\mathfrak{a}_n,\mathfrak{b}_n,\mathfrak{c}_n,\mathfrak{d}_n,\mathfrak{e}_n,\mathfrak{f}_n,\mathfrak{g}_n,\mathfrak{h}_n,\mathfrak{i}_n), \\
\mathfrak{h}_2 & = \beta_3 (\mathfrak{a}_n,\mathfrak{b}_n,\mathfrak{c}_n,\mathfrak{d}_n,\mathfrak{e}_n,\mathfrak{f}_n,\mathfrak{g}_n,\mathfrak{h}_n,\mathfrak{i}_n),
\end{align*}
where the nine matrices $\mathfrak{a}_n$, $\mathfrak{b}_n$, $\mathfrak{c}_n$, $\mathfrak{d}_n$, $\mathfrak{e}_n$, $\mathfrak{f}_n$, $\mathfrak{g}_n$, $\mathfrak{h}_n$, and $\mathfrak{i}_n$ are the ones we obtained in the proof of (f) (as words in $\mathfrak{x}_3, \mathfrak{y}_3, \mathfrak{h}_3$). Therefore $\mathfrak{x}_2, \mathfrak{y}_2$, and $\mathfrak{h}_2$ are contained in $\langle \mathfrak{a}_n,\mathfrak{b}_n,\mathfrak{c}_n,\mathfrak{d}_n,\mathfrak{e}_n,\mathfrak{f}_n,\mathfrak{g}_n,\mathfrak{h}_n, \mathfrak{i}_n \rangle$. Hence, $\langle \mathfrak{x}_2, \mathfrak{y}_2, \mathfrak{h}_2 \rangle = \mathfrak{H}$ is contained in $\langle \mathfrak{a}_n$, $\mathfrak{b}_n$, $\mathfrak{c}_n$, $\mathfrak{d}_n$, $\mathfrak{e}_n$, $\mathfrak{f}_n$, $\mathfrak{g}_n$, $\mathfrak{h}_n$, $\mathfrak{i}_n \rangle$. It is clear that $\langle \mathfrak{a}_n$, $\mathfrak{b}_n$, $\mathfrak{c}_n$, $\mathfrak{d}_n$, $\mathfrak{e}_n$, $\mathfrak{f}_n$, $\mathfrak{g}_n$, $\mathfrak{h}_n$, $\mathfrak{i}_n \rangle$ is contained in $\langle \mathfrak{x}_2, \mathfrak{y}_2, \mathfrak{h}_2 \rangle = \mathfrak{H}$, since the nine matrices $\mathfrak{a}_n$, $\mathfrak{b}_n$, $\ldots$, $\mathfrak{h}_n$ are words in $\mathfrak{x}_3, \mathfrak{y}_3, \mathfrak{h}_3$. Thus $\mathfrak{H} = \langle \mathfrak{a}_n,\mathfrak{b}_n,\mathfrak{c}_n,\mathfrak{d}_n,\mathfrak{e}_n,\mathfrak{f}_n,\mathfrak{g}_n,\mathfrak{h}_n, \mathfrak{i}_n \rangle$. Recall that $\mathfrak{z}_n = (\mathfrak{b}_n \mathfrak{c}_n)^3$. Now we have
\begin{align*}
\mathfrak{H} = \langle \mathfrak{a}_n,\mathfrak{b}_n,\mathfrak{c}_n,\mathfrak{d}_n,\mathfrak{e}_n,\mathfrak{f}_n,\mathfrak{g}_n,\mathfrak{h}_n, \mathfrak{i}_n, \mathfrak{z}_n \rangle.
\end{align*}
Observe that these generating ten matrices $\mathfrak{a}_n,\mathfrak{b}_n,\mathfrak{c}_n,\mathfrak{d}_n,\mathfrak{e}_n,\mathfrak{f}_n,\mathfrak{g}_n,\mathfrak{h}_n, \mathfrak{i}_n, \mathfrak{z}_n$ satisfy the relations $R(G_\ell)$. Therefore, the matrix group $\mathfrak{H}$ which is generated by these ten matrices is isomorphic to some quotient group of the finitely presented group $G_\ell$ which has $R(G_\ell)$ as its defining relations. We proved in (f) that $G_\ell \cong 2\Fi_{22}$, and therefore $\mathfrak{H}$ is isomorphic to some quotient group of $2\Fi_{22}$. Therefore, $\mathfrak{H}$ is isomorphic to $1$, $\Fi_{22}$, or $2\Fi_{22}$. Notice that $\mathfrak{z}_n$ is an involution (hence not an identity element) which in the center of this matrix group $\mathfrak{H}$. Among the three choices $1$, $\Fi_{22}$, and $2\Fi_{22}$, the only group with nontrivial center is $2\Fi_{22}$. Therefore $\mathfrak{H} \cong 2\Fi_{22}$.

Hence, the natural homomorphism $\vartheta : \mathfrak{H} \rightarrow G_\ell$ given by $\vartheta(\mathfrak{a}_n) = a_\ell$, $\vartheta(\mathfrak{b}_n) = b_\ell$, $\vartheta(\mathfrak{c}_n) = c_\ell$, $\vartheta(\mathfrak{d}_n) = d_\ell$,$\vartheta(\mathfrak{e}_n) = e_\ell$, $\vartheta(\mathfrak{f}_n) = f_\ell$, $\vartheta(\mathfrak{g}_n) = g_\ell$,$\vartheta(\mathfrak{h}_n) = h_\ell$, $\vartheta(\mathfrak{i}_n) = i_\ell$, and $\vartheta(\mathfrak{z}_n) = z_\ell$ is an isomorphism. Note that $x_0, y_0$, and $h_0$ are expressed by the three words $\beta_1, \beta_2$, and $\beta_3$ in terms of $a_\ell, b_\ell, c_\ell, d_\ell, e_\ell, f_\ell, g_\ell, h_\ell$, and $i_\ell$, and also that $\mathfrak{x}_2, \mathfrak{y}_2$, and $\mathfrak{h}_2$ are expressed by the three words $\beta_1, \beta_2$, and $\beta_3$ in terms of $\mathfrak{a}_n,\mathfrak{b}_n,\mathfrak{c}_n,\mathfrak{d}_n,\mathfrak{e}_n,\mathfrak{f}_n,\mathfrak{g}_n,\mathfrak{h}_n$, and $\mathfrak{i}_n$. Since $\vartheta$ is an isomorphism, it preserves the multiplication of elements, and therefore we have $\vartheta(\mathfrak{x}_2) = x_0, \vartheta(\mathfrak{y}_2) = y_0$, and $\vartheta(\mathfrak{h}_2) = h_0$.

The irreducibility of $\mathfrak{H}$ is checked by means of the \textsc{Magma} command \\
$\verb"IsIrreducible"$.

\end{proof}

\begin{corollary}\label{cor. H(Fi_23)} Keep the notations in Theorem \ref{thm. H(Fi_23)_embed}. Let $K = \GF(17)$. Let $H := G_\ell =\langle x_0, y_0, h_0 \rangle$, and $D_H := \langle x_0, y_0 \rangle$. Then, $H$ and $D_H$ have orders $2^{18} \cdot 3^9 \cdot 5^2 \cdot 7 \cdot 11 \cdot 13$ and $2^{18} \cdot 3^2 \cdot 5 \cdot 7 \cdot 11$, respectively. Then the following statements hold:


\begin{enumerate}
\item[\rm(a)] Each Sylow $2$-subgroup $S$ of $H$ has a unique maximal elementary abelian normal subgroup $A$ of order $2^{11}$. Also, $N_H(A) \cong D = C_E (z)$.

\item[\rm(b)] There is a Sylow $2$-subgroup $S$ such that its maximal elementary abelian normal subgroup $A$ of order $2^{11}$ satisfies $N_H(A) = \langle x_0, y_0 \rangle = D_H$.

\item[\rm(c)] There is an isomorphism $\phi_D : D_H \rightarrow D$ such that $\phi_D(x_0) = x_1$ and $\phi_D(y_0) = y_1$.


\item[\rm(d)] A system of representatives $w_i$ of the $114$ conjugacy classes of $H$ and the corresponding centralizer orders $|C_H(w_i)|$ are given in Table \ref{Fi_23cc H}


\item[\rm(e)] The character table of $H$ is given in Table \ref{Fi_23ct_H}.

\end{enumerate}

\end{corollary}

\begin{proof}

By Theorem \ref{thm. H(Fi_23)_embed}, the finitely presented group $G_\ell =$ $\langle a_\ell$, $b_\ell$, $c_\ell$, $d_\ell$, $e_\ell$, $f_\ell$, $g_\ell$, $h_\ell$, $i_\ell$, $z_\ell \rangle$ has a faithful permutation representation $PG_\ell$ of degree $28160$, with stabilizer $\langle b_\ell z_\ell$, $c_\ell$, $d_\ell$, $e_\ell$, $f_\ell z_\ell$, $g_\ell$, $h_\ell$, $i_\ell \rangle$. We also have that $x_0, y_0$, and $h_0$ are expressed as words in terms of $a_\ell,b_\ell,c_\ell,d_\ell,e_\ell,f_\ell,g_\ell,h_\ell,i_\ell$, and that $G_\ell = \langle x_0, y_0, h_0 \rangle$. Therefore, it makes sense to say $H := G_\ell = \langle x_0, y_0, h_0 \rangle$, and computations are also possible in $G_\ell$, by means of the faithful permutation representation of $G_\ell$, and the corresponding permutations for $x_0$, $y_0$, and $h_0$. The orders of $H$ and $D_H := \langle x_0, y_0 \rangle$ are checked with \textsc{Magma}.

(a) By means of the \textsc{Magma} commands $\verb"S:=SylowSubgroup(H, 2)"$ and $$\verb"Subgroups(S:Al:=``Normal'',IsElementaryAbelian:=true)",$$ we obtain a Sylow $2$-subgroup $S$ of $H$, and the unique maximal elementary abelian subgroup $A$ of $S$ of order $2^{11}$. By the \textsc{Magma} command $\verb"IsIsomorphic"$, we can also verify that $N_H(A) \cong D$.

(b) Let $S$ be any Sylow $2$-subgroup of $D_H$, and let $A$ be as in (a). Then, it can be checked by the \textsc{Magma} that $N_H(A) = D_H$.

(c) Recall that the matrix group $\langle \mathfrak{x}_2, \mathfrak{y}_2 \rangle$ is a faithful representation of $D = \langle x_1, y_1 \rangle$ (by $x_1 \mapsto \mathfrak{x}_2$ and $y_1 \mapsto \mathfrak{y}_2$), since so is $\langle \mathfrak{p}_2, \mathfrak{q}_2, \mathfrak{x}_2, \mathfrak{y}_2 \rangle$. We showed that $\mathfrak{H} = \langle \mathfrak{x}_2, \mathfrak{y}_2, \mathfrak{h}_2 \rangle$ is a faithful representation of $H = \langle x_0, y_0, h_0 \rangle$ (by $x_0 \mapsto \mathfrak{x}_2$, $y_0 \mapsto \mathfrak{y}_2$, and $h_0 \mapsto \mathfrak{h}_2$). Therefore, $\langle \mathfrak{x}_2, \mathfrak{y}_2 \rangle$ is a faithful representation of $\langle x_0, y_0 \rangle = D_H$ (by $x_0 \mapsto \mathfrak{x}_2$ and $y_0 \mapsto \mathfrak{y}_2$). Therefore there is an isomorphism $\phi_D$ from $D_H$ to $D$ such that $\phi_D (x_0) = x_1$ and $\phi_D (y_0) = y_1$.

(d) and (e) Checked by means of \textsc{Magma} and Kratzer's Algorithm 5.3.18 of \cite{michler}.

\end{proof}


\newpage 

\section{Construction of Fischer's simple group $\Fi_{23}$}\label{sec. Fi_23}

By Lemma \ref{l. E(Fi_23)classes} and Corollary \ref{cor. H(Fi_23)} the amalgam $H \leftarrow D \rightarrow E$ constructed in sections \ref{sec. M23-extensions} and \ref{sec. H(Fi_23)} satisfies the main condition of G. Michler's Algorithm 7.4.8 of \cite{michler}. Therefore we can apply the Algorithm 7.4.8 of \cite{michler} 
to give here a new existence proof for Fischer's simple group $\Fi_{23}$.

The readers should be aware of the abusive notation used in the following theorem. The symbols $\chi, \tau, \chi_i, \tau_i, \mathfrak{V}, \mathfrak{W}, \mathfrak{V}_i, \mathfrak{W}_i, U_i, S_i, \mathcal{T}$ appearing in the following theorem have nothing to do with those of same notations appearing in Theorem \ref{thm. H(Fi_23)}.

\begin{theorem}\label{thm. Fi_23} Keep the notations in Lemma \ref{l. E(Fi_23)classes} and Corollary \ref{cor. H(Fi_23)}. Let $K = \GF(17)$. Using the notations of the character tables of \ref{Fi_23ct_H}, \ref{Fi_23ct_D}, and \ref{Fi_23ct_E} of $H, D, E$, the following statements hold:

\begin{enumerate}
\item[\rm(a)] The smallest degree of a nontrivial pair $(\chi, \tau)\in mf \mbox{char}_{\mathbb{C}}(H) \times mf \mbox{char}_{\mathbb{C}}(E)$ of compatible characters which divides the group order of $\Fi_{23}$ is $782$.

\item[\rm(b)] There is exactly one compatible pair $(\chi, \tau) \in mf \mbox{char}_{\mathbb{C}}(H) \times mf \mbox{char}_{\mathbb{C}}(E)$ of degree $782$ of the groups $H = \langle D_H, h \rangle$ and $E = \langle D, e \rangle$:
\begin{align*}
(\chi, \tau) = (\chi_{1} + \chi_{\bf 3} + \chi_{4}, \tau_{1} + \tau_{2} + \tau_{\bf 10} +\tau_{\bf 11})
\end{align*}
with common restriction
\begin{align*}
\chi_{|{D_H}} = \tau_{|D}
= \psi_{1} + \psi_{1} + \psi_{2} + \psi_{6}
+ \psi_{\bf 9} + \psi_{\bf 10} + \psi_{15},
\end{align*}
where irreducible characters with bold face indices denote faithful irreducible characters. 

\item[\rm(c)] Let $\mathfrak{V}$ and $\mathfrak{W}$ be the up-to-isomorphism uniquely determined faithful semi-simple multiplicity-free $782$-dimensional modules of $H$ and $E$ over $F=\GF(17)$ corresponding to the compatible pair $\chi, \tau$, respectively.

Let $\kappa_{\mathfrak{V}} : H \rightarrow \GL_{782}(17)$ and $\kappa_{\mathfrak{W}} : E \rightarrow \GL_{782}(17)$ be the representations of $H$ and $E$ afforded by the modules $\mathfrak{V}$ and $\mathfrak{W}$, respectively.

Let $\mathfrak{x} = \kappa_{\mathfrak{V}}(x_0)$, $\mathfrak{y} = \kappa_{\mathfrak{V}}(y_0)$, $\mathfrak{h} = \kappa_{\mathfrak{V}}(h_0)$ in $\kappa_{\mathfrak{V}}(H) \le \GL_{782}(17)$. Then the following assertions hold:

$\mathfrak{V}_{|{D_H}} \cong \mathfrak{W}_{|D}$, and there is a transformation matrix $\mathcal{T} \in \GL_{782}(17)$ such that
\begin{align*}
\mathfrak{x} = \mathcal{T}^{-1} \kappa_{\mathfrak{W}}(x) \mathcal{T}, \quad
\mathfrak{y} = \mathcal{T}^{-1} \kappa_{\mathfrak{W}}(y) \mathcal{T}.
\end{align*}




\item[\rm(d)] Let $Y = \GL_{782}(17)$, $\mathfrak D = \langle \mathfrak x, \mathfrak y
\rangle$, $\mathfrak H = \langle \mathfrak x, \mathfrak y,
\mathfrak h \rangle$. Let $\mathcal D = C_Y(\mathfrak D)$ and
$\mathcal H = C_Y(\mathfrak H)$. Let $\mathfrak e_1 = \mathcal
T^{-1} \kappa_\mathfrak W (e) \mathcal T$. Let $\mathfrak E =
\langle \mathfrak D, \mathfrak e_1 \rangle$ and $\mathcal E =
C_Y(\mathfrak E)$. Then the following statements hold:
\begin{enumerate}
\item[\rm(1)] There is an isomorphism $$\gamma: \mathcal D
\rightarrow \mathcal D_1 = \GL_2(17) \times {K^{*}}^5 \le
\GL_7(17).$$

\item[\rm(2)] $\mathcal H_1 = \gamma(\mathcal H)$ is generated by
the three diagonal matrices\\
$a_1 = diag(3, 1, 1, 1, 1, 1, 1)$, $a_2 = diag(1, 1, 1, 1, 3, 3, 1)$ and\\
$a_3 = diag(1, 3, 3, 3, 1, 1, 3)$.

\item[\rm(3)] $\mathcal E_1 = \gamma(\mathcal E)$ is generated by
the four diagonal matrices\\
$b_1 = a_1$, $b_2 = diag(1, 3, 3, 1, 1,1, 1)$, \\
$b_3 =  diag(1, 1, 1, 3, 3, 1, 1)$ and $b_4 =  diag(1, 1, 1, 1, 1, 3, 3)$.

\item[\rm(4)] $\mathcal D$ has $321 \times 16$ $\mathcal
H$-$\mathcal E$ double cosets.

\item[\rm(5)] The free product $H*_DE$ of $H$ and $E$ with
amalgamated subgroup $D$ has exactly one irreducible
$782$-dimensional representation over $K$ whose Sylow
$2$-subgroups have the same exponent as the ones of $H$. It corresponds to the $\mathcal H$-$\mathcal E$ double coset representative
$$\mathcal F = diag(u,1^{21},1^{77},1^{176},1^{176},16^{330}) \in \GL_{782}(17) \quad
{where}$$
{\renewcommand{\arraystretch}{0.5}
$$
u = \left( \begin{array}{*{2}{c@{\,}}c}
1 &  1\\
9 & 14
\end{array} \right).
$$
}

Let $\mathfrak e = \mathcal F^{-1} \mathfrak e_1 \mathcal F$ and $\mathfrak G = \langle \mathfrak x, \mathfrak y, \mathfrak h, \mathfrak e \rangle$. The proof of this Sylow $2$-subgroup test for $\mathfrak{G}$ is split into two parts. The first half is the order test for the elements $\mathfrak{h} \mathfrak{e}$ and $\mathfrak{y}\mathfrak{e}\mathfrak{h}$; these two matrices have to have orders of elements in $\Fi_{23}$ (it turned out that $\mathfrak{h} \mathfrak{e}$ has order $12$, and $\mathfrak{y}\mathfrak{e}\mathfrak{h}$ has order $13$). The second half of the proof is done in {\rm (e)} of this theorem.

\item[\rm(6)] The four generating matrices of $\mathfrak G$ are documented in \cite{kim1}.
\end{enumerate}

\item[\rm(e)] $\mathfrak G$ has a faithful permutation
representation $P\mathfrak G$ of degree $31671$ with stabilizer
$\mathfrak H = \langle \mathfrak x, \mathfrak y, \mathfrak h
\rangle$.

\item[\rm(f)] $\mathfrak G$ is a finite simple group of order
$2^{18}\cdot 3^{13}\cdot 5^2\cdot7\cdot11\cdot13\cdot17\cdot23$
with centralizer $C_{\mathfrak G}(\mathfrak z) = \mathfrak H \cong
H$ of the involution $\mathfrak z = (\mathfrak x \mathfrak
y^2)^7$.

\item[\rm(g)]  $\mathfrak G$ has $98$ conjugacy classes
$\mathfrak{g_i}^{\mathfrak{G}}$ with representatives
$\mathfrak{g_i}$ and centralizer orders
$|C_{\mathfrak{G}}(\mathfrak{g_i})|$ as given in Table
\ref{Fi_23cc}.

\item[\rm(h)] The character table of $\mathfrak G$ coincides with
that of $\Fi_{23}$ in the Atlas \cite{atlas}, its p. 178-179.

\end{enumerate}

\end{theorem}

\begin{proof}

(a) The character tables of the groups $H, D$, and $E$ are stated in the Appendix \ref{Fi_23ct_H}, \ref{Fi_23ct_D}, and \ref{Fi_23ct_E}. In the following we use their notations. Using \textsc{Magma} and the character tables of $H,D$, and $E$ and the fusion of the classes of $D_H \cong D$ in $H$ and $D$ in $E$, an application of Kratzer's Algorithm 7.3.10 of \cite{michler} yields the compatible pair stated in assertion (a), dividing the group order of $\Fi_{23}$; the group order of $\Fi_{23}$ is taken from Atlas \cite{atlas}.

(b) The application of Kratzer's Algorithm 7.3.10 of \cite{michler} also shows that the pairs $(\chi, \tau)$ of (b) is the only compatible pair of degree $782$ with respect to the fusion of the $D$-classes into the $H$- and into the $E$-classes.

(c) 
In order to construct the faithful irreducible representation
$\mathfrak V$ corresponding to the character $\chi = \chi_{1} + \chi_{3} + \chi_{4}$ of degree $782$, the author employed the \textsc{Magma} command
$\verb"LowIndexSubgroups(PH, 4000)"$ using the faithful
permutation representation $PH$ of $H$ of degree $28160$ (see statements (f) and (g) of Theorem \ref{thm. H(Fi_23)_embed} for construction of $PH$). \textsc{Magma} found a subgroup $U_1$ of $H$ such that the followings hold.

$U_1$ is of index $3510$ in $H$, and $\chi_{4}$ (dimension $429$) is a constituent of the permutation character $(1_{U_1})^{H}$. By $\verb"GetShortGens"$ the author obtained $U_1 = \langle (y_0)^7$, $(y_0 h_0 y_0^2 h_0)^7$, $(y_0 h_0 y_0 h_0 y_0)^7$, $(h_0 y_0^2 h_0 y_0)^7$, $(h_0 y_0 h_0 y_0^2)^7$, $(y_0^3 h_0^2 y_0^2)^{11} \rangle$. By applying the Meat-axe Algorithm to the permutation module $(1_{U})^{H}$ the author obtained the irreducible $KH$-module $\mathfrak{V}_{4}$ corresponding to $\chi_{4}$.

Note that $\chi_{1}$ is the trivial character. Notice also that $\chi_{3}$ is the unique irreducible character of $H$ of degree $352$, and that we already have an irreducible representation of $H$ of degree $352$, namely $\mathfrak{H} = \langle \mathfrak{x}_2, \mathfrak{y}_2, \mathfrak{h}_2 \rangle$. Thus, we can use $\mathfrak{H}$ as a representation of $H$ corresponding to $\chi_3$. 

In order to construct the faithful irreducible representation
$\mathfrak W$ corresponding to the character $\tau = \tau_{1} + \tau_{2} + \tau_{10} +\tau_{11}$ of degree $782$, the author employed the \textsc{Magma} command
$\verb"LowIndexSubgroups(PE, 1500)"$ using the faithful
permutation presentation $PE$ of $E$ of degree $1012$ (see Lemma \ref{l. E(Fi_23)classes} for construction of $PE$). \textsc{Magma} found subgroups $S_1, S_2$, and $S_3$ such that the followings hold.

$S_1$ is of index $506$ in $E$, and $\tau_{2}$ (dimension $22$) and $\tau_{10}$ (dimension $253$) are constituents of the permutation character $(1_{S_1})^{E}$. By $\verb"GetShortGens"$ we get $S_1 =$ $\langle (y_1^3 e_1 y_1)^3$, $(y_1^2 x_1 y_1 x_1 y_1)^3$, $(y_1^4 e_1 x_1)^4$, $(y_1 e_1 x_1 y_1^3)^4 \rangle$. By applying the Meat-axe Algorithm to the permutation module $(1_{S_1})^{e}$ the author obtained the irreducible $KE$-modules $\mathfrak{W}_{2}$ and $\mathfrak{W}_{10}$ corresponding to $\tau_{2}$ and $\tau_{10}$, respectively.

$S_2$ is of index $1012$ in $E$, and $\tau_{11}$ (dimension $506$) is a constituent of the permutation character $(1_{S_2})^{E}$. By $\verb"GetShortGens"$ we get $S_2 =$ $\langle (x_1 e_1 y_1)^4$, $(x_1 y_1^3 x_1 e_1)^7$, $(y_1^2 x_1 y_1 x_1 y_1)^2 \rangle$. By applying the Meat-axe Algorithm to the permutation module $(1_{S_2})^{E}$ the author obtained the irreducible $KE$-module $\mathfrak{W}_{11}$ corresponding to $\tau_{11}$.

Therefore we can obtain the representations $\kappa_{\mathfrak{V}} : H \rightarrow \GL_{782}(17)$ and $\kappa_{\mathfrak{W}} : E \rightarrow \GL_{782}(17)$ of $H$ and $E$ afforded by the modules $\mathfrak{V}$ and $\mathfrak{W}$, respectively, as follows. As before, $\mathfrak{V} = \mathfrak{V}_1 \oplus \mathfrak{V}_{3} \oplus \mathfrak{V}_4$, where $\mathfrak{V}_1$ is the trivial $KH$-module, and $\mathfrak{V}_3$ corresponds to the representation $\kappa_{\mathfrak{V}_3} : H \rightarrow \GL_{352}(17)$ given by $\kappa_{\mathfrak{V}_3}(x_0) = \mathfrak{x}_2$, $\kappa_{\mathfrak{V}_3}(y_0) = \mathfrak{y}_2$, and $\kappa_{\mathfrak{V}_3}(h_0) = \mathfrak{h}_2$. Then, the matrices for $\kappa_{\mathfrak{V}}$ can be obtained by diagonal joining, for example,
\begin{align*}
\kappa_{\mathfrak{V}} (x_0) = \mbox{diag}(\kappa_{\mathfrak{V}_1} (x_0), \kappa_{\mathfrak{V}_3} (x_0), \kappa_{\mathfrak{V}_4} (x_0)),
\end{align*}
where $\kappa_{\mathfrak{V}_1} : H \rightarrow \GL_{1}(17)$ and $\kappa_{\mathfrak{V}_4} : H \rightarrow \GL_{429}(17)$ are the representations of $H$ afforded by the $KH$-modules $\mathfrak{V}_1$ and $\mathfrak{V}_4$, respectively. And similarly for $y_0$ and $h_0$, too. For $E$ side, let $\mathfrak{W} = \mathfrak{W}_1 \oplus \mathfrak{W}_{2} \oplus \mathfrak{W}_{10} \oplus \mathfrak{W}_{11}$, where $\mathfrak{W}_1$ is the trivial $KE$-module. Then, the matrices for $\kappa_{\mathfrak{W}}$ can be obtained by diagonal joining, for example,
\begin{align*}
\kappa_{\mathfrak{W}} (x_1) = \mbox{diag}(\kappa_{\mathfrak{W}_1} (x_1), \kappa_{\mathfrak{W}_2} (x_1), \kappa_{\mathfrak{W}_{10}} (x_1), \kappa_{\mathfrak{W}_{11}} (x_1)),
\end{align*}
where $\kappa_{\mathfrak{W}_1} : E \rightarrow \GL_{1}(17)$, $\kappa_{\mathfrak{W}_{2}} : E \rightarrow \GL_{22}(17)$, $\kappa_{\mathfrak{W}_{10}} : E \rightarrow \GL_{253}(17)$, and $\kappa_{\mathfrak{W}_{11}} : E \rightarrow \GL_{506}(17)$ are the representations of $E$ afforded by the $KE$-modules $\mathfrak{W}_1$, $\mathfrak{W}_2$, $\mathfrak{W}_{10}$, and $\mathfrak{W}_{11}$, respectively. And similarly for $y_1$ and $e_1$, too.

$\chi_{|{D_H}} = \tau_{|D}$ means $\mathfrak{V}_{|D_H} \cong \mathfrak{W}_{|D}$. Recall that two representations of a group being isomorphic means that there is some matrix $\mathcal{T}$ such that for every element of the group, the matrix for the element corresponding to one of the two representations is conjugate to the matrix for the same element corresponding to the other representation by $\mathcal{T}$. However, here we have two isomorphic groups $D_H$ and $D$ instead of identical groups. So, employing an isomorphism $\phi_D : D_H \rightarrow D$ as obtained in Corollary \ref{cor. H(Fi_23)}(c), now we can see that $\mathfrak{V}_{|D_H} \cong \mathfrak{W}_{|D}$ means that there is some $\mathcal{T} \in \GL_{782}(17)$ such that ${\kappa_\mathfrak{V}}_{|D_H}(\phi_D(g)) = \mathcal{T}_1^{-1} {\kappa_\mathfrak{W}}_{|T}(g) \mathcal{T}_1$ for all $g \in D_H$.

Assuming we have such $\mathcal{T}$, we get
\begin{align*}
\mathfrak{x} & = {\kappa_\mathfrak{V}}_{|D_H} (x_0)
= {\kappa_\mathfrak{V}}_{|D_H} (\phi_D(x_1))
= \mathcal{T}^{-1} {\kappa_\mathfrak{W}}_{|D}(x_1) \mathcal{T}, \quad \mbox{and} \\
\mathfrak{y} & = {\kappa_\mathfrak{V}}_{|D_H} (y_0)
= {\kappa_\mathfrak{V}}_{|D_H} (\phi_D(y_1))
= \mathcal{T}^{-1} {\kappa_\mathfrak{W}}_{|D}(y_1) \mathcal{T},
\end{align*}
thus $\mathfrak{x}= \mathcal{T}^{-1} \kappa_{\mathfrak{W}}(x_1) \mathcal{T}$ and $\mathfrak{y} = \mathcal{T}^{-1} \kappa_{\mathfrak{W}}(y_1) \mathcal{T}$, as desired in the statement (1). Knowing that such $\mathcal{T}$ exists, we can apply the Parker's isomorphism test of Proposition 6.1.6 of
\cite{michler} by means of the \textsc{Magma} command
$$\verb"IsIsomorphic(GModule(sub<Y|W(x1),W(y1)>),GModule(sub<Y|V(x0),V(y0)>))",$$
which gives the boolean value, which is $\verb"true"$ in this case, and the desired transformation matrix $\mathcal{T}$. 

By assertion (b) and Corollary 7.2.4 of \cite{michler} this
transformation matrix ${\mathcal T}$ has to be multiplied by a
block diagonal matrix ${\mathcal S}_0$ of $\GL_{782}(17)$. In order to
calculate its entries one has to get the composition factors of
the restrictions ${\chi_i}_{|D_H}$ and ${\tau_j}_{|D}$ to $D_H$
and $D$, respectively. From the fusion and the $3$ character
tables \ref{Fi_23ct_H}, \ref{Fi_23ct_D} and
\ref{Fi_23ct_E} follows that:
\begin{eqnarray*}
&&{\chi_{1}}_{|D_H} = \psi_{1}, \quad {\chi_{3}}_{|D_H} = \psi_{9} + \psi_{10}, \quad
{\chi_{4}}_{|D_H} = \psi_{1} + \psi_{2} + \psi_{6} + \psi_{15}, \quad \mbox{and}\\
&&{\tau_{1}}_{|D} = \psi_{1}, \quad {\tau_{2}}_{|D} = \psi_{1} + \psi_{2}, \quad {\tau_{10}}_{|D} = \psi_{6} + \psi_{9}, \quad {\tau_{11}}_{|D} = \psi_{10} + \psi_{15}.\\
\end{eqnarray*}

Thus, by applying the Meat-axe Algorithm to $\chi_{4}|_{D_H}$, we get the $KD_H$-modules $\mathfrak{U}_{2}, \mathfrak{U}_{6}$, and $\mathfrak{U}_{15}$ of $D_H \cong D$, such that ${\mathfrak{V}_{4}}_{|D_H} \cong \mathfrak{U}_1 \oplus \mathfrak{U}_2 \oplus \mathfrak{U}_6 \oplus \mathfrak{U}_{15}$, where $\mathfrak{U}_1$ is the trivial $KD_H$-module. We already have the constituents $KD_H$-modules $\mathfrak{U}_{9}$ and $\mathfrak{U}_{10}$ of the restriction of $\mathfrak{V}$ to $D_H$, so that ${\mathfrak{V}_{3}}_{|D_H} \cong \mathfrak{U}_{9} \oplus \mathfrak{U}_{10}$. Thus we also have ${\mathfrak{W}_{2}}_{|D} \cong \mathfrak{U}_{1} \oplus \mathfrak{U}_{2}$, ${\mathfrak{W}_{10}}_{|D} \cong \mathfrak{U}_{6} \oplus \mathfrak{U}_{9}$, and ${\mathfrak{W}_{11}}_{|D} \cong \mathfrak{U}_{10} \oplus \mathfrak{U}_{15}$. It is clear that ${\mathfrak{V}_{1}}_{|D_H} = \mathfrak{U}_{1} = {\mathfrak{W}_{1}}_{|D}$. Notice that the $KD_H$-modules can be thought of as $KD$-modules and vice versa, via the isomorphism $\phi_D : D_H \rightarrow D$. Now we have
\begin{align*}
{\mathfrak{V}}_{|D_H}
\cong \mathfrak{U}_{1} \oplus \mathfrak{U}_{1} \oplus \mathfrak{U}_{2} \oplus \mathfrak{U}_{6} \oplus \mathfrak{U}_{9} \oplus \mathfrak{U}_{10} \oplus \mathfrak{U}_{15}
\cong {\mathfrak{W}}_{|D}.
\end{align*}
Therefore, there is a transformation matrix $\mathcal{S}_0 \in \GL_{782}(17)$ such that
\begin{align*}
\mathcal{S}_0^{-1} \kappa_\mathfrak{V} (x_0) \mathcal{S}_0
& = \mbox{diag}(\kappa_{\mathfrak{U}_{1}} (x_0), \kappa_{\mathfrak{U}_{1}} (x_0),
\kappa_{\mathfrak{U}_{2}} (x_0),\kappa_{\mathfrak{U}_{6}} (x_0),\kappa_{\mathfrak{U}_{9}} (x_0),
\kappa_{\mathfrak{U}_{10}} (x_0),\kappa_{\mathfrak{U}_{15}} (x_0)),\\
\mbox{ and} \quad \quad \quad \quad & \\
\mathcal{S}_0^{-1} \kappa_\mathfrak{V} (y_0) \mathcal{S}_0
& = \mbox{diag}(\kappa_{\mathfrak{U}_{1}} (y_0), \kappa_{\mathfrak{U}_{1}} (y_0),
\kappa_{\mathfrak{U}_{2}} (y_0),\kappa_{\mathfrak{U}_{6}} (y_0),\kappa_{\mathfrak{U}_{9}} (y_0),
\kappa_{\mathfrak{U}_{10}} (y_0),\kappa_{\mathfrak{U}_{15}} (y_0)),
\end{align*}
where $\kappa_{\mathfrak{U}_{1}} : D \rightarrow \GL_{1}(17)$, $\kappa_{\mathfrak{U}_{2}} : D \rightarrow \GL_{21}(17)$, $\kappa_{\mathfrak{U}_{6}} : D \rightarrow \GL_{77}(17)$,  $\kappa_{\mathfrak{U}_{9}} : D \rightarrow \GL_{176}(17)$, $\kappa_{\mathfrak{U}_{10}} : D \rightarrow \GL_{176}(17)$, and $\kappa_{\mathfrak{U}_{15}} : D \rightarrow \GL_{330}(17)$ are the representations of $D$ afforded by the modules $\mathfrak{U}_{1}$, $\mathfrak{U}_{2}$, $\mathfrak{U}_{6}$, $\mathfrak{U}_{9}$, $\mathfrak{U}_{10}$, and $\mathfrak{U}_{15}$, respectively. This matrix $\mathcal{S}_0$ can be obtained by applying Parker's isomorphism test to the two modules $\mathfrak{V}$ and $\mathfrak{U}_{1} \oplus \mathfrak{U}_{1} \oplus \mathfrak{U}_{2} \oplus \mathfrak{U}_{6} \oplus \mathfrak{U}_{9} \oplus \mathfrak{U}_{10} \oplus \mathfrak{U}_{15}$. To be precise, the author did this Parker's isomorphism test first for lower right $781$ by $781$ submatrices of the relevant matrices, and then enlarged the transformation matrix by diagonally joining $1$ in the upper left corner. By this way we can be guaranteed to have right restrictions as we expect. We can assume that we started with $\mathcal{S}_0^{-1} \kappa_\mathfrak{V} (x_0) \mathcal{S}_0$, $\mathcal{S}_0^{-1} \kappa_\mathfrak{V} (y_0) \mathcal{S}_0$, and $\mathcal{S}_0^{-1} \kappa_\mathfrak{V} (h_0) \mathcal{S}_0$ as $\mathfrak{x}$, $\mathfrak{y}$, and $\mathfrak{h}$, at the beginning of the statement (c). Hence, $\mathfrak{x}$ and $\mathfrak{y}$ are now assumed to be in the block diagonal form as follows, from the beginning (then $\mathfrak{h}$ would be in a certain block form; although not block diagonal, it still carries the structure of $\mathfrak{V}_{1} \oplus \mathfrak{V}_{3} \oplus \mathfrak{V}_{4}$):
\begin{align*}
\mathfrak{x}
& = \mbox{diag}(\kappa_{\mathfrak{U}_{1}} (x_0), \kappa_{\mathfrak{U}_{1}} (x_0),
\kappa_{\mathfrak{U}_{2}} (x_0),\kappa_{\mathfrak{U}_{6}} (x_0),\kappa_{\mathfrak{U}_{9}} (x_0),
\kappa_{\mathfrak{U}_{10}} (x_0),\kappa_{\mathfrak{U}_{15}} (x_0)), \mbox{ and} \\
\mathfrak{y}
& = \mbox{diag}(\kappa_{\mathfrak{U}_{1}} (y_0), \kappa_{\mathfrak{U}_{1}} (y_0),
\kappa_{\mathfrak{U}_{2}} (y_0),\kappa_{\mathfrak{U}_{6}} (y_0),\kappa_{\mathfrak{U}_{9}} (y_0),
\kappa_{\mathfrak{U}_{10}} (y_0),\kappa_{\mathfrak{U}_{15}} (y_0)).
\end{align*}




(d) Let $Y = \GL_{782}(17)$, $\mathcal D = C_Y(\mathfrak D)$,
$\mathcal H = C_Y(\mathfrak H)$ and $\mathcal E = C_Y(\mathfrak
E)$. By (b) the restrictions of the compatible characters $(\chi,
\tau)$ are not multiplicity free. Therefore Theorem 7.2.2 of
\cite{michler} asserts that one has to determine the $\mathcal
H$-$\mathcal E$ double cosets of $\mathcal D$ in order to find all
the suitable representations of degree $782$ of the free product
$H*_DE$ with amalgamated subgroup $D$. How a double coset representative
gives rise to a representation will be explained later in this proof.

For each integer $k$ let $v^k$ denote the diagonal matrix of
$\GL_k(17)$ having all diagonal entries equal to $v \in K^{*}$.

From (c) we know that each element of $\mathfrak{D} = \langle \mathfrak{x}_0, \mathfrak{y}_0 \rangle$ is of the block diagonal form as follows:
\begin{align*}
\mathfrak{g}
& = \mbox{diag}(\kappa_{\mathfrak{U}_{1}} (g), \kappa_{\mathfrak{U}_{1}} (g),
\kappa_{\mathfrak{U}_{2}} (g),\kappa_{\mathfrak{U}_{6}} (g),\kappa_{\mathfrak{U}_{9}} (g),
\kappa_{\mathfrak{U}_{10}} (g),\kappa_{\mathfrak{U}_{15}} (g)),
\end{align*}
for any element $g \in D_H$. Note that all of $\kappa_{\mathfrak{U}_{1}}, \kappa_{\mathfrak{U}_{2}}, \kappa_{\mathfrak{U}_{6}}, \kappa_{\mathfrak{U}_{9}}, \kappa_{\mathfrak{U}_{10}}$, and $\kappa_{\mathfrak{U}_{15}}$ are irreducible representations of $D_H$. Therefore, by the Theorem of Artin-Wedderburn (Theorem 2.1.27 of \cite{michler})
we know that each element $\mathcal V$ of $\mathcal D = C_Y(\mathfrak{D})$ can be
represented as a blocked diagonal matrix
$$\mathcal V = diag(a,b^{21},c^{77},d^{176},e^{176},f^{330}) \in \GL_{782}(17)$$
where $a \in \GL_2(17)$ and all $b, c, d, e, f \in K^{*}$ are
uniquely determined by $\mathcal V$.

The map $\gamma : \mathcal D \rightarrow \mathcal D_1 := \GL_2(17)
\times {K^{*}}^5$ defined by
$$\gamma(\mathcal V) = diag(a,b,c,d,e,f) \in \GL_7(17)$$
is an isomorphism. Thus assertion (1) of (d) holds.

Let $\mathcal H_1 = \gamma(\mathcal H)$, $\mathcal E_1 =
\gamma(\mathcal E)$. Recall from the proof of (c) that
\begin{eqnarray*}
&&{\chi_{1}}_{|D_H} = \psi_{1}, \quad {\chi_{3}}_{|D_H} = \psi_{9} + \psi_{10}, \quad
{\chi_{4}}_{|D_H} = \psi_{1} + \psi_{2} + \psi_{6} + \psi_{15}, \quad \mbox{and}\\
&&{\tau_{1}}_{|D} = \psi_{1}, \quad {\tau_{2}}_{|D} = \psi_{1} + \psi_{2}, \quad {\tau_{10}}_{|D} = \psi_{6} + \psi_{9}, \quad {\tau_{11}}_{|D} = \psi_{10} + \psi_{15}.
\end{eqnarray*}

Since $3 \in K$ generates the multiplicative group $K^{*}$ of $K$
it follows from these restrictions that $\mathcal H_1 = \langle
a_1, a_2, a_3 \rangle$ and $\mathcal E_1 = \langle b_1, b_2, b_3,
b_4\rangle$ where the generators $a_i$ and $b_j$ denote the
diagonal matrices of $\GL_7(17)$ given in assertions (2) and (3)
of statement (d), respectively. In particular, $\mathcal{H}_1$ and $\mathcal{E}_1$ are abelian.

Recall that we have to determine the $\mathcal H$-$\mathcal E$ double cosets of $\mathcal D$.
Since $\gamma$ is an isomorphism it suffices to determine the
$\mathcal H_1$-$\mathcal E_1$ double cosets of $\mathcal D_1$. Let
$\mathcal A$ and $\mathcal B$ be the direct factors $\GL_2(17)$
and ${K^{*}}^5$ of $\mathcal D_1$, respectively, so $\mathcal D_1 = \mathcal A \times \mathcal B$. Let $\Phi_{\mathcal{A}} : \mathcal D_1 \rightarrow \mathcal{A}$ and $\Phi_{\mathcal{B}} : \mathcal D_1 \rightarrow \mathcal{B}$ be the projection mappings, i.e. $\Phi_{\mathcal{A}} (ab) = a$ and $\Phi_{\mathcal{B}} (ab) = b$ for all $a \in \mathcal{A}$ and $b \in \mathcal{B}$. Thus, for any element $g$ of $\mathcal{D}_1$, the projection $\Phi_{\mathcal{A}} (g)$ is just the upper left $2$ by $2$ submatrix of $g$, and $\Phi_{\mathcal{B}} (g)$ the lower right $5$ by $5$ submatrix of $g$.

Let $\lambda_1, \lambda_2$ be two representatives for a single $\mathcal{H}_1$-$\mathcal{E}_1$ double coset, i.e. $\mathcal{H}_1 \lambda_1 \mathcal{E}_1 = \mathcal{H}_1 \lambda_2 \mathcal{E}_1$. Then $\lambda_2 \in \mathcal{H}_1 \lambda_1 \mathcal{E}_1$, and therefore $\lambda_2 = a \lambda_1 b$ for some $a \in \mathcal{A}$ and $b \in \mathcal{B}$. Note that any element of $\mathcal{D}_1$ can uniquely be written as multiplication of an element in $\mathcal{A}$ and an element in $\mathcal{B}$. Therefore, there exist $\xi_1, \delta_1, \theta_1$ in $\mathcal{A}$ and $ \xi_2, \delta_2, \theta_2$ in $\mathcal{B}$ such that $\lambda = \theta_1 \theta_2$, $a = \xi_1 \xi_2$ and $b = \psi_1 \delta_2$. Then, $\lambda_2 = a \lambda_1 b$ implies $\lambda_2 = (\xi_1 \xi_2) (\theta_1 \theta_2) (\delta_1 \delta_2)$.

Observe that elements of $\mathcal{B}$ commutes both with elements of $\mathcal{A}$ and those of $\mathcal{B}$. Therefore, $\lambda_2 = (\xi_1 \xi_2) (\theta_1 \theta_2) (\psi_1 \psi_2)$ implies $\lambda_2 = (\xi_1 \theta_1 \psi_1) (\xi_2 \theta_2 \psi_2)$. Since $\xi_1 \theta_1 \psi_1 \in \mathcal{A}$, $\xi_2 \theta_2 \psi_2 \in \mathcal{B}$, and $\Phi_{\mathcal{B}} (\lambda_1) = \theta_2$, we have $\Phi_{\mathcal{B}} (\lambda_2) = \xi_2 \theta_2 \psi_2 = \xi_2 \Phi_{\mathcal{B}} (\lambda_1) \psi_2$. Let $\mathcal{H}_2 = \Phi_{\mathcal{B}} (\mathcal{H}_1)$ and $\mathcal{E}_2 = \Phi_{\mathcal{B}} (\mathcal{E}_1)$. Then, $\Phi_{\mathcal{B}} (\lambda_2) = \xi_2 \Phi_{\mathcal{B}} (\lambda_1) \psi_2$ means that $\Phi_{\mathcal{B}} (\lambda_1)$ and $\Phi_{\mathcal{B}} (\lambda_2)$ represent a single $\mathcal{H}_2$-$\mathcal{E}_2$ double coset in $\mathcal{B} = \Phi_{\mathcal{B}} (\mathcal{D}_1)$.

Since $\mathcal{B}$ is an abelian group, we get that $\mathcal{H}_2 \mathcal{E}_2$ is a group, and that $\mathcal{H}_2$-$\mathcal{E}_2$ double cosets are just $\mathcal{H}_2 \mathcal{E}_2$-cosets. Observe that the group $\mathcal{H}_2 \mathcal{E}_2$ are generated by $\Phi_{\mathcal{B}} (a_i)$ and $\Phi_{\mathcal{B}} (b_j)$, where
\begin{eqnarray*}
&& \Phi_{\mathcal{B}} (a_1) = \mbox{Id}(\GL_{5}(17)),
\Phi_{\mathcal{B}} (a_2) = \mbox{diag}(1,1,3,3,1),
\Phi_{\mathcal{B}} (a_3) = \mbox{diag}(3,3,1,1,3), \\
&& \Phi_{\mathcal{B}} (b_1) = \mbox{Id}(\GL_{5}(17)),
\Phi_{\mathcal{B}} (b_2) = \mbox{diag}(3,1,1,1,1), \\
&& \Phi_{\mathcal{B}} (b_3) = \mbox{diag}(1,3,3,1,1),
\Phi_{\mathcal{B}} (b_4) = \mbox{diag}(1,1,1,3,3).
\end{eqnarray*}
It is checked in \textsc{Magma} that the group $\mathcal{H}_2 \mathcal{E}_2$ has order $16^4$, and that the set of sixteen matrices $\zeta_j = \mbox{diag}(1,1,1,1,j)$, where $j$ runs through $1,2,3,\ldots,16$, serves as a complete set of representatives of distinct $\mathcal{H}_2 \mathcal{E}_2$-cosets in $\mathcal{B}$.

What is just proved is that for any $\mathcal{H}_1$-$\mathcal{E}_1$ double coset representative $\lambda$, we can find some $a \in \mathcal{H}_1$ and $b \in \mathcal{E}_1$ such that $\Phi_{\mathcal{B}} (a) \Phi_{\mathcal{B}} (\lambda) \Phi_{\mathcal{B}}(b)$ is of the form $\zeta_j$, for some $j$ (in the above argument, we had $\xi_2 = \Phi_{\mathcal{B}} (a)$ and $\psi_2 = \Phi_{\mathcal{B}} (b)$). Now, $a \lambda b$ represents the same $\mathcal{H}_1$-$\mathcal{E}_1$ double coset as $\lambda$, and we now have that $\Phi_{\mathcal{B}} (a \lambda b) = \Phi_{\mathcal{B}} (a) \Phi_{\mathcal{B}} (\lambda) \Phi_{\mathcal{B}}(b)$ is of the form $\zeta_j$. Hence, we can now assume that any double coset representative $\lambda$ satisfies $\Phi_{\mathcal{B}} (\lambda) =\zeta_j$ for some $j$. Now, suppose that $\lambda_1$ and $\lambda_2$ represent a single $\mathcal{H}_1$-$\mathcal{E}_1$ double coset. We proved above that $\Phi_{\mathcal{B}} (\lambda_1) = \zeta_j = \Phi_{\mathcal{B}} (\lambda_2)$ for some $j$.

Fix any $j \in \{1,2,3,\ldots, 16\}$. Let's find out how many distinct $\mathcal{H}_1$-$\mathcal{E}_1$ double coset representative $\lambda$ there are such that $\Phi_{\mathcal{B}} (\lambda) = \zeta_j$. Let $\lambda_1$ and $\lambda_2$ be $\mathcal{H}_1$-$\mathcal{E}_1$ double coset representatives such that $\Phi_{\mathcal{B}} (\lambda_1) = \zeta_j = \Phi_{\mathcal{B}} (\lambda_2)$ for this fixed $j$. Then $\lambda_2 \in \mathcal{H}_1 \lambda_1 \mathcal{E}_1$, and therefore $\lambda_2 = a \lambda_1 b$ for some $a \in \mathcal{H}_1$ and $b \in \mathcal{E}_1$. Observe that $\mathcal{H}_1$ is abelian and is generated by $a_1, a_2, a_3$. Therefore $a = a_1^{\pi_1} a_2^{\pi_2} a_3^{\pi_3}$ for some $\pi_1, \pi_2, \pi_3$ in $\{1,2,3,\ldots, 16\}$. Similarly, $\mathcal{E}_1$ is abelian and is generated by $b_1, b_2, b_3, b_4$. Therefore $b = b_1^{\sigma_1} b_2^{\sigma_2} b_3^{\sigma_3} b_4^{\sigma_4}$ for some $\sigma_1, \sigma_2, \sigma_3, \sigma_4$ in $\{1,2,3,\ldots, 16\}$. We can observe by computation that
\begin{align*}
a & = a_1^{\pi_1} a_2^{\pi_2} a_3^{\pi_3}
= \mbox{diag}(3^{\pi_1}, 3^{\pi_3}, 3^{\pi_3}, 3^{\pi_3}, 3^{\pi_2}, 3^{\pi_2}, 3^{\pi_3}), \quad \mbox{and} \\
b & = b_1^{\sigma_1} b_2^{\sigma_2} b_3^{\sigma_3} b_4^{\sigma_4}
= \mbox{diag}(3^{\sigma_1}, 3^{\sigma_2}, 3^{\sigma_2}, 3^{\sigma_3}, 3^{\sigma_3}, 3^{\sigma_4}, 3^{\sigma_4}).
\end{align*}

For $i=1,2$, we know that $\lambda_i = \mbox{diag}(\Phi_{\mathcal{A}}(\lambda_i), \Phi_{\mathcal{B}}(\lambda_i))$ and that $\Phi_{\mathcal{B}}(\lambda_i) = \zeta_j = \mbox{diag}(1,1,1,1,j)$. Let $\Phi_{\mathcal{A}}(\lambda_1) = \left( \begin{smallmatrix} A & B \\ C & D \end{smallmatrix} \right)$ and $\Phi_{\mathcal{A}}(\lambda_2) = \left( \begin{smallmatrix} A' & B' \\ C' & D' \end{smallmatrix} \right)$. Thus, we can observe that
\begin{align*}
a \lambda_1 b
& = a [ \mbox{diag}(\Phi_{\mathcal{A}}(\lambda_1), \Phi_{\mathcal{B}}(\lambda_1)) ] b
= a [ \mbox{diag} (\left( \begin{smallmatrix} A & B \\ C & D \end{smallmatrix} \right), 1,1,1,1,j) ] b \\
& = \mbox{diag} (
\left( \begin{smallmatrix} 3^{\pi_1} & 0 \\ 0 & 3^{\pi_3} \end{smallmatrix} \right)
\left( \begin{smallmatrix} A & B \\ C & D \end{smallmatrix} \right)
\left( \begin{smallmatrix} 3^{\sigma_1} & 0 \\ 0 & 3^{\sigma_2} \end{smallmatrix} \right),
3^{\pi_3 + \sigma_2},
3^{\pi_3 + \sigma_3},
3^{\pi_2 + \sigma_3},
3^{\pi_2 + \sigma_4},
j 3^{\pi_3 + \sigma_4}
)
\end{align*}
We also have
\begin{align*}
\lambda_2 = \mbox{diag}(\Phi_{\mathcal{A}}(\lambda_2), \Phi_{\mathcal{B}}(\lambda_2))
= \mbox{diag} (\left( \begin{smallmatrix} A' & B' \\ C' & D' \end{smallmatrix} \right), 1,1,1,1,j),
\end{align*}
and therefore from $\lambda_2 = a \lambda_1 b$ we get $\left( \begin{smallmatrix} 3^{\pi_1} & 0 \\ 0 & 3^{\pi_3} \end{smallmatrix} \right)
\left( \begin{smallmatrix} A & B \\ C & D \end{smallmatrix} \right)
\left( \begin{smallmatrix} 3^{\sigma_1} & 0 \\ 0 & 3^{\sigma_2} \end{smallmatrix} \right) = \left( \begin{smallmatrix} A' & B' \\ C' & D' \end{smallmatrix} \right)$ and $3^{\pi_3 + \sigma_2} = 3^{\pi_3 + \sigma_3}= 3^{\pi_2 + \sigma_3} = 3^{\pi_2 + \sigma_4} = 3^{\pi_3 + \sigma_4}=1$. It is easy to see that $3^{\pi_3 + \sigma_2} = 3^{\pi_3 + \sigma_3}= 3^{\pi_2 + \sigma_3} = 3^{\pi_2 + \sigma_4} = 3^{\pi_3 + \sigma_4}=1$ implies
\begin{align*}
\pi_3 = -\sigma_2, \quad \pi_3 = -\sigma_3, \quad \pi_2 = -\sigma_3,
\quad \pi_2 = -\sigma_4, \quad \pi_3 = - \sigma_4,
\end{align*}
so that $\pi_2 = \pi_3 = -\sigma_2 = -\sigma_3 = -\sigma_4$.

Now let's look at
{\renewcommand{\arraystretch}{0.5}
\begin{align*}
\left( \begin{array}{*{2}{c@{\,}}c}
3^{\pi_1} & 0\\
0 & 3^{\pi_3}\\
\end{array} \right)
\left( \begin{array}{*{2}{c@{\,}}c}
A & B\\
C & D\\
\end{array} \right)
\left( \begin{array}{*{2}{c@{\,}}c}
3^{\sigma_1} & 0\\
0 & 3^{\sigma_2}\\
\end{array} \right) =
\left( \begin{array}{*{2}{c@{\,}}c}
A' & B'\\
C' & D'\\
\end{array} \right).
\end{align*}
}
Since $\pi_1, \pi_3, \sigma_1$ are arbitrary and $\sigma_2 = -\pi_3$, the above equation can be rewritten as
{\renewcommand{\arraystretch}{0.5}
\begin{align*}
\left( \begin{array}{*{2}{c@{\,}}c}
s & 0\\
0 & u\\
\end{array} \right)
\left( \begin{array}{*{2}{c@{\,}}c}
A & B\\
C & D\\
\end{array} \right)
\left( \begin{array}{*{2}{c@{\,}}c}
t & 0\\
0 & u^{-1}\\
\end{array} \right) =
\left( \begin{array}{*{2}{c@{\,}}c}
A' & B'\\
C' & D'\\
\end{array} \right),
\end{align*}
}
where $s,t,u$ are any elements in $K^*$. By computation we get $\left( \begin{smallmatrix} s & 0 \\ 0 & u \end{smallmatrix} \right)
\left( \begin{smallmatrix} A & B \\ C & D \end{smallmatrix} \right)
\left( \begin{smallmatrix} t & 0 \\ 0 & u^{-1} \end{smallmatrix} \right)$ $=$ $\left( \begin{smallmatrix} (stA) & (su^{-1}B) \\ (tuC) & D \end{smallmatrix} \right)$. Hence, the above equation becomes
{\renewcommand{\arraystretch}{0.5}
\begin{align*}
(*) : \quad \left( \begin{array}{*{2}{c@{\,}}c}
(stA) & (su^{-1}B)\\
(tuC) & D\\
\end{array} \right) =
\left( \begin{array}{*{2}{c@{\,}}c}
A' & B'\\
C' & D'\\
\end{array} \right).
\end{align*}
}
Define a relation $\sim$ by :  $\left( \begin{smallmatrix} A & B \\ C & D \end{smallmatrix} \right) \sim \left( \begin{smallmatrix} A' & B' \\ C' & D' \end{smallmatrix} \right)$ if $(*)$ holds. Then, it is easy to prove that $\sim$ is an equivalence relation. It is easy to observe that we have $\left( \begin{smallmatrix} A & B \\ C & D \end{smallmatrix} \right) \sim \left( \begin{smallmatrix} A' & B' \\ C' & D' \end{smallmatrix} \right)$ if and only if $D=D'$ and $\left( \begin{smallmatrix} A' & B' \\ C' & D \end{smallmatrix} \right)$ can be obtained by multiplying some numbers$\in K^*$ to the first column and the first row of $\left( \begin{smallmatrix} A & B \\ C & D \end{smallmatrix}\right)$. Therefore, it is just an easy computation by hand to get the complete list of representatives for distinct equivalence classes of $\sim$. The list of representatives is presented as follows.

For cases $D=0$,
{\renewcommand{\arraystretch}{0.5}
\begin{align*}
& \left( \begin{array}{*{2}{c@{\,}}c}
0 & 1 \\
1 & 0 \\
\end{array} \right), \quad \mbox{and} \quad
\left( \begin{array}{*{2}{c@{\,}}c}
1 & 1\\
c & 0\\
\end{array} \right)
\end{align*}
}
are the representatives for all possible distinct classes, where $c$ runs through $K^*$.

This gives $1+16 = 17$ cases. For cases $D\neq 0$ and at least one of $A,B,C$ is $0$,
{\renewcommand{\arraystretch}{0.5}
\begin{align*}
& \left( \begin{array}{*{2}{c@{\,}}c}
1 & 0 \\
1 & c \\
\end{array} \right), \quad
\left( \begin{array}{*{2}{c@{\,}}c}
1 & 1\\
0 & c\\
\end{array} \right), \quad
\left( \begin{array}{*{2}{c@{\,}}c}
0 & 1\\
1 & c\\
\end{array} \right), \quad \mbox{and} \quad
\left( \begin{array}{*{2}{c@{\,}}c}
1 & 0\\
0 & c\\
\end{array} \right)
\end{align*}
}
are the representatives for all possible distinct classes, where $c$ runs through $K^*$. This gives $4\times 16 = 64$ cases.

For cases when $A,B,C,D$ are all nonzero,
{\renewcommand{\arraystretch}{0.5}
\begin{align*}
& \left( \begin{array}{*{2}{c@{\,}}c}
1 & 1 \\
c & d \\
\end{array} \right)
\end{align*}
}
are the representatives for all possible distinct classes, where $c$ and $d$ runs through $K^*$ and $c\neq d$. Here we have $16^2 - 16= 240$ cases. Now we have the complete list of equivalence classes of $\sim$, and there are $17 + 64 + 240 = 321$ classes.

What is just proved is that, for any fixed $j \in \{1,2,3,\ldots,16\}$, there are $321$ distinct $\mathcal{H}_1$-$\mathcal{E}_1$ double coset representatives $\lambda$ such that $\Phi_\mathcal{B} (\lambda) = \zeta_j$. Therefore, the number of $\mathcal{H}_1$-$\mathcal{E}_1$ double cosets in $\mathcal{D}_1$ is exactly $16 \times 321 = 5136$ as desired in (4). We also have a method to enumerate all the double coset representatives; namely, each double coset representative is of the form
\begin{align*}
{M}_{k,j} = \mbox{diag}(m_k,1,1,1,1,j),
\end{align*}
where $j$ ranges in $K^*$ and $m_k \in \GL_{2}(17)$ runs through the $321$ choices for $\left( \begin{smallmatrix} A & B \\ C & D \end{smallmatrix} \right)$ described above. Since there are $321$ choices for $m_k$, the subscript $k$ can be considered to range in $\{1,2,\ldots,321\}$. For each $k$, let $\mathfrak{m}_{k,j} = \gamma^{-1} (M_{k,j})$, $\mathfrak{e}_{k,j} = \mathfrak{m}_{k,j}^{-1} \mathfrak{e}_1 \mathfrak{m}_{k,j}$, and
$$\mathfrak G_{k,j} = \langle \mathfrak x, \mathfrak y, \mathfrak h, \mathfrak e_{k,j}\rangle.$$
As stated in Theorem 7.2.2 of \cite{michler}, this matrix group $\mathfrak{G}_{k,j}$ is the representation which corresponds to the double coset representative $\mathfrak{m}_{k,j}$.

As mentioned in the statement (5), only first half of the Sylow $2$-subgroup test of Step 5(c) of Algorithm 7.4.8 of \cite{michler} for each $\mathfrak G_{k,j}$ is done here.
For each $k,j$, the author checked the orders of the matrices $\mathfrak{h} \mathfrak e_{k,j}$ and $\mathfrak{y} \mathfrak e_{k,j}\mathfrak{h}$ with \textsc{Magma}. If either of the two matrices
has an order which can't be the order of an element of $\Fi_{23}$ (information on largest order of elements in $\Fi_{23}$ is obtained in Atlas \cite{atlas}), then the particular pair $(k,j)$ is discarded.
Through this test, only one case survived, namely the case is with $j = 16$ and $k = k_0$ such that
{\renewcommand{\arraystretch}{0.5}
\begin{align*}
m_{k_0} = \left( \begin{array}{*{2}{c@{\,}}c}
1 &  1\\
9 & 14
\end{array} \right).
\end{align*}
}
For this $k_0$, let $\mathfrak{e} = \mathfrak{e}_{k,j}$ and $\mathfrak{G} = \mathfrak{G}_{k_0,16} = \langle \mathfrak{x}, \mathfrak{y}, \mathfrak{h}, \mathfrak{e} \rangle$. It is checked with MAGMA that the two matrices $\mathfrak{h}\mathfrak{e}$ and $\mathfrak{y} \mathfrak{e} \mathfrak{h}$ have orders $12$ and $13$, respectively. So (5) is done. The four generating matrices $\mathfrak{x}, \mathfrak{y}, \mathfrak{h}, \mathfrak{e}$ are documented in the author's website \cite{kim1} as mentioned in (6).


(e) Using the algorithm described in the proof of Theorem 6.2.1 of \cite{michler} implemented in \textsc{Magma}, a faithful permutation representation $P\mathfrak
G$ of $\mathfrak G$ of degree $31671$ with stabilizer $\mathfrak H$ is obtained. 
In order to use this algorithm for getting permutation representation, the author had to look for a short-word matrix in terms of $\mathfrak{x}, \mathfrak{y}, \mathfrak{h}, \mathfrak{e}$ which has order $23$ (since $23$ divides the order of $\mathfrak{G}$, but not the order of the permutation stabilizer $\mathfrak{H}$); the matrix $\mathfrak{x} \mathfrak{y} \mathfrak{e} \mathfrak{h}^2$ of order $23$ is used.


In particular, it is shown by means of \textsc{Magma} that $|\mathfrak G| = 2^{18}\cdot 3^{13}\cdot 5^2\cdot
7\cdot 11\cdot 13\cdot 17\cdot 23$, using the faithful permutation representation $P\mathfrak{G}$.

(f) Let $\mathfrak z = (\mathfrak x \mathfrak y^2)^7$. Using the faithful permutation representation $P\mathfrak
G$ of degree $31671$, it is verified that $C_{\mathfrak G}(\mathfrak z) = \mathfrak{H} \cong H$.

(g) Using the faithful permutation representation $P\mathfrak G$
of degree $31671$ and \\ Kratzer's Algorithm 5.3.18 of \cite{michler}, the representatives of all the conjugacy classes of $\mathfrak G$ are obtained using \textsc{Magma}, see Table \ref{Fi_23cc}.

(h) Furthermore, character table of $\mathfrak
G$ is computed by means of the above permutation representation
$P\mathfrak G$ and \textsc{Magma}. It coincides with the one of $\Fi_{23}$
in \cite{atlas}, p. 178 -179. The character table of $\mathfrak G$
implies that $\mathfrak G$ is a simple group. This completes the
proof.

\end{proof}

\begin{remark} Let $E_1 = V_1\rtimes \M_{23}$ be
the split extension of $\M_{23}$ by its simple module $V_1$ of
dimension $11$ over $F = \GF(2)$. Then $E_1$ has a unique
class of $2$-central involutions represented by some element $z'_1$.
However, when applying the Algorithm \ref{alg. simple} to $E_1$,
an overgroup of $C_{E_1}(z'_1)$ of odd index satisfying all conditions of
Algorithm \ref{alg. simple} was not found. Hence Algorithm \ref{alg. simple} can't be applied for further steps.
\end{remark}

\begin{remark} Let $E_2 = V_2\rtimes \M_{23}$ be
the split extension of $\M_{23}$ by its simple module $V_2$ of
dimension $11$ over $F = \GF(2)$. Then $E_2$ has three
classes of $2$-central involutions represented by some elements $z'_2$, $z'_3$, and $z'_4$.
However, when applying the Algorithm \ref{alg. simple} to $E_2$,
an overgroup of $C_{E_i}(z'_i)$ of odd index satisfying all conditions of
Algorithm \ref{alg. simple} was not found, for any $i=2,3,4$. Hence Algorithm \ref{alg. simple} can't be applied for further steps.
\end{remark}



\begin{appendix}

\newpage

\section{Representatives of conjugacy classes}\label{CFJcc}



\begin{cclass}\label{Fi_23cc E} Conjugacy classes of $E(\Fi_{23}) = E = \langle x_1,y_1,e_1 \rangle$, with subscripts dropped
{ \setlength{\arraycolsep}{1mm}
\renewcommand{\baselinestretch}{0.5}
\scriptsize
$$
 \begin{array}{|c|r|r|c|c|c|c|c|c|
}\hline
 \mbox{Class} &
 \multicolumn{1}{c|}{\mbox{Representative}} &
 \multicolumn{1}{c|}{|\mbox{Centralizer}|}
 &
 2{\rm P}
 & 3{\rm P}
 & 5{\rm P}
 & 7{\rm P}
 & 11{\rm P}
 & 23{\rm P}
 \\\hline
 1 & 1 & 2^{18}\cdot3^{2}\cdot5\cdot7\cdot11\cdot23
 & 1 & 1 & 1 & 1 & 1 & 1 \\\hline

 2_{1} & (y)^{7} & 2^{18}\cdot3^{2}\cdot5\cdot7\cdot11
 & 1 & 2_{1} & 2_{1} & 2_{1} & 2_{1} & 2_{1} \\\hline

 2_{2} & (y^{3}e)^{6} & 2^{18}\cdot3^{2}\cdot5\cdot7
 & 1 & 2_{2} & 2_{2} & 2_{2} & 2_{2} & 2_{2} \\\hline

 2_{3} & (xye)^{8} & 2^{18}\cdot3^{2}\cdot5
 & 1 & 2_{3} & 2_{3} & 2_{3} & 2_{3} & 2_{3} \\\hline

 2_{4} & e & 2^{14}\cdot3\cdot7
 & 1 & 2_{4} & 2_{4} & 2_{4} & 2_{4} & 2_{4} \\\hline

 2_{5} & x & 2^{14}\cdot3
 & 1 & 2_{5} & 2_{5} & 2_{5} & 2_{5} & 2_{5} \\\hline

 3 & (y^{3}e)^{4} & 2^{7}\cdot3^{2}\cdot5
 & 3 & 1 & 3 & 3 & 3 & 3 \\\hline

 4_{1} & (xy^{2}xey^{2})^{7} & 2^{11}\cdot3\cdot7
 & 2_{2} & 4_{1} & 4_{1} & 4_{1} & 4_{1} & 4_{1} \\\hline

 4_{2} & (xye)^{4} & 2^{12}\cdot3
 & 2_{3} & 4_{2} & 4_{2} & 4_{2} & 4_{2} & 4_{2} \\\hline

 4_{3} & (y^{2}e)^{2} & 2^{12}\cdot3
 & 2_{3} & 4_{3} & 4_{3} & 4_{3} & 4_{3} & 4_{3} \\\hline

 4_{4} & (y^{3}e)^{3} & 2^{11}\cdot3
 & 2_{2} & 4_{4} & 4_{4} & 4_{4} & 4_{4} & 4_{4} \\\hline

 4_{5} & (xey)^{2} & 2^{9}
 & 2_{5} & 4_{5} & 4_{5} & 4_{5} & 4_{5} & 4_{5} \\\hline

 4_{6} & (xyey^{3})^{2} & 2^{9}
 & 2_{5} & 4_{6} & 4_{6} & 4_{6} & 4_{6} & 4_{6} \\\hline

 4_{7} & xe & 2^{8}
 & 2_{4} & 4_{7} & 4_{7} & 4_{7} & 4_{7} & 4_{7} \\\hline

 5 & (xyxye)^{2} & 2^{3}\cdot3\cdot5
 & 5 & 5 & 1 & 5 & 5 & 5 \\\hline

 6_{1} & (xyxey^{2})^{5} & 2^{7}\cdot3^{2}\cdot5
 & 3 & 2_{3} & 6_{1} & 6_{1} & 6_{1} & 6_{1} \\\hline

 6_{2} & (xyexe)^{2} & 2^{7}\cdot3^{2}
 & 3 & 2_{3} & 6_{2} & 6_{2} & 6_{2} & 6_{2} \\\hline

 6_{3} & xyey^{3}e & 2^{7}\cdot3^{2}
 & 3 & 2_{1} & 6_{3} & 6_{3} & 6_{3} & 6_{3} \\\hline

 6_{4} & (y^{3}e)^{2} & 2^{6}\cdot3^{2}
 & 3 & 2_{2} & 6_{4} & 6_{4} & 6_{4} & 6_{4} \\\hline

 6_{5} & xyxexey^{2} & 2^{6}\cdot3^{2}
 & 3 & 2_{3} & 6_{5} & 6_{5} & 6_{5} & 6_{5} \\\hline

 6_{6} & y^{4}e & 2^{5}\cdot3
 & 3 & 2_{4} & 6_{6} & 6_{6} & 6_{6} & 6_{6} \\\hline

 6_{7} & xyxeyxe & 2^{5}\cdot3
 & 3 & 2_{5} & 6_{7} & 6_{7} & 6_{7} & 6_{7} \\\hline

 7_{1} & (y)^{2} & 2^{3}\cdot7
 & 7_{1} & 7_{2} & 7_{2} & 1 & 7_{1} & 7_{1} \\\hline

 7_{2} & (ye)^{2} & 2^{3}\cdot7
 & 7_{2} & 7_{1} & 7_{1} & 1 & 7_{2} & 7_{2} \\\hline

 8_{1} & (xye)^{2} & 2^{7}
 & 4_{2} & 8_{1} & 8_{1} & 8_{1} & 8_{1} & 8_{1} \\\hline

 8_{2} & y^{2}e & 2^{7}
 & 4_{3} & 8_{2} & 8_{2} & 8_{2} & 8_{2} & 8_{2} \\\hline

 8_{3} & xy^{2}ey^{2}e & 2^{7}
 & 4_{3} & 8_{3} & 8_{3} & 8_{3} & 8_{3} & 8_{3} \\\hline

 8_{4} & xey & 2^{5}
 & 4_{5} & 8_{4} & 8_{4} & 8_{4} & 8_{4} & 8_{4} \\\hline

 8_{5} & xyey^{3} & 2^{5}
 & 4_{6} & 8_{5} & 8_{5} & 8_{5} & 8_{5} & 8_{5} \\\hline

 10_{1} & (xyxey^{2})^{3} & 2^{3}\cdot3\cdot5
 & 5 & 10_{1} & 2_{3} & 10_{1} & 10_{1} & 10_{1} \\\hline

 10_{2} & xyxye & 2^{3}\cdot5
 & 5 & 10_{2} & 2_{2} & 10_{2} & 10_{2} & 10_{2} \\\hline

 10_{3} & xyxyey & 2^{3}\cdot5
 & 5 & 10_{3} & 2_{1} & 10_{3} & 10_{3} & 10_{3} \\\hline

 11_{1} & (xy^{3})^{2} & 2\cdot11
 & 11_{2} & 11_{1} & 11_{1} & 11_{2} & 1 & 11_{1} \\\hline

 11_{2} & (xy^{3})^{4} & 2\cdot11
 & 11_{1} & 11_{2} & 11_{2} & 11_{1} & 1 & 11_{2} \\\hline

 12_{1} & xyexe & 2^{5}\cdot3
 & 6_{2} & 4_{2} & 12_{1} & 12_{1} & 12_{1} & 12_{1} \\\hline

 12_{2} & xyxyxey^{2} & 2^{5}\cdot3
 & 6_{2} & 4_{3} & 12_{2} & 12_{2} & 12_{2} & 12_{2} \\\hline

 12_{3} & y^{3}e & 2^{4}\cdot3
 & 6_{4} & 4_{4} & 12_{3} & 12_{3} & 12_{3} & 12_{3} \\\hline

 12_{4} & y^{3}ey^{2}e & 2^{4}\cdot3
 & 6_{4} & 4_{1} & 12_{4} & 12_{4} & 12_{4} & 12_{4} \\\hline

 14_{1} & xyey^{2} & 2^{3}\cdot7
 & 7_{2} & 14_{2} & 14_{2} & 2_{2} & 14_{1} & 14_{1} \\\hline

 14_{2} & xy^{5} & 2^{3}\cdot7
 & 7_{1} & 14_{1} & 14_{1} & 2_{2} & 14_{2} & 14_{2} \\\hline

 14_{3} & y & 2^{2}\cdot7
 & 7_{1} & 14_{4} & 14_{4} & 2_{1} & 14_{3} & 14_{3} \\\hline

 14_{4} & ye & 2^{2}\cdot7
 & 7_{2} & 14_{3} & 14_{3} & 2_{1} & 14_{4} & 14_{4} \\\hline

 14_{5} & xyxe & 2^{2}\cdot7
 & 7_{2} & 14_{6} & 14_{6} & 2_{4} & 14_{5} & 14_{5} \\\hline

 14_{6} & xy^{3}xe & 2^{2}\cdot7
 & 7_{1} & 14_{5} & 14_{5} & 2_{4} & 14_{6} & 14_{6} \\\hline

 15_{1} & xy^{3}e & 2\cdot3\cdot5
 & 15_{1} & 5 & 3 & 15_{2} & 15_{2} & 15_{1} \\\hline

 15_{2} & (xy^{3}ey)^{2} & 2\cdot3\cdot5
 & 15_{2} & 5 & 3 & 15_{1} & 15_{1} & 15_{2} \\\hline

 16_{1} & xye & 2^{5}
 & 8_{1} & 16_{1} & 16_{2} & 16_{2} & 16_{1} & 16_{2} \\\hline

 16_{2} & xey^{3} & 2^{5}
 & 8_{1} & 16_{2} & 16_{1} & 16_{1} & 16_{2} & 16_{1} \\\hline

 22_{1} & xy^{3} & 2\cdot11
 & 11_{1} & 22_{1} & 22_{1} & 22_{2} & 2_{1} & 22_{1} \\\hline

 22_{2} & xey^{2} & 2\cdot11
 & 11_{2} & 22_{2} & 22_{2} & 22_{1} & 2_{1} & 22_{2} \\\hline

 23_{1} & xeyey & 23
 & 23_{1} & 23_{1} & 23_{2} & 23_{2} & 23_{2} & 1 \\\hline

 23_{2} & xy^{2}eyxe & 23
 & 23_{2} & 23_{2} & 23_{1} & 23_{1} & 23_{1} & 1 \\\hline

 28_{1} & xy^{2}xey^{2} & 2^{2}\cdot7
 & 14_{1} & 28_{2} & 28_{2} & 4_{1} & 28_{1} & 28_{1} \\\hline

 28_{2} & xey^{3}ey & 2^{2}\cdot7
 & 14_{2} & 28_{1} & 28_{1} & 4_{1} & 28_{2} & 28_{2} \\\hline

 30_{1} & xyxey^{2} & 2\cdot3\cdot5
 & 15_{1} & 10_{1} & 6_{1} & 30_{2} & 30_{2} & 30_{1} \\\hline

 30_{2} & xy^{3}ey & 2\cdot3\cdot5
 & 15_{2} & 10_{1} & 6_{1} & 30_{1} & 30_{1} & 30_{2} \\\hline

\end{array}
$$
}
\end{cclass}


\newpage
\begin{cclass}\label{Fi_23cc D} Conjugacy classes of $D(\Fi_{23}) = D = \langle x_1, y_1 \rangle \cong \langle x_0, y_0 \rangle$, with subscripts dropped
{ \setlength{\arraycolsep}{1mm}
\renewcommand{\baselinestretch}{0.5}
\scriptsize
$$
 \begin{array}{|c|r|r|c|c|c|c|c|
}\hline
 \mbox{Class} &
 \multicolumn{1}{c|}{\mbox{Representative}} &
 \multicolumn{1}{c|}{|\mbox{Centralizer}|}
 &
 2{\rm P}
 & 3{\rm P}
 & 5{\rm P}
 & 7{\rm P}
 & 11{\rm P}
 \\\hline
 1 & 1 & 2^{18}\cdot3^{2}\cdot5\cdot7\cdot11
 & 1 & 1 & 1 & 1 & 1 \\\hline

 2_{1} & (xy^{2})^{7} & 2^{18}\cdot3^{2}\cdot5\cdot7\cdot11
 & 1 & 2_{1} & 2_{1} & 2_{1} & 2_{1} \\\hline

 2_{2} & (y)^{7} & 2^{17}\cdot3^{2}\cdot5\cdot7
 & 1 & 2_{2} & 2_{2} & 2_{2} & 2_{2} \\\hline

 2_{3} & (xy^{5})^{7} & 2^{17}\cdot3^{2}\cdot5\cdot7
 & 1 & 2_{3} & 2_{3} & 2_{3} & 2_{3} \\\hline

 2_{4} & (xyxy^{2})^{8} & 2^{18}\cdot3\cdot5
 & 1 & 2_{4} & 2_{4} & 2_{4} & 2_{4} \\\hline

 2_{5} & (xyxy^{3})^{6} & 2^{18}\cdot3\cdot5
 & 1 & 2_{5} & 2_{5} & 2_{5} & 2_{5} \\\hline

 2_{6} & (xyxy^{2}xy^{2})^{4} & 2^{16}\cdot3^{2}
 & 1 & 2_{6} & 2_{6} & 2_{6} & 2_{6} \\\hline

 2_{7} & x & 2^{14}\cdot3
 & 1 & 2_{7} & 2_{7} & 2_{7} & 2_{7} \\\hline

 2_{8} & (xy^{3}xy^{4})^{3} & 2^{14}\cdot3
 & 1 & 2_{8} & 2_{8} & 2_{8} & 2_{8} \\\hline

 2_{9} & (xy^{2}xy^{3}xy^{2}xy^{5})^{2} & 2^{13}
 & 1 & 2_{9} & 2_{9} & 2_{9} & 2_{9} \\\hline

 3 & (xyxy^{3})^{4} & 2^{7}\cdot3^{2}
 & 3 & 1 & 3 & 3 & 3 \\\hline

 4_{1} & (xyxy^{3})^{3} & 2^{11}\cdot3
 & 2_{5} & 4_{1} & 4_{1} & 4_{1} & 4_{1} \\\hline

 4_{2} & (xyxy^{3}xyxy^{4})^{3} & 2^{11}\cdot3
 & 2_{5} & 4_{2} & 4_{2} & 4_{2} & 4_{2} \\\hline

 4_{3} & (xyxy^{2})^{4} & 2^{12}
 & 2_{4} & 4_{3} & 4_{3} & 4_{3} & 4_{3} \\\hline

 4_{4} & (xyxyxy^{2}xy^{6})^{2} & 2^{12}
 & 2_{4} & 4_{4} & 4_{4} & 4_{4} & 4_{4} \\\hline

 4_{5} & (xyxy^{2}xy^{2})^{2} & 2^{10}\cdot3
 & 2_{6} & 4_{5} & 4_{5} & 4_{5} & 4_{5} \\\hline

 4_{6} & (xyxy^{2}xyxy^{6})^{3} & 2^{10}\cdot3
 & 2_{6} & 4_{6} & 4_{6} & 4_{6} & 4_{6} \\\hline

 4_{7} & xy^{2}xy^{7}xy^{5} & 2^{10}
 & 2_{5} & 4_{7} & 4_{7} & 4_{7} & 4_{7} \\\hline

 4_{8} & (xyxy^{3}xy^{2})^{2} & 2^{9}
 & 2_{7} & 4_{8} & 4_{8} & 4_{8} & 4_{8} \\\hline

 4_{9} & (xyxy^{2}xy^{4})^{2} & 2^{9}
 & 2_{7} & 4_{9} & 4_{9} & 4_{9} & 4_{9} \\\hline

 4_{10} & xy^{2}xy^{3}xy^{5} & 2^{8}
 & 2_{8} & 4_{10} & 4_{10} & 4_{10} & 4_{10} \\\hline

 4_{11} & xy^{2}xy^{3}xy^{2}xy^{5} & 2^{8}
 & 2_{9} & 4_{11} & 4_{11} & 4_{11} & 4_{11} \\\hline

 4_{12} & xyxyxy^{4}xy^{4}xy^{2} & 2^{8}
 & 2_{9} & 4_{12} & 4_{12} & 4_{12} & 4_{12} \\\hline

 4_{13} & xyxyxyxy^{2}xy^{2} & 2^{7}
 & 2_{8} & 4_{13} & 4_{13} & 4_{13} & 4_{13} \\\hline

 5 & (xyxy^{5})^{2} & 2^{3}\cdot5
 & 5 & 5 & 1 & 5 & 5 \\\hline

 6_{1} & (xyxy^{2}xyxy^{6})^{2} & 2^{7}\cdot3^{2}
 & 3 & 2_{6} & 6_{1} & 6_{1} & 6_{1} \\\hline

 6_{2} & xy^{2}xy^{2}xy^{6}xy^{3} & 2^{7}\cdot3^{2}
 & 3 & 2_{6} & 6_{2} & 6_{2} & 6_{2} \\\hline

 6_{3} & xyxy^{2}xy^{3}xy^{2}xy^{5} & 2^{7}\cdot3^{2}
 & 3 & 2_{1} & 6_{3} & 6_{3} & 6_{3} \\\hline

 6_{4} & xyxyxy^{2}xyxyxy^{3} & 2^{5}\cdot3^{2}
 & 3 & 2_{6} & 6_{4} & 6_{4} & 6_{4} \\\hline

 6_{5} & xyxy^{2}xy^{6}xy^{4} & 2^{5}\cdot3^{2}
 & 3 & 2_{6} & 6_{5} & 6_{5} & 6_{5} \\\hline

 6_{6} & xyxy^{3}xyxy^{4}xy^{3} & 2^{5}\cdot3^{2}
 & 3 & 2_{3} & 6_{6} & 6_{6} & 6_{6} \\\hline

 6_{7} & xyxyxy^{4}xy^{3}xyxy^{2} & 2^{5}\cdot3^{2}
 & 3 & 2_{2} & 6_{7} & 6_{7} & 6_{7} \\\hline

 6_{8} & (xyxy^{3})^{2} & 2^{6}\cdot3
 & 3 & 2_{5} & 6_{8} & 6_{8} & 6_{8} \\\hline

 6_{9} & xyxy^{5}xy^{2} & 2^{6}\cdot3
 & 3 & 2_{4} & 6_{9} & 6_{9} & 6_{9} \\\hline

 6_{10} & xy^{3}xy^{4} & 2^{5}\cdot3
 & 3 & 2_{8} & 6_{10} & 6_{10} & 6_{10} \\\hline

 6_{11} & xy^{4}xy^{6} & 2^{5}\cdot3
 & 3 & 2_{7} & 6_{11} & 6_{11} & 6_{11} \\\hline

 7_{1} & (y)^{2} & 2^{2}\cdot7
 & 7_{1} & 7_{2} & 7_{2} & 1 & 7_{1} \\\hline

 7_{2} & (xy^{2})^{2} & 2^{2}\cdot7
 & 7_{2} & 7_{1} & 7_{1} & 1 & 7_{2} \\\hline

 8_{1} & (xyxy^{2})^{2} & 2^{7}
 & 4_{3} & 8_{1} & 8_{1} & 8_{1} & 8_{1} \\\hline

 8_{2} & xyxyxy^{2}xy^{6} & 2^{7}
 & 4_{4} & 8_{2} & 8_{2} & 8_{2} & 8_{2} \\\hline

 8_{3} & xyxy^{6}xy^{3}xy^{4} & 2^{7}
 & 4_{4} & 8_{3} & 8_{3} & 8_{3} & 8_{3} \\\hline

 8_{4} & xyxy^{2}xy^{2} & 2^{6}
 & 4_{5} & 8_{4} & 8_{4} & 8_{4} & 8_{4} \\\hline

 8_{5} & xy^{2}xy^{3}xy^{3} & 2^{6}
 & 4_{3} & 8_{5} & 8_{5} & 8_{5} & 8_{5} \\\hline

 8_{6} & xy^{2}xy^{2}xy^{2}xy^{3} & 2^{6}
 & 4_{5} & 8_{6} & 8_{6} & 8_{6} & 8_{6} \\\hline

 8_{7} & xyxy^{3}xy^{2} & 2^{5}
 & 4_{8} & 8_{7} & 8_{7} & 8_{7} & 8_{7} \\\hline

 8_{8} & xyxy^{2}xy^{4} & 2^{5}
 & 4_{9} & 8_{8} & 8_{8} & 8_{8} & 8_{8} \\\hline

 10_{1} & xyxy^{5} & 2^{3}\cdot5
 & 5 & 10_{2} & 2_{3} & 10_{2} & 10_{1} \\\hline

 10_{2} & xyxy^{6} & 2^{3}\cdot5
 & 5 & 10_{1} & 2_{3} & 10_{1} & 10_{2} \\\hline

 10_{3} & xyxy^{3}xy^{3} & 2^{3}\cdot5
 & 5 & 10_{3} & 2_{5} & 10_{3} & 10_{3} \\\hline

 10_{4} & xyxy^{2}xyxy^{3} & 2^{3}\cdot5
 & 5 & 10_{6} & 2_{2} & 10_{6} & 10_{4} \\\hline

 10_{5} & xy^{2}xy^{3}xy^{4} & 2^{3}\cdot5
 & 5 & 10_{5} & 2_{4} & 10_{5} & 10_{5} \\\hline

 10_{6} & xy^{2}xy^{4}xy^{3} & 2^{3}\cdot5
 & 5 & 10_{4} & 2_{2} & 10_{4} & 10_{6} \\\hline

 10_{7} & xyxyxy^{4}xy^{3}xy^{2} & 2^{3}\cdot5
 & 5 & 10_{7} & 2_{1} & 10_{7} & 10_{7} \\\hline
 11_{1} & (xy^{3})^{2} & 2\cdot11
 & 11_{2} & 11_{1} & 11_{1} & 11_{2} & 1 \\\hline

 11_{2} & (xy^{3})^{4} & 2\cdot11
 & 11_{1} & 11_{2} & 11_{2} & 11_{1} & 1 \\\hline

 12_{1} & xyxy^{2}xyxy^{6} & 2^{5}\cdot3
 & 6_{1} & 4_{6} & 12_{1} & 12_{1} & 12_{1} \\\hline

 12_{2} & xyxyxy^{2}xy^{7} & 2^{5}\cdot3
 & 6_{1} & 4_{5} & 12_{2} & 12_{2} & 12_{2} \\\hline

 12_{3} & xyxy^{3} & 2^{4}\cdot3
 & 6_{8} & 4_{1} & 12_{3} & 12_{3} & 12_{3} \\\hline

 12_{4} & xyxy^{3}xyxy^{4} & 2^{4}\cdot3
 & 6_{8} & 4_{2} & 12_{4} & 12_{4} & 12_{4} \\\hline

 14_{1} & y & 2^{2}\cdot7
 & 7_{1} & 14_{6} & 14_{6} & 2_{2} & 14_{1} \\\hline

 14_{2} & xy^{2} & 2^{2}\cdot7
 & 7_{2} & 14_{5} & 14_{5} & 2_{1} & 14_{2} \\\hline

 14_{3} & xy^{5} & 2^{2}\cdot7
 & 7_{1} & 14_{4} & 14_{4} & 2_{3} & 14_{3} \\\hline

 14_{4} & xy^{2}xy^{3} & 2^{2}\cdot7
 & 7_{2} & 14_{3} & 14_{3} & 2_{3} & 14_{4} \\\hline

 14_{5} & xy^{2}xy^{4} & 2^{2}\cdot7
 & 7_{1} & 14_{2} & 14_{2} & 2_{1} & 14_{5} \\\hline

 14_{6} & xy^{3}xy^{5} & 2^{2}\cdot7
 & 7_{2} & 14_{1} & 14_{1} & 2_{2} & 14_{6} \\\hline

 16_{1} & xyxy^{2} & 2^{5}
 & 8_{1} & 16_{1} & 16_{2} & 16_{2} & 16_{1} \\\hline

 16_{2} & xyxyxy^{2} & 2^{5}
 & 8_{1} & 16_{2} & 16_{1} & 16_{1} & 16_{2} \\\hline

 22_{1} & xy^{3} & 2\cdot11
 & 11_{1} & 22_{1} & 22_{1} & 22_{2} & 2_{1} \\\hline

 22_{2} & xyxyxy^{2}xy^{2} & 2\cdot11
 & 11_{2} & 22_{2} & 22_{2} & 22_{1} & 2_{1} \\\hline

\end{array}
$$

}

\end{cclass}


\newpage

\begin{cclass}\label{Fi_23cc H_2} Conjugacy classes of $H_2 = \langle p_2, q_2, h_2 \rangle$, with subscripts dropped
{ \setlength{\arraycolsep}{1mm}
\renewcommand{\baselinestretch}{0.5}
\scriptsize
$$
 \begin{array}{|c|r|r|c|c|c|
}\hline
 \mbox{Class} &
 \multicolumn{1}{c|}{\mbox{Representative}} &
 \multicolumn{1}{c|}{|\mbox{Centralizer}|}
 &
 2{\rm P}
 & 3{\rm P}
 & 5{\rm P}
 \\\hline
 1 & 1 & 2^{18}\cdot3^{4}\cdot5
 & 1 & 1 & 1 \\\hline

 2_{1} & (pq)^{8} & 2^{18}\cdot3^{4}\cdot5
 & 1 & 2_{1} & 2_{1} \\\hline

 2_{2} & (p^{2}q)^{5} & 2^{18}\cdot3^{4}\cdot5
 & 1 & 2_{2} & 2_{2} \\\hline

 2_{3} & (p^{2}h^{2}q)^{5} & 2^{18}\cdot3^{4}\cdot5
 & 1 & 2_{3} & 2_{3} \\\hline

 2_{4} & (q)^{3} & 2^{17}\cdot3^{4}\cdot5
 & 1 & 2_{4} & 2_{4} \\\hline

 2_{5} & (p^{2}q^{2}h)^{5} & 2^{17}\cdot3^{4}\cdot5
 & 1 & 2_{5} & 2_{5} \\\hline

 2_{6} & (qh)^{6} & 2^{16}\cdot3^{2}
 & 1 & 2_{6} & 2_{6} \\\hline

 2_{7} & (p^{2}hqh)^{6} & 2^{16}\cdot3^{2}
 & 1 & 2_{7} & 2_{7} \\\hline

 2_{8} & (p^{2}qp^{2}h)^{3} & 2^{16}\cdot3^{2}
 & 1 & 2_{8} & 2_{8} \\\hline

 2_{9} & (phphq^{2}h)^{3} & 2^{16}\cdot3^{2}
 & 1 & 2_{9} & 2_{9} \\\hline

 2_{10} & (pqh)^{6} & 2^{17}\cdot3
 & 1 & 2_{10} & 2_{10} \\\hline

 2_{11} & (pqhq)^{6} & 2^{17}\cdot3
 & 1 & 2_{11} & 2_{11} \\\hline

 2_{12} & (pqph)^{6} & 2^{15}\cdot3^{2}
 & 1 & 2_{12} & 2_{12} \\\hline

 2_{13} & (p^{2}hq^{2})^{3} & 2^{15}\cdot3^{2}
 & 1 & 2_{13} & 2_{13} \\\hline

 2_{14} & (p^{2}h^{2}qh)^{4} & 2^{16}\cdot3
 & 1 & 2_{14} & 2_{14} \\\hline

 2_{15} & (ph)^{5} & 2^{11}\cdot3^{2}\cdot5
 & 1 & 2_{15} & 2_{15} \\\hline

 2_{16} & (phqh)^{5} & 2^{11}\cdot3^{2}\cdot5
 & 1 & 2_{16} & 2_{16} \\\hline

 2_{17} & (p)^{2} & 2^{13}\cdot3
 & 1 & 2_{17} & 2_{17} \\\hline

 2_{18} & (pq^{2})^{6} & 2^{13}\cdot3
 & 1 & 2_{18} & 2_{18} \\\hline

 2_{19} & (p^{2}qp^{2}qh)^{3} & 2^{13}\cdot3
 & 1 & 2_{19} & 2_{19} \\\hline

 2_{20} & (p^{2}qpqph)^{3} & 2^{13}\cdot3
 & 1 & 2_{20} & 2_{20} \\\hline

 2_{21} & p^{2}q^{3} & 2^{12}\cdot3
 & 1 & 2_{21} & 2_{21} \\\hline

 2_{22} & (pqpq^{2}h)^{3} & 2^{12}\cdot3
 & 1 & 2_{22} & 2_{22} \\\hline

 2_{23} & p^{3}h^{2}qhqpq & 2^{13}
 & 1 & 2_{23} & 2_{23} \\\hline

 2_{24} & (p^{2}hph^{2})^{3} & 2^{10}\cdot3
 & 1 & 2_{24} & 2_{24} \\\hline

 3_{1} & h & 2^{8}\cdot3^{4}
 & 3_{1} & 1 & 3_{1} \\\hline

 3_{2} & (pqh)^{4} & 2^{9}\cdot3^{3}
 & 3_{2} & 1 & 3_{2} \\\hline

 3_{3} & (q)^{2} & 2^{6}\cdot3^{3}
 & 3_{3} & 1 & 3_{3} \\\hline

 4_{1} & (p^{2}qph)^{6} & 2^{12}\cdot3^{3}
 & 2_{1} & 4_{1} & 4_{1} \\\hline

 4_{2} & (phq)^{5} & 2^{11}\cdot3^{2}\cdot5
 & 2_{2} & 4_{2} & 4_{2} \\\hline

 4_{3} & (p^{3}hqh)^{3} & 2^{11}\cdot3^{2}\cdot5
 & 2_{2} & 4_{3} & 4_{3} \\\hline

 4_{4} & (pq)^{4} & 2^{12}\cdot3
 & 2_{1} & 4_{4} & 4_{4} \\\hline

 4_{5} & (pqh)^{3} & 2^{11}\cdot3
 & 2_{10} & 4_{5} & 4_{5} \\\hline

 4_{6} & (q^{2}h)^{3} & 2^{11}\cdot3
 & 2_{10} & 4_{6} & 4_{6} \\\hline

 4_{7} & (pqphq)^{3} & 2^{11}\cdot3
 & 2_{10} & 4_{7} & 4_{7} \\\hline

 4_{8} & (pq^{4}h)^{3} & 2^{11}\cdot3
 & 2_{11} & 4_{8} & 4_{8} \\\hline

 4_{9} & (pqphpqh)^{3} & 2^{11}\cdot3
 & 2_{11} & 4_{9} & 4_{9} \\\hline

 4_{10} & (p^{2}qpq^{2}h^{2})^{3} & 2^{11}\cdot3
 & 2_{10} & 4_{10} & 4_{10} \\\hline

 4_{11} & (qh)^{3} & 2^{10}\cdot3
 & 2_{6} & 4_{11} & 4_{11} \\\hline

 4_{12} & (qh^{2})^{3} & 2^{10}\cdot3
 & 2_{10} & 4_{12} & 4_{12} \\\hline

 4_{13} & (pqhq)^{3} & 2^{10}\cdot3
 & 2_{11} & 4_{13} & 4_{13} \\\hline

 4_{14} & (p^{2}hqh)^{3} & 2^{10}\cdot3
 & 2_{7} & 4_{14} & 4_{14} \\\hline

 4_{15} & (pq^{2}h^{2})^{3} & 2^{10}\cdot3
 & 2_{10} & 4_{15} & 4_{15} \\\hline

 4_{16} & p^{3}hph & 2^{10}\cdot3
 & 2_{7} & 4_{16} & 4_{16} \\\hline

 4_{17} & (p^{2}hph^{2}q)^{3} & 2^{10}\cdot3
 & 2_{2} & 4_{17} & 4_{17} \\\hline

 4_{18} & (p^{2}hphphq)^{3} & 2^{10}\cdot3
 & 2_{6} & 4_{18} & 4_{18} \\\hline

 4_{19} & (p^{2}hphq)^{2} & 2^{11}
 & 2_{11} & 4_{19} & 4_{19} \\\hline

 4_{20} & (pqpqhq)^{2} & 2^{11}
 & 2_{11} & 4_{20} & 4_{20} \\\hline

 4_{21} & (pq^{2})^{3} & 2^{9}\cdot3
 & 2_{18} & 4_{21} & 4_{21} \\\hline

 4_{22} & (p^{3}q)^{3} & 2^{9}\cdot3
 & 2_{17} & 4_{22} & 4_{22} \\\hline

 4_{23} & (pqph)^{3} & 2^{9}\cdot3
 & 2_{12} & 4_{23} & 4_{23} \\\hline

 4_{24} & (p^{3}qh)^{3} & 2^{9}\cdot3
 & 2_{18} & 4_{24} & 4_{24} \\\hline

 4_{25} & (p^{2}qhq)^{3} & 2^{9}\cdot3
 & 2_{12} & 4_{25} & 4_{25} \\\hline

 4_{26} & (p^{2}hph)^{3} & 2^{9}\cdot3
 & 2_{17} & 4_{26} & 4_{26} \\\hline

 4_{27} & (p^{2}h^{2}qh)^{2} & 2^{10}
 & 2_{14} & 4_{27} & 4_{27} \\\hline

 4_{28} & p^{3}qph^{2}qh & 2^{10}
 & 2_{14} & 4_{28} & 4_{28} \\\hline

 4_{29} & p & 2^{9}
 & 2_{17} & 4_{29} & 4_{29} \\\hline

 4_{30} & pq^{3} & 2^{9}
 & 2_{18} & 4_{30} & 4_{30} \\\hline

 4_{31} & phqh^{2} & 2^{9}
 & 2_{17} & 4_{31} & 4_{31} \\\hline

 4_{32} & p^{3}qp^{2}h & 2^{9}
 & 2_{11} & 4_{32} & 4_{32} \\\hline

 4_{33} & p^{2}q^{2}ph^{2} & 2^{9}
 & 2_{10} & 4_{33} & 4_{33} \\\hline

 4_{34} & p^{2}qhqhq & 2^{9}
 & 2_{17} & 4_{34} & 4_{34} \\\hline

 4_{35} & p^{2}hqhq^{2} & 2^{9}
 & 2_{17} & 4_{35} & 4_{35} \\\hline

 4_{36} & pqphqph & 2^{9}
 & 2_{18} & 4_{36} & 4_{36} \\\hline

 4_{37} & pq^{2}hphqh & 2^{9}
 & 2_{17} & 4_{37} & 4_{37} \\\hline

 4_{38} & p^{2}hphpq^{2}h & 2^{9}
 & 2_{17} & 4_{38} & 4_{38} \\\hline

 4_{39} & pq^{2}h^{2}q^{2}ph & 2^{9}
 & 2_{10} & 4_{39} & 4_{39} \\\hline

 4_{40} & pq^{2}phq & 2^{8}
 & 2_{17} & 4_{40} & 4_{40} \\\hline

 4_{41} & pqpqpqph & 2^{8}
 & 2_{17} & 4_{41} & 4_{41} \\\hline

 5 & (ph)^{2} & 2^{4}\cdot5
 & 5 & 5 & 1 \\\hline

 6_{1} & q^{3}h & 2^{8}\cdot3^{4}
 & 3_{1} & 2_{4} & 6_{5} \\\hline

\end{array}
$$

\newpage
 \normalsize
 Conjugacy classes of $H_2 = \langle p_2, q_2, h_2 \rangle$, with subscripts dropped
 (continued)
 \scriptsize
 $$
 \begin{array}{|c|r|r|c|c|c|
}\hline
 \mbox{Class} &
 \multicolumn{1}{c|}{\mbox{Representative}} &
 \multicolumn{1}{c|}{|\mbox{Centralizer}|}
 &
 2{\rm P}
 & 3{\rm P}
 & 5{\rm P}
 \\\hline

 6_{2} & (pqpqh)^{3} & 2^{8}\cdot3^{4}
 & 3_{1} & 2_{5} & 6_{3} \\\hline

 6_{3} & (pq^{2}ph)^{3} & 2^{8}\cdot3^{4}
 & 3_{1} & 2_{5} & 6_{2} \\\hline

 6_{4} & (phphq)^{2} & 2^{8}\cdot3^{4}
 & 3_{1} & 2_{1} & 6_{4} \\\hline

 6_{5} & q^{3}h^{2} & 2^{8}\cdot3^{4}
 & 3_{1} & 2_{4} & 6_{1} \\\hline

 6_{6} & (p^{2}qhqh)^{3} & 2^{8}\cdot3^{4}
 & 3_{1} & 2_{3} & 6_{6} \\\hline

 6_{7} & (pqhqph^{2})^{3} & 2^{8}\cdot3^{4}
 & 3_{1} & 2_{2} & 6_{7} \\\hline

 6_{8} & (p^{2}qph)^{4} & 2^{9}\cdot3^{3}
 & 3_{2} & 2_{1} & 6_{8} \\\hline

 6_{9} & (p^{3}hqh)^{2} & 2^{9}\cdot3^{3}
 & 3_{2} & 2_{2} & 6_{9} \\\hline

 6_{10} & p^{2}hph^{2}pq^{2} & 2^{9}\cdot3^{3}
 & 3_{2} & 2_{3} & 6_{10} \\\hline

 6_{11} & p^{3}qhpq & 2^{8}\cdot3^{3}
 & 3_{2} & 2_{5} & 6_{11} \\\hline

 6_{12} & p^{2}hpqhpq & 2^{8}\cdot3^{3}
 & 3_{2} & 2_{4} & 6_{12} \\\hline

 6_{13} & (qh)^{2} & 2^{8}\cdot3^{2}
 & 3_{1} & 2_{6} & 6_{13} \\\hline

 6_{14} & (pqph)^{2} & 2^{8}\cdot3^{2}
 & 3_{1} & 2_{12} & 6_{16} \\\hline

 6_{15} & (p^{2}hqh)^{2} & 2^{8}\cdot3^{2}
 & 3_{1} & 2_{7} & 6_{15} \\\hline

 6_{16} & (pq^{2}ph^{2})^{2} & 2^{8}\cdot3^{2}
 & 3_{1} & 2_{12} & 6_{14} \\\hline

 6_{17} & phqph^{2} & 2^{8}\cdot3^{2}
 & 3_{1} & 2_{13} & 6_{18} \\\hline

 6_{18} & phqhph^{2} & 2^{8}\cdot3^{2}
 & 3_{1} & 2_{13} & 6_{17} \\\hline

 6_{19} & pq^{3}h^{2}phq & 2^{8}\cdot3^{2}
 & 3_{1} & 2_{8} & 6_{19} \\\hline

 6_{20} & p^{3}hpqh^{2}q^{2} & 2^{8}\cdot3^{2}
 & 3_{1} & 2_{9} & 6_{20} \\\hline

 6_{21} & (pq^{2}hq^{2})^{2} & 2^{6}\cdot3^{3}
 & 3_{3} & 2_{2} & 6_{21} \\\hline

 6_{22} & p^{2}hphpq & 2^{6}\cdot3^{3}
 & 3_{3} & 2_{3} & 6_{22} \\\hline

 6_{23} & q^{2}hqhqh & 2^{6}\cdot3^{3}
 & 3_{3} & 2_{1} & 6_{23} \\\hline

 6_{24} & p^{2}qh^{2} & 2^{7}\cdot3^{2}
 & 3_{2} & 2_{6} & 6_{24} \\\hline

 6_{25} & p^{2}qp^{2}h & 2^{7}\cdot3^{2}
 & 3_{2} & 2_{8} & 6_{25} \\\hline

 6_{26} & p^{2}hphph & 2^{7}\cdot3^{2}
 & 3_{2} & 2_{7} & 6_{26} \\\hline

 6_{27} & q^{2}h^{2}qhqh^{2} & 2^{7}\cdot3^{2}
 & 3_{2} & 2_{9} & 6_{27} \\\hline

 6_{28} & q & 2^{5}\cdot3^{3}
 & 3_{3} & 2_{4} & 6_{28} \\\hline

 6_{29} & p^{3}qp^{2}hph & 2^{5}\cdot3^{3}
 & 3_{3} & 2_{5} & 6_{29} \\\hline

 6_{30} & (pqh)^{2} & 2^{8}\cdot3
 & 3_{2} & 2_{10} & 6_{30} \\\hline

 6_{31} & (pqhq)^{2} & 2^{8}\cdot3
 & 3_{2} & 2_{11} & 6_{31} \\\hline

 6_{32} & p^{2}hq^{2} & 2^{6}\cdot3^{2}
 & 3_{2} & 2_{13} & 6_{32} \\\hline

 6_{33} & phqhqh^{2} & 2^{6}\cdot3^{2}
 & 3_{2} & 2_{15} & 6_{33} \\\hline

 6_{34} & ph^{2}q^{2}h^{2} & 2^{6}\cdot3^{2}
 & 3_{2} & 2_{16} & 6_{34} \\\hline

 6_{35} & p^{2}qpq^{2}ph & 2^{6}\cdot3^{2}
 & 3_{2} & 2_{12} & 6_{35} \\\hline

 6_{36} & p^{3}h^{2}pq & 2^{7}\cdot3
 & 3_{2} & 2_{14} & 6_{36} \\\hline

 6_{37} & qhqh^{2} & 2^{5}\cdot3^{2}
 & 3_{3} & 2_{6} & 6_{37} \\\hline

 6_{38} & p^{3}qpq & 2^{5}\cdot3^{2}
 & 3_{3} & 2_{13} & 6_{41} \\\hline

 6_{39} & pqhq^{2}ph & 2^{5}\cdot3^{2}
 & 3_{3} & 2_{12} & 6_{43} \\\hline

 6_{40} & phphq^{2}h & 2^{5}\cdot3^{2}
 & 3_{3} & 2_{9} & 6_{40} \\\hline

 6_{41} & p^{2}q^{2}h^{2}qh & 2^{5}\cdot3^{2}
 & 3_{3} & 2_{13} & 6_{38} \\\hline

 6_{42} & p^{2}qhphpq & 2^{5}\cdot3^{2}
 & 3_{3} & 2_{7} & 6_{42} \\\hline

 6_{43} & p^{2}hqphpq & 2^{5}\cdot3^{2}
 & 3_{3} & 2_{12} & 6_{39} \\\hline

 6_{44} & pqh^{2}qhph & 2^{5}\cdot3^{2}
 & 3_{3} & 2_{8} & 6_{44} \\\hline

 6_{45} & p^{2}h & 2^{6}\cdot3
 & 3_{3} & 2_{17} & 6_{45} \\\hline

 6_{46} & (pq^{2})^{2} & 2^{6}\cdot3
 & 3_{3} & 2_{18} & 6_{46} \\\hline

 6_{47} & p^{2}qp^{2}qh & 2^{6}\cdot3
 & 3_{3} & 2_{19} & 6_{47} \\\hline

 6_{48} & p^{2}qpqph & 2^{6}\cdot3
 & 3_{3} & 2_{20} & 6_{48} \\\hline

 6_{49} & p^{2}qpq & 2^{4}\cdot3^{2}
 & 3_{3} & 2_{15} & 6_{49} \\\hline

 6_{50} & p^{2}qhpq^{2} & 2^{4}\cdot3^{2}
 & 3_{3} & 2_{16} & 6_{50} \\\hline

 6_{51} & p^{2}q^{3}h & 2^{5}\cdot3
 & 3_{3} & 2_{21} & 6_{51} \\\hline

 6_{52} & p^{2}hph^{2} & 2^{5}\cdot3
 & 3_{2} & 2_{24} & 6_{52} \\\hline

 6_{53} & pqpq^{2}h & 2^{5}\cdot3
 & 3_{3} & 2_{22} & 6_{53} \\\hline

 8_{1} & (p^{2}qph)^{3} & 2^{7}\cdot3
 & 4_{1} & 8_{1} & 8_{1} \\\hline

 8_{2} & (pqhqhq)^{3} & 2^{7}\cdot3
 & 4_{1} & 8_{2} & 8_{2} \\\hline

 8_{3} & (pq)^{2} & 2^{7}
 & 4_{4} & 8_{3} & 8_{3} \\\hline

 8_{4} & pqpqhq & 2^{7}
 & 4_{20} & 8_{4} & 8_{4} \\\hline

 8_{5} & pqh^{2}qh & 2^{7}
 & 4_{20} & 8_{5} & 8_{5} \\\hline

 8_{6} & p^{3}hqph & 2^{7}
 & 4_{20} & 8_{6} & 8_{6} \\\hline

 8_{7} & pqphphq & 2^{7}
 & 4_{20} & 8_{7} & 8_{7} \\\hline

 8_{8} & pq^{2}hq & 2^{6}
 & 4_{11} & 8_{8} & 8_{8} \\\hline

 8_{9} & p^{2}qpq^{2} & 2^{6}
 & 4_{11} & 8_{9} & 8_{9} \\\hline

 8_{10} & p^{2}q^{2}ph & 2^{6}
 & 4_{16} & 8_{10} & 8_{10} \\\hline

 8_{11} & p^{2}hphq & 2^{6}
 & 4_{19} & 8_{11} & 8_{11} \\\hline

 8_{12} & p^{2}h^{2}qh & 2^{6}
 & 4_{27} & 8_{12} & 8_{12} \\\hline

 8_{13} & p^{3}qphq & 2^{6}
 & 4_{27} & 8_{13} & 8_{13} \\\hline

 8_{14} & pqpqpqh & 2^{6}
 & 4_{16} & 8_{14} & 8_{14} \\\hline

 9 & (pqpqh)^{2} & 2^{3}\cdot3^{2}
 & 9 & 3_{1} & 9 \\\hline

 10_{1} & p^{2}q & 2^{4}\cdot5
 & 5 & 10_{1} & 2_{2} \\\hline

 10_{2} & p^{2}qh & 2^{4}\cdot5
 & 5 & 10_{2} & 2_{1} \\\hline

 10_{3} & p^{2}h^{2}q & 2^{4}\cdot5
 & 5 & 10_{3} & 2_{3} \\\hline

 10_{4} & ph & 2^{3}\cdot5
 & 5 & 10_{4} & 2_{15} \\\hline

\end{array}
$$

\newpage
 \normalsize
 Conjugacy classes of $H_2 = \langle p_2, q_2, h_2 \rangle$, with subscripts dropped
 (continued)
 \scriptsize
 $$
 \begin{array}{|c|r|r|c|c|c|
}\hline
 \mbox{Class} &
 \multicolumn{1}{c|}{\mbox{Representative}} &
 \multicolumn{1}{c|}{|\mbox{Centralizer}|}
 &
 2{\rm P}
 & 3{\rm P}
 & 5{\rm P}
 \\\hline

 10_{5} & p^{2}q^{2} & 2^{3}\cdot5
 & 5 & 10_{5} & 2_{4} \\\hline

 10_{6} & phqh & 2^{3}\cdot5
 & 5 & 10_{6} & 2_{16} \\\hline

 10_{7} & p^{2}q^{2}h & 2^{3}\cdot5
 & 5 & 10_{7} & 2_{5} \\\hline

 12_{1} & (p^{2}qph)^{2} & 2^{7}\cdot3^{2}
 & 6_{8} & 4_{1} & 12_{1} \\\hline

 12_{2} & (pqhqhq)^{2} & 2^{7}\cdot3^{2}
 & 6_{8} & 4_{1} & 12_{2} \\\hline

 12_{3} & p^{3}hpq & 2^{5}\cdot3^{3}
 & 6_{4} & 4_{1} & 12_{3} \\\hline

 12_{4} & p^{3}hqh & 2^{6}\cdot3^{2}
 & 6_{9} & 4_{3} & 12_{4} \\\hline

 12_{5} & ph^{2}qhqh & 2^{6}\cdot3^{2}
 & 6_{9} & 4_{2} & 12_{5} \\\hline

 12_{6} & p^{2}h^{2}pqhpq & 2^{6}\cdot3^{2}
 & 6_{8} & 4_{1} & 12_{6} \\\hline

 12_{7} & pqh & 2^{6}\cdot3
 & 6_{30} & 4_{5} & 12_{7} \\\hline

 12_{8} & q^{2}h & 2^{6}\cdot3
 & 6_{30} & 4_{6} & 12_{8} \\\hline

 12_{9} & pqphq & 2^{6}\cdot3
 & 6_{30} & 4_{7} & 12_{9} \\\hline

 12_{10} & pq^{4}h & 2^{6}\cdot3
 & 6_{31} & 4_{8} & 12_{10} \\\hline

 12_{11} & pqphpqh & 2^{6}\cdot3
 & 6_{31} & 4_{9} & 12_{11} \\\hline

 12_{12} & p^{2}qpq^{2}h^{2} & 2^{6}\cdot3
 & 6_{30} & 4_{10} & 12_{12} \\\hline

 12_{13} & pq^{2}hq^{2} & 2^{4}\cdot3^{2}
 & 6_{21} & 4_{2} & 12_{13} \\\hline

 12_{14} & pqpqhpq^{2} & 2^{4}\cdot3^{2}
 & 6_{21} & 4_{3} & 12_{14} \\\hline

 12_{15} & qh & 2^{5}\cdot3
 & 6_{13} & 4_{11} & 12_{15} \\\hline

 12_{16} & qh^{2} & 2^{5}\cdot3
 & 6_{30} & 4_{12} & 12_{16} \\\hline

 12_{17} & pqph & 2^{5}\cdot3
 & 6_{14} & 4_{23} & 12_{24} \\\hline

 12_{18} & pqhq & 2^{5}\cdot3
 & 6_{31} & 4_{13} & 12_{18} \\\hline

 12_{19} & p^{2}qhq & 2^{5}\cdot3
 & 6_{14} & 4_{25} & 12_{23} \\\hline

 12_{20} & p^{2}hqh & 2^{5}\cdot3
 & 6_{15} & 4_{14} & 12_{20} \\\hline

 12_{21} & pq^{2}h^{2} & 2^{5}\cdot3
 & 6_{30} & 4_{15} & 12_{21} \\\hline

 12_{22} & phphq & 2^{5}\cdot3
 & 6_{4} & 4_{4} & 12_{22} \\\hline

 12_{23} & pq^{2}ph^{2} & 2^{5}\cdot3
 & 6_{16} & 4_{25} & 12_{19} \\\hline

 12_{24} & p^{2}q^{2}p^{2}h & 2^{5}\cdot3
 & 6_{16} & 4_{23} & 12_{17} \\\hline

 12_{25} & p^{2}hph^{2}q & 2^{5}\cdot3
 & 6_{9} & 4_{17} & 12_{25} \\\hline

 12_{26} & p^{3}qh^{2}pq & 2^{5}\cdot3
 & 6_{15} & 4_{16} & 12_{26} \\\hline

 12_{27} & p^{2}hphphq & 2^{5}\cdot3
 & 6_{13} & 4_{18} & 12_{27} \\\hline

 12_{28} & pq^{2} & 2^{4}\cdot3
 & 6_{46} & 4_{21} & 12_{28} \\\hline

 12_{29} & p^{3}q & 2^{4}\cdot3
 & 6_{45} & 4_{22} & 12_{29} \\\hline

 12_{30} & p^{3}qh & 2^{4}\cdot3
 & 6_{46} & 4_{24} & 12_{30} \\\hline

 12_{31} & p^{2}hph & 2^{4}\cdot3
 & 6_{45} & 4_{26} & 12_{31} \\\hline

 16_{1} & pq & 2^{5}
 & 8_{3} & 16_{1} & 16_{2} \\\hline

 16_{2} & p^{3}q^{2} & 2^{5}
 & 8_{3} & 16_{2} & 16_{1} \\\hline

 18_{1} & pqpqh & 2^{3}\cdot3^{2}
 & 9 & 6_{2} & 18_{3} \\\hline

 18_{2} & pqph^{2} & 2^{3}\cdot3^{2}
 & 9 & 6_{1} & 18_{5} \\\hline

 18_{3} & pq^{2}ph & 2^{3}\cdot3^{2}
 & 9 & 6_{3} & 18_{1} \\\hline

 18_{4} & p^{2}qhqh & 2^{3}\cdot3^{2}
 & 9 & 6_{6} & 18_{4} \\\hline

 18_{5} & pq^{2}pqh & 2^{3}\cdot3^{2}
 & 9 & 6_{5} & 18_{2} \\\hline

 18_{6} & pqhqph^{2} & 2^{3}\cdot3^{2}
 & 9 & 6_{7} & 18_{6} \\\hline

 18_{7} & p^{3}hqphq & 2^{3}\cdot3^{2}
 & 9 & 6_{4} & 18_{7} \\\hline

 20_{1} & phq & 2^{3}\cdot5
 & 10_{1} & 20_{1} & 4_{2} \\\hline

 20_{2} & phq^{2}hq & 2^{3}\cdot5
 & 10_{1} & 20_{2} & 4_{3} \\\hline

 24_{1} & p^{2}qph & 2^{4}\cdot3
 & 12_{1} & 8_{1} & 24_{1} \\\hline

 24_{2} & pqhqhq & 2^{4}\cdot3
 & 12_{2} & 8_{2} & 24_{2} \\\hline

\end{array}
$$

}

\end{cclass}


\newpage
\begin{cclass}\label{Fi_23cc D_2} Conjugacy classes of $D_2 = \langle p_2, q_2 \rangle \cong \langle p_1, q_1 \rangle$, with subscripts dropped
{ \setlength{\arraycolsep}{1mm}
\renewcommand{\baselinestretch}{0.5}
\scriptsize
$$
 \begin{array}{|c|r|r|c|c|c|
}\hline
 \mbox{Class} &
 \multicolumn{1}{c|}{\mbox{Representative}} &
 \multicolumn{1}{c|}{|\mbox{Centralizer}|}
 &
 2{\rm P}
 & 3{\rm P}
 & 5{\rm P}
 \\\hline
 1 & 1 & 2^{18}\cdot3\cdot5
 & 1 & 1 & 1 \\\hline

 2_{1} & (pq)^{8} & 2^{18}\cdot3\cdot5
 & 1 & 2_{1} & 2_{1} \\\hline

 2_{2} & (p^{2}q)^{5} & 2^{18}\cdot3\cdot5
 & 1 & 2_{2} & 2_{2} \\\hline

 2_{3} & (p^{2}qpq^{2}pq^{2})^{5} & 2^{18}\cdot3\cdot5
 & 1 & 2_{3} & 2_{3} \\\hline

 2_{4} & (q)^{3} & 2^{17}\cdot3\cdot5
 & 1 & 2_{4} & 2_{4} \\\hline

 2_{5} & (p^{3}q^{2}p^{2}qpq)^{3} & 2^{17}\cdot3\cdot5
 & 1 & 2_{5} & 2_{5} \\\hline

 2_{6} & (p^{2}qpq^{2})^{4} & 2^{16}\cdot3
 & 1 & 2_{6} & 2_{6} \\\hline

 2_{7} & (p^{3}qpq^{4})^{3} & 2^{16}\cdot3
 & 1 & 2_{7} & 2_{7} \\\hline

 2_{8} & (p^{3}qp^{2}qpqpq)^{4} & 2^{16}\cdot3
 & 1 & 2_{8} & 2_{8} \\\hline

 2_{9} & (p^{3}qpqp^{2}q^{4})^{3} & 2^{16}\cdot3
 & 1 & 2_{9} & 2_{9} \\\hline

 2_{10} & (p^{3}qp^{2}q)^{4} & 2^{17}
 & 1 & 2_{10} & 2_{10} \\\hline

 2_{11} & (p^{2}qpq^{2}pqp^{2}q^{2})^{2} & 2^{17}
 & 1 & 2_{11} & 2_{11} \\\hline

 2_{12} & (p^{3}qpq)^{3} & 2^{15}\cdot3
 & 1 & 2_{12} & 2_{12} \\\hline

 2_{13} & (p^{3}qpqp^{2}q)^{3} & 2^{15}\cdot3
 & 1 & 2_{13} & 2_{13} \\\hline

 2_{14} & (p^{3}qp^{2}q^{2}pq)^{4} & 2^{16}
 & 1 & 2_{14} & 2_{14} \\\hline

 2_{15} & (pq^{2})^{6} & 2^{13}\cdot3
 & 1 & 2_{15} & 2_{15} \\\hline

 2_{16} & (p^{3}q)^{6} & 2^{13}\cdot3
 & 1 & 2_{16} & 2_{16} \\\hline

 2_{17} & (pqpqpqpq^{2})^{3} & 2^{13}\cdot3
 & 1 & 2_{17} & 2_{17} \\\hline

 2_{18} & (p^{3}qpq^{2}pqpq^{4})^{3} & 2^{13}\cdot3
 & 1 & 2_{18} & 2_{18} \\\hline

 2_{19} & (p^{3}q^{2}pqpq)^{4} & 2^{14}
 & 1 & 2_{19} & 2_{19} \\\hline

 2_{20} & (p^{3}q^{2}pqpq^{2})^{2} & 2^{14}
 & 1 & 2_{20} & 2_{20} \\\hline

 2_{21} & (pq^{2}pq^{5})^{3} & 2^{12}\cdot3
 & 1 & 2_{21} & 2_{21} \\\hline

 2_{22} & (p^{3}qpq^{2}pqpq)^{3} & 2^{12}\cdot3
 & 1 & 2_{22} & 2_{22} \\\hline

 2_{23} & (p^{2}q^{2}pqpqpq^{2}pq^{2})^{2} & 2^{13}
 & 1 & 2_{23} & 2_{23} \\\hline

 2_{24} & p^{3}qpqp^{2}qpq^{2}pqpqpq^{2} & 2^{13}
 & 1 & 2_{24} & 2_{24} \\\hline

 2_{25} & (p^{2}qpq)^{3} & 2^{11}\cdot3
 & 1 & 2_{25} & 2_{25} \\\hline

 2_{26} & (p^{3}q^{2}pq^{2}pq^{2})^{3} & 2^{11}\cdot3
 & 1 & 2_{26} & 2_{26} \\\hline

 2_{27} & (p)^{2} & 2^{12}
 & 1 & 2_{27} & 2_{27} \\\hline

 2_{28} & (pq^{3})^{2} & 2^{12}
 & 1 & 2_{28} & 2_{28} \\\hline

 2_{29} & p^{3}qp^{2}qpqpq^{2}pqp^{2}q & 2^{12}
 & 1 & 2_{29} & 2_{29} \\\hline

 2_{30} & p^{3}qpqpqpqpqpqpqpq & 2^{12}
 & 1 & 2_{30} & 2_{30} \\\hline

 2_{31} & p^{3}qpq^{2}pqp^{2}qpqp^{2}q^{2}p^{2}q & 2^{12}
 & 1 & 2_{31} & 2_{31} \\\hline

 2_{32} & p^{2}q^{3} & 2^{11}
 & 1 & 2_{32} & 2_{32} \\\hline

 2_{33} & p^{3}q^{2}pqpq^{2}pqp^{2}q & 2^{11}
 & 1 & 2_{33} & 2_{33} \\\hline

 2_{34} & p^{3}qpqpqpq^{2}pqpqp^{3}q^{2} & 2^{10}
 & 1 & 2_{34} & 2_{34} \\\hline

 3 & (q)^{2} & 2^{6}\cdot3
 & 3 & 1 & 3 \\\hline

 4_{1} & (p^{2}qpq^{4})^{3} & 2^{11}\cdot3
 & 2_{2} & 4_{1} & 4_{1} \\\hline

 4_{2} & (p^{3}qp^{3}qpq)^{3} & 2^{11}\cdot3
 & 2_{2} & 4_{2} & 4_{2} \\\hline

 4_{3} & (pq)^{4} & 2^{12}
 & 2_{1} & 4_{3} & 4_{3} \\\hline

 4_{4} & (pqpq^{2}pq^{2})^{2} & 2^{12}
 & 2_{1} & 4_{4} & 4_{4} \\\hline

 4_{5} & (p^{3}qp^{2}q)^{2} & 2^{11}
 & 2_{10} & 4_{5} & 4_{5} \\\hline

 4_{6} & (p^{2}q^{2}pqpqpq^{2})^{2} & 2^{11}
 & 2_{10} & 4_{6} & 4_{6} \\\hline

 4_{7} & p^{3}q^{2}pq^{2}pq^{2}pqpq & 2^{11}
 & 2_{10} & 4_{7} & 4_{7} \\\hline

 4_{8} & p^{3}qpq^{2}pqp^{2}qp^{2}q^{2} & 2^{11}
 & 2_{11} & 4_{8} & 4_{8} \\\hline

 4_{9} & p^{3}qp^{2}qp^{3}qpqp^{2}q^{2} & 2^{11}
 & 2_{10} & 4_{9} & 4_{9} \\\hline

 4_{10} & p^{3}qp^{2}q^{2}p^{3}qpqp^{2}q & 2^{11}
 & 2_{11} & 4_{10} & 4_{10} \\\hline

 4_{11} & p^{3}qpqpq^{2}p^{3}q^{2}p^{2}q & 2^{11}
 & 2_{11} & 4_{11} & 4_{11} \\\hline

 4_{12} & p^{3}qp^{2}qpqp^{2}q^{2}p^{3}qpq & 2^{11}
 & 2_{11} & 4_{12} & 4_{12} \\\hline

 4_{13} & (pq^{2})^{3} & 2^{9}\cdot3
 & 2_{15} & 4_{13} & 4_{13} \\\hline

 4_{14} & (p^{3}q)^{3} & 2^{9}\cdot3
 & 2_{16} & 4_{14} & 4_{14} \\\hline

 4_{15} & (p^{3}qpq^{2}pq)^{3} & 2^{9}\cdot3
 & 2_{15} & 4_{15} & 4_{15} \\\hline

 4_{16} & (p^{3}qp^{3}qpq^{2})^{3} & 2^{9}\cdot3
 & 2_{16} & 4_{16} & 4_{16} \\\hline

 4_{17} & (p^{2}qpq^{2})^{2} & 2^{10}
 & 2_{6} & 4_{17} & 4_{17} \\\hline
 4_{18} & (p^{3}qpqpq^{2})^{2} & 2^{10}
 & 2_{10} & 4_{18} & 4_{18} \\\hline

 4_{19} & (p^{3}qp^{2}q^{2}pq)^{2} & 2^{10}
 & 2_{14} & 4_{19} & 4_{19} \\\hline

 4_{20} & (p^{3}qp^{2}qpqpq)^{2} & 2^{10}
 & 2_{8} & 4_{20} & 4_{20} \\\hline

 4_{21} & p^{2}qpqp^{2}qpq^{2} & 2^{10}
 & 2_{8} & 4_{21} & 4_{21} \\\hline

 4_{22} & p^{3}qp^{2}qpqpq^{2}pq^{2} & 2^{10}
 & 2_{14} & 4_{22} & 4_{22} \\\hline

 4_{23} & p^{3}qpqpqp^{2}qpq^{4} & 2^{10}
 & 2_{10} & 4_{23} & 4_{23} \\\hline

 4_{24} & p^{3}qpqpqpq^{2}pqp^{2}q & 2^{10}
 & 2_{11} & 4_{24} & 4_{24} \\\hline

 4_{25} & p^{3}qp^{2}q^{2}p^{3}q^{2}pqpq & 2^{10}
 & 2_{11} & 4_{25} & 4_{25} \\\hline

 4_{26} & p^{3}qpq^{2}pqp^{3}q^{2}pq^{2} & 2^{10}
 & 2_{10} & 4_{26} & 4_{26} \\\hline

 4_{27} & p^{2}qp^{2}q^{2}pq^{2}pq^{2}pqp^{2}q^{2} & 2^{10}
 & 2_{2} & 4_{27} & 4_{27} \\\hline

 4_{28} & p^{3}q^{2}pqp^{2}q^{2}p^{2}q^{2}pqpq^{2} & 2^{10}
 & 2_{6} & 4_{28} & 4_{28} \\\hline

 4_{29} & (p^{3}q^{2}pqpq)^{2} & 2^{9}
 & 2_{19} & 4_{29} & 4_{29} \\\hline

 4_{30} & (p^{2}qpq^{2}p^{2}q^{2})^{2} & 2^{9}
 & 2_{19} & 4_{30} & 4_{30} \\\hline

 4_{31} & p^{3}q^{2}pq^{2}p^{2}q & 2^{9}
 & 2_{16} & 4_{31} & 4_{31} \\\hline

 4_{32} & p^{2}qpqpq^{5} & 2^{9}
 & 2_{16} & 4_{32} & 4_{32} \\\hline

 4_{33} & p^{3}qpqpqp^{2}q^{2} & 2^{9}
 & 2_{15} & 4_{33} & 4_{33} \\\hline

 4_{34} & p^{2}qpq^{2}pqp^{2}q^{2} & 2^{9}
 & 2_{11} & 4_{34} & 4_{34} \\\hline

 4_{35} & p^{3}q^{2}p^{2}q^{2}pqpq & 2^{9}
 & 2_{16} & 4_{35} & 4_{35} \\\hline

\end{array}
$$

\newpage
 \normalsize
 Conjugacy classes of $D_2 = \langle p_2, q_2 \rangle$, with subscripts dropped
 (continued)
 \scriptsize
 $$
 \begin{array}{|c|r|r|c|c|c|
}\hline
 \mbox{Class} &
 \multicolumn{1}{c|}{\mbox{Representative}} &
 \multicolumn{1}{c|}{|\mbox{Centralizer}|}
 &
 2{\rm P}
 & 3{\rm P}
 & 5{\rm P}
 \\\hline

 4_{36} & p^{2}q^{2}pqpqpqpq^{2} & 2^{9}
 & 2_{14} & 4_{36} & 4_{36} \\\hline

 4_{37} & p^{3}qp^{3}q^{2}pqp^{2}q & 2^{9}
 & 2_{16} & 4_{37} & 4_{37} \\\hline

 4_{38} & p^{3}qp^{2}q^{2}pq^{2}p^{2}q & 2^{9}
 & 2_{13} & 4_{38} & 4_{38} \\\hline

 4_{39} & p^{3}q^{2}pqpqpqpqpq & 2^{9}
 & 2_{16} & 4_{39} & 4_{39} \\\hline

 4_{40} & p^{3}qp^{3}qpq^{2}pqpq^{2} & 2^{9}
 & 2_{15} & 4_{40} & 4_{40} \\\hline

 4_{41} & p^{3}qp^{2}qpqp^{2}qp^{2}q^{2} & 2^{9}
 & 2_{14} & 4_{41} & 4_{41} \\\hline

 4_{42} & p^{3}q^{2}pqpqp^{2}qpqpq & 2^{9}
 & 2_{11} & 4_{42} & 4_{42} \\\hline

 4_{43} & p^{3}qp^{2}q^{2}pq^{2}pq^{3}pq & 2^{9}
 & 2_{13} & 4_{43} & 4_{43} \\\hline

 4_{44} & p^{2}q^{2}pqp^{2}q^{2}pqpq^{2}pq & 2^{9}
 & 2_{16} & 4_{44} & 4_{44} \\\hline

 4_{45} & p^{3}qp^{2}q^{2}p^{2}qpqpqp^{2}q & 2^{9}
 & 2_{10} & 4_{45} & 4_{45} \\\hline

 4_{46} & p^{2}qpqpq^{2} & 2^{8}
 & 2_{16} & 4_{46} & 4_{46} \\\hline

 4_{47} & p^{3}q^{2}pqpq^{2} & 2^{8}
 & 2_{20} & 4_{47} & 4_{47} \\\hline

 4_{48} & p^{3}qpqp^{3}q^{2} & 2^{8}
 & 2_{19} & 4_{48} & 4_{48} \\\hline

 4_{49} & p^{2}qpqpq^{2}p^{2}q^{2} & 2^{8}
 & 2_{20} & 4_{49} & 4_{49} \\\hline

 4_{50} & p^{2}qpqpqp^{2}qpq^{2} & 2^{8}
 & 2_{20} & 4_{50} & 4_{50} \\\hline

 4_{51} & p^{2}q^{2}pqpqpq^{2}pq^{2} & 2^{8}
 & 2_{23} & 4_{51} & 4_{51} \\\hline

 4_{52} & p^{3}qpqp^{3}qpqpqpq & 2^{8}
 & 2_{16} & 4_{52} & 4_{52} \\\hline

 4_{53} & p^{2}qp^{2}q^{2}pqpqp^{2}qpq & 2^{8}
 & 2_{19} & 4_{53} & 4_{53} \\\hline

 4_{54} & p^{2}qpqp^{2}q^{2}pqpqpq^{2} & 2^{8}
 & 2_{23} & 4_{54} & 4_{54} \\\hline

 4_{55} & p & 2^{7}
 & 2_{27} & 4_{55} & 4_{55} \\\hline

 4_{56} & pq^{3} & 2^{7}
 & 2_{28} & 4_{56} & 4_{56} \\\hline

 4_{57} & p^{2}qpqpqpq^{2} & 2^{7}
 & 2_{27} & 4_{57} & 4_{57} \\\hline

 4_{58} & p^{2}q^{2}p^{2}q^{2}pqpq & 2^{7}
 & 2_{20} & 4_{58} & 4_{58} \\\hline

 4_{59} & p^{2}qpqpqpq^{5} & 2^{7}
 & 2_{28} & 4_{59} & 4_{59} \\\hline

 5 & (p^{2}q)^{2} & 2^{3}\cdot5
 & 5 & 5 & 1 \\\hline

 6_{1} & (pq^{2})^{2} & 2^{6}\cdot3
 & 3 & 2_{15} & 6_{1} \\\hline

 6_{2} & (p^{3}q)^{2} & 2^{6}\cdot3
 & 3 & 2_{16} & 6_{2} \\\hline

 6_{3} & p^{3}qpq^{3} & 2^{6}\cdot3
 & 3 & 2_{2} & 6_{3} \\\hline

 6_{4} & pqpqpqpq^{2} & 2^{6}\cdot3
 & 3 & 2_{17} & 6_{4} \\\hline

 6_{5} & p^{3}q^{2}p^{2}qpq^{2} & 2^{6}\cdot3
 & 3 & 2_{3} & 6_{5} \\\hline

 6_{6} & p^{3}q^{2}p^{2}q^{4}pq & 2^{6}\cdot3
 & 3 & 2_{1} & 6_{6} \\\hline

 6_{7} & p^{3}qpq^{2}pqpq^{4} & 2^{6}\cdot3
 & 3 & 2_{18} & 6_{7} \\\hline

 6_{8} & q & 2^{5}\cdot3
 & 3 & 2_{4} & 6_{8} \\\hline

 6_{9} & p^{3}qpq & 2^{5}\cdot3
 & 3 & 2_{12} & 6_{10} \\\hline

 6_{10} & p^{3}qp^{2}qpq & 2^{5}\cdot3
 & 3 & 2_{12} & 6_{9} \\\hline
 6_{11} & p^{3}qpqp^{2}q & 2^{5}\cdot3
 & 3 & 2_{13} & 6_{19} \\\hline

 6_{12} & p^{3}qpq^{4} & 2^{5}\cdot3
 & 3 & 2_{7} & 6_{12} \\\hline

 6_{13} & p^{3}q^{4}pq & 2^{5}\cdot3
 & 3 & 2_{6} & 6_{13} \\\hline

 6_{14} & pq^{2}pq^{5} & 2^{5}\cdot3
 & 3 & 2_{21} & 6_{14} \\\hline

 6_{15} & p^{3}q^{2}p^{2}qpq & 2^{5}\cdot3
 & 3 & 2_{5} & 6_{15} \\\hline

 6_{16} & p^{3}qpq^{2}pqpq & 2^{5}\cdot3
 & 3 & 2_{22} & 6_{16} \\\hline

 6_{17} & p^{3}qpqp^{2}q^{4} & 2^{5}\cdot3
 & 3 & 2_{9} & 6_{17} \\\hline

 6_{18} & p^{3}qpqpq^{3}pq & 2^{5}\cdot3
 & 3 & 2_{8} & 6_{18} \\\hline

 6_{19} & p^{3}q^{2}p^{2}qpqp^{2}q & 2^{5}\cdot3
 & 3 & 2_{13} & 6_{11} \\\hline

 6_{20} & p^{2}qpq & 2^{4}\cdot3
 & 3 & 2_{25} & 6_{20} \\\hline

 6_{21} & p^{3}q^{2}pq^{2}pq^{2} & 2^{4}\cdot3
 & 3 & 2_{26} & 6_{21} \\\hline

 8_{1} & (pq)^{2} & 2^{7}
 & 4_{3} & 8_{1} & 8_{1} \\\hline

 8_{2} & p^{3}qp^{2}q & 2^{7}
 & 4_{5} & 8_{2} & 8_{2} \\\hline

 8_{3} & p^{2}q^{2}pqpq & 2^{7}
 & 4_{5} & 8_{3} & 8_{3} \\\hline

 8_{4} & pqpq^{2}pq^{2} & 2^{7}
 & 4_{4} & 8_{4} & 8_{4} \\\hline

 8_{5} & p^{2}qp^{2}qpq^{2} & 2^{7}
 & 4_{4} & 8_{5} & 8_{5} \\\hline

 8_{6} & p^{2}q^{2}pqpq^{4} & 2^{7}
 & 4_{5} & 8_{6} & 8_{6} \\\hline

 8_{7} & p^{3}qpqpq^{2}p^{2}q^{2} & 2^{7}
 & 4_{5} & 8_{7} & 8_{7} \\\hline

 8_{8} & p^{2}qpq^{2} & 2^{6}
 & 4_{17} & 8_{8} & 8_{8} \\\hline

 8_{9} & p^{3}qpq^{2} & 2^{6}
 & 4_{3} & 8_{9} & 8_{9} \\\hline

 8_{10} & p^{3}qpqpq^{2} & 2^{6}
 & 4_{18} & 8_{16} & 8_{10} \\\hline

 8_{11} & p^{2}qpqpqpq & 2^{6}
 & 4_{17} & 8_{11} & 8_{11} \\\hline

 8_{12} & p^{3}qp^{2}q^{2}pq & 2^{6}
 & 4_{19} & 8_{12} & 8_{12} \\\hline

 8_{13} & p^{3}qp^{2}qpqpq & 2^{6}
 & 4_{20} & 8_{13} & 8_{13} \\\hline

 8_{14} & p^{3}qp^{2}q^{2}pq^{2} & 2^{6}
 & 4_{19} & 8_{14} & 8_{14} \\\hline

 8_{15} & p^{2}q^{2}pqpqpq^{2} & 2^{6}
 & 4_{6} & 8_{15} & 8_{15} \\\hline

 8_{16} & p^{3}qpqpq^{5} & 2^{6}
 & 4_{18} & 8_{10} & 8_{16} \\\hline

 8_{17} & p^{3}qpq^{2}p^{2}qpq & 2^{6}
 & 4_{20} & 8_{17} & 8_{17} \\\hline

 8_{18} & p^{3}q^{2}pqpq & 2^{5}
 & 4_{29} & 8_{18} & 8_{18} \\\hline

 8_{19} & p^{2}qpq^{2}p^{2}q^{2} & 2^{5}
 & 4_{30} & 8_{19} & 8_{19} \\\hline

 10_{1} & p^{2}q & 2^{3}\cdot5
 & 5 & 10_{1} & 2_{2} \\\hline

 10_{2} & p^{2}q^{2} & 2^{3}\cdot5
 & 5 & 10_{3} & 2_{4} \\\hline

 10_{3} & pqpq^{3} & 2^{3}\cdot5
 & 5 & 10_{2} & 2_{4} \\\hline

 10_{4} & p^{2}qpq^{2}pq^{2} & 2^{3}\cdot5
 & 5 & 10_{4} & 2_{3} \\\hline

 10_{5} & p^{2}q^{2}pq^{2}pq^{2} & 2^{3}\cdot5
 & 5 & 10_{5} & 2_{1} \\\hline

 10_{6} & p^{3}qp^{3}qp^{2}q & 2^{3}\cdot5
 & 5 & 10_{7} & 2_{5} \\\hline

 10_{7} & p^{3}q^{2}pqpqpq & 2^{3}\cdot5
 & 5 & 10_{6} & 2_{5} \\\hline

\end{array}
$$

\newpage
 \normalsize
 Conjugacy classes of $D_2 = \langle p_2, q_2 \rangle$, with subscripts dropped
 (continued)
 \scriptsize
 $$
 \begin{array}{|c|r|r|c|c|c|
}\hline
 \mbox{Class} &
 \multicolumn{1}{c|}{\mbox{Representative}} &
 \multicolumn{1}{c|}{|\mbox{Centralizer}|}
 &
 2{\rm P}
 & 3{\rm P}
 & 5{\rm P}
 \\\hline

 12_{1} & pq^{2} & 2^{4}\cdot3
 & 6_{1} & 4_{13} & 12_{1} \\\hline

 12_{2} & p^{3}q & 2^{4}\cdot3
 & 6_{2} & 4_{14} & 12_{2} \\\hline

 12_{3} & p^{2}qpq^{4} & 2^{4}\cdot3
 & 6_{3} & 4_{1} & 12_{3} \\\hline

 12_{4} & p^{3}qpq^{2}pq & 2^{4}\cdot3
 & 6_{1} & 4_{15} & 12_{4} \\\hline

 12_{5} & p^{3}qp^{3}qpq & 2^{4}\cdot3
 & 6_{3} & 4_{2} & 12_{5} \\\hline

 12_{6} & p^{3}qp^{3}qpq^{2} & 2^{4}\cdot3
 & 6_{2} & 4_{16} & 12_{6} \\\hline

 16_{1} & pq & 2^{5}
 & 8_{1} & 16_{1} & 16_{2} \\\hline

 16_{2} & p^{3}q^{2} & 2^{5}
 & 8_{1} & 16_{2} & 16_{1} \\\hline

 \end{array}
 $$

}

\end{cclass}


\newpage
\begin{cclass}\label{Fi_23cc H} Conjugacy classes of $H(\Fi_{23}) = H = \langle x_0,y_0,h_0 \rangle$, with subscripts dropped
{ \setlength{\arraycolsep}{1mm}
\renewcommand{\baselinestretch}{0.5}
\scriptsize
$$
 \begin{array}{|c|r|r|c|c|c|c|c|c|
}\hline
 \mbox{Class} &
 \multicolumn{1}{c|}{\mbox{Representative}} &
 \multicolumn{1}{c|}{|\mbox{Centralizer}|}
 &
 2{\rm P}
 & 3{\rm P}
 & 5{\rm P}
 & 7{\rm P}
 & 11{\rm P}
 & 13{\rm P}
 \\\hline
 1 & 1 & 2^{18}\cdot3^{9}\cdot5^{2}\cdot7\cdot11\cdot13
 & 1 & 1 & 1 & 1 & 1 & 1 \\\hline

 2_{1} & (xy^{2})^{7} & 2^{18}\cdot3^{9}\cdot5^{2}\cdot7\cdot11\cdot13
 & 1 & 2_{1} & 2_{1} & 2_{1} & 2_{1} & 2_{1} \\\hline

 2_{2} & (y)^{7} & 2^{17}\cdot3^{6}\cdot5\cdot7\cdot11
 & 1 & 2_{2} & 2_{2} & 2_{2} & 2_{2} & 2_{2} \\\hline

 2_{3} & (xyhy)^{11} & 2^{17}\cdot3^{6}\cdot5\cdot7\cdot11
 & 1 & 2_{3} & 2_{3} & 2_{3} & 2_{3} & 2_{3} \\\hline

 2_{4} & x & 2^{18}\cdot3^{4}\cdot5
 & 1 & 2_{4} & 2_{4} & 2_{4} & 2_{4} & 2_{4} \\\hline

 2_{5} & (yh)^{10} & 2^{18}\cdot3^{4}\cdot5
 & 1 & 2_{5} & 2_{5} & 2_{5} & 2_{5} & 2_{5} \\\hline

 2_{6} & (xh)^{6} & 2^{16}\cdot3^{3}
 & 1 & 2_{6} & 2_{6} & 2_{6} & 2_{6} & 2_{6} \\\hline

 3_{1} & (xyxh)^{10} & 2^{9}\cdot3^{7}\cdot5\cdot7
 & 3_{1} & 1 & 3_{1} & 3_{1} & 3_{1} & 3_{1} \\\hline

 3_{2} & h & 2^{8}\cdot3^{9}
 & 3_{2} & 1 & 3_{2} & 3_{2} & 3_{2} & 3_{2} \\\hline

 3_{3} & (xyxy^{3})^{4} & 2^{7}\cdot3^{7}
 & 3_{3} & 1 & 3_{3} & 3_{3} & 3_{3} & 3_{3} \\\hline

 3_{4} & (xyhyh)^{6} & 2^{4}\cdot3^{7}
 & 3_{4} & 1 & 3_{4} & 3_{4} & 3_{4} & 3_{4} \\\hline

 4_{1} & (xy^{2}xh)^{3} & 2^{12}\cdot3^{3}
 & 2_{4} & 4_{1} & 4_{1} & 4_{1} & 4_{1} & 4_{1} \\\hline

 4_{2} & (yh)^{5} & 2^{11}\cdot3^{2}\cdot5
 & 2_{5} & 4_{2} & 4_{2} & 4_{2} & 4_{2} & 4_{2} \\\hline

 4_{3} & (xhy)^{5} & 2^{11}\cdot3^{2}\cdot5
 & 2_{5} & 4_{3} & 4_{3} & 4_{3} & 4_{3} & 4_{3} \\\hline

 4_{4} & (y^{2}h)^{4} & 2^{12}\cdot3
 & 2_{4} & 4_{4} & 4_{4} & 4_{4} & 4_{4} & 4_{4} \\\hline

 4_{5} & (xh)^{3} & 2^{10}\cdot3^{2}
 & 2_{6} & 4_{5} & 4_{5} & 4_{5} & 4_{5} & 4_{5} \\\hline

 4_{6} & (xyh)^{2} & 2^{10}\cdot3^{2}
 & 2_{6} & 4_{6} & 4_{6} & 4_{6} & 4_{6} & 4_{6} \\\hline

 4_{7} & (xh^{2}y^{3})^{3} & 2^{10}\cdot3
 & 2_{5} & 4_{7} & 4_{7} & 4_{7} & 4_{7} & 4_{7} \\\hline

 5 & (yh)^{4} & 2^{4}\cdot3\cdot5^{2}
 & 5 & 5 & 1 & 5 & 5 & 5 \\\hline

 6_{1} & (y^{3}h^{2})^{5} & 2^{9}\cdot3^{7}\cdot5\cdot7
 & 3_{1} & 2_{1} & 6_{1} & 6_{1} & 6_{1} & 6_{1} \\\hline

 6_{2} & (xy^{2}h^{2}yhy)^{3} & 2^{8}\cdot3^{9}
 & 3_{2} & 2_{1} & 6_{2} & 6_{2} & 6_{2} & 6_{2} \\\hline

 6_{3} & (xyxh)^{5} & 2^{8}\cdot3^{5}\cdot5
 & 3_{1} & 2_{2} & 6_{3} & 6_{3} & 6_{3} & 6_{3} \\\hline

 6_{4} & (xhy^{3})^{5} & 2^{8}\cdot3^{5}\cdot5
 & 3_{1} & 2_{3} & 6_{4} & 6_{4} & 6_{4} & 6_{4} \\\hline

 6_{5} & xyh^{2}yh^{2}xh^{2}xh & 2^{7}\cdot3^{7}
 & 3_{3} & 2_{1} & 6_{5} & 6_{5} & 6_{5} & 6_{5} \\\hline

 6_{6} & (xyxyh)^{3} & 2^{8}\cdot3^{6}
 & 3_{2} & 2_{3} & 6_{6} & 6_{6} & 6_{6} & 6_{6} \\\hline

 6_{7} & (xyxhy)^{3} & 2^{8}\cdot3^{6}
 & 3_{2} & 2_{2} & 6_{7} & 6_{7} & 6_{7} & 6_{7} \\\hline

 6_{8} & (xyhyh)^{3} & 2^{4}\cdot3^{7}
 & 3_{4} & 2_{1} & 6_{8} & 6_{8} & 6_{8} & 6_{8} \\\hline

 6_{9} & (xy^{2}xh)^{2} & 2^{8}\cdot3^{4}
 & 3_{2} & 2_{4} & 6_{9} & 6_{9} & 6_{9} & 6_{9} \\\hline

 6_{10} & (xy^{2}xyh)^{3} & 2^{8}\cdot3^{4}
 & 3_{2} & 2_{5} & 6_{10} & 6_{10} & 6_{10} & 6_{10} \\\hline

 6_{11} & (yhyh^{2})^{4} & 2^{9}\cdot3^{3}
 & 3_{1} & 2_{4} & 6_{11} & 6_{11} & 6_{11} & 6_{11} \\\hline

 6_{12} & (xh^{2}y^{3})^{2} & 2^{9}\cdot3^{3}
 & 3_{1} & 2_{5} & 6_{12} & 6_{12} & 6_{12} & 6_{12} \\\hline

 6_{13} & y^{7}h & 2^{5}\cdot3^{5}
 & 3_{3} & 2_{2} & 6_{13} & 6_{13} & 6_{13} & 6_{13} \\\hline

 6_{14} & xyxyxhyxhy^{2} & 2^{5}\cdot3^{5}
 & 3_{3} & 2_{3} & 6_{14} & 6_{14} & 6_{14} & 6_{14} \\\hline

 6_{15} & (xh)^{2} & 2^{8}\cdot3^{3}
 & 3_{2} & 2_{6} & 6_{15} & 6_{15} & 6_{15} & 6_{15} \\\hline

 6_{16} & (xhxhy^{2}h)^{2} & 2^{8}\cdot3^{3}
 & 3_{2} & 2_{6} & 6_{16} & 6_{16} & 6_{16} & 6_{16} \\\hline

 6_{17} & xyhxhxh & 2^{7}\cdot3^{3}
 & 3_{3} & 2_{6} & 6_{17} & 6_{17} & 6_{17} & 6_{17} \\\hline

 6_{18} & (xy^{2}xhyhy)^{2} & 2^{7}\cdot3^{3}
 & 3_{3} & 2_{6} & 6_{18} & 6_{18} & 6_{18} & 6_{18} \\\hline

 6_{19} & xyxy^{2}xyxyhy^{2} & 2^{7}\cdot3^{3}
 & 3_{1} & 2_{6} & 6_{19} & 6_{19} & 6_{19} & 6_{19} \\\hline

 6_{20} & (xyxy^{3})^{2} & 2^{6}\cdot3^{3}
 & 3_{3} & 2_{5} & 6_{20} & 6_{20} & 6_{20} & 6_{20} \\\hline

 6_{21} & xh^{2}yh^{2}y & 2^{6}\cdot3^{3}
 & 3_{3} & 2_{4} & 6_{21} & 6_{21} & 6_{21} & 6_{21} \\\hline

 6_{22} & xyxhy^{2}xyh & 2^{5}\cdot3^{3}
 & 3_{3} & 2_{6} & 6_{22} & 6_{22} & 6_{22} & 6_{22} \\\hline

 6_{23} & xyxh^{2}yhyh & 2^{5}\cdot3^{3}
 & 3_{3} & 2_{6} & 6_{23} & 6_{23} & 6_{23} & 6_{23} \\\hline

 6_{24} & (xyxyhyh)^{2} & 2^{4}\cdot3^{3}
 & 3_{4} & 2_{6} & 6_{24} & 6_{24} & 6_{24} & 6_{24} \\\hline

 6_{25} & xyh^{2}y^{2}xh & 2^{4}\cdot3^{3}
 & 3_{4} & 2_{6} & 6_{25} & 6_{25} & 6_{25} & 6_{25} \\\hline

 7 & (y)^{2} & 2^{2}\cdot3\cdot7
 & 7 & 7 & 7 & 1 & 7 & 7 \\\hline

 8_{1} & (yhyh^{2})^{3} & 2^{7}\cdot3
 & 4_{1} & 8_{1} & 8_{1} & 8_{1} & 8_{1} & 8_{1} \\\hline

 8_{2} & (xyhy^{2}h)^{3} & 2^{7}\cdot3
 & 4_{1} & 8_{2} & 8_{2} & 8_{2} & 8_{2} & 8_{2} \\\hline

 8_{3} & (y^{2}h)^{2} & 2^{7}
 & 4_{4} & 8_{3} & 8_{3} & 8_{3} & 8_{3} & 8_{3} \\\hline

 8_{4} & xyh & 2^{6}
 & 4_{6} & 8_{4} & 8_{4} & 8_{4} & 8_{4} & 8_{4} \\\hline

 8_{5} & xyh^{2}yh & 2^{6}
 & 4_{6} & 8_{5} & 8_{5} & 8_{5} & 8_{5} & 8_{5} \\\hline

 9_{1} & (xyxyh)^{2} & 2^{3}\cdot3^{4}
 & 9_{1} & 3_{2} & 9_{1} & 9_{1} & 9_{1} & 9_{1} \\\hline

 9_{2} & (xyhy^{3})^{2} & 2^{2}\cdot3^{4}
 & 9_{2} & 3_{2} & 9_{2} & 9_{2} & 9_{2} & 9_{2} \\\hline

 9_{3} & (xyhyh)^{2} & 2\cdot3^{3}
 & 9_{3} & 3_{4} & 9_{3} & 9_{3} & 9_{3} & 9_{3} \\\hline
 10_{1} & (y^{3}h^{2})^{3} & 2^{4}\cdot3\cdot5^{2}
 & 5 & 10_{1} & 2_{1} & 10_{1} & 10_{1} & 10_{1} \\\hline

 10_{2} & yh^{2} & 2^{3}\cdot3\cdot5
 & 5 & 10_{2} & 2_{2} & 10_{2} & 10_{2} & 10_{2} \\\hline

 10_{3} & (xhy^{3})^{3} & 2^{3}\cdot3\cdot5
 & 5 & 10_{3} & 2_{3} & 10_{3} & 10_{3} & 10_{3} \\\hline

 10_{4} & (yh)^{2} & 2^{4}\cdot5
 & 5 & 10_{4} & 2_{5} & 10_{4} & 10_{4} & 10_{4} \\\hline

 10_{5} & xh^{2}y & 2^{4}\cdot5
 & 5 & 10_{5} & 2_{4} & 10_{5} & 10_{5} & 10_{5} \\\hline

 11_{1} & (xy^{3})^{2} & 2^{2}\cdot11
 & 11_{2} & 11_{1} & 11_{1} & 11_{2} & 1 & 11_{2} \\\hline

 11_{2} & (xy^{3})^{4} & 2^{2}\cdot11
 & 11_{1} & 11_{2} & 11_{2} & 11_{1} & 1 & 11_{1} \\\hline

 12_{1} & (yhyh^{2})^{2} & 2^{7}\cdot3^{2}
 & 6_{11} & 4_{1} & 12_{1} & 12_{1} & 12_{1} & 12_{1} \\\hline

 12_{2} & (xyhy^{2}h)^{2} & 2^{7}\cdot3^{2}
 & 6_{11} & 4_{1} & 12_{2} & 12_{2} & 12_{2} & 12_{2} \\\hline

 12_{3} & xy^{2}xh & 2^{5}\cdot3^{3}
 & 6_{9} & 4_{1} & 12_{3} & 12_{3} & 12_{3} & 12_{3} \\\hline

 12_{4} & xy^{2}xy^{3}h & 2^{6}\cdot3^{2}
 & 6_{12} & 4_{3} & 12_{4} & 12_{4} & 12_{4} & 12_{4} \\\hline

 12_{5} & xyxy^{3}xh^{2} & 2^{6}\cdot3^{2}
 & 6_{12} & 4_{2} & 12_{5} & 12_{5} & 12_{5} & 12_{5} \\\hline

 12_{6} & xy^{2}xhxhyhy & 2^{6}\cdot3^{2}
 & 6_{11} & 4_{1} & 12_{6} & 12_{6} & 12_{6} & 12_{6} \\\hline

 12_{7} & xh & 2^{5}\cdot3^{2}
 & 6_{15} & 4_{5} & 12_{7} & 12_{7} & 12_{7} & 12_{7} \\\hline

 12_{8} & xyxh^{2}y & 2^{5}\cdot3^{2}
 & 6_{15} & 4_{6} & 12_{8} & 12_{8} & 12_{8} & 12_{8} \\\hline

 12_{9} & xhxhy^{2}h & 2^{5}\cdot3^{2}
 & 6_{16} & 4_{6} & 12_{9} & 12_{9} & 12_{9} & 12_{9} \\\hline

 12_{10} & xy^{2}xhyhy & 2^{5}\cdot3^{2}
 & 6_{18} & 4_{5} & 12_{10} & 12_{10} & 12_{10} & 12_{10} \\\hline

 12_{11} & xhxh^{2}y^{4} & 2^{5}\cdot3^{2}
 & 6_{18} & 4_{6} & 12_{11} & 12_{11} & 12_{11} & 12_{11} \\\hline

\end{array}
$$

\newpage
 \normalsize
 Conjugacy classes of $H = \langle x_0,y_0,h_0 \rangle$, with subscripts dropped
 (continued)
 \scriptsize
 $$
 \begin{array}{|c|r|r|c|c|c|c|c|c|
}\hline
 \mbox{Class} &
 \multicolumn{1}{c|}{\mbox{Representative}} &
 \multicolumn{1}{c|}{|\mbox{Centralizer}|}
 &
 2{\rm P}
 & 3{\rm P}
 & 5{\rm P}
 & 7{\rm P}
 & 11{\rm P}
 & 13{\rm P}
 \\\hline

 12_{12} & xhy^{4}hyh & 2^{5}\cdot3^{2}
 & 6_{16} & 4_{5} & 12_{12} & 12_{12} & 12_{12} & 12_{12} \\\hline

 12_{13} & xyxy^{3} & 2^{4}\cdot3^{2}
 & 6_{20} & 4_{2} & 12_{13} & 12_{13} & 12_{13} & 12_{13} \\\hline

 12_{14} & xyxhy^{3}h & 2^{4}\cdot3^{2}
 & 6_{20} & 4_{3} & 12_{14} & 12_{14} & 12_{14} & 12_{14} \\\hline

 12_{15} & xh^{2}y^{3} & 2^{5}\cdot3
 & 6_{12} & 4_{7} & 12_{15} & 12_{15} & 12_{15} & 12_{15} \\\hline

 12_{16} & xy^{3}hy^{2}h & 2^{5}\cdot3
 & 6_{9} & 4_{4} & 12_{16} & 12_{16} & 12_{16} & 12_{16} \\\hline

 12_{17} & xyxyhyh & 2^{3}\cdot3^{2}
 & 6_{24} & 4_{6} & 12_{17} & 12_{17} & 12_{17} & 12_{17} \\\hline

 12_{18} & xy^{3}h^{2}xh & 2^{3}\cdot3^{2}
 & 6_{24} & 4_{5} & 12_{18} & 12_{18} & 12_{18} & 12_{18} \\\hline

 13_{1} & (xy^{3}h)^{2} & 2\cdot13
 & 13_{2} & 13_{1} & 13_{2} & 13_{2} & 13_{2} & 1 \\\hline

 13_{2} & (xy^{3}h)^{4} & 2\cdot13
 & 13_{1} & 13_{2} & 13_{1} & 13_{1} & 13_{1} & 1 \\\hline

 14_{1} & xy^{2} & 2^{2}\cdot3\cdot7
 & 7 & 14_{1} & 14_{1} & 2_{1} & 14_{1} & 14_{1} \\\hline

 14_{2} & y & 2^{2}\cdot7
 & 7 & 14_{2} & 14_{2} & 2_{2} & 14_{2} & 14_{2} \\\hline

 14_{3} & xh^{2}y^{2} & 2^{2}\cdot7
 & 7 & 14_{3} & 14_{3} & 2_{3} & 14_{3} & 14_{3} \\\hline

 15 & (xyxh)^{2} & 2^{2}\cdot3\cdot5
 & 15 & 5 & 3_{1} & 15 & 15 & 15 \\\hline

 16_{1} & y^{2}h & 2^{5}
 & 8_{3} & 16_{1} & 16_{2} & 16_{2} & 16_{1} & 16_{2} \\\hline

 16_{2} & xyxyhy & 2^{5}
 & 8_{3} & 16_{2} & 16_{1} & 16_{1} & 16_{2} & 16_{1} \\\hline

 18_{1} & xhy^{2}hy^{2}h & 2^{3}\cdot3^{4}
 & 9_{1} & 6_{2} & 18_{1} & 18_{1} & 18_{1} & 18_{1} \\\hline

 18_{2} & xy^{2}h^{2}yhy & 2^{2}\cdot3^{4}
 & 9_{2} & 6_{2} & 18_{2} & 18_{2} & 18_{2} & 18_{2} \\\hline

 18_{3} & xyxyh & 2^{3}\cdot3^{3}
 & 9_{1} & 6_{6} & 18_{5} & 18_{3} & 18_{5} & 18_{3} \\\hline

 18_{4} & xyxhy & 2^{3}\cdot3^{3}
 & 9_{1} & 6_{7} & 18_{6} & 18_{4} & 18_{6} & 18_{4} \\\hline

 18_{5} & xy^{2}xhy & 2^{3}\cdot3^{3}
 & 9_{1} & 6_{6} & 18_{3} & 18_{5} & 18_{3} & 18_{5} \\\hline

 18_{6} & xhy^{3}h^{2} & 2^{3}\cdot3^{3}
 & 9_{1} & 6_{7} & 18_{4} & 18_{6} & 18_{4} & 18_{6} \\\hline

 18_{7} & xyhy^{3} & 2^{2}\cdot3^{3}
 & 9_{2} & 6_{6} & 18_{7} & 18_{7} & 18_{7} & 18_{7} \\\hline

 18_{8} & xyxh^{2}xh & 2^{2}\cdot3^{3}
 & 9_{2} & 6_{7} & 18_{8} & 18_{8} & 18_{8} & 18_{8} \\\hline

 18_{9} & xy^{2}xyh & 2^{3}\cdot3^{2}
 & 9_{1} & 6_{10} & 18_{9} & 18_{9} & 18_{9} & 18_{9} \\\hline

 18_{10} & xh^{2}yhy & 2^{3}\cdot3^{2}
 & 9_{1} & 6_{9} & 18_{10} & 18_{10} & 18_{10} & 18_{10} \\\hline

 18_{11} & xyhyh & 2\cdot3^{3}
 & 9_{3} & 6_{8} & 18_{11} & 18_{11} & 18_{11} & 18_{11} \\\hline

 20_{1} & yh & 2^{3}\cdot5
 & 10_{4} & 20_{1} & 4_{2} & 20_{1} & 20_{1} & 20_{1} \\\hline

 20_{2} & xhy & 2^{3}\cdot5
 & 10_{4} & 20_{2} & 4_{3} & 20_{2} & 20_{2} & 20_{2} \\\hline

 21 & xhyhy & 2\cdot3\cdot7
 & 21 & 7 & 21 & 3_{1} & 21 & 21 \\\hline

 22_{1} & xy^{3} & 2^{2}\cdot11
 & 11_{1} & 22_{1} & 22_{1} & 22_{3} & 2_{1} & 22_{3} \\\hline

 22_{2} & xyhy & 2^{2}\cdot11
 & 11_{1} & 22_{2} & 22_{2} & 22_{5} & 2_{3} & 22_{5} \\\hline

 22_{3} & xyxh^{2} & 2^{2}\cdot11
 & 11_{2} & 22_{3} & 22_{3} & 22_{1} & 2_{1} & 22_{1} \\\hline

 22_{4} & xyxhy^{2} & 2^{2}\cdot11
 & 11_{2} & 22_{4} & 22_{4} & 22_{6} & 2_{2} & 22_{6} \\\hline

 22_{5} & xyhxh^{2} & 2^{2}\cdot11
 & 11_{2} & 22_{5} & 22_{5} & 22_{2} & 2_{3} & 22_{2} \\\hline

 22_{6} & xy^{2}xy^{2}h & 2^{2}\cdot11
 & 11_{1} & 22_{6} & 22_{6} & 22_{4} & 2_{2} & 22_{4} \\\hline
 24_{1} & yhyh^{2} & 2^{4}\cdot3
 & 12_{1} & 8_{1} & 24_{1} & 24_{1} & 24_{1} & 24_{1} \\\hline

 24_{2} & xyhy^{2}h & 2^{4}\cdot3
 & 12_{2} & 8_{2} & 24_{2} & 24_{2} & 24_{2} & 24_{2} \\\hline

 26_{1} & xy^{3}h & 2\cdot13
 & 13_{1} & 26_{1} & 26_{2} & 26_{2} & 26_{2} & 2_{1} \\\hline

 26_{2} & xhy^{2}h & 2\cdot13
 & 13_{2} & 26_{2} & 26_{1} & 26_{1} & 26_{1} & 2_{1} \\\hline

 30_{1} & xyxh & 2^{2}\cdot3\cdot5
 & 15 & 10_{2} & 6_{3} & 30_{1} & 30_{1} & 30_{1} \\\hline

 30_{2} & xhy^{3} & 2^{2}\cdot3\cdot5
 & 15 & 10_{3} & 6_{4} & 30_{2} & 30_{2} & 30_{2} \\\hline

 30_{3} & y^{3}h^{2} & 2^{2}\cdot3\cdot5
 & 15 & 10_{1} & 6_{1} & 30_{3} & 30_{3} & 30_{3} \\\hline

 42 & xhy^{2}hy & 2\cdot3\cdot7
 & 21 & 14_{1} & 42 & 6_{1} & 42 & 42 \\\hline

\end{array}
$$

}

\end{cclass}


\newpage
\begin{cclass}\label{Fi_23cc} Conjugacy classes of $\mathfrak{G} = \langle \mathfrak{x},\mathfrak{y},\mathfrak{h},\mathfrak{e} \rangle$
{ \setlength{\arraycolsep}{1mm}
\renewcommand{\baselinestretch}{0.5}
\scriptsize
$$
 \begin{array}{|c|r|r|c|c|c|c|c|c|c|c|
}\hline
 \mbox{Class} &
 \multicolumn{1}{c|}{\mbox{Representative}} &
 \multicolumn{1}{c|}{|\mbox{Centralizer}|}
 &
 2{\rm P}
 & 3{\rm P}
 & 5{\rm P}
 & 7{\rm P}
 & 11{\rm P}
 & 13{\rm P}
 & 17{\rm P}
 & 23{\rm P}
 \\\hline
 1 & 1 & 2^{18}\cdot3^{13}\cdot5^{2}\cdot7\cdot11\cdot13\cdot17\cdot23
 & 1 & 1 & 1 & 1 & 1 & 1 & 1 & 1 \\\hline

 2_{1} & (y)^{7} & 2^{18}\cdot3^{9}\cdot5^{2}\cdot7\cdot11\cdot13
 & 1 & 2_{1} & 2_{1} & 2_{1} & 2_{1} & 2_{1} & 2_{1} & 2_{1} \\\hline

 2_{2} & e & 2^{18}\cdot3^{6}\cdot5\cdot7\cdot11
 & 1 & 2_{2} & 2_{2} & 2_{2} & 2_{2} & 2_{2} & 2_{2} & 2_{2} \\\hline

 2_{3} & x & 2^{18}\cdot3^{5}\cdot5
 & 1 & 2_{3} & 2_{3} & 2_{3} & 2_{3} & 2_{3} & 2_{3} & 2_{3} \\\hline

 3_{1} & (he)^{4} & 2^{9}\cdot3^{10}\cdot5\cdot7\cdot13
 & 3_{1} & 1 & 3_{1} & 3_{1} & 3_{1} & 3_{1} & 3_{1} & 3_{1} \\\hline

 3_{2} & h & 2^{10}\cdot3^{13}
 & 3_{2} & 1 & 3_{2} & 3_{2} & 3_{2} & 3_{2} & 3_{2} & 3_{2} \\\hline

 3_{3} & (xyeh)^{4} & 2^{7}\cdot3^{10}\cdot5
 & 3_{3} & 1 & 3_{3} & 3_{3} & 3_{3} & 3_{3} & 3_{3} & 3_{3} \\\hline

 3_{4} & (xyhyh)^{6} & 2^{4}\cdot3^{10}
 & 3_{4} & 1 & 3_{4} & 3_{4} & 3_{4} & 3_{4} & 3_{4} & 3_{4} \\\hline

 4_{1} & (xhy)^{5} & 2^{11}\cdot3^{4}\cdot5\cdot7
 & 2_{2} & 4_{1} & 4_{1} & 4_{1} & 4_{1} & 4_{1} & 4_{1} & 4_{1} \\\hline

 4_{2} & (xyh)^{2} & 2^{12}\cdot3^{4}
 & 2_{3} & 4_{2} & 4_{2} & 4_{2} & 4_{2} & 4_{2} & 4_{2} & 4_{2} \\\hline

 4_{3} & xe & 2^{11}\cdot3^{2}\cdot5
 & 2_{2} & 4_{3} & 4_{3} & 4_{3} & 4_{3} & 4_{3} & 4_{3} & 4_{3} \\\hline

 4_{4} & (xh)^{3} & 2^{12}\cdot3^{2}
 & 2_{3} & 4_{4} & 4_{4} & 4_{4} & 4_{4} & 4_{4} & 4_{4} & 4_{4} \\\hline

 5 & (yh)^{4} & 2^{4}\cdot3^{2}\cdot5^{2}\cdot7
 & 5 & 5 & 1 & 5 & 5 & 5 & 5 & 5 \\\hline

 6_{1} & (xyxh)^{5} & 2^{9}\cdot3^{7}\cdot5\cdot7
 & 3_{1} & 2_{1} & 6_{1} & 6_{1} & 6_{1} & 6_{1} & 6_{1} & 6_{1} \\\hline

 6_{2} & (xyxhy)^{3} & 2^{8}\cdot3^{9}
 & 3_{2} & 2_{1} & 6_{2} & 6_{2} & 6_{2} & 6_{2} & 6_{2} & 6_{2} \\\hline

 6_{3} & (he)^{2} & 2^{9}\cdot3^{5}\cdot5
 & 3_{1} & 2_{2} & 6_{3} & 6_{3} & 6_{3} & 6_{3} & 6_{3} & 6_{3} \\\hline

 6_{4} & (yehe)^{3} & 2^{9}\cdot3^{6}
 & 3_{2} & 2_{2} & 6_{4} & 6_{4} & 6_{4} & 6_{4} & 6_{4} & 6_{4} \\\hline

 6_{5} & xheh^{2}e & 2^{7}\cdot3^{7}
 & 3_{3} & 2_{1} & 6_{5} & 6_{5} & 6_{5} & 6_{5} & 6_{5} & 6_{5} \\\hline

 6_{6} & (xh)^{2} & 2^{10}\cdot3^{5}
 & 3_{2} & 2_{3} & 6_{6} & 6_{6} & 6_{6} & 6_{6} & 6_{6} & 6_{6} \\\hline

 6_{7} & (xehy)^{5} & 2^{7}\cdot3^{5}\cdot5
 & 3_{3} & 2_{3} & 6_{7} & 6_{7} & 6_{7} & 6_{7} & 6_{7} & 6_{7} \\\hline

 6_{8} & (xhye)^{3} & 2^{8}\cdot3^{5}
 & 3_{2} & 2_{3} & 6_{8} & 6_{8} & 6_{8} & 6_{8} & 6_{8} & 6_{8} \\\hline

 6_{9} & (xhe)^{2} & 2^{9}\cdot3^{4}
 & 3_{1} & 2_{3} & 6_{9} & 6_{9} & 6_{9} & 6_{9} & 6_{9} & 6_{9} \\\hline

 6_{10} & (xyhyh)^{3} & 2^{4}\cdot3^{7}
 & 3_{4} & 2_{1} & 6_{10} & 6_{10} & 6_{10} & 6_{10} & 6_{10} & 6_{10} \\\hline

 6_{11} & (xyeh)^{2} & 2^{6}\cdot3^{5}
 & 3_{3} & 2_{2} & 6_{11} & 6_{11} & 6_{11} & 6_{11} & 6_{11} & 6_{11} \\\hline

 6_{12} & (xyexe)^{2} & 2^{7}\cdot3^{4}
 & 3_{3} & 2_{3} & 6_{12} & 6_{12} & 6_{12} & 6_{12} & 6_{12} & 6_{12} \\\hline

 6_{13} & xh^{2}ehe & 2^{6}\cdot3^{4}
 & 3_{3} & 2_{3} & 6_{13} & 6_{13} & 6_{13} & 6_{13} & 6_{13} & 6_{13} \\\hline

 6_{14} & xyhyehe & 2^{4}\cdot3^{5}
 & 3_{4} & 2_{3} & 6_{14} & 6_{14} & 6_{14} & 6_{14} & 6_{14} & 6_{14} \\\hline

 6_{15} & (xh^{2}ey)^{2} & 2^{4}\cdot3^{4}
 & 3_{4} & 2_{3} & 6_{15} & 6_{15} & 6_{15} & 6_{15} & 6_{15} & 6_{15} \\\hline

 7 & (y)^{2} & 2^{3}\cdot3\cdot5\cdot7
 & 7 & 7 & 7 & 1 & 7 & 7 & 7 & 7 \\\hline

 8_{1} & xyh & 2^{7}\cdot3
 & 4_{2} & 8_{1} & 8_{1} & 8_{1} & 8_{1} & 8_{1} & 8_{1} & 8_{1} \\\hline

 8_{2} & (xye)^{2} & 2^{7}\cdot3
 & 4_{4} & 8_{2} & 8_{2} & 8_{2} & 8_{2} & 8_{2} & 8_{2} & 8_{2} \\\hline

 8_{3} & xey & 2^{7}\cdot3
 & 4_{2} & 8_{3} & 8_{3} & 8_{3} & 8_{3} & 8_{3} & 8_{3} & 8_{3} \\\hline

 9_{1} & (xyh^{2}e)^{3} & 2^{3}\cdot3^{7}
 & 9_{1} & 3_{2} & 9_{1} & 9_{1} & 9_{1} & 9_{1} & 9_{1} & 9_{1} \\\hline

 9_{2} & (xhye)^{2} & 2^{3}\cdot3^{6}
 & 9_{2} & 3_{2} & 9_{2} & 9_{2} & 9_{2} & 9_{2} & 9_{2} & 9_{2} \\\hline

 9_{3} & (xyhey)^{4} & 2^{3}\cdot3^{6}
 & 9_{3} & 3_{2} & 9_{3} & 9_{3} & 9_{3} & 9_{3} & 9_{3} & 9_{3} \\\hline

 9_{4} & (yh^{2}eh^{2})^{2} & 2\cdot3^{6}
 & 9_{4} & 3_{2} & 9_{4} & 9_{4} & 9_{4} & 9_{4} & 9_{4} & 9_{4} \\\hline

 9_{5} & (xyhyh)^{2} & 2\cdot3^{4}
 & 9_{5} & 3_{4} & 9_{5} & 9_{5} & 9_{5} & 9_{5} & 9_{5} & 9_{5} \\\hline

 10_{1} & yh^{2} & 2^{4}\cdot3\cdot5^{2}
 & 5 & 10_{1} & 2_{1} & 10_{1} & 10_{1} & 10_{1} & 10_{1} & 10_{1} \\\hline

 10_{2} & (yh)^{2} & 2^{4}\cdot3\cdot5
 & 5 & 10_{2} & 2_{2} & 10_{2} & 10_{2} & 10_{2} & 10_{2} & 10_{2} \\\hline

 10_{3} & xh^{2}y & 2^{4}\cdot3\cdot5
 & 5 & 10_{3} & 2_{3} & 10_{3} & 10_{3} & 10_{3} & 10_{3} & 10_{3} \\\hline

 11 & (xy^{3})^{2} & 2^{2}\cdot11
 & 11 & 11 & 11 & 11 & 1 & 11 & 11 & 11 \\\hline

 12_{1} & (xeyh^{2})^{5} & 2^{6}\cdot3^{3}\cdot5
 & 6_{3} & 4_{1} & 12_{1} & 12_{1} & 12_{1} & 12_{1} & 12_{1} & 12_{1} \\\hline

 12_{2} & (xheh)^{2} & 2^{7}\cdot3^{3}
 & 6_{9} & 4_{2} & 12_{2} & 12_{2} & 12_{2} & 12_{2} & 12_{2} & 12_{2} \\\hline

 12_{3} & xyxh^{2}y & 2^{7}\cdot3^{3}
 & 6_{6} & 4_{2} & 12_{3} & 12_{3} & 12_{3} & 12_{3} & 12_{3} & 12_{3} \\\hline

 12_{4} & xehe & 2^{5}\cdot3^{4}
 & 6_{8} & 4_{2} & 12_{4} & 12_{4} & 12_{4} & 12_{4} & 12_{4} & 12_{4} \\\hline

 12_{5} & (xyhey)^{3} & 2^{5}\cdot3^{4}
 & 6_{4} & 4_{1} & 12_{5} & 12_{5} & 12_{5} & 12_{5} & 12_{5} & 12_{5} \\\hline

 12_{6} & xh & 2^{7}\cdot3^{2}
 & 6_{6} & 4_{4} & 12_{6} & 12_{6} & 12_{6} & 12_{6} & 12_{6} & 12_{6} \\\hline

 12_{7} & xhe & 2^{7}\cdot3^{2}
 & 6_{9} & 4_{2} & 12_{7} & 12_{7} & 12_{7} & 12_{7} & 12_{7} & 12_{7} \\\hline

 12_{8} & xyxhxhe & 2^{5}\cdot3^{3}
 & 6_{12} & 4_{2} & 12_{8} & 12_{8} & 12_{8} & 12_{8} & 12_{8} & 12_{8} \\\hline

 12_{9} & he & 2^{6}\cdot3^{2}
 & 6_{3} & 4_{3} & 12_{9} & 12_{9} & 12_{9} & 12_{9} & 12_{9} & 12_{9} \\\hline

 12_{10} & xyeh & 2^{4}\cdot3^{3}
 & 6_{11} & 4_{1} & 12_{10} & 12_{10} & 12_{10} & 12_{10} & 12_{10} & 12_{10}
\\\hline

 12_{11} & xyexe & 2^{5}\cdot3^{2}
 & 6_{12} & 4_{4} & 12_{11} & 12_{11} & 12_{11} & 12_{11} & 12_{11} & 12_{11}
\\\hline

 12_{12} & xhexeh & 2^{5}\cdot3^{2}
 & 6_{8} & 4_{4} & 12_{12} & 12_{12} & 12_{12} & 12_{12} & 12_{12} & 12_{12}
\\\hline

 12_{13} & xh^{2}ey & 2^{3}\cdot3^{3}
 & 6_{15} & 4_{2} & 12_{13} & 12_{13} & 12_{13} & 12_{13} & 12_{13} & 12_{13}
\\\hline

 12_{14} & y^{3}e & 2^{4}\cdot3^{2}
 & 6_{11} & 4_{3} & 12_{14} & 12_{14} & 12_{14} & 12_{14} & 12_{14} & 12_{14}
\\\hline

 12_{15} & xeyhey & 2^{3}\cdot3^{2}
 & 6_{15} & 4_{4} & 12_{15} & 12_{15} & 12_{15} & 12_{15} & 12_{15} & 12_{15}
\\\hline

 13_{1} & yeh & 2\cdot3\cdot13
 & 13_{2} & 13_{1} & 13_{2} & 13_{2} & 13_{2} & 1 & 13_{1} & 13_{1} \\\hline

 13_{2} & (xy^{2}eh)^{2} & 2\cdot3\cdot13
 & 13_{1} & 13_{2} & 13_{1} & 13_{1} & 13_{1} & 1 & 13_{2} & 13_{2} \\\hline

 14_{1} & y & 2^{2}\cdot3\cdot7
 & 7 & 14_{1} & 14_{1} & 2_{1} & 14_{1} & 14_{1} & 14_{1} & 14_{1} \\\hline

 14_{2} & xyxe & 2^{3}\cdot7
 & 7 & 14_{2} & 14_{2} & 2_{2} & 14_{2} & 14_{2} & 14_{2} & 14_{2} \\\hline

 15_{1} & (xyxh)^{2} & 2^{3}\cdot3^{2}\cdot5
 & 15_{1} & 5 & 3_{1} & 15_{1} & 15_{1} & 15_{1} & 15_{1} & 15_{1} \\\hline

 15_{2} & (xehy)^{2} & 2\cdot3^{2}\cdot5
 & 15_{2} & 5 & 3_{3} & 15_{2} & 15_{2} & 15_{2} & 15_{2} & 15_{2} \\\hline

 16_{1} & xye & 2^{5}
 & 8_{2} & 16_{1} & 16_{2} & 16_{2} & 16_{1} & 16_{2} & 16_{1} & 16_{2} \\\hline

 16_{2} & yheh & 2^{5}
 & 8_{2} & 16_{2} & 16_{1} & 16_{1} & 16_{2} & 16_{1} & 16_{2} & 16_{1} \\\hline

 17 & xhey & 17
 & 17 & 17 & 17 & 17 & 17 & 17 & 1 & 17 \\\hline

 18_{1} & xyxhy & 2^{3}\cdot3^{4}
 & 9_{2} & 6_{2} & 18_{1} & 18_{1} & 18_{1} & 18_{1} & 18_{1} & 18_{1} \\\hline

 18_{2} & xeyxeh & 2^{2}\cdot3^{4}
 & 9_{3} & 6_{2} & 18_{2} & 18_{2} & 18_{2} & 18_{2} & 18_{2} & 18_{2} \\\hline

 18_{3} & xhye & 2^{3}\cdot3^{3}
 & 9_{2} & 6_{8} & 18_{3} & 18_{3} & 18_{3} & 18_{3} & 18_{3} & 18_{3} \\\hline

 18_{4} & yehe & 2^{3}\cdot3^{3}
 & 9_{2} & 6_{4} & 18_{4} & 18_{4} & 18_{4} & 18_{4} & 18_{4} & 18_{4} \\\hline

 18_{5} & (xyhey)^{2} & 2^{3}\cdot3^{3}
 & 9_{3} & 6_{4} & 18_{5} & 18_{5} & 18_{5} & 18_{5} & 18_{5} & 18_{5} \\\hline

  18_{6} & xhehy & 2^{3}\cdot3^{3}
 & 9_{1} & 6_{8} & 18_{6} & 18_{6} & 18_{6} & 18_{6} & 18_{6} & 18_{6} \\\hline

\end{array}
$$

\newpage
 \normalsize
 Conjugacy classes of $\mathfrak{G} = \langle \mathfrak{x},\mathfrak{y},\mathfrak{h},\mathfrak{e} \rangle$
 (continued)
 \scriptsize
 $$
 \begin{array}{|c|r|r|c|c|c|c|c|c|c|c|
}\hline
 \mbox{Class} &
 \multicolumn{1}{c|}{\mbox{Representative}} &
 \multicolumn{1}{c|}{|\mbox{Centralizer}|}
 &
 2{\rm P}
 & 3{\rm P}
 & 5{\rm P}
 & 7{\rm P}
 & 11{\rm P}
 & 13{\rm P}
 & 17{\rm P}
 & 23{\rm P}
 \\\hline
 18_{7} & xyhyh & 2\cdot3^{3}
 & 9_{5} & 6_{10} & 18_{7} & 18_{7} & 18_{7} & 18_{7} & 18_{7} & 18_{7} \\\hline

 18_{8} & yh^{2}eh^{2} & 2\cdot3^{3}
 & 9_{4} & 6_{8} & 18_{8} & 18_{8} & 18_{8} & 18_{8} & 18_{8} & 18_{8} \\\hline

 20_{1} & xhy & 2^{3}\cdot3\cdot5
 & 10_{2} & 20_{1} & 4_{1} & 20_{1} & 20_{1} & 20_{1} & 20_{1} & 20_{1} \\\hline

 20_{2} & yh & 2^{3}\cdot5
 & 10_{2} & 20_{2} & 4_{3} & 20_{2} & 20_{2} & 20_{2} & 20_{2} & 20_{2} \\\hline

 21 & yhe & 2\cdot3\cdot7
 & 21 & 7 & 21 & 3_{1} & 21 & 21 & 21 & 21 \\\hline

 22_{1} & xy^{3} & 2^{2}\cdot11
 & 11 & 22_{1} & 22_{1} & 22_{3} & 2_{1} & 22_{3} & 22_{3} & 22_{1} \\\hline

 22_{2} & xyhy & 2^{2}\cdot11
 & 11 & 22_{2} & 22_{2} & 22_{2} & 2_{2} & 22_{2} & 22_{2} & 22_{2} \\\hline

 22_{3} & xey^{2} & 2^{2}\cdot11
 & 11 & 22_{3} & 22_{3} & 22_{1} & 2_{1} & 22_{1} & 22_{1} & 22_{3} \\\hline

 23_{1} & xyeh^{2} & 23
 & 23_{1} & 23_{1} & 23_{2} & 23_{2} & 23_{2} & 23_{1} & 23_{2} & 1 \\\hline

 23_{2} & yh^{2}ye & 23
 & 23_{2} & 23_{2} & 23_{1} & 23_{1} & 23_{1} & 23_{2} & 23_{1} & 1 \\\hline

 24_{1} & xheh & 2^{4}\cdot3
 & 12_{2} & 8_{3} & 24_{1} & 24_{1} & 24_{1} & 24_{1} & 24_{1} & 24_{1} \\\hline

 24_{2} & xyhyxe & 2^{4}\cdot3
 & 12_{6} & 8_{2} & 24_{2} & 24_{2} & 24_{2} & 24_{2} & 24_{2} & 24_{2} \\\hline

 24_{3} & xyhy^{2}h & 2^{4}\cdot3
 & 12_{7} & 8_{1} & 24_{3} & 24_{3} & 24_{3} & 24_{3} & 24_{3} & 24_{3} \\\hline

 26_{1} & xy^{3}h & 2\cdot13
 & 13_{1} & 26_{1} & 26_{2} & 26_{2} & 26_{2} & 2_{1} & 26_{1} & 26_{1} \\\hline

 26_{2} & xy^{2}eh & 2\cdot13
 & 13_{2} & 26_{2} & 26_{1} & 26_{1} & 26_{1} & 2_{1} & 26_{2} & 26_{2} \\\hline

 27 & xyh^{2}e & 3^{3}
 & 27 & 9_{1} & 27 & 27 & 27 & 27 & 27 & 27 \\\hline

 28 & xhyhxe & 2^{2}\cdot7
 & 14_{2} & 28 & 28 & 4_{1} & 28 & 28 & 28 & 28 \\\hline

 30_{1} & xhy^{3} & 2^{3}\cdot3\cdot5
 & 15_{1} & 10_{2} & 6_{3} & 30_{1} & 30_{1} & 30_{1} & 30_{1} & 30_{1} \\\hline

 30_{2} & xyxh & 2^{2}\cdot3\cdot5
 & 15_{1} & 10_{1} & 6_{1} & 30_{2} & 30_{2} & 30_{2} & 30_{2} & 30_{2} \\\hline

 30_{3} & xehy & 2\cdot3\cdot5
 & 15_{2} & 10_{3} & 6_{7} & 30_{3} & 30_{3} & 30_{3} & 30_{3} & 30_{3} \\\hline

 35 & y^{2}he & 5\cdot7
 & 35 & 35 & 7 & 5 & 35 & 35 & 35 & 35 \\\hline

 36_{1} & xyhey & 2^{2}\cdot3^{2}
 & 18_{5} & 12_{5} & 36_{1} & 36_{1} & 36_{1} & 36_{1} & 36_{1} & 36_{1}
\\\hline

 36_{2} & xy^{2}eyh & 2^{2}\cdot3^{2}
 & 18_{6} & 12_{4} & 36_{2} & 36_{2} & 36_{2} & 36_{2} & 36_{2} & 36_{2}
\\\hline

 39_{1} & xhyhe & 3\cdot13
 & 39_{2} & 13_{2} & 39_{2} & 39_{2} & 39_{2} & 3_{1} & 39_{1} & 39_{1} \\\hline

 39_{2} & xh^{2}xey & 3\cdot13
 & 39_{1} & 13_{1} & 39_{1} & 39_{1} & 39_{1} & 3_{1} & 39_{2} & 39_{2} \\\hline

 42 & yeyeh & 2\cdot3\cdot7
 & 21 & 14_{1} & 42 & 6_{1} & 42 & 42 & 42 & 42 \\\hline

 60 & xeyh^{2} & 2^{2}\cdot3\cdot5
 & 30_{1} & 20_{1} & 12_{1} & 60 & 60 & 60 & 60 & 60 \\\hline

 \end{array}
$$

}

\end{cclass}

\newpage
\section{Character Tables of Local Subgroups of $\Fi_{23}$ and $2\Fi_{22}$}




\begin{ctab}\label{Fi_23ct_E} Character table of $E(\Fi_{23}) = E = \langle x_1,y_1,e_1 \rangle$
{
\renewcommand{\baselinestretch}{0.5}

{\tiny
\setlength{\arraycolsep}{0,3mm}
$$
\begin{array}{r|rrrrrrrrrrrrrrrrrrrrrrrrr}
2&18&18&18&18&14&14&7&11&12&12&11&9&9&8&3&7&7&7&6&6&5&5&3&3\\
3&2&2&2&2&1&1&2&1&1&1&1&.&.&.&1&2&2&2&2&2&1&1&.&.\\
5&1&1&1&1&.&.&1&.&.&.&.&.&.&.&1&1&.&.&.&.&.&.&.&.\\
7&1&1&1&.&1&.&.&1&.&.&.&.&.&.&.&.&.&.&.&.&.&.&1&1\\
11&1&1&.&.&.&.&.&.&.&.&.&.&.&.&.&.&.&.&.&.&.&.&.&.\\
23&1&.&.&.&.&.&.&.&.&.&.&.&.&.&.&.&.&.&.&.&.&.&.&.\\
\hline
&1a&2a&2b&2c&2d&2e&3a&4a&4b&4c&4d&4e&4f&4g&5a&6a&6b&6c&6d&6e&6f
&6g&7a&7b\\
\hline
2P&1a&1a&1a&1a&1a&1a&3a&2b&2c&2c&2b&2e&2e&2d&5a&3a&3a&3a&3a&3a&3a
&3a&7a&7b\\
3P&1a&2a&2b&2c&2d&2e&1a&4a&4b&4c&4d&4e&4f&4g&5a&2c&2c&2a&2b&2c&2d
&2e&7b&7a\\
5P&1a&2a&2b&2c&2d&2e&3a&4a&4b&4c&4d&4e&4f&4g&1a&6a&6b&6c&6d&6e&6f
&6g&7b&7a\\
7P&1a&2a&2b&2c&2d&2e&3a&4a&4b&4c&4d&4e&4f&4g&5a&6a&6b&6c&6d&6e&6f
&6g&1a&1a\\
11P&1a&2a&2b&2c&2d&2e&3a&4a&4b&4c&4d&4e&4f&4g&5a&6a&6b&6c&6d&6e&6f
&6g&7a&7b\\
23P&1a&2a&2b&2c&2d&2e&3a&4a&4b&4c&4d&4e&4f&4g&5a&6a&6b&6c&6d&6e&6f
&6g&7a&7b\\
\hline
X.1&1&1&1&1&1&1&1&1&1&1&1&1&1&1&1&1&1&1&1&1&1&1&1&1\\
X.2&22&22&22&22&6&6&4&6&6&6&6&2&2&2&2&4&4&4&4&4&.&.&1&1\\
X.3&45&45&45&45&-3&-3&.&-3&-3&-3&-3&1&1&1&.&.&.&.&.&.&.&.&A&\bar{A}\\
X.4&45&45&45&45&-3&-3&.&-3&-3&-3&-3&1&1&1&.&.&.&.&.&.&.&.&\bar{A}&A\\
X.5&230&230&230&230&22&22&5&22&22&22&22&2&2&2&.&5&5&5&5&5&1&1&-1&-1\\
X.6&231&231&231&231&7&7&6&7&7&7&7&-1&-1&-1&1&6&6&6&6&6&-2&-2&.&.\\
X.7&231&231&231&231&7&7&-3&7&7&7&7&-1&-1&-1&1&-3&-3&-3&-3&-3&1&1&.
&.\\
X.8&231&231&231&231&7&7&-3&7&7&7&7&-1&-1&-1&1&-3&-3&-3&-3&-3&1&1&.
&.\\
X.9&253&253&253&253&13&13&1&13&13&13&13&1&1&1&-2&1&1&1&1&1&1&1&1&1\\
X.10&253&-99&29&-3&29&-3&10&-15&-3&5&1&5&-3&1&3&.&-6&.&2&.&2&.&1&1\\
X.11&506&154&26&-6&42&10&11&14&-6&2&-2&6&-2&2&1&9&3&1&-1&-3&3&1&2&2\\
X.12&770&770&770&770&-14&-14&5&-14&-14&-14&-14&-2&-2&-2&.&5&5&5&5&5
&1&1&.&.\\
X.13&770&770&770&770&-14&-14&5&-14&-14&-14&-14&-2&-2&-2&.&5&5&5&5&5
&1&1&.&.\\
X.14&896&896&896&896&.&.&-4&.&.&.&.&.&.&.&1&-4&-4&-4&-4&-4&.&.&.&.\\
X.15&896&896&896&896&.&.&-4&.&.&.&.&.&.&.&1&-4&-4&-4&-4&-4&.&.&.&.\\
X.16&990&990&990&990&-18&-18&.&-18&-18&-18&-18&2&2&2&.&.&.&.&.&.&.&.
&\bar{A}&A\\
X.17&990&990&990&990&-18&-18&.&-18&-18&-18&-18&2&2&2&.&.&.&.&.&.&.&.
&A&\bar{A}\\
X.18&1035&1035&1035&1035&27&27&.&27&27&27&27&-1&-1&-1&.&.&.&.&.&.&.
&.&-1&-1\\
X.19&1288&-56&-56&8&56&-8&10&.&8&-8&.&4&4&-4&3&-10&2&-2&-2&2&2&-2&.
&.\\
X.20&1518&-594&174&-18&62&-2&15&-34&-2&14&-2&2&2&2&3&-15&-9&9&3&-3
&-1&1&-1&-1\\
X.21&2024&2024&2024&2024&8&8&-1&8&8&8&8&.&.&.&-1&-1&-1&-1&-1&-1&-1
&-1&1&1\\
X.22&2530&-990&290&-30&-46&18&10&18&18&2&-14&-2&-2&-2&.&.&-6&.&2&.&2
&.&A&\bar{A}\\
X.23&2530&-990&290&-30&-46&18&10&18&18&2&-14&-2&-2&-2&.&.&-6&.&2&.&2
&.&\bar{A}&A\\
X.24&3542&-1386&406&-42&70&6&5&-42&6&22&-10&2&2&2&-3&15&-3&-9&1&3&1
&3&.&.\\
X.25&3542&1078&182&-42&70&-26&14&42&-10&14&-6&2&-6&-2&2&6&6&-2&2&-6
&-2&-2&.&.\\
X.26&3542&-1386&406&-42&70&6&5&-42&6&22&-10&2&2&2&-3&-15&-3&9&1&-3&1
&-3&.&.\\
X.27&3795&-1485&435&-45&83&-13&15&-41&-13&11&7&3&-5&-1&.&15&-9&-9&3
&3&-1&-1&1&1\\
X.28&3795&-1485&435&-45&-13&19&15&-1&19&11&-17&-5&3&-1&.&-15&-9&9&3
&-3&-1&1&1&1\\
X.29&5313&-2079&609&-63&49&17&-15&-35&17&25&-19&-3&5&1&3&15&9&-9&-3
&3&1&-1&.&.\\
X.30&7084&2156&364&-84&140&76&19&28&-20&-4&-4&4&4&4&-1&21&3&5&-5&-3
&-1&1&.&.\\
X.31&8855&-3465&1015&-105&7&39&-10&-21&39&31&-37&3&-5&-1&.&.&6&.&-2
&.&-2&.&.&.\\
X.32&10120&3080&520&-120&168&40&-5&56&-24&8&-8&.&.&.&.&-15&3&-7&7&-3
&3&1&-2&-2\\
X.33&10626&3234&546&-126&-14&-46&6&14&2&10&-2&-2&6&2&1&-6&6&-6&6&-6
&-2&2&.&.\\
X.34&10626&3234&546&-126&-14&-46&-3&14&2&10&-2&-2&6&2&1&3&-3&3&-3&3
&1&-1&.&.\\
X.35&10626&3234&546&-126&-14&-46&-3&14&2&10&-2&-2&6&2&1&3&-3&3&-3&3
&1&-1&.&.\\
X.36&11385&-4455&1305&-135&-87&9&.&45&9&-15&-3&5&-3&1&.&.&.&.&.&.&.
&.&A&\bar{A}\\
X.37&11385&-4455&1305&-135&-87&9&.&45&9&-15&-3&5&-3&1&.&.&.&.&.&.&.
&.&\bar{A}&A\\
X.38&12880&-560&-560&80&112&-16&10&.&16&-16&.&8&8&-8&.&-10&2&-2&-2&2
&-2&2&.&.\\
X.39&12880&-560&-560&80&-112&16&10&.&-16&16&.&.&.&.&.&-10&2&-2&-2&2
&2&-2&.&.\\
X.40&12880&-560&-560&80&-112&16&10&.&-16&16&.&.&.&.&.&-10&2&-2&-2&2
&2&-2&.&.\\
X.41&14168&-616&-616&88&168&-24&20&.&24&-24&.&-4&-4&4&3&-20&4&-4&-4
&4&.&.&.&.\\
X.42&14168&4312&728&-168&56&-72&11&56&-8&24&-8&.&.&.&-2&9&3&1&-1&-3
&-1&-3&.&.\\
X.43&17710&5390&910&-210&-98&62&25&-70&14&-26&10&-6&2&-2&.&15&9&-1&1
&-9&1&-1&.&.\\
X.44&20608&-896&-896&128&.&.&-20&.&.&.&.&.&.&.&3&20&-4&4&4&-4&.&.&.
&.\\
X.45&20608&-896&-896&128&.&.&-20&.&.&.&.&.&.&.&3&20&-4&4&4&-4&.&.&.
&.\\
X.46&22770&-8910&2610&-270&162&-30&.&-78&-30&18&18&-2&-2&-2&.&.&.&.
&.&.&.&.&-1&-1\\
X.47&22770&6930&1170&-270&-126&-30&.&-42&18&-6&6&6&-2&2&.&.&.&.&.&.
&.&.&F&\bar{F}\\
X.48&22770&6930&1170&-270&-126&-30&.&-42&18&-6&6&6&-2&2&.&.&.&.&.&.
&.&.&\bar{F}&F\\
X.49&26565&-10395&3045&-315&-91&5&15&49&5&-19&1&-3&5&1&.&15&-9&-9&3
&3&-1&-1&.&.\\
X.50&28336&8624&1456&-336&112&112&-14&.&-16&-16&.&.&.&.&1&-6&-6&2&-2
&6&-2&-2&.&.\\
X.51&30360&-11880&3480&-360&-8&-8&-15&8&-8&-8&8&.&.&.&.&-15&9&9&-3
&-3&1&1&1&1\\
X.52&32384&9856&1664&-384&.&.&-16&.&.&.&.&.&.&.&-1&-24&.&-8&8&.&.&.
&2&2\\
X.53&35420&10780&1820&-420&28&-36&5&28&-4&12&-4&-4&-4&-4&.&15&-3&7
&-7&3&1&3&.&.\\
X.54&56672&-2464&-2464&352&224&-32&-10&.&32&-32&.&.&.&.&-3&10&-2&2&2
&-2&2&-2&.&.\\
X.55&57960&-2520&-2520&360&-168&24&.&.&-24&24&.&4&4&-4&.&.&.&.&.&.&.
&.&.&.\\
X.56&70840&-3080&-3080&440&-56&8&10&.&-8&8&.&-4&-4&4&.&-10&2&-2&-2&2
&-2&2&.&.\\
\end{array}
$$
} }

\newpage
Character table of $E(\Fi_{23})$ (continued){
\renewcommand{\baselinestretch}{0.5}

{\tiny
\setlength{\arraycolsep}{0,3mm}
$$
\begin{array}{r|rrrrrrrrrrrrrrrrrrrrrrrrrrr}
2&7&7&7&5&5&3&3&3&1&1&5&5&4&4&3&3&2&2&2&2&1&1&5&5&1&1\\
3&.&.&.&.&.&1&.&.&.&.&1&1&1&1&.&.&.&.&.&.&1&1&.&.&.&.\\
5&.&.&.&.&.&1&1&1&.&.&.&.&.&.&.&.&.&.&.&.&1&1&.&.&.&.\\
7&.&.&.&.&.&.&.&.&.&.&.&.&.&.&1&1&1&1&1&1&.&.&.&.&.&.\\
11&.&.&.&.&.&.&.&.&1&1&.&.&.&.&.&.&.&.&.&.&.&.&.&.&1&1\\
23&.&.&.&.&.&.&.&.&.&.&.&.&.&.&.&.&.&.&.&.&.&.&.&.&.&.\\
\hline
&8a&8b&8c&8d&8e&10a&10b&10c&11a&11b&12a&12b&12c&12d&14a&14b&14c
&14d&14e&14f&15a&15b&16a&16b&22a&22b\\
\hline
2P&4b&4c&4c&4e&4f&5a&5a&5a&11b&11a&6b&6b&6d&6d&7b&7a&7a&7b&7b&7a
&15a&15b&8a&8a&11a&11b\\
3P&8a&8b&8c&8d&8e&10a&10b&10c&11a&11b&4b&4c&4d&4a&14b&14a&14d&14c
&14f&14e&5a&5a&16a&16b&22a&22b\\
5P&8a&8b&8c&8d&8e&2c&2b&2a&11a&11b&12a&12b&12c&12d&14b&14a&14d&14c
&14f&14e&3a&3a&16b&16a&22a&22b\\
7P&8a&8b&8c&8d&8e&10a&10b&10c&11b&11a&12a&12b&12c&12d&2b&2b&2a&2a
&2d&2d&15b&15a&16b&16a&22b&22a\\
11P&8a&8b&8c&8d&8e&10a&10b&10c&1a&1a&12a&12b&12c&12d&14a&14b&14c
&14d&14e&14f&15b&15a&16a&16b&2a&2a\\
23P&8a&8b&8c&8d&8e&10a&10b&10c&11a&11b&12a&12b&12c&12d&14a&14b&14c
&14d&14e&14f&15a&15b&16b&16a&22a&22b\\
\hline
X.1&1&1&1&1&1&1&1&1&1&1&1&1&1&1&1&1&1&1&1&1&1&1&1&1&1&1\\
X.2&2&2&2&.&.&2&2&2&.&.&.&.&.&.&1&1&1&1&-1&-1&-1&-1&.&.&.&.\\
X.3&1&1&1&-1&-1&.&.&.&1&1&.&.&.&.&\bar{A}&A&A&\bar{A}&-\bar{A}&-A&.
&.&-1&-1&1&1\\
X.4&1&1&1&-1&-1&.&.&.&1&1&.&.&.&.&A&\bar{A}&\bar{A}&A&-A&-\bar{A}&.
&.&-1&-1&1&1\\
X.5&2&2&2&.&.&.&.&.&-1&-1&1&1&1&1&-1&-1&-1&-1&1&1&.&.&.&.&-1&-1\\
X.6&-1&-1&-1&-1&-1&1&1&1&.&.&-2&-2&-2&-2&.&.&.&.&.&.&1&1&-1&-1&.&.\\
X.7&-1&-1&-1&-1&-1&1&1&1&.&.&1&1&1&1&.&.&.&.&.&.&B&\bar{B}&-1&-1&.
&.\\
X.8&-1&-1&-1&-1&-1&1&1&1&.&.&1&1&1&1&.&.&.&.&.&.&\bar{B}&B&-1&-1&.
&.\\
X.9&1&1&1&-1&-1&-2&-2&-2&.&.&1&1&1&1&1&1&1&1&-1&-1&1&1&-1&-1&.&.\\
X.10&1&-3&1&1&1&-3&-1&1&.&.&.&2&-2&.&1&1&-1&-1&1&1&.&.&-1&-1&.&.\\
X.11&-2&2&-2&.&.&-1&1&-1&.&.&-3&-1&1&-1&-2&-2&.&.&.&.&1&1&.&.&.&.\\
X.12&-2&-2&-2&.&.&.&.&.&.&.&1&1&1&1&.&.&.&.&.&.&.&.&.&.&.&.\\
X.13&-2&-2&-2&.&.&.&.&.&.&.&1&1&1&1&.&.&.&.&.&.&.&.&.&.&.&.\\
X.14&.&.&.&.&.&1&1&1&D&\bar{D}&.&.&.&.&.&.&.&.&.&.&1&1&.&.&\bar{D}&D\\
X.15&.&.&.&.&.&1&1&1&\bar{D}&D&.&.&.&.&.&.&.&.&.&.&1&1&.&.&D&\bar{D}\\
X.16&2&2&2&.&.&.&.&.&.&.&.&.&.&.&A&\bar{A}&\bar{A}&A&A&\bar{A}&.&.&.
&.&.&.\\
X.17&2&2&2&.&.&.&.&.&.&.&.&.&.&.&\bar{A}&A&A&\bar{A}&\bar{A}&A&.&.&.
&.&.&.\\
X.18&-1&-1&-1&1&1&.&.&.&1&1&.&.&.&.&-1&-1&-1&-1&-1&-1&.&.&1&1&1&1\\
X.19&.&.&.&-2&2&3&-1&-1&1&1&2&-2&.&.&.&.&.&.&.&.&.&.&.&.&-1&-1\\
X.20&2&-2&-2&.&.&-3&-1&1&.&.&1&-1&1&-1&-1&-1&1&1&-1&-1&.&.&.&.&.&.\\
X.21&.&.&.&.&.&-1&-1&-1&.&.&-1&-1&-1&-1&1&1&1&1&1&1&-1&-1&.&.&.&.\\
X.22&-2&2&2&.&.&.&.&.&.&.&.&2&-2&.&\bar{A}&A&-A&-\bar{A}&\bar{A}&A&.
&.&.&.&.&.\\
X.23&-2&2&2&.&.&.&.&.&.&.&.&2&-2&.&A&\bar{A}&-\bar{A}&-A&A&\bar{A}&.
&.&.&.&.&.\\
X.24&2&-2&-2&.&.&3&1&-1&.&.&3&1&-1&-3&.&.&.&.&.&.&.&.&.&.&.&.\\
X.25&2&2&-2&.&.&-2&2&-2&.&.&2&2&.&.&.&.&.&.&.&.&-1&-1&.&.&.&.\\
X.26&2&-2&-2&.&.&3&1&-1&.&.&-3&1&-1&3&.&.&.&.&.&.&.&.&.&.&.&.\\
X.27&-1&-1&3&-1&-1&.&.&.&.&.&-1&-1&1&1&1&1&-1&-1&-1&-1&.&.&1&1&.&.\\
X.28&-1&3&-1&-1&-1&.&.&.&.&.&1&-1&1&-1&1&1&-1&-1&1&1&.&.&1&1&.&.\\
X.29&1&1&-3&-1&-1&-3&-1&1&.&.&-1&1&-1&1&.&.&.&.&.&.&.&.&1&1&.&.\\
X.30&-4&.&.&.&.&1&-1&1&.&.&1&-1&-1&1&.&.&.&.&.&.&-1&-1&.&.&.&.\\
X.31&-1&-1&3&1&1&.&.&.&.&.&.&-2&2&.&.&.&.&.&.&.&.&.&-1&-1&.&.\\
X.32&.&.&.&.&.&.&.&.&.&.&-3&-1&1&-1&2&2&.&.&.&.&.&.&.&.&.&.\\
X.33&-2&-2&2&.&.&-1&1&-1&.&.&2&-2&-2&2&.&.&.&.&.&.&1&1&.&.&.&.\\
X.34&-2&-2&2&.&.&-1&1&-1&.&.&-1&1&1&-1&.&.&.&.&.&.&B&\bar{B}&.&.&.&.\\
X.35&-2&-2&2&.&.&-1&1&-1&.&.&-1&1&1&-1&.&.&.&.&.&.&\bar{B}&B&.&.&.&.\\
X.36&1&-3&1&-1&-1&.&.&.&.&.&.&.&.&.&\bar{A}&A&-A&-\bar{A}&-\bar{A}
&-A&.&.&1&1&.&.\\
X.37&1&-3&1&-1&-1&.&.&.&.&.&.&.&.&.&A&\bar{A}&-\bar{A}&-A&-A
&-\bar{A}&.&.&1&1&.&.\\
X.38&.&.&.&.&.&.&.&.&-1&-1&-2&2&.&.&.&.&.&.&.&.&.&.&.&.&1&1\\
X.39&.&.&.&.&.&.&.&.&-1&-1&2&-2&.&.&.&.&.&.&.&.&.&.&E&\bar{E}&1&1\\
X.40&.&.&.&.&.&.&.&.&-1&-1&2&-2&.&.&.&.&.&.&.&.&.&.&\bar{E}&E&1&1\\
X.41&.&.&.&2&-2&3&-1&-1&.&.&.&.&.&.&.&.&.&.&.&.&.&.&.&.&.&.\\
X.42&.&.&.&.&.&2&-2&2&.&.&1&3&1&-1&.&.&.&.&.&.&1&1&.&.&.&.\\
X.43&2&-2&2&.&.&.&.&.&.&.&-1&1&1&-1&.&.&.&.&.&.&.&.&.&.&.&.\\
X.44&.&.&.&.&.&3&-1&-1&D&\bar{D}&.&.&.&.&.&.&.&.&.&.&.&.&.&.
&-\bar{D}&-D\\
X.45&.&.&.&.&.&3&-1&-1&\bar{D}&D&.&.&.&.&.&.&.&.&.&.&.&.&.&.&-D
&-\bar{D}\\
X.46&-2&2&2&.&.&.&.&.&.&.&.&.&.&.&-1&-1&1&1&1&1&.&.&.&.&.&.\\
X.47&-2&2&-2&.&.&.&.&.&.&.&.&.&.&.&-\bar{F}&-F&.&.&.&.&.&.&.&.&.&.\\
X.48&-2&2&-2&.&.&.&.&.&.&.&.&.&.&.&-F&-\bar{F}&.&.&.&.&.&.&.&.&.&.\\
X.49&1&1&-3&1&1&.&.&.&.&.&-1&-1&1&1&.&.&.&.&.&.&.&.&-1&-1&.&.\\
X.50&.&.&.&.&.&-1&1&-1&.&.&2&2&.&.&.&.&.&.&.&.&1&1&.&.&.&.\\
X.51&.&.&.&.&.&.&.&.&.&.&1&1&-1&-1&1&1&-1&-1&-1&-1&.&.&.&.&.&.\\
X.52&.&.&.&.&.&1&-1&1&.&.&.&.&.&.&-2&-2&.&.&.&.&-1&-1&.&.&.&.\\
X.53&4&.&.&.&.&.&.&.&.&.&-1&-3&-1&1&.&.&.&.&.&.&.&.&.&.&.&.\\
X.54&.&.&.&.&.&-3&1&1&.&.&2&-2&.&.&.&.&.&.&.&.&.&.&.&.&.&.\\
X.55&.&.&.&2&-2&.&.&.&1&1&.&.&.&.&.&.&.&.&.&.&.&.&.&.&-1&-1\\
X.56&.&.&.&-2&2&.&.&.&.&.&-2&2&.&.&.&.&.&.&.&.&.&.&.&.&.&.\\
\end{array}
$$
} }

\newpage
Character table of $E(\Fi_{23})$ (continued){
\renewcommand{\baselinestretch}{0.5}

{\tiny
\setlength{\arraycolsep}{0,3mm}
$$
\begin{array}{r|rrrrrrr}
2&.&.&2&2&1&1\\
3&.&.&.&.&1&1\\
5&.&.&.&.&1&1\\
7&.&.&1&1&.&.\\
11&.&.&.&.&.&.\\
23&1&1&.&.&.&.\\
\hline
&23a&23b&28a&28b&30a&30b\\
\hline
2P&23a&23b&14a&14b&15a&15b\\
3P&23a&23b&28b&28a&10a&10a\\
5P&23b&23a&28b&28a&6a&6a\\
7P&23b&23a&4a&4a&30b&30a\\
11P&23b&23a&28a&28b&30b&30a\\
23P&1a&1a&28a&28b&30a&30b\\
\hline
X.1&1&1&1&1&1&1\\
X.2&-1&-1&-1&-1&-1&-1\\
X.3&-1&-1&-\bar{A}&-A&.&.\\
X.4&-1&-1&-A&-\bar{A}&.&.\\
X.5&.&.&1&1&.&.\\
X.6&1&1&.&.&1&1\\
X.7&1&1&.&.&B&\bar{B}\\
X.8&1&1&.&.&\bar{B}&B\\
X.9&.&.&-1&-1&1&1\\
X.10&.&.&-1&-1&.&.\\
X.11&.&.&.&.&-1&-1\\
X.12&C&\bar{C}&.&.&.&.\\
X.13&\bar{C}&C&.&.&.&.\\
X.14&-1&-1&.&.&1&1\\
X.15&-1&-1&.&.&1&1\\
X.16&1&1&A&\bar{A}&.&.\\
X.17&1&1&\bar{A}&A&.&.\\
X.18&.&.&-1&-1&.&.\\
X.19&.&.&.&.&.&.\\
X.20&.&.&1&1&.&.\\
X.21&.&.&1&1&-1&-1\\
X.22&.&.&-\bar{A}&-A&.&.\\
X.23&.&.&-A&-\bar{A}&.&.\\
X.24&.&.&.&.&.&.\\
X.25&.&.&.&.&1&1\\
X.26&.&.&.&.&.&.\\
X.27&.&.&1&1&.&.\\
X.28&.&.&-1&-1&.&.\\
X.29&.&.&.&.&.&.\\
X.30&.&.&.&.&1&1\\
X.31&.&.&.&.&.&.\\
X.32&.&.&.&.&.&.\\
X.33&.&.&.&.&-1&-1\\
X.34&.&.&.&.&-B&-\bar{B}\\
X.35&.&.&.&.&-\bar{B}&-B\\
X.36&.&.&\bar{A}&A&.&.\\
X.37&.&.&A&\bar{A}&.&.\\
X.38&.&.&.&.&.&.\\
X.39&.&.&.&.&.&.\\
X.40&.&.&.&.&.&.\\
X.41&.&.&.&.&.&.\\
X.42&.&.&.&.&-1&-1\\
X.43&.&.&.&.&.&.\\
X.44&.&.&.&.&.&.\\
X.45&.&.&.&.&.&.\\
X.46&.&.&-1&-1&.&.\\
X.47&.&.&.&.&.&.\\
X.48&.&.&.&.&.&.\\
X.49&.&.&.&.&.&.\\
X.50&.&.&.&.&-1&-1\\
X.51&.&.&1&1&.&.\\
X.52&.&.&.&.&1&1\\
X.53&.&.&.&.&.&.\\
X.54&.&.&.&.&.&.\\
X.55&.&.&.&.&.&.\\
X.56&.&.&.&.&.&.\\
\end{array}
$$
} }
$
A = \zeta(7)_7^4 + \zeta(7)_7^2 + \zeta(7)_7,
B = -2\zeta(15)_3 \zeta(15)_5^3 - 2\zeta(15)_3 \zeta(15)_5^2 - \zeta(15)_3 -
\zeta(15)_5^3 - \zeta(15)_5^2 - 1,
C = \zeta(23)_{23}^{18} + \zeta(23)_{23}^{16} + \zeta(23)_{23}^{13} + \zeta(23)_{23}^{12} +
\zeta(23)_{23}^9 + \zeta(23)_{23}^8 + \zeta(23)_{23}^6 + \zeta(23)_{23}^4 + \zeta(23)_{23}^3 +
\zeta(23)_{23}^2 + \zeta(23)_{23},
D = \zeta(11)_{11}^9 + \zeta(11)_{11}^5 + \zeta(11)_{11}^4 + \zeta(11)_{11}^3 + \zeta(11)_{11},
E = -2 \zeta(8)_8^3 - 2 \zeta(8)_8,
F = 2 \zeta(7)_7^4 + 2 \zeta(7)_7^2 + 2\zeta(7)_7
$.
\end{ctab}

\newpage


\begin{ctab}\label{Fi_23ct_D} Character table of $D(\Fi_{23}) = D = \langle x_1,y_1 \rangle \cong \langle x_0, y_0 \rangle$
{
\renewcommand{\baselinestretch}{0.5}

{\tiny
\setlength{\arraycolsep}{0,3mm}
$$
\begin{array}{r|rrrrrrrrrrrrrrrrrrrrrrr}
2&18&18&17&17&18&18&16&14&14&13&7&11&11&12&12&10&10&10&9&9&8&8\\
3&2&2&2&2&1&1&2&1&1&.&2&1&1&.&.&1&1&.&.&.&.&.\\
5&1&1&1&1&1&1&.&.&.&.&.&.&.&.&.&.&.&.&.&.&.&.\\
7&1&1&1&1&.&.&.&.&.&.&.&.&.&.&.&.&.&.&.&.&.&.\\
11&1&1&.&.&.&.&.&.&.&.&.&.&.&.&.&.&.&.&.&.&.&.\\
\hline
&1a&2a&2b&2c&2d&2e&2f&2g&2h&2i&3a&4a&4b&4c&4d&4e&4f&4g&4h&4i&4j
&4k\\
\hline
2P&1a&1a&1a&1a&1a&1a&1a&1a&1a&1a&3a&2e&2e&2d&2d&2f&2f&2e&2g&2g&2h
&2i\\
3P&1a&2a&2b&2c&2d&2e&2f&2g&2h&2i&1a&4a&4b&4c&4d&4e&4f&4g&4h&4i&4j
&4k\\
5P&1a&2a&2b&2c&2d&2e&2f&2g&2h&2i&3a&4a&4b&4c&4d&4e&4f&4g&4h&4i&4j
&4k\\
7P&1a&2a&2b&2c&2d&2e&2f&2g&2h&2i&3a&4a&4b&4c&4d&4e&4f&4g&4h&4i&4j
&4k\\
11P&1a&2a&2b&2c&2d&2e&2f&2g&2h&2i&3a&4a&4b&4c&4d&4e&4f&4g&4h&4i&4j
&4k\\
\hline
X.1&1&1&1&1&1&1&1&1&1&1&1&1&1&1&1&1&1&1&1&1&1&1\\
X.2&21&21&21&21&21&21&21&5&5&5&3&5&5&5&5&5&5&5&1&1&1&1\\
X.3&45&45&45&45&45&45&45&-3&-3&-3&.&-3&-3&-3&-3&-3&-3&-3&1&1&1&1\\
X.4&45&45&45&45&45&45&45&-3&-3&-3&.&-3&-3&-3&-3&-3&-3&-3&1&1&1&1\\
X.5&55&55&55&55&55&55&55&7&7&7&1&7&7&7&7&7&7&7&3&3&3&-1\\
X.6&77&77&-35&-35&13&13&-3&13&13&-3&5&-7&-7&-3&5&1&1&1&-3&5&1&1\\
X.7&99&99&99&99&99&99&99&3&3&3&.&3&3&3&3&3&3&3&3&3&3&-1\\
X.8&154&154&154&154&154&154&154&10&10&10&1&10&10&10&10&10&10&10&-2
&-2&-2&2\\
X.9&176&-176&-64&64&-16&16&.&-16&16&.&5&8&-8&.&.&4&-4&.&.&.&.&4\\
X.10&176&-176&64&-64&-16&16&.&-16&16&.&5&-8&8&.&.&4&-4&.&.&.&.&4\\
X.11&210&210&210&210&210&210&210&2&2&2&3&2&2&2&2&2&2&2&-2&-2&-2&-2\\
X.12&231&231&231&231&231&231&231&7&7&7&-3&7&7&7&7&7&7&7&-1&-1&-1&-1\\
X.13&280&280&280&280&280&280&280&-8&-8&-8&1&-8&-8&-8&-8&-8&-8&-8&.&.
&.&.\\
X.14&280&280&280&280&280&280&280&-8&-8&-8&1&-8&-8&-8&-8&-8&-8&-8&.&.
&.&.\\
X.15&330&330&90&90&10&10&-6&26&26&10&6&6&6&-6&2&-2&-2&-2&-2&6&2&2\\
X.16&385&385&-175&-175&65&65&-15&17&17&1&-2&-11&-11&1&9&5&5&-3&5&-3
&1&1\\
X.17&385&385&385&385&385&385&385&1&1&1&-2&1&1&1&1&1&1&1&1&1&1&1\\
X.18&385&385&-175&-175&65&65&-15&17&17&1&7&-11&-11&1&9&5&5&-3&5&-3&1
&1\\
X.19&616&616&-56&-56&-24&-24&8&24&24&-8&4&.&.&8&-8&.&.&.&4&4&-4&4\\
X.20&616&616&-56&-56&-24&-24&8&24&24&-8&4&.&.&8&-8&.&.&.&4&4&-4&-4\\
X.21&616&616&-280&-280&104&104&-24&8&8&8&-5&-8&-8&8&8&8&8&-8&.&.&.&.\\
X.22&616&616&-280&-280&104&104&-24&8&8&8&-5&-8&-8&8&8&8&8&-8&.&.&.&.\\
X.23&672&-672&.&.&32&-32&.&-32&32&.&6&.&.&.&.&-8&8&.&.&.&.&.\\
X.24&693&693&-315&-315&117&117&-27&21&21&5&.&-15&-15&5&13&9&9&-7&-3
&5&1&1\\
X.25&770&770&-350&-350&130&130&-30&-14&-14&18&5&2&2&18&2&10&10&-14
&-2&-2&-2&-2\\
X.26&990&990&270&270&30&30&-18&-18&-18&-2&.&-6&-6&14&6&-6&-6&2&-6&2
&-2&2\\
X.27&990&990&270&270&30&30&-18&-18&-18&-2&.&-6&-6&14&6&-6&-6&2&-6&2
&-2&2\\
X.28&1056&-1056&-384&384&-96&96&.&-32&32&.&3&16&-16&.&.&8&-8&.&.&.&.
&.\\
X.29&1056&-1056&384&-384&-96&96&.&-32&32&.&3&-16&16&.&.&8&-8&.&.&.&.
&.\\
X.30&1155&1155&-525&-525&195&195&-45&35&35&-13&3&-17&-17&-13&11&-1
&-1&7&-5&3&-1&-1\\
X.31&1155&1155&-525&-525&195&195&-45&-13&-13&3&3&7&7&3&-5&-1&-1&-1&7
&-1&3&-1\\
X.32&1760&-1760&640&-640&-160&160&.&32&-32&.&5&16&-16&.&.&-8&8&.&.&.
&.&.\\
X.33&1760&-1760&-640&640&-160&160&.&32&-32&.&5&-16&16&.&.&-8&8&.&.&.
&.&.\\
X.34&1760&-1760&640&-640&-160&160&.&32&-32&.&5&16&-16&.&.&-8&8&.&.&.
&.&.\\
X.35&1760&-1760&-640&640&-160&160&.&32&-32&.&5&-16&16&.&.&-8&8&.&.&.
&.&.\\
X.36&1980&1980&540&540&60&60&-36&60&60&28&.&12&12&-4&12&-12&-12&-4&4
&4&4&.\\
X.37&2310&2310&630&630&70&70&-42&22&22&-26&6&18&18&-10&14&2&2&-6&-6
&2&-2&-2\\
X.38&2310&2310&-1050&-1050&390&390&-90&22&22&-10&-3&-10&-10&-10&6&-2
&-2&6&2&2&2&-2\\
X.39&2310&2310&630&630&70&70&-42&22&22&38&6&-6&-6&-10&-18&2&2&2&2&-6
&-2&2\\
X.40&2310&2310&630&630&70&70&-42&-10&-10&6&6&-6&-6&22&14&-14&-14&2&2
&-6&-2&-2\\
X.41&2464&-2464&-896&896&-224&224&.&-32&32&.&7&16&-16&.&.&8&-8&.&.&.
&.&.\\
X.42&2464&-2464&896&-896&-224&224&.&-32&32&.&-2&-16&16&.&.&8&-8&.&.
&.&.&.\\
X.43&2464&-2464&-896&896&-224&224&.&-32&32&.&-2&16&-16&.&.&8&-8&.&.
&.&.&.\\
X.44&2464&-2464&896&-896&-224&224&.&-32&32&.&7&-16&16&.&.&8&-8&.&.&.
&.&.\\
X.45&2640&-2640&960&-960&-240&240&.&16&-16&.&3&8&-8&.&.&-4&4&.&.&.&.
&-4\\
X.46&2640&2640&720&720&80&80&-48&16&16&16&-6&.&.&16&16&-16&-16&.&.&.
&.&.\\
X.47&2640&-2640&-960&960&-240&240&.&16&-16&.&3&-8&8&.&.&-4&4&.&.&.&.
&-4\\
X.48&3360&-3360&.&.&160&-160&.&-32&32&.&-6&.&.&.&.&-8&8&.&.&.&.&.\\
X.49&3360&-3360&.&.&160&-160&.&-32&32&.&-6&.&.&.&.&-8&8&.&.&.&.&.\\
X.50&3465&3465&-1575&-1575&585&585&-135&9&9&-7&.&-3&-3&-7&1&-3&-3&5
&1&-7&-3&1\\
X.51&3465&3465&-1575&-1575&585&585&-135&-39&-39&9&.&21&21&9&-15&-3
&-3&-3&-3&5&1&1\\
X.52&3696&-3696&-1344&1344&-336&336&.&-16&16&.&-3&8&-8&.&.&4&-4&.&.
&.&.&-4\\
X.53&3696&-3696&1344&-1344&-336&336&.&-16&16&.&-3&-8&8&.&.&4&-4&.&.
&.&.&-4\\
X.54&4620&4620&1260&1260&140&140&-84&44&44&12&-6&12&12&-20&-4&4&4&-4
&-4&-4&-4&.\\
X.55&5544&5544&-504&-504&-216&-216&72&24&24&-8&.&.&.&8&-8&.&.&.&4&4
&-4&-4\\
X.56&5544&5544&-504&-504&-216&-216&72&24&24&-8&.&.&.&8&-8&.&.&.&4&4
&-4&4\\
X.57&6160&-6160&-2240&2240&-560&560&.&16&-16&.&-5&-8&8&.&.&-4&4&.&.
&.&.&4\\
X.58&6160&6160&-560&-560&-240&-240&80&48&48&-16&4&.&.&16&-16&.&.&.
&-8&-8&8&.\\
X.59&6160&-6160&2240&-2240&-560&560&.&16&-16&.&-5&8&-8&.&.&-4&4&.&.
&.&.&4\\
X.60&6160&6160&-560&-560&-240&-240&80&-48&-48&16&4&.&.&-16&16&.&.&.
&.&.&.&.\\
X.61&6160&6160&-560&-560&-240&-240&80&-48&-48&16&4&.&.&-16&16&.&.&.
&.&.&.&.\\
X.62&6720&-6720&.&.&320&-320&.&-64&64&.&6&.&.&.&.&-16&16&.&.&.&.&.\\
X.63&6720&-6720&.&.&320&-320&.&64&-64&.&6&.&.&.&.&16&-16&.&.&.&.&.\\
X.64&6930&6930&1890&1890&210&210&-126&-30&-30&-46&.&6&6&2&10&6&6&-2
&6&-2&2&2\\
X.65&6930&6930&1890&1890&210&210&-126&-30&-30&18&.&-18&-18&2&-22&6&6
&6&-2&6&2&-2\\
X.66&7392&-7392&.&.&352&-352&.&32&-32&.&-6&.&.&.&.&8&-8&.&.&.&.&.\\
X.67&8064&-8064&.&.&384&-384&.&.&.&.&.&.&.&.&.&.&.&.&.&.&.&.\\
X.68&8064&-8064&.&.&384&-384&.&.&.&.&.&.&.&.&.&.&.&.&.&.&.&.\\
X.69&9856&9856&-896&-896&-384&-384&128&.&.&.&-8&.&.&.&.&.&.&.&.&.&.
&.\\
\end{array}
$$
} }

\newpage
Character table of $D(\Fi_{23})$ (continued){
\renewcommand{\baselinestretch}{0.5}

{\tiny
\setlength{\arraycolsep}{0,3mm}
$$
\begin{array}{r|rrrrrrrrrrrrrrrrrrrrrrrrrrrrrrr}
2&8&7&3&7&7&7&5&5&5&5&6&6&5&5&2&2&7&7&7&6&6&6&5&5&3&3&3&3&3&3\\
3&.&.&.&2&2&2&2&2&2&2&1&1&1&1&.&.&.&.&.&.&.&.&.&.&.&.&.&.&.&.\\
5&.&.&1&.&.&.&.&.&.&.&.&.&.&.&.&.&.&.&.&.&.&.&.&.&1&1&1&1&1&1\\
7&.&.&.&.&.&.&.&.&.&.&.&.&.&.&1&1&.&.&.&.&.&.&.&.&.&.&.&.&.&.\\
11&.&.&.&.&.&.&.&.&.&.&.&.&.&.&.&.&.&.&.&.&.&.&.&.&.&.&.&.&.&.\\
\hline
&4l&4m&5a&6a&6b&6c&6d&6e&6f&6g&6h&6i&6j&6k&7a&7b&8a&8b&8c&8d&8e
&8f&8g&8h&10a&10b&10c&10d&10e&10f\\
\hline
2P&2i&2h&5a&3a&3a&3a&3a&3a&3a&3a&3a&3a&3a&3a&7a&7b&4c&4d&4d&4e&4c
&4e&4h&4i&5a&5a&5a&5a&5a&5a\\
3P&4l&4m&5a&2f&2f&2a&2f&2f&2c&2b&2e&2d&2h&2g&7b&7a&8a&8b&8c&8d&8e
&8f&8g&8h&10b&10a&10c&10f&10e&10d\\
5P&4l&4m&1a&6a&6b&6c&6d&6e&6f&6g&6h&6i&6j&6k&7b&7a&8a&8b&8c&8d&8e
&8f&8g&8h&2c&2c&2e&2b&2d&2b\\
7P&4l&4m&5a&6a&6b&6c&6d&6e&6f&6g&6h&6i&6j&6k&1a&1a&8a&8b&8c&8d&8e
&8f&8g&8h&10b&10a&10c&10f&10e&10d\\
11P&4l&4m&5a&6a&6b&6c&6d&6e&6f&6g&6h&6i&6j&6k&7a&7b&8a&8b&8c&8d&8e
&8f&8g&8h&10a&10b&10c&10d&10e&10f\\
\hline
X.1&1&1&1&1&1&1&1&1&1&1&1&1&1&1&1&1&1&1&1&1&1&1&1&1&1&1&1&1&1&1\\
X.2&1&1&1&3&3&3&3&3&3&3&3&3&-1&-1&.&.&1&1&1&1&1&1&-1&-1&1&1&1&1&1&1\\
X.3&1&1&.&.&.&.&.&.&.&.&.&.&.&.&A&\bar{A}&1&1&1&1&1&1&-1&-1&.&.&.&.
&.&.\\
X.4&1&1&.&.&.&.&.&.&.&.&.&.&.&.&\bar{A}&A&1&1&1&1&1&1&-1&-1&.&.&.&.
&.&.\\
X.5&-1&-1&.&1&1&1&1&1&1&1&1&1&1&1&-1&-1&3&3&3&-1&-1&-1&1&1&.&.&.&.
&.&.\\
X.6&1&1&2&-3&-3&5&-3&-3&1&1&1&1&1&1&.&.&1&-3&1&-1&1&-1&1&1&.&.&-2&.
&-2&.\\
X.7&-1&-1&-1&.&.&.&.&.&.&.&.&.&.&.&1&1&3&3&3&-1&-1&-1&-1&-1&-1&-1
&-1&-1&-1&-1\\
X.8&2&2&-1&1&1&1&1&1&1&1&1&1&1&1&.&.&-2&-2&-2&2&2&2&.&.&-1&-1&-1&-1
&-1&-1\\
X.9&-4&.&1&-3&3&-5&3&-3&1&-1&1&-1&1&-1&1&1&.&.&.&-2&.&2&.&.&-1&-1&1
&1&-1&1\\
X.10&-4&.&1&-3&3&-5&-3&3&-1&1&1&-1&1&-1&1&1&.&.&.&2&.&-2&.&.&1&1&1
&-1&-1&-1\\
X.11&-2&-2&.&3&3&3&3&3&3&3&3&3&-1&-1&.&.&-2&-2&-2&-2&-2&-2&.&.&.&.&.
&.&.&.\\
X.12&-1&-1&1&-3&-3&-3&-3&-3&-3&-3&-3&-3&1&1&.&.&-1&-1&-1&-1&-1&-1&-1
&-1&1&1&1&1&1&1\\
X.13&.&.&.&1&1&1&1&1&1&1&1&1&1&1&.&.&.&.&.&.&.&.&.&.&.&.&.&.&.&.\\
X.14&.&.&.&1&1&1&1&1&1&1&1&1&1&1&.&.&.&.&.&.&.&.&.&.&.&.&.&.&.&.\\
X.15&2&2&.&6&6&6&.&.&.&.&-2&-2&2&2&1&1&-2&2&-2&.&-2&.&.&.&.&.&.&.&.
&.\\
X.16&1&1&.&6&6&-2&.&.&-4&-4&2&2&2&2&.&.&1&1&-3&-1&1&-1&-1&-1&.&.&.&.
&.&.\\
X.17&1&1&.&-2&-2&-2&-2&-2&-2&-2&-2&-2&-2&-2&.&.&1&1&1&1&1&1&1&1&.&.
&.&.&.&.\\
X.18&1&1&.&-9&-9&7&-3&-3&5&5&-1&-1&-1&-1&.&.&1&1&-3&-1&1&-1&-1&-1&.
&.&.&.&.&.\\
X.19&4&-4&1&-4&-4&4&2&2&-2&-2&.&.&.&.&.&.&.&.&.&.&.&.&2&-2&-1&-1&1
&-1&1&-1\\
X.20&-4&4&1&-4&-4&4&2&2&-2&-2&.&.&.&.&.&.&.&.&.&.&.&.&-2&2&-1&-1&1
&-1&1&-1\\
X.21&.&.&1&3&3&-5&3&3&-1&-1&-1&-1&-1&-1&.&.&.&.&.&.&.&.&.&.&C&-C&-1
&-C&-1&C\\
X.22&.&.&1&3&3&-5&3&3&-1&-1&-1&-1&-1&-1&.&.&.&.&.&.&.&.&.&.&-C&C&-1
&C&-1&-C\\
X.23&.&.&2&6&-6&-6&.&.&.&.&-2&2&2&-2&.&.&.&.&.&.&.&.&.&.&.&.&-2&.&2
&.\\
X.24&1&1&-2&.&.&.&.&.&.&.&.&.&.&.&.&.&1&-3&1&-1&1&-1&1&1&.&.&2&.&2&.\\
X.25&-2&-2&.&-3&-3&5&-3&-3&1&1&1&1&1&1&.&.&-2&2&2&2&-2&2&.&.&.&.&.&.
&.&.\\
X.26&2&2&.&.&.&.&.&.&.&.&.&.&.&.&A&\bar{A}&2&2&-2&.&-2&.&.&.&.&.&.&.
&.&.\\
X.27&2&2&.&.&.&.&.&.&.&.&.&.&.&.&\bar{A}&A&2&2&-2&.&-2&.&.&.&.&.&.&.
&.&.\\
X.28&.&.&1&3&-3&-3&3&-3&-3&3&3&-3&-1&1&-1&-1&.&.&.&.&.&.&.&.&-1&-1&1
&1&-1&1\\
X.29&.&.&1&3&-3&-3&-3&3&3&-3&3&-3&-1&1&-1&-1&.&.&.&.&.&.&.&.&1&1&1
&-1&-1&-1\\
X.30&-1&-1&.&3&3&3&-3&-3&-3&-3&3&3&-1&-1&.&.&-1&-1&3&1&-1&1&-1&-1&.
&.&.&.&.&.\\
X.31&-1&-1&.&3&3&3&-3&-3&-3&-3&3&3&-1&-1&.&.&3&-1&-5&1&-1&1&1&1&.&.
&.&.&.&.\\
X.32&.&.&.&-3&3&-5&-3&3&-1&1&1&-1&1&-1&A&\bar{A}&.&.&.&.&.&.&.&.&.&.
&.&.&.&.\\
X.33&.&.&.&-3&3&-5&3&-3&1&-1&1&-1&1&-1&\bar{A}&A&.&.&.&.&.&.&.&.&.&.
&.&.&.&.\\
X.34&.&.&.&-3&3&-5&-3&3&-1&1&1&-1&1&-1&\bar{A}&A&.&.&.&.&.&.&.&.&.&.
&.&.&.&.\\
X.35&.&.&.&-3&3&-5&3&-3&1&-1&1&-1&1&-1&A&\bar{A}&.&.&.&.&.&.&.&.&.&.
&.&.&.&.\\
X.36&.&.&.&.&.&.&.&.&.&.&.&.&.&.&-1&-1&-4&.&.&.&.&.&.&.&.&.&.&.&.&.\\
X.37&-2&-2&.&6&6&6&.&.&.&.&-2&-2&-2&-2&.&.&2&2&-2&.&2&.&.&.&.&.&.&.
&.&.\\
X.38&-2&-2&.&-3&-3&-3&3&3&3&3&-3&-3&1&1&.&.&2&-2&-2&2&-2&2&.&.&.&.&.
&.&.&.\\
X.39&2&2&.&6&6&6&.&.&.&.&-2&-2&-2&-2&.&.&2&-2&2&.&-2&.&.&.&.&.&.&.&.
&.\\
X.40&-2&-2&.&6&6&6&.&.&.&.&-2&-2&2&2&.&.&2&-2&2&.&2&.&.&.&.&.&.&.&.
&.\\
X.41&.&.&-1&-9&9&-7&3&-3&5&-5&-1&1&-1&1&.&.&.&.&.&.&.&.&.&.&1&1&-1
&-1&1&-1\\
X.42&.&.&-1&6&-6&2&.&.&4&-4&2&-2&2&-2&.&.&.&.&.&.&.&.&.&.&-1&-1&-1&1
&1&1\\
X.43&.&.&-1&6&-6&2&.&.&-4&4&2&-2&2&-2&.&.&.&.&.&.&.&.&.&.&1&1&-1&-1
&1&-1\\
X.44&.&.&-1&-9&9&-7&-3&3&-5&5&-1&1&-1&1&.&.&.&.&.&.&.&.&.&.&-1&-1&-1
&1&1&1\\
X.45&4&.&.&3&-3&-3&-3&3&3&-3&3&-3&-1&1&1&1&.&.&.&-2&.&2&.&.&.&.&.&.
&.&.\\
X.46&.&.&.&-6&-6&-6&.&.&.&.&2&2&-2&-2&1&1&.&.&.&.&.&.&.&.&.&.&.&.&.
&.\\
X.47&4&.&.&3&-3&-3&3&-3&-3&3&3&-3&-1&1&1&1&.&.&.&2&.&-2&.&.&.&.&.&.
&.&.\\
X.48&.&.&.&-6&6&6&.&.&.&.&2&-2&2&-2&.&.&.&.&.&.&.&.&.&.&.&.&.&.&.&.\\
X.49&.&.&.&-6&6&6&.&.&.&.&2&-2&2&-2&.&.&.&.&.&.&.&.&.&.&.&.&.&.&.&.\\
X.50&1&1&.&.&.&.&.&.&.&.&.&.&.&.&.&.&-3&5&1&-1&1&-1&1&1&.&.&.&.&.&.\\
X.51&1&1&.&.&.&.&.&.&.&.&.&.&.&.&.&.&1&-3&1&-1&1&-1&-1&-1&.&.&.&.&.
&.\\
X.52&4&.&1&-3&3&3&-3&3&3&-3&-3&3&1&-1&.&.&.&.&.&2&.&-2&.&.&-1&-1&1&1
&-1&1\\
X.53&4&.&1&-3&3&3&3&-3&-3&3&-3&3&1&-1&.&.&.&.&.&-2&.&2&.&.&1&1&1&-1
&-1&-1\\
X.54&.&.&.&-6&-6&-6&.&.&.&.&2&2&2&2&.&.&4&.&.&.&.&.&.&.&.&.&.&.&.&.\\
X.55&-4&4&-1&.&.&.&.&.&.&.&.&.&.&.&.&.&.&.&.&.&.&.&2&-2&1&1&-1&1&-1
&1\\
X.56&4&-4&-1&.&.&.&.&.&.&.&.&.&.&.&.&.&.&.&.&.&.&.&-2&2&1&1&-1&1&-1
&1\\
X.57&-4&.&.&3&-3&5&-3&3&-1&1&-1&1&-1&1&.&.&.&.&.&-2&.&2&.&.&.&.&.&.
&.&.\\
X.58&.&.&.&-4&-4&4&2&2&-2&-2&.&.&.&.&.&.&.&.&.&.&.&.&.&.&.&.&.&.&.&.\\
X.59&-4&.&.&3&-3&5&3&-3&1&-1&-1&1&-1&1&.&.&.&.&.&2&.&-2&.&.&.&.&.&.
&.&.\\
X.60&.&.&.&-4&-4&4&2&2&-2&-2&.&.&.&.&.&.&.&.&.&.&.&.&.&.&.&.&.&.&.&.\\
X.61&.&.&.&-4&-4&4&2&2&-2&-2&.&.&.&.&.&.&.&.&.&.&.&.&.&.&.&.&.&.&.&.\\
X.62&.&.&.&6&-6&-6&.&.&.&.&-2&2&-2&2&.&.&.&.&.&.&.&.&.&.&.&.&.&.&.&.\\
X.63&.&.&.&6&-6&-6&.&.&.&.&-2&2&2&-2&.&.&.&.&.&.&.&.&.&.&.&.&.&.&.&.\\
X.64&2&2&.&.&.&.&.&.&.&.&.&.&.&.&.&.&-2&-2&2&.&-2&.&.&.&.&.&.&.&.&.\\
X.65&-2&-2&.&.&.&.&.&.&.&.&.&.&.&.&.&.&-2&2&-2&.&2&.&.&.&.&.&.&.&.&.\\
X.66&.&.&2&-6&6&6&.&.&.&.&2&-2&-2&2&.&.&.&.&.&.&.&.&.&.&.&.&-2&.&2&.\\
X.67&.&.&-1&.&.&.&.&.&.&.&.&.&.&.&.&.&.&.&.&.&.&.&.&.&C&-C&1&C&-1&-C\\
X.68&.&.&-1&.&.&.&.&.&.&.&.&.&.&.&.&.&.&.&.&.&.&.&.&.&-C&C&1&-C&-1&C\\
X.69&.&.&1&8&8&-8&-4&-4&4&4&.&.&.&.&.&.&.&.&.&.&.&.&.&.&-1&-1&1&-1&1
&-1\\
\end{array}
$$
} }

\newpage
Character table of $D(\Fi_{23})$ (continued){
\renewcommand{\baselinestretch}{0.5}

{\tiny
\setlength{\arraycolsep}{0,3mm}
$$
\begin{array}{r|rrrrrrrrrrrrrrrrrr}
2&3&1&1&5&5&4&4&2&2&2&2&2&2&5&5&1&1\\
3&.&.&.&1&1&1&1&.&.&.&.&.&.&.&.&.&.\\
5&1&.&.&.&.&.&.&.&.&.&.&.&.&.&.&.&.\\
7&.&.&.&.&.&.&.&1&1&1&1&1&1&.&.&.&.\\
11&.&1&1&.&.&.&.&.&.&.&.&.&.&.&.&1&1\\
\hline
&10g&11a&11b&12a&12b&12c&12d&14a&14b&14c&14d&14e&14f&16a&16b&22a
&22b\\
\hline
2P&5a&11b&11a&6a&6a&6h&6h&7a&7b&7a&7b&7a&7b&8a&8a&11a&11b\\
3P&10g&11a&11b&4f&4e&4a&4b&14f&14e&14d&14c&14b&14a&16a&16b&22a&22b\\
5P&2a&11a&11b&12a&12b&12c&12d&14f&14e&14d&14c&14b&14a&16b&16a&22a
&22b\\
7P&10g&11b&11a&12a&12b&12c&12d&2b&2a&2c&2c&2a&2b&16b&16a&22b&22a\\
11P&10g&1a&1a&12a&12b&12c&12d&14a&14b&14c&14d&14e&14f&16a&16b&2a&2a\\
\hline
X.1&1&1&1&1&1&1&1&1&1&1&1&1&1&1&1&1&1\\
X.2&1&-1&-1&-1&-1&-1&-1&.&.&.&.&.&.&-1&-1&-1&-1\\
X.3&.&1&1&.&.&.&.&A&\bar{A}&A&\bar{A}&A&\bar{A}&-1&-1&1&1\\
X.4&.&1&1&.&.&.&.&\bar{A}&A&\bar{A}&A&\bar{A}&A&-1&-1&1&1\\
X.5&.&.&.&1&1&1&1&-1&-1&-1&-1&-1&-1&1&1&.&.\\
X.6&2&.&.&1&1&-1&-1&.&.&.&.&.&.&-1&-1&.&.\\
X.7&-1&.&.&.&.&.&.&1&1&1&1&1&1&-1&-1&.&.\\
X.8&-1&.&.&1&1&1&1&.&.&.&.&.&.&.&.&.&.\\
X.9&-1&.&.&-1&1&-1&1&-1&-1&1&1&-1&-1&.&.&.&.\\
X.10&-1&.&.&-1&1&1&-1&1&-1&-1&-1&-1&1&.&.&.&.\\
X.11&.&1&1&-1&-1&-1&-1&.&.&.&.&.&.&.&.&1&1\\
X.12&1&.&.&1&1&1&1&.&.&.&.&.&.&-1&-1&.&.\\
X.13&.&B&\bar{B}&1&1&1&1&.&.&.&.&.&.&.&.&\bar{B}&B\\
X.14&.&\bar{B}&B&1&1&1&1&.&.&.&.&.&.&.&.&B&\bar{B}\\
X.15&.&.&.&-2&-2&.&.&-1&1&-1&-1&1&-1&.&.&.&.\\
X.16&.&.&.&2&2&-2&-2&.&.&.&.&.&.&1&1&.&.\\
X.17&.&.&.&-2&-2&-2&-2&.&.&.&.&.&.&1&1&.&.\\
X.18&.&.&.&-1&-1&1&1&.&.&.&.&.&.&1&1&.&.\\
X.19&1&.&.&.&.&.&.&.&.&.&.&.&.&.&.&.&.\\
X.20&1&.&.&.&.&.&.&.&.&.&.&.&.&.&.&.&.\\
X.21&1&.&.&-1&-1&1&1&.&.&.&.&.&.&.&.&.&.\\
X.22&1&.&.&-1&-1&1&1&.&.&.&.&.&.&.&.&.&.\\
X.23&-2&1&1&2&-2&.&.&.&.&.&.&.&.&.&.&-1&-1\\
X.24&-2&.&.&.&.&.&.&.&.&.&.&.&.&-1&-1&.&.\\
X.25&.&.&.&1&1&-1&-1&.&.&.&.&.&.&.&.&.&.\\
X.26&.&.&.&.&.&.&.&-A&\bar{A}&-A&-\bar{A}&A&-\bar{A}&.&.&.&.\\
X.27&.&.&.&.&.&.&.&-\bar{A}&A&-\bar{A}&-A&\bar{A}&-A&.&.&.&.\\
X.28&-1&.&.&1&-1&1&-1&1&1&-1&-1&1&1&.&.&.&.\\
X.29&-1&.&.&1&-1&-1&1&-1&1&1&1&1&-1&.&.&.&.\\
X.30&.&.&.&-1&-1&1&1&.&.&.&.&.&.&1&1&.&.\\
X.31&.&.&.&-1&-1&1&1&.&.&.&.&.&.&-1&-1&.&.\\
X.32&.&.&.&-1&1&1&-1&A&-\bar{A}&-A&-\bar{A}&-A&\bar{A}&.&.&.&.\\
X.33&.&.&.&-1&1&-1&1&-\bar{A}&-A&\bar{A}&A&-\bar{A}&-A&.&.&.&.\\
X.34&.&.&.&-1&1&1&-1&\bar{A}&-A&-\bar{A}&-A&-\bar{A}&A&.&.&.&.\\
X.35&.&.&.&-1&1&-1&1&-A&-\bar{A}&A&\bar{A}&-A&-\bar{A}&.&.&.&.\\
X.36&.&.&.&.&.&.&.&1&-1&1&1&-1&1&.&.&.&.\\
X.37&.&.&.&2&2&.&.&.&.&.&.&.&.&.&.&.&.\\
X.38&.&.&.&1&1&-1&-1&.&.&.&.&.&.&.&.&.&.\\
X.39&.&.&.&2&2&.&.&.&.&.&.&.&.&.&.&.&.\\
X.40&.&.&.&-2&-2&.&.&.&.&.&.&.&.&.&.&.&.\\
X.41&1&.&.&1&-1&1&-1&.&.&.&.&.&.&.&.&.&.\\
X.42&1&.&.&-2&2&2&-2&.&.&.&.&.&.&.&.&.&.\\
X.43&1&.&.&-2&2&-2&2&.&.&.&.&.&.&.&.&.&.\\
X.44&1&.&.&1&-1&-1&1&.&.&.&.&.&.&.&.&.&.\\
X.45&.&.&.&1&-1&-1&1&1&-1&-1&-1&-1&1&.&.&.&.\\
X.46&.&.&.&2&2&.&.&-1&1&-1&-1&1&-1&.&.&.&.\\
X.47&.&.&.&1&-1&1&-1&-1&-1&1&1&-1&-1&.&.&.&.\\
X.48&.&B&\bar{B}&2&-2&.&.&.&.&.&.&.&.&.&.&-\bar{B}&-B\\
X.49&.&\bar{B}&B&2&-2&.&.&.&.&.&.&.&.&.&.&-B&-\bar{B}\\
X.50&.&.&.&.&.&.&.&.&.&.&.&.&.&-1&-1&.&.\\
X.51&.&.&.&.&.&.&.&.&.&.&.&.&.&1&1&.&.\\
X.52&-1&.&.&-1&1&-1&1&.&.&.&.&.&.&.&.&.&.\\
X.53&-1&.&.&-1&1&1&-1&.&.&.&.&.&.&.&.&.&.\\
X.54&.&.&.&-2&-2&.&.&.&.&.&.&.&.&.&.&.&.\\
X.55&-1&.&.&.&.&.&.&.&.&.&.&.&.&.&.&.&.\\
X.56&-1&.&.&.&.&.&.&.&.&.&.&.&.&.&.&.&.\\
X.57&.&.&.&1&-1&1&-1&.&.&.&.&.&.&.&.&.&.\\
X.58&.&.&.&.&.&.&.&.&.&.&.&.&.&.&.&.&.\\
X.59&.&.&.&1&-1&-1&1&.&.&.&.&.&.&.&.&.&.\\
X.60&.&.&.&.&.&.&.&.&.&.&.&.&.&D&\bar{D}&.&.\\
X.61&.&.&.&.&.&.&.&.&.&.&.&.&.&\bar{D}&D&.&.\\
X.62&.&-1&-1&-2&2&.&.&.&.&.&.&.&.&.&.&1&1\\
X.63&.&-1&-1&2&-2&.&.&.&.&.&.&.&.&.&.&1&1\\
X.64&.&.&.&.&.&.&.&.&.&.&.&.&.&.&.&.&.\\
X.65&.&.&.&.&.&.&.&.&.&.&.&.&.&.&.&.&.\\
X.66&-2&.&.&-2&2&.&.&.&.&.&.&.&.&.&.&.&.\\
X.67&1&1&1&.&.&.&.&.&.&.&.&.&.&.&.&-1&-1\\
X.68&1&1&1&.&.&.&.&.&.&.&.&.&.&.&.&-1&-1\\
X.69&1&.&.&.&.&.&.&.&.&.&.&.&.&.&.&.&.\\
\end{array}
$$
} }
$
A = -\zeta(7)_7^4 - \zeta(7)_7^2 - \zeta(7)_7 - 1,
B = -\zeta(11)_{11}^9 - \zeta(11)_{11}^5 - \zeta(11)_{11}^4 - \zeta(11)_{11}^3 - \zeta(11)_{11}
- 1,
C = 2\zeta(5)_5^3 + 2\zeta(5)_5^2 + 1,
D = -2\zeta(8)_8^3 - 2\zeta(8)_8
$.
\end{ctab}

\newpage


\begin{ctab}\label{Fi_23ct_H} Character table of $H(\Fi_{23}) = H = \langle x_0,y_0,h_0 \rangle$
{
\renewcommand{\baselinestretch}{0.5}

{\tiny
\setlength{\arraycolsep}{0,3mm}
$$
\begin{array}{r|rrrrrrrrrrrrrrrrrrr}
2&18&18&17&17&18&18&16&9&8&7&4&12&11&11&12&10&10&10\\
3&9&9&6&6&4&4&3&7&9&7&7&3&2&2&1&2&2&1\\
5&2&2&1&1&1&1&.&1&.&.&.&.&1&1&.&.&.&.\\
7&1&1&1&1&.&.&.&1&.&.&.&.&.&.&.&.&.&.\\
11&1&1&1&1&.&.&.&.&.&.&.&.&.&.&.&.&.&.\\
13&1&1&.&.&.&.&.&.&.&.&.&.&.&.&.&.&.&.\\
\hline
&1a&2a&2b&2c&2d&2e&2f&3a&3b&3c&3d&4a&4b&4c&4d&4e&4f&4g\\
\hline
2P&1a&1a&1a&1a&1a&1a&1a&3a&3b&3c&3d&2d&2e&2e&2d&2f&2f&2e\\
\end{array}
$$
} }

\newpage
Character table of $H(\Fi_{23})$ (continued){
\renewcommand{\baselinestretch}{0.5}

{\tiny
\setlength{\arraycolsep}{0,3mm}
$$
\begin{array}{r|rrrrrrrrrrrrrrrrrrrrrrr}
2&4&9&8&8&8&7&8&8&4&8&8&9&9&5&5&8&8&7&7&7&6&6\\
3&1&7&9&5&5&7&6&6&7&4&4&3&3&5&5&3&3&3&3&3&3&3\\
5&2&1&.&1&1&.&.&.&.&.&.&.&.&.&.&.&.&.&.&.&.&.\\
7&.&1&.&.&.&.&.&.&.&.&.&.&.&.&.&.&.&.&.&.&.&.\\
11&.&.&.&.&.&.&.&.&.&.&.&.&.&.&.&.&.&.&.&.&.&.\\
13&.&.&.&.&.&.&.&.&.&.&.&.&.&.&.&.&.&.&.&.&.&.\\
\hline
&5a&6_{1}&6_{2}&6_{3}&6_{4}&6_{5}&6_{6}&6_{7}&6_{8}&6_{9}&6_{10}
&6_{11}&6_{12}&6_{13}&6_{14}&6_{15}&6_{16}&6_{17}&6_{18}&6_{19}&6_{20}
&6_{21}\\
\hline
2P&5a&3a&3b&3a&3a&3c&3b&3b&3d&3b&3b&3a&3a&3c&3c&3b&3b&3c&3c&3a&3c
&3c\\
\end{array}
$$
} }

\newpage
Character table of $H(\Fi_{23})$ (continued){
\renewcommand{\baselinestretch}{0.5}

{\tiny
\setlength{\arraycolsep}{0,3mm}
$$
\begin{array}{r|rrrrrrrrrrrrrrrrrrrrrrrrrrrrr}
2&5&5&4&4&2&7&7&7&6&6&3&2&1&4&3&3&4&4&2&2&7&7&5&6&6&6&5&5\\
3&3&3&3&3&1&1&1&.&.&.&4&4&3&1&1&1&.&.&.&.&2&2&3&2&2&2&2&2\\
5&.&.&.&.&.&.&.&.&.&.&.&.&.&2&1&1&1&1&.&.&.&.&.&.&.&.&.&.\\
7&.&.&.&.&1&.&.&.&.&.&.&.&.&.&.&.&.&.&.&.&.&.&.&.&.&.&.&.\\
11&.&.&.&.&.&.&.&.&.&.&.&.&.&.&.&.&.&.&1&1&.&.&.&.&.&.&.&.\\
13&.&.&.&.&.&.&.&.&.&.&.&.&.&.&.&.&.&.&.&.&.&.&.&.&.&.&.&.\\
\hline
&6_{22}&6_{23}&6_{24}&6_{25}&7a&8a&8b&8c&8d&8e&9a&9b&9c&10a&10b
&10c&10d&10e&11a&11b&12a&12b&12c&12d&12e&12f&12g&12h\\
\hline
2P&3c&3c&3d&3d&7a&4a&4a&4d&4f&4f&9a&9b&9c&5a&5a&5a&5a&5a&11b&11a
&6_{11}&6_{11}&6_{9}&6_{12}&6_{12}&6_{11}&6_{15}&6_{15}\\
\end{array}
$$
} }

\newpage
Character table of $H(\Fi_{23})$ (continued){
\renewcommand{\baselinestretch}{0.5}

{\tiny
\setlength{\arraycolsep}{0,3mm}
$$
\begin{array}{r|rrrrrrrrrrrrrrrrrrrrrrrrrrrr}
2&5&5&5&5&4&4&5&5&3&3&1&1&2&2&2&2&5&5&3&2&3&3&3&3&2&2&3\\
3&2&2&2&2&2&2&1&1&2&2&.&.&1&.&.&1&.&.&4&4&3&3&3&3&3&3&2\\
5&.&.&.&.&.&.&.&.&.&.&.&.&.&.&.&1&.&.&.&.&.&.&.&.&.&.&.\\
7&.&.&.&.&.&.&.&.&.&.&.&.&1&1&1&.&.&.&.&.&.&.&.&.&.&.&.\\
11&.&.&.&.&.&.&.&.&.&.&.&.&.&.&.&.&.&.&.&.&.&.&.&.&.&.&.\\
13&.&.&.&.&.&.&.&.&.&.&1&1&.&.&.&.&.&.&.&.&.&.&.&.&.&.&.\\
\hline
&12i&12j&12k&12l&12m&12n&12o&12p&12q&12r&13a&13b&14a&14b&14c&15a
&16a&16b&18a&18b&18c&18d&18e&18f&18g&18h&18i\\
\hline
2P&6_{16}&6_{18}&6_{18}&6_{16}&6_{20}&6_{20}&6_{12}&6_{9}&6_{24}
&6_{24}&13b&13a&7a&7a&7a&15a&8c&8c&9a&9b&9a&9a&9a&9a&9b&9b&9a\\
\end{array}
$$
} }

\newpage
Character table of $H(\Fi_{23})$ (continued){
\renewcommand{\baselinestretch}{0.5}

{\tiny
\setlength{\arraycolsep}{0,3mm}
$$
\begin{array}{r|rrrrrrrrrrrrrrrrrrrr}
2&3&1&3&3&1&2&2&2&2&2&2&4&4&1&1&2&2&2&1\\
3&2&3&.&.&1&.&.&.&.&.&.&1&1&.&.&1&1&1&1\\
5&.&.&1&1&.&.&.&.&.&.&.&.&.&.&.&1&1&1&.\\
7&.&.&.&.&1&.&.&.&.&.&.&.&.&.&.&.&.&.&1\\
11&.&.&.&.&.&1&1&1&1&1&1&.&.&.&.&.&.&.&.\\
13&.&.&.&.&.&.&.&.&.&.&.&.&.&1&1&.&.&.&.\\
\hline
&18j&18k&20a&20b&21a&22a&22b&22c&22d&22e&22f&24a&24b&26a&26b&30a
&30b&30c&42a\\
\hline
2P&9a&9c&10d&10d&21a&11a&11a&11b&11b&11b&11a&12a&12b&13a&13b&15a
&15a&15a&21a\\
\end{array}
$$
} }

\newpage
Character table of $H(\Fi_{23})$ (continued){
\renewcommand{\baselinestretch}{0.5}

{\tiny
\setlength{\arraycolsep}{0,3mm}
$$
\begin{array}{r|rrrrrrrrrrrrrrrrrrr}
2&18&18&17&17&18&18&16&9&8&7&4&12&11&11&12&10&10&10\\
3&9&9&6&6&4&4&3&7&9&7&7&3&2&2&1&2&2&1\\
5&2&2&1&1&1&1&.&1&.&.&.&.&1&1&.&.&.&.\\
7&1&1&1&1&.&.&.&1&.&.&.&.&.&.&.&.&.&.\\
11&1&1&1&1&.&.&.&.&.&.&.&.&.&.&.&.&.&.\\
13&1&1&.&.&.&.&.&.&.&.&.&.&.&.&.&.&.&.\\
\hline
&1a&2a&2b&2c&2d&2e&2f&3a&3b&3c&3d&4a&4b&4c&4d&4e&4f&4g\\
\hline
2P&1a&1a&1a&1a&1a&1a&1a&3a&3b&3c&3d&2d&2e&2e&2d&2f&2f&2e\\
\end{array}
$$
} }

Character table of $H(\Fi_{23})$ (continued){
\renewcommand{\baselinestretch}{0.5}

{\tiny
\setlength{\arraycolsep}{0,3mm}
$$
\begin{array}{r|rrrrrrrrrrrrrrrrrrrrrrrr}
2&4&9&8&8&8&7&8&8&4&8&8&9&9&5&5&8&8&7&7&7&6&6&5\\
3&1&7&9&5&5&7&6&6&7&4&4&3&3&5&5&3&3&3&3&3&3&3&3\\
5&2&1&.&1&1&.&.&.&.&.&.&.&.&.&.&.&.&.&.&.&.&.&.\\
7&.&1&.&.&.&.&.&.&.&.&.&.&.&.&.&.&.&.&.&.&.&.&.\\
11&.&.&.&.&.&.&.&.&.&.&.&.&.&.&.&.&.&.&.&.&.&.&.\\
13&.&.&.&.&.&.&.&.&.&.&.&.&.&.&.&.&.&.&.&.&.&.&.\\
\hline
&5a&6_{1}&6_{2}&6_{3}&6_{4}&6_{5}&6_{6}&6_{7}&6_{8}&6_{9}&6_{10}
&6_{11}&6_{12}&6_{13}&6_{14}&6_{15}&6_{16}&6_{17}&6_{18}&6_{19}&6_{20}
&6_{21}&6_{22}\\
\hline
2P&5a&3a&3b&3a&3a&3c&3b&3b&3d&3b&3b&3a&3a&3c&3c&3b&3b&3c&3c&3a&3c
&3c&3c\\
\end{array}
$$
} }

\newpage
Character table of $H(\Fi_{23})$ (continued){
\renewcommand{\baselinestretch}{0.5}

{\tiny
\setlength{\arraycolsep}{0,3mm}
$$
\begin{array}{r|rrrrrrrrrrrrrrrrrrrrrrrrrrrrr}
2&5&4&4&2&7&7&7&6&6&3&2&1&4&3&3&4&4&2&2&7&7&5&6&6&6&5&5&5\\
3&3&3&3&1&1&1&.&.&.&4&4&3&1&1&1&.&.&.&.&2&2&3&2&2&2&2&2&2\\
5&.&.&.&.&.&.&.&.&.&.&.&.&2&1&1&1&1&.&.&.&.&.&.&.&.&.&.&.\\
7&.&.&.&1&.&.&.&.&.&.&.&.&.&.&.&.&.&.&.&.&.&.&.&.&.&.&.&.\\
11&.&.&.&.&.&.&.&.&.&.&.&.&.&.&.&.&.&1&1&.&.&.&.&.&.&.&.&.\\
13&.&.&.&.&.&.&.&.&.&.&.&.&.&.&.&.&.&.&.&.&.&.&.&.&.&.&.&.\\
\hline
&6_{23}&6_{24}&6_{25}&7a&8a&8b&8c&8d&8e&9a&9b&9c&10a&10b&10c&10d
&10e&11a&11b&12a&12b&12c&12d&12e&12f&12g&12h&12i\\
\hline
2P&3c&3d&3d&7a&4a&4a&4d&4f&4f&9a&9b&9c&5a&5a&5a&5a&5a&11b&11a
&6_{11}&6_{11}&6_{9}&6_{12}&6_{12}&6_{11}&6_{15}&6_{15}&6_{16}\\
\end{array}
$$
} }

Character table of $H(\Fi_{23})$ (continued){
\renewcommand{\baselinestretch}{0.5}

{\tiny
\setlength{\arraycolsep}{0,3mm}
$$
\begin{array}{r|rrrrrrrrrrrrrrrrrrrrrrrrrrr}
2&5&5&5&4&4&5&5&3&3&1&1&2&2&2&2&5&5&3&2&3&3&3&3&2&2&3\\
3&2&2&2&2&2&1&1&2&2&.&.&1&.&.&1&.&.&4&4&3&3&3&3&3&3&2\\
5&.&.&.&.&.&.&.&.&.&.&.&.&.&.&1&.&.&.&.&.&.&.&.&.&.&.\\
7&.&.&.&.&.&.&.&.&.&.&.&1&1&1&.&.&.&.&.&.&.&.&.&.&.&.\\
11&.&.&.&.&.&.&.&.&.&.&.&.&.&.&.&.&.&.&.&.&.&.&.&.&.&.\\
13&.&.&.&.&.&.&.&.&.&1&1&.&.&.&.&.&.&.&.&.&.&.&.&.&.&.\\
\hline
&12j&12k&12l&12m&12n&12o&12p&12q&12r&13a&13b&14a&14b&14c&15a&16a
&16b&18a&18b&18c&18d&18e&18f&18g&18h&18i\\
\hline
2P&6_{18}&6_{18}&6_{16}&6_{20}&6_{20}&6_{12}&6_{9}&6_{24}&6_{24}
&13b&13a&7a&7a&7a&15a&8c&8c&9a&9b&9a&9a&9a&9a&9b&9b&9a\\
\end{array}
$$
} }

\newpage
Character table of $H(\Fi_{23})$ (continued){
\renewcommand{\baselinestretch}{0.5}

{\tiny
\setlength{\arraycolsep}{0,3mm}
$$
\begin{array}{r|rrrrrrrrrrrrrrrrrrrr}
2&3&1&3&3&1&2&2&2&2&2&2&4&4&1&1&2&2&2&1\\
3&2&3&.&.&1&.&.&.&.&.&.&1&1&.&.&1&1&1&1\\
5&.&.&1&1&.&.&.&.&.&.&.&.&.&.&.&1&1&1&.\\
7&.&.&.&.&1&.&.&.&.&.&.&.&.&.&.&.&.&.&1\\
11&.&.&.&.&.&1&1&1&1&1&1&.&.&.&.&.&.&.&.\\
13&.&.&.&.&.&.&.&.&.&.&.&.&.&1&1&.&.&.&.\\
\hline
&18j&18k&20a&20b&21a&22a&22b&22c&22d&22e&22f&24a&24b&26a&26b&30a
&30b&30c&42a\\
\hline
2P&9a&9c&10d&10d&21a&11a&11a&11b&11b&11b&11a&12a&12b&13a&13b&15a
&15a&15a&21a\\
\end{array}
$$
} } where
$
A = -\zeta(11)_{11}^9 - \zeta(11)_{11}^5 - \zeta(11)_{11}^4 - \zeta(11)_{11}^3 - \zeta(11)_{11}
- 1,
B = 6\zeta(3)_3 + 3,
C = \zeta(13)_{13}^{11} + \zeta(13)_{13}^8 + \zeta(13)_{13}^7 + \zeta(13)_{13}^6 +
\zeta(13)_{13}^5 + \zeta(13)_{13}^2 + 1,
D = -\zeta(13)_{13}^{11} - \zeta(13)_{13}^8 - \zeta(13)_{13}^7 - \zeta(13)_{13}^6 -
\zeta(13)_{13}^5 - \zeta(13)_{13}^2,
E = -2\zeta(8)_8^3 - 2\zeta(8)_8
$.
\end{ctab}

\newpage


\begin{ctab}\label{Fi_23ct_H_2} Character table of $H(2\Fi_{22}) = H_2 = \langle p_2,q_2,h_2 \rangle$
{
\renewcommand{\baselinestretch}{0.5}

{\tiny
\setlength{\arraycolsep}{0,3mm}
$$
\begin{array}{r|rrrrrrrrrrrrrrrrrrrrrrr}
2&18&18&18&18&17&17&16&16&16&16&17&17&15&15&16&11&11&13&13&13&13
&12\\
3&4&4&4&4&4&4&2&2&2&2&1&1&2&2&1&2&2&1&1&1&1&1\\
5&1&1&1&1&1&1&.&.&.&.&.&.&.&.&.&1&1&.&.&.&.&.\\
\hline
&1a&2a&2b&2c&2d&2e&2f&2g&2h&2i&2j&2k&2l&2m&2n&2o&2p&2q&2r&2s&2t
&2u\\
\hline
2P&1a&1a&1a&1a&1a&1a&1a&1a&1a&1a&1a&1a&1a&1a&1a&1a&1a&1a&1a&1a&1a
&1a\\
3P&1a&2a&2b&2c&2d&2e&2f&2g&2h&2i&2j&2k&2l&2m&2n&2o&2p&2q&2r&2s&2t
&2u\\
5P&1a&2a&2b&2c&2d&2e&2f&2g&2h&2i&2j&2k&2l&2m&2n&2o&2p&2q&2r&2s&2t
&2u\\
\hline
X.1&1&1&1&1&1&1&1&1&1&1&1&1&1&1&1&1&1&1&1&1&1&1\\
\end{array}
$$
} }

\newpage
Character table of $H(2\Fi_{22})$ (continued){
\renewcommand{\baselinestretch}{0.5}

{\tiny
\setlength{\arraycolsep}{0,3mm}
$$
\begin{array}{r|rrrrrrrrrrrrrrrrrrrrrrrrr}
2&12&13&10&8&9&6&12&11&11&12&11&11&11&11&11&11&10&10&10&10&10&10
&10&10\\
3&1&.&1&4&3&3&3&2&2&1&1&1&1&1&1&1&1&1&1&1&1&1&1&1\\
5&.&.&.&.&.&.&.&1&1&.&.&.&.&.&.&.&.&.&.&.&.&.&.&.\\
\hline
&2v&2w&2x&3a&3b&3c&4_{1}&4_{2}&4_{3}&4_{4}&4_{5}&4_{6}&4_{7}
&4_{8}&4_{9}&4_{10}&4_{11}&4_{12}&4_{13}&4_{14}&4_{15}&4_{16}&4_{17}
&4_{18}\\
\hline
2P&1a&1a&1a&3a&3b&3c&2a&2b&2b&2a&2j&2j&2j&2k&2k&2j&2f&2j&2k&2g&2j
&2g&2b&2f\\
3P&2v&2w&2x&1a&1a&1a&4_{1}&4_{2}&4_{3}&4_{4}&4_{5}&4_{6}&4_{7}
&4_{8}&4_{9}&4_{10}&4_{11}&4_{12}&4_{13}&4_{14}&4_{15}&4_{16}&4_{17}
&4_{18}\\
5P&2v&2w&2x&3a&3b&3c&4_{1}&4_{2}&4_{3}&4_{4}&4_{5}&4_{6}&4_{7}
&4_{8}&4_{9}&4_{10}&4_{11}&4_{12}&4_{13}&4_{14}&4_{15}&4_{16}&4_{17}
&4_{18}\\
\hline
X.1&1&1&1&1&1&1&1&1&1&1&1&1&1&1&1&1&1&1&1&1&1&1&1&1\\
\end{array}
$$
} }

\newpage
Character table of $H(2\Fi_{22})$ (continued){
\renewcommand{\baselinestretch}{0.5}

{\tiny
\setlength{\arraycolsep}{0,3mm}
$$
\begin{array}{r|rrrrrrrrrrrrrrrrrrrrrrrrrrrr}
2&11&11&9&9&9&9&9&9&10&10&9&9&9&9&9&9&9&9&9&9&9&8&8&4&8&8&8\\
3&.&.&1&1&1&1&1&1&.&.&.&.&.&.&.&.&.&.&.&.&.&.&.&.&4&4&4\\
5&.&.&.&.&.&.&.&.&.&.&.&.&.&.&.&.&.&.&.&.&.&.&.&1&.&.&.\\
\hline
&4_{19}&4_{20}&4_{21}&4_{22}&4_{23}&4_{24}&4_{25}&4_{26}&4_{27}
&4_{28}&4_{29}&4_{30}&4_{31}&4_{32}&4_{33}&4_{34}&4_{35}&4_{36}&4_{37}
&4_{38}&4_{39}&4_{40}&4_{41}&5a&6_{1}&6_{2}&6_{3}\\
\hline
2P&2k&2k&2r&2q&2l&2r&2l&2q&2n&2n&2q&2r&2q&2k&2j&2q&2q&2r&2q&2q&2j
&2q&2q&5a&3a&3a&3a\\
3P&4_{19}&4_{20}&4_{21}&4_{22}&4_{23}&4_{24}&4_{25}&4_{26}&4_{27}
&4_{28}&4_{29}&4_{30}&4_{31}&4_{32}&4_{33}&4_{34}&4_{35}&4_{36}&4_{37}
&4_{38}&4_{39}&4_{40}&4_{41}&5a&2d&2e&2e\\
5P&4_{19}&4_{20}&4_{21}&4_{22}&4_{23}&4_{24}&4_{25}&4_{26}&4_{27}
&4_{28}&4_{29}&4_{30}&4_{31}&4_{32}&4_{33}&4_{34}&4_{35}&4_{36}&4_{37}
&4_{38}&4_{39}&4_{40}&4_{41}&1a&6_{5}&6_{3}&6_{2}\\
\hline
X.1&1&1&1&1&1&1&1&1&1&1&1&1&1&1&1&1&1&1&1&1&1&1&1&1&1&1&1\\
\end{array}
$$
} }

\newpage
Character table of $H(2\Fi_{22})$ (continued){
\renewcommand{\baselinestretch}{0.5}

{\tiny
\setlength{\arraycolsep}{0,3mm}
$$
\begin{array}{r|rrrrrrrrrrrrrrrrrrrrrrrrrrr}
2&8&8&8&8&9&9&9&8&8&8&8&8&8&8&8&8&8&6&6&6&7&7&7&7&5&5\\
3&4&4&4&4&3&3&3&3&3&2&2&2&2&2&2&2&2&3&3&3&2&2&2&2&3&3\\
5&.&.&.&.&.&.&.&.&.&.&.&.&.&.&.&.&.&.&.&.&.&.&.&.&.&.\\
\hline
&6_{4}&6_{5}&6_{6}&6_{7}&6_{8}&6_{9}&6_{10}&6_{11}&6_{12}&6_{13}
&6_{14}&6_{15}&6_{16}&6_{17}&6_{18}&6_{19}&6_{20}&6_{21}&6_{22}&6_{23}
&6_{24}&6_{25}&6_{26}&6_{27}&6_{28}&6_{29}\\
\hline
2P&3a&3a&3a&3a&3b&3b&3b&3b&3b&3a&3a&3a&3a&3a&3a&3a&3a&3c&3c&3c&3b
&3b&3b&3b&3c&3c\\
3P&2a&2d&2c&2b&2a&2b&2c&2e&2d&2f&2l&2g&2l&2m&2m&2h&2i&2b&2c&2a&2f
&2h&2g&2i&2d&2e\\
5P&6_{4}&6_{1}&6_{6}&6_{7}&6_{8}&6_{9}&6_{10}&6_{11}&6_{12}&6_{13}
&6_{16}&6_{15}&6_{14}&6_{18}&6_{17}&6_{19}&6_{20}&6_{21}&6_{22}&6_{23}
&6_{24}&6_{25}&6_{26}&6_{27}&6_{28}&6_{29}\\
\hline
X.1&1&1&1&1&1&1&1&1&1&1&1&1&1&1&1&1&1&1&1&1&1&1&1&1&1&1\\
\end{array}
$$
} }

\newpage
Character table of $H(2\Fi_{22})$ (continued){
\renewcommand{\baselinestretch}{0.5}

{\tiny
\setlength{\arraycolsep}{0,3mm}
$$
\begin{array}{r|rrrrrrrrrrrrrrrrrrrrrrrrrrrrr}
2&8&8&6&6&6&6&7&5&5&5&5&5&5&5&5&6&6&6&6&4&4&5&5&5&7&7&7&7\\
3&1&1&2&2&2&2&1&2&2&2&2&2&2&2&2&1&1&1&1&2&2&1&1&1&1&1&.&.\\
5&.&.&.&.&.&.&.&.&.&.&.&.&.&.&.&.&.&.&.&.&.&.&.&.&.&.&.&.\\
\hline
&6_{30}&6_{31}&6_{32}&6_{33}&6_{34}&6_{35}&6_{36}&6_{37}&6_{38}
&6_{39}&6_{40}&6_{41}&6_{42}&6_{43}&6_{44}&6_{45}&6_{46}&6_{47}&6_{48}
&6_{49}&6_{50}&6_{51}&6_{52}&6_{53}&8a&8b&8c&8d\\
\hline
2P&3b&3b&3b&3b&3b&3b&3b&3c&3c&3c&3c&3c&3c&3c&3c&3c&3c&3c&3c&3c&3c
&3c&3b&3c&4_{1}&4_{1}&4_{4}&4_{20}\\
3P&2j&2k&2m&2o&2p&2l&2n&2f&2m&2l&2i&2m&2g&2l&2h&2q&2r&2s&2t&2o&2p
&2u&2x&2v&8a&8b&8c&8d\\
5P&6_{30}&6_{31}&6_{32}&6_{33}&6_{34}&6_{35}&6_{36}&6_{37}&6_{41}
&6_{43}&6_{40}&6_{38}&6_{42}&6_{39}&6_{44}&6_{45}&6_{46}&6_{47}&6_{48}
&6_{49}&6_{50}&6_{51}&6_{52}&6_{53}&8a&8b&8c&8d\\
\hline
X.1&1&1&1&1&1&1&1&1&1&1&1&1&1&1&1&1&1&1&1&1&1&1&1&1&1&1&1&1\\
\end{array}
$$
} }

\newpage
Character table of $H(2\Fi_{22})$ (continued){
\renewcommand{\baselinestretch}{0.5}

{\tiny
\setlength{\arraycolsep}{0,3mm}
$$
\begin{array}{r|rrrrrrrrrrrrrrrrrrrrrrrrrrrr}
2&7&7&7&6&6&6&6&6&6&6&3&4&4&4&3&3&3&3&7&7&5&6&6&6&6&6&6\\
3&.&.&.&.&.&.&.&.&.&.&2&.&.&.&.&.&.&.&2&2&3&2&2&2&1&1&1\\
5&.&.&.&.&.&.&.&.&.&.&.&1&1&1&1&1&1&1&.&.&.&.&.&.&.&.&.\\
\hline
&8e&8f&8g&8h&8i&8j&8k&8l&8m&8n&9a&10a&10b&10c&10d&10e&10f&10g
&12_{1}&12_{2}&12_{3}&12_{4}&12_{5}&12_{6}&12_{7}&12_{8}&12_{9}\\
\hline
2P&4_{20}&4_{20}&4_{20}&4_{11}&4_{11}&4_{16}&4_{19}&4_{27}&4_{27}
&4_{16}&9a&5a&5a&5a&5a&5a&5a&5a&6_{8}&6_{8}&6_{4}&6_{9}&6_{9}&6_{8}
&6_{30}&6_{30}&6_{30}\\
3P&8e&8f&8g&8h&8i&8j&8k&8l&8m&8n&3a&10a&10b&10c&10d&10e&10f&10g
&4_{1}&4_{1}&4_{1}&4_{3}&4_{2}&4_{1}&4_{5}&4_{6}&4_{7}\\
5P&8e&8f&8g&8h&8i&8j&8k&8l&8m&8n&9a&2b&2a&2c&2o&2d&2p&2e&12_{1}
&12_{2}&12_{3}&12_{4}&12_{5}&12_{6}&12_{7}&12_{8}&12_{9}\\
\hline
X.1&1&1&1&1&1&1&1&1&1&1&1&1&1&1&1&1&1&1&1&1&1&1&1&1&1&1&1\\
\end{array}
$$
} }

\newpage
Character table of $H(2\Fi_{22})$ (continued){
\renewcommand{\baselinestretch}{0.5}

{\tiny
\setlength{\arraycolsep}{0,3mm}
$$
\begin{array}{r|rrrrrrrrrrrrrrrrrrrrrr}
2&6&6&6&4&4&5&5&5&5&5&5&5&5&5&5&5&5&5&4&4&4\\
3&1&1&1&2&2&1&1&1&1&1&1&1&1&1&1&1&1&1&1&1&1\\
5&.&.&.&.&.&.&.&.&.&.&.&.&.&.&.&.&.&.&.&.&.\\
\hline
&12_{10}&12_{11}&12_{12}&12_{13}&12_{14}&12_{15}&12_{16}&12_{17}
&12_{18}&12_{19}&12_{20}&12_{21}&12_{22}&12_{23}&12_{24}&12_{25}&12_{26}
&12_{27}&12_{28}&12_{29}&12_{30}\\
\hline
2P&6_{31}&6_{31}&6_{30}&6_{21}&6_{21}&6_{13}&6_{30}&6_{14}&6_{31}
&6_{14}&6_{15}&6_{30}&6_{4}&6_{16}&6_{16}&6_{9}&6_{15}&6_{13}&6_{46}
&6_{45}&6_{46}\\
3P&4_{8}&4_{9}&4_{10}&4_{2}&4_{3}&4_{11}&4_{12}&4_{23}&4_{13}
&4_{25}&4_{14}&4_{15}&4_{4}&4_{25}&4_{23}&4_{17}&4_{16}&4_{18}&4_{21}
&4_{22}&4_{24}\\
5P&12_{10}&12_{11}&12_{12}&12_{13}&12_{14}&12_{15}&12_{16}&12_{24}
&12_{18}&12_{23}&12_{20}&12_{21}&12_{22}&12_{19}&12_{17}&12_{25}&12_{26}
&12_{27}&12_{28}&12_{29}&12_{30}\\
\hline
X.1&1&1&1&1&1&1&1&1&1&1&1&1&1&1&1&1&1&1&1&1&1\\
\end{array}
$$
} }

\newpage
Character table of $H(2\Fi_{22})$ (continued){
\renewcommand{\baselinestretch}{0.5}

{\tiny
\setlength{\arraycolsep}{0,3mm}
$$
\begin{array}{r|rrrrrrrrrrrrrrr}
2&4&5&5&3&3&3&3&3&3&3&3&3&4&4\\
3&1&.&.&2&2&2&2&2&2&2&.&.&1&1\\
5&.&.&.&.&.&.&.&.&.&.&1&1&.&.\\
\hline
&12_{31}&16a&16b&18a&18b&18c&18d&18e&18f&18g&20a&20b&24a&24b\\
\hline
2P&6_{45}&8c&8c&9a&9a&9a&9a&9a&9a&9a&10a&10a&12_{1}&12_{2}\\
3P&4_{26}&16a&16b&6_{2}&6_{1}&6_{3}&6_{6}&6_{5}&6_{7}&6_{4}&20a
&20b&8a&8b\\
5P&12_{31}&16b&16a&18c&18e&18a&18d&18b&18f&18g&4_{2}&4_{3}&24a&24b\\
\hline
X.1&1&1&1&1&1&1&1&1&1&1&1&1&1&1\\
\end{array}
$$
} }

\newpage
Character table of $H(2\Fi_{22})$ (continued){
\renewcommand{\baselinestretch}{0.5}

{\tiny
\setlength{\arraycolsep}{0,3mm}
$$
\begin{array}{r|rrrrrrrrrrrrrrrrrrrr}
2&18&18&18&18&17&17&16&16&16&16&17&17&15&15&16&11&11&13&13\\
3&4&4&4&4&4&4&2&2&2&2&1&1&2&2&1&2&2&1&1\\
5&1&1&1&1&1&1&.&.&.&.&.&.&.&.&.&1&1&.&.\\
\hline
&1a&2a&2b&2c&2d&2e&2f&2g&2h&2i&2j&2k&2l&2m&2n&2o&2p&2q&2r\\
\hline
2P&1a&1a&1a&1a&1a&1a&1a&1a&1a&1a&1a&1a&1a&1a&1a&1a&1a&1a&1a\\
3P&1a&2a&2b&2c&2d&2e&2f&2g&2h&2i&2j&2k&2l&2m&2n&2o&2p&2q&2r\\
5P&1a&2a&2b&2c&2d&2e&2f&2g&2h&2i&2j&2k&2l&2m&2n&2o&2p&2q&2r\\
\hline
X.96&320&320&-320&-320&.&.&-64&64&64&-64&64&-64&.&.&.&.&.&.&.\\
\end{array}
$$
} }

\newpage
Character table of $H(2\Fi_{22})$ (continued){
\renewcommand{\baselinestretch}{0.5}

{\tiny
\setlength{\arraycolsep}{0,3mm}
$$
\begin{array}{r|rrrrrrrrrrrrrrrrrrrrrrrr}
2&13&13&12&12&13&10&8&9&6&12&11&11&12&11&11&11&11&11&11&10&10&10
&10\\
3&1&1&1&1&.&1&4&3&3&3&2&2&1&1&1&1&1&1&1&1&1&1&1\\
5&.&.&.&.&.&.&.&.&.&.&1&1&.&.&.&.&.&.&.&.&.&.&.\\
\hline
&2s&2t&2u&2v&2w&2x&3a&3b&3c&4_{1}&4_{2}&4_{3}&4_{4}&4_{5}&4_{6}
&4_{7}&4_{8}&4_{9}&4_{10}&4_{11}&4_{12}&4_{13}&4_{14}\\
\hline
2P&1a&1a&1a&1a&1a&1a&3a&3b&3c&2a&2b&2b&2a&2j&2j&2j&2k&2k&2j&2f&2j
&2k&2g\\
3P&2s&2t&2u&2v&2w&2x&1a&1a&1a&4_{1}&4_{2}&4_{3}&4_{4}&4_{5}&4_{6}
&4_{7}&4_{8}&4_{9}&4_{10}&4_{11}&4_{12}&4_{13}&4_{14}\\
5P&2s&2t&2u&2v&2w&2x&3a&3b&3c&4_{1}&4_{2}&4_{3}&4_{4}&4_{5}&4_{6}
&4_{7}&4_{8}&4_{9}&4_{10}&4_{11}&4_{12}&4_{13}&4_{14}\\
\hline
X.96&.&.&.&.&.&.&-4&14&-4&.&.&.&.&.&-16&16&.&.&.&.&.&.&.\\
\end{array}
$$
} }

\newpage
Character table of $H(2\Fi_{22})$ (continued){
\renewcommand{\baselinestretch}{0.5}

{\tiny
\setlength{\arraycolsep}{0,3mm}
$$
\begin{array}{r|rrrrrrrrrrrrrrrrrrrrrrrrrrr}
2&10&10&10&10&11&11&9&9&9&9&9&9&10&10&9&9&9&9&9&9&9&9&9&9&9&8\\
3&1&1&1&1&.&.&1&1&1&1&1&1&.&.&.&.&.&.&.&.&.&.&.&.&.&.\\
5&.&.&.&.&.&.&.&.&.&.&.&.&.&.&.&.&.&.&.&.&.&.&.&.&.&.\\
\hline
&4_{15}&4_{16}&4_{17}&4_{18}&4_{19}&4_{20}&4_{21}&4_{22}&4_{23}
&4_{24}&4_{25}&4_{26}&4_{27}&4_{28}&4_{29}&4_{30}&4_{31}&4_{32}&4_{33}
&4_{34}&4_{35}&4_{36}&4_{37}&4_{38}&4_{39}&4_{40}\\
\hline
2P&2j&2g&2b&2f&2k&2k&2r&2q&2l&2r&2l&2q&2n&2n&2q&2r&2q&2k&2j&2q&2q
&2r&2q&2q&2j&2q\\
3P&4_{15}&4_{16}&4_{17}&4_{18}&4_{19}&4_{20}&4_{21}&4_{22}&4_{23}
&4_{24}&4_{25}&4_{26}&4_{27}&4_{28}&4_{29}&4_{30}&4_{31}&4_{32}&4_{33}
&4_{34}&4_{35}&4_{36}&4_{37}&4_{38}&4_{39}&4_{40}\\
5P&4_{15}&4_{16}&4_{17}&4_{18}&4_{19}&4_{20}&4_{21}&4_{22}&4_{23}
&4_{24}&4_{25}&4_{26}&4_{27}&4_{28}&4_{29}&4_{30}&4_{31}&4_{32}&4_{33}
&4_{34}&4_{35}&4_{36}&4_{37}&4_{38}&4_{39}&4_{40}\\
\hline
X.96&.&.&.&.&.&.&.&.&.&.&.&.&.&.&.&.&.&.&.&.&.&.&.&.&.&.\\
\end{array}
$$
} }

\newpage
Character table of $H(2\Fi_{22})$ (continued){
\renewcommand{\baselinestretch}{0.5}

{\tiny
\setlength{\arraycolsep}{0,3mm}
$$
\begin{array}{r|rrrrrrrrrrrrrrrrrrrrrrrrrr}
2&8&4&8&8&8&8&8&8&8&9&9&9&8&8&8&8&8&8&8&8&8&8&6&6&6\\
3&.&.&4&4&4&4&4&4&4&3&3&3&3&3&2&2&2&2&2&2&2&2&3&3&3\\
5&.&1&.&.&.&.&.&.&.&.&.&.&.&.&.&.&.&.&.&.&.&.&.&.&.\\
\hline
&4_{41}&5a&6_{1}&6_{2}&6_{3}&6_{4}&6_{5}&6_{6}&6_{7}&6_{8}&6_{9}
&6_{10}&6_{11}&6_{12}&6_{13}&6_{14}&6_{15}&6_{16}&6_{17}&6_{18}&6_{19}
&6_{20}&6_{21}&6_{22}&6_{23}\\
\hline
2P&2q&5a&3a&3a&3a&3a&3a&3a&3a&3b&3b&3b&3b&3b&3a&3a&3a&3a&3a&3a&3a
&3a&3c&3c&3c\\
3P&4_{41}&5a&2d&2e&2e&2a&2d&2c&2b&2a&2b&2c&2e&2d&2f&2l&2g&2l&2m&2m
&2h&2i&2b&2c&2a\\
5P&4_{41}&1a&6_{5}&6_{3}&6_{2}&6_{4}&6_{1}&6_{6}&6_{7}&6_{8}&6_{9}
&6_{10}&6_{11}&6_{12}&6_{13}&6_{16}&6_{15}&6_{14}&6_{18}&6_{17}&6_{19}
&6_{20}&6_{21}&6_{22}&6_{23}\\
\hline
X.96&.&.&F&F&\bar{F}&-4&\bar{F}&4&4&14&-14&-14&.&.&-4&\bar{G}&4&G&G
&\bar{G}&4&-4&4&4&-4\\
\end{array}
$$
} }

\newpage
Character table of $H(2\Fi_{22})$ (continued){
\renewcommand{\baselinestretch}{0.5}

{\tiny
\setlength{\arraycolsep}{0,3mm}
$$
\begin{array}{r|rrrrrrrrrrrrrrrrrrrrrrrrrrrr}
2&7&7&7&7&5&5&8&8&6&6&6&6&7&5&5&5&5&5&5&5&5&6&6&6&6&4&4\\
3&2&2&2&2&3&3&1&1&2&2&2&2&1&2&2&2&2&2&2&2&2&1&1&1&1&2&2\\
5&.&.&.&.&.&.&.&.&.&.&.&.&.&.&.&.&.&.&.&.&.&.&.&.&.&.&.\\
\hline
&6_{24}&6_{25}&6_{26}&6_{27}&6_{28}&6_{29}&6_{30}&6_{31}&6_{32}
&6_{33}&6_{34}&6_{35}&6_{36}&6_{37}&6_{38}&6_{39}&6_{40}&6_{41}&6_{42}
&6_{43}&6_{44}&6_{45}&6_{46}&6_{47}&6_{48}&6_{49}&6_{50}\\
\hline
2P&3b&3b&3b&3b&3c&3c&3b&3b&3b&3b&3b&3b&3b&3c&3c&3c&3c&3c&3c&3c&3c
&3c&3c&3c&3c&3c&3c\\
3P&2f&2h&2g&2i&2d&2e&2j&2k&2m&2o&2p&2l&2n&2f&2m&2l&2i&2m&2g&2l&2h
&2q&2r&2s&2t&2o&2p\\
5P&6_{24}&6_{25}&6_{26}&6_{27}&6_{28}&6_{29}&6_{30}&6_{31}&6_{32}
&6_{33}&6_{34}&6_{35}&6_{36}&6_{37}&6_{41}&6_{43}&6_{40}&6_{38}&6_{42}
&6_{39}&6_{44}&6_{45}&6_{46}&6_{47}&6_{48}&6_{49}&6_{50}\\
\hline
X.96&2&-2&-2&2&.&.&-2&2&.&.&.&.&.&2&B&\bar{B}&2&\bar{B}&-2&B&-2&.&.
&.&.&.&.\\
\end{array}
$$
} }

\newpage
Character table of $H(2\Fi_{22})$ (continued){
\renewcommand{\baselinestretch}{0.5}

{\tiny
\setlength{\arraycolsep}{0,3mm}
$$
\begin{array}{r|rrrrrrrrrrrrrrrrrrrrrrrrrrrr}
2&5&5&5&7&7&7&7&7&7&7&6&6&6&6&6&6&6&3&4&4&4&3&3&3&3&7&7\\
3&1&1&1&1&1&.&.&.&.&.&.&.&.&.&.&.&.&2&.&.&.&.&.&.&.&2&2\\
5&.&.&.&.&.&.&.&.&.&.&.&.&.&.&.&.&.&.&1&1&1&1&1&1&1&.&.\\
\hline
&6_{51}&6_{52}&6_{53}&8a&8b&8c&8d&8e&8f&8g&8h&8i&8j&8k&8l&8m&8n
&9a&10a&10b&10c&10d&10e&10f&10g&12_{1}&12_{2}\\
\hline
2P&3c&3b&3c&4_{1}&4_{1}&4_{4}&4_{20}&4_{20}&4_{20}&4_{20}&4_{11}
&4_{11}&4_{16}&4_{19}&4_{27}&4_{27}&4_{16}&9a&5a&5a&5a&5a&5a&5a&5a
&6_{8}&6_{8}\\
3P&2u&2x&2v&8a&8b&8c&8d&8e&8f&8g&8h&8i&8j&8k&8l&8m&8n&3a&10a&10b
&10c&10d&10e&10f&10g&4_{1}&4_{1}\\
5P&6_{51}&6_{52}&6_{53}&8a&8b&8c&8d&8e&8f&8g&8h&8i&8j&8k&8l&8m&8n
&9a&2b&2a&2c&2o&2d&2p&2e&12_{1}&12_{2}\\
\hline
X.96&.&.&.&.&.&.&.&.&.&.&.&.&.&.&.&.&.&-1&.&.&.&.&.&.&.&6&-6\\
\end{array}
$$
} }

\newpage
Character table of $H(2\Fi_{22})$ (continued){
\renewcommand{\baselinestretch}{0.5}

{\tiny
\setlength{\arraycolsep}{0,3mm}
$$
\begin{array}{r|rrrrrrrrrrrrrrrrrrrrrrr}
2&5&6&6&6&6&6&6&6&6&6&4&4&5&5&5&5&5&5&5&5&5&5\\
3&3&2&2&2&1&1&1&1&1&1&2&2&1&1&1&1&1&1&1&1&1&1\\
5&.&.&.&.&.&.&.&.&.&.&.&.&.&.&.&.&.&.&.&.&.&.\\
\hline
&12_{3}&12_{4}&12_{5}&12_{6}&12_{7}&12_{8}&12_{9}&12_{10}
&12_{11}&12_{12}&12_{13}&12_{14}&12_{15}&12_{16}&12_{17}&12_{18}&12_{19}
&12_{20}&12_{21}&12_{22}&12_{23}&12_{24}\\
\hline
2P&6_{4}&6_{9}&6_{9}&6_{8}&6_{30}&6_{30}&6_{30}&6_{31}&6_{31}
&6_{30}&6_{21}&6_{21}&6_{13}&6_{30}&6_{14}&6_{31}&6_{14}&6_{15}&6_{30}
&6_{4}&6_{16}&6_{16}\\
3P&4_{1}&4_{3}&4_{2}&4_{1}&4_{5}&4_{6}&4_{7}&4_{8}&4_{9}&4_{10}
&4_{2}&4_{3}&4_{11}&4_{12}&4_{23}&4_{13}&4_{25}&4_{14}&4_{15}&4_{4}
&4_{25}&4_{23}\\
5P&12_{3}&12_{4}&12_{5}&12_{6}&12_{7}&12_{8}&12_{9}&12_{10}
&12_{11}&12_{12}&12_{13}&12_{14}&12_{15}&12_{16}&12_{24}&12_{18}&12_{23}
&12_{20}&12_{21}&12_{22}&12_{19}&12_{17}\\
\hline
X.96&.&.&.&.&.&2&-2&.&.&.&.&.&.&.&.&.&.&.&.&.&.&.\\
\end{array}
$$
} }

\newpage
Character table of $H(2\Fi_{22})$ (continued){
\renewcommand{\baselinestretch}{0.5}

{\tiny
\setlength{\arraycolsep}{0,3mm}
$$
\begin{array}{r|rrrrrrrrrrrrrrrrrrrrr}
2&5&5&5&4&4&4&4&5&5&3&3&3&3&3&3&3&3&3&4&4\\
3&1&1&1&1&1&1&1&.&.&2&2&2&2&2&2&2&.&.&1&1\\
5&.&.&.&.&.&.&.&.&.&.&.&.&.&.&.&.&1&1&.&.\\
\hline
&12_{25}&12_{26}&12_{27}&12_{28}&12_{29}&12_{30}&12_{31}&16a&16b
&18a&18b&18c&18d&18e&18f&18g&20a&20b&24a&24b\\
\hline
2P&6_{9}&6_{15}&6_{13}&6_{46}&6_{45}&6_{46}&6_{45}&8c&8c&9a&9a&9a
&9a&9a&9a&9a&10a&10a&12_{1}&12_{2}\\
3P&4_{17}&4_{16}&4_{18}&4_{21}&4_{22}&4_{24}&4_{26}&16a&16b&6_{2}
&6_{1}&6_{3}&6_{6}&6_{5}&6_{7}&6_{4}&20a&20b&8a&8b\\
5P&12_{25}&12_{26}&12_{27}&12_{28}&12_{29}&12_{30}&12_{31}&16b&16a
&18c&18e&18a&18d&18b&18f&18g&4_{2}&4_{3}&24a&24b\\
\hline
X.96&.&.&.&.&.&.&.&.&.&D&D&\bar{D}&1&\bar{D}&1&-1&.&.&.&.\\
\end{array}
$$
} } where
{\scriptsize $
A = -12\zeta(3)_3 - 6,
B = 4\zeta(3)_3 + 2,
C = 2\zeta(3)_3 + 2,
D = 2\zeta(3)_3 + 1,
E = -2\zeta(8)_8^3 - 2\zeta(8)_8,
F = -24\zeta(3)_3 - 12,
G = -8\zeta(3)_3 - 4,
H = -16\zeta(3)_3 - 8,
I = -48\zeta(3)_3 - 24,
J = 4\zeta(3)_3 + 1,
K = -36\zeta(3)_3 - 18
$}.
\end{ctab}

\newpage


\begin{ctab}\label{Fi_23ct_D_2} Character table of $D(2\Fi_{22}) = D_2 = \langle p_2,q_2 \rangle \cong \langle p_1, q_1 \rangle$
{
\renewcommand{\baselinestretch}{0.5}

{\tiny
\setlength{\arraycolsep}{0,3mm}
$$
\begin{array}{r|rrrrrrrrrrrrrrrrrrrrrrrr}
2&18&18&18&18&17&17&16&16&16&16&17&17&15&15&16&13&13&13&13&14&14
&12&12\\
3&1&1&1&1&1&1&1&1&1&1&.&.&1&1&.&1&1&1&1&.&.&1&1\\
5&1&1&1&1&1&1&.&.&.&.&.&.&.&.&.&.&.&.&.&.&.&.&.\\
\hline
&1a&2_{1}&2_{2}&2_{3}&2_{4}&2_{5}&2_{6}&2_{7}&2_{8}&2_{9}&2_{10}
&2_{11}&2_{12}&2_{13}&2_{14}&2_{15}&2_{16}&2_{17}&2_{18}&2_{19}&2_{20}
&2_{21}&2_{22}\\
\hline
2P&1a&1a&1a&1a&1a&1a&1a&1a&1a&1a&1a&1a&1a&1a&1a&1a&1a&1a&1a&1a&1a
&1a&1a\\
3P&1a&2_{1}&2_{2}&2_{3}&2_{4}&2_{5}&2_{6}&2_{7}&2_{8}&2_{9}&2_{10}
&2_{11}&2_{12}&2_{13}&2_{14}&2_{15}&2_{16}&2_{17}&2_{18}&2_{19}&2_{20}
&2_{21}&2_{22}\\
5P&1a&2_{1}&2_{2}&2_{3}&2_{4}&2_{5}&2_{6}&2_{7}&2_{8}&2_{9}&2_{10}
&2_{11}&2_{12}&2_{13}&2_{14}&2_{15}&2_{16}&2_{17}&2_{18}&2_{19}&2_{20}
&2_{21}&2_{22}\\
\hline
X.1&1&1&1&1&1&1&1&1&1&1&1&1&1&1&1&1&1&1&1&1&1&1&1\\
\end{array}
$$
} }

\newpage
Character table of $D(2\Fi_{22})$ (continued){
\renewcommand{\baselinestretch}{0.5}

{\tiny
\setlength{\arraycolsep}{0,3mm}
$$
\begin{array}{r|rrrrrrrrrrrrrrrrrrrrrrrrrr}
2&13&13&11&11&12&12&12&12&12&11&11&10&6&11&11&12&12&11&11&11&11
&11&11&11&11\\
3&.&.&1&1&.&.&.&.&.&.&.&.&1&1&1&.&.&.&.&.&.&.&.&.&.\\
5&.&.&.&.&.&.&.&.&.&.&.&.&.&.&.&.&.&.&.&.&.&.&.&.&.\\
\hline
&2_{23}&2_{24}&2_{25}&2_{26}&2_{27}&2_{28}&2_{29}&2_{30}&2_{31}
&2_{32}&2_{33}&2_{34}&3a&4_{1}&4_{2}&4_{3}&4_{4}&4_{5}&4_{6}&4_{7}
&4_{8}&4_{9}&4_{10}&4_{11}&4_{12}\\
\hline
2P&1a&1a&1a&1a&1a&1a&1a&1a&1a&1a&1a&1a&3a&2_{2}&2_{2}&2_{1}&2_{1}
&2_{10}&2_{10}&2_{10}&2_{11}&2_{10}&2_{11}&2_{11}&2_{11}\\
3P&2_{23}&2_{24}&2_{25}&2_{26}&2_{27}&2_{28}&2_{29}&2_{30}&2_{31}
&2_{32}&2_{33}&2_{34}&1a&4_{1}&4_{2}&4_{3}&4_{4}&4_{5}&4_{6}&4_{7}
&4_{8}&4_{9}&4_{10}&4_{11}&4_{12}\\
5P&2_{23}&2_{24}&2_{25}&2_{26}&2_{27}&2_{28}&2_{29}&2_{30}&2_{31}
&2_{32}&2_{33}&2_{34}&3a&4_{1}&4_{2}&4_{3}&4_{4}&4_{5}&4_{6}&4_{7}
&4_{8}&4_{9}&4_{10}&4_{11}&4_{12}\\
\hline
X.1&1&1&1&1&1&1&1&1&1&1&1&1&1&1&1&1&1&1&1&1&1&1&1&1&1\\
\end{array}
$$
} }

\newpage
Character table of $D(2\Fi_{22})$ (continued){
\renewcommand{\baselinestretch}{0.5}

{\tiny
\setlength{\arraycolsep}{0,3mm}
$$
\begin{array}{r|rrrrrrrrrrrrrrrrrrrrrrrrrrrr}
2&9&9&9&9&10&10&10&10&10&10&10&10&10&10&10&10&9&9&9&9&9&9&9&9&9&9
&9\\
3&1&1&1&1&.&.&.&.&.&.&.&.&.&.&.&.&.&.&.&.&.&.&.&.&.&.&.\\
5&.&.&.&.&.&.&.&.&.&.&.&.&.&.&.&.&.&.&.&.&.&.&.&.&.&.&.\\
\hline
&4_{13}&4_{14}&4_{15}&4_{16}&4_{17}&4_{18}&4_{19}&4_{20}&4_{21}
&4_{22}&4_{23}&4_{24}&4_{25}&4_{26}&4_{27}&4_{28}&4_{29}&4_{30}&4_{31}
&4_{32}&4_{33}&4_{34}&4_{35}&4_{36}&4_{37}&4_{38}&4_{39}\\
\hline
2P&2_{15}&2_{16}&2_{15}&2_{16}&2_{6}&2_{10}&2_{14}&2_{8}&2_{8}
&2_{14}&2_{10}&2_{11}&2_{11}&2_{10}&2_{2}&2_{6}&2_{19}&2_{19}&2_{16}
&2_{16}&2_{15}&2_{11}&2_{16}&2_{14}&2_{16}&2_{13}&2_{16}\\
3P&4_{13}&4_{14}&4_{15}&4_{16}&4_{17}&4_{18}&4_{19}&4_{20}&4_{21}
&4_{22}&4_{23}&4_{24}&4_{25}&4_{26}&4_{27}&4_{28}&4_{29}&4_{30}&4_{31}
&4_{32}&4_{33}&4_{34}&4_{35}&4_{36}&4_{37}&4_{38}&4_{39}\\
5P&4_{13}&4_{14}&4_{15}&4_{16}&4_{17}&4_{18}&4_{19}&4_{20}&4_{21}
&4_{22}&4_{23}&4_{24}&4_{25}&4_{26}&4_{27}&4_{28}&4_{29}&4_{30}&4_{31}
&4_{32}&4_{33}&4_{34}&4_{35}&4_{36}&4_{37}&4_{38}&4_{39}\\
\hline
X.1&1&1&1&1&1&1&1&1&1&1&1&1&1&1&1&1&1&1&1&1&1&1&1&1&1&1&1\\
\end{array}
$$
} }

\newpage
Character table of $D(2\Fi_{22})$ (continued){
\renewcommand{\baselinestretch}{0.5}

{\tiny
\setlength{\arraycolsep}{0,3mm}
$$
\begin{array}{r|rrrrrrrrrrrrrrrrrrrrrrrrrrrrr}
2&9&9&9&9&9&9&8&8&8&8&8&8&8&8&8&7&7&7&7&7&3&6&6&6&6&6&6&6\\
3&.&.&.&.&.&.&.&.&.&.&.&.&.&.&.&.&.&.&.&.&.&1&1&1&1&1&1&1\\
5&.&.&.&.&.&.&.&.&.&.&.&.&.&.&.&.&.&.&.&.&1&.&.&.&.&.&.&.\\
\hline
&4_{40}&4_{41}&4_{42}&4_{43}&4_{44}&4_{45}&4_{46}&4_{47}&4_{48}
&4_{49}&4_{50}&4_{51}&4_{52}&4_{53}&4_{54}&4_{55}&4_{56}&4_{57}&4_{58}
&4_{59}&5a&6a&6b&6c&6d&6e&6f&6g\\
\hline
2P&2_{15}&2_{14}&2_{11}&2_{13}&2_{16}&2_{10}&2_{16}&2_{20}&2_{19}
&2_{20}&2_{20}&2_{23}&2_{16}&2_{19}&2_{23}&2_{27}&2_{28}&2_{27}&2_{20}
&2_{28}&5a&3a&3a&3a&3a&3a&3a&3a\\
3P&4_{40}&4_{41}&4_{42}&4_{43}&4_{44}&4_{45}&4_{46}&4_{47}&4_{48}
&4_{49}&4_{50}&4_{51}&4_{52}&4_{53}&4_{54}&4_{55}&4_{56}&4_{57}&4_{58}
&4_{59}&5a&2_{15}&2_{16}&2_{2}&2_{17}&2_{3}&2_{1}&2_{18}\\
5P&4_{40}&4_{41}&4_{42}&4_{43}&4_{44}&4_{45}&4_{46}&4_{47}&4_{48}
&4_{49}&4_{50}&4_{51}&4_{52}&4_{53}&4_{54}&4_{55}&4_{56}&4_{57}&4_{58}
&4_{59}&1a&6a&6b&6c&6d&6e&6f&6g\\
\hline
X.1&1&1&1&1&1&1&1&1&1&1&1&1&1&1&1&1&1&1&1&1&1&1&1&1&1&1&1&1\\
\end{array}
$$
} }

\newpage
Character table of $D(2\Fi_{22})$ (continued){
\renewcommand{\baselinestretch}{0.5}

{\tiny
\setlength{\arraycolsep}{0,3mm}
$$
\begin{array}{r|rrrrrrrrrrrrrrrrrrrrrrrrrrrrrr}
2&5&5&5&5&5&5&5&5&5&5&5&5&4&4&7&7&7&7&7&7&7&6&6&6&6&6&6&6&6\\
3&1&1&1&1&1&1&1&1&1&1&1&1&1&1&.&.&.&.&.&.&.&.&.&.&.&.&.&.&.\\
5&.&.&.&.&.&.&.&.&.&.&.&.&.&.&.&.&.&.&.&.&.&.&.&.&.&.&.&.&.\\
\hline
&6h&6i&6j&6k&6l&6m&6n&6o&6p&6q&6r&6s&6t&6u&8a&8b&8c&8d&8e&8f&8g
&8h&8i&8j&8k&8l&8m&8n&8o\\
\hline
2P&3a&3a&3a&3a&3a&3a&3a&3a&3a&3a&3a&3a&3a&3a&4_{3}&4_{5}&4_{5}
&4_{4}&4_{4}&4_{5}&4_{5}&4_{17}&4_{3}&4_{18}&4_{17}&4_{19}&4_{20}
&4_{19}&4_{6}\\
3P&2_{4}&2_{12}&2_{12}&2_{13}&2_{7}&2_{6}&2_{21}&2_{5}&2_{22}
&2_{9}&2_{8}&2_{13}&2_{25}&2_{26}&8a&8b&8c&8d&8e&8f&8g&8h&8i&8p&8k&8l
&8m&8n&8o\\
5P&6h&6j&6i&6s&6l&6m&6n&6o&6p&6q&6r&6k&6t&6u&8a&8b&8c&8d&8e&8f&8g
&8h&8i&8j&8k&8l&8m&8n&8o\\
\hline
X.1&1&1&1&1&1&1&1&1&1&1&1&1&1&1&1&1&1&1&1&1&1&1&1&1&1&1&1&1&1\\
\end{array}
$$
} }

\newpage
Character table of $D(2\Fi_{22})$ (continued){
\renewcommand{\baselinestretch}{0.5}

{\tiny
\setlength{\arraycolsep}{0,3mm}
$$
\begin{array}{r|rrrrrrrrrrrrrrrrrrrr}
2&6&6&5&5&3&3&3&3&3&3&3&4&4&4&4&4&4&5&5\\
3&.&.&.&.&.&.&.&.&.&.&.&1&1&1&1&1&1&.&.\\
5&.&.&.&.&1&1&1&1&1&1&1&.&.&.&.&.&.&.&.\\
\hline
&8p&8q&8r&8s&10a&10b&10c&10d&10e&10f&10g&12a&12b&12c&12d&12e&12f
&16a&16b\\
\hline
2P&4_{18}&4_{20}&4_{29}&4_{30}&5a&5a&5a&5a&5a&5a&5a&6a&6b&6c&6a&6c
&6b&8a&8a\\
3P&8j&8q&8r&8s&10a&10c&10b&10d&10e&10g&10f&4_{13}&4_{14}&4_{1}
&4_{15}&4_{2}&4_{16}&16a&16b\\
5P&8p&8q&8r&8s&2_{2}&2_{4}&2_{4}&2_{3}&2_{1}&2_{5}&2_{5}&12a&12b
&12c&12d&12e&12f&16b&16a\\
\hline
X.1&1&1&1&1&1&1&1&1&1&1&1&1&1&1&1&1&1&1&1\\
\end{array}
$$
} }

\newpage
Character table of $D(2\Fi_{22})$ (continued){
\renewcommand{\baselinestretch}{0.5}

{\tiny
\setlength{\arraycolsep}{0,3mm}
$$
\begin{array}{r|rrrrrrrrrrrrrrrrrrrrrr}
2&18&18&18&18&17&17&16&16&16&16&17&17&15&15&16&13&13&13&13&14&14\\
3&1&1&1&1&1&1&1&1&1&1&.&.&1&1&.&1&1&1&1&.&.\\
5&1&1&1&1&1&1&.&.&.&.&.&.&.&.&.&.&.&.&.&.&.\\
\hline
&1a&2_{1}&2_{2}&2_{3}&2_{4}&2_{5}&2_{6}&2_{7}&2_{8}&2_{9}&2_{10}
&2_{11}&2_{12}&2_{13}&2_{14}&2_{15}&2_{16}&2_{17}&2_{18}&2_{19}&2_{20}\\
\hline
2P&1a&1a&1a&1a&1a&1a&1a&1a&1a&1a&1a&1a&1a&1a&1a&1a&1a&1a&1a&1a&1a\\
3P&1a&2_{1}&2_{2}&2_{3}&2_{4}&2_{5}&2_{6}&2_{7}&2_{8}&2_{9}&2_{10}
&2_{11}&2_{12}&2_{13}&2_{14}&2_{15}&2_{16}&2_{17}&2_{18}&2_{19}&2_{20}\\
5P&1a&2_{1}&2_{2}&2_{3}&2_{4}&2_{5}&2_{6}&2_{7}&2_{8}&2_{9}&2_{10}
&2_{11}&2_{12}&2_{13}&2_{14}&2_{15}&2_{16}&2_{17}&2_{18}&2_{19}&2_{20}\\
\hline
X.94&120&120&120&120&-120&-120&48&48&48&48&24&24&-48&-48&-24&12&12
&12&12&-8&-8\\
\end{array}
$$
} }

\newpage
Character table of $D(2\Fi_{22})$ (continued){
\renewcommand{\baselinestretch}{0.5}

{\tiny
\setlength{\arraycolsep}{0,3mm}
$$
\begin{array}{r|rrrrrrrrrrrrrrrrrrrrrrrrr}
2&12&12&13&13&11&11&12&12&12&12&12&11&11&10&6&11&11&12&12&11&11
&11&11&11\\
3&1&1&.&.&1&1&.&.&.&.&.&.&.&.&1&1&1&.&.&.&.&.&.&.\\
5&.&.&.&.&.&.&.&.&.&.&.&.&.&.&.&.&.&.&.&.&.&.&.&.\\
\hline
&2_{21}&2_{22}&2_{23}&2_{24}&2_{25}&2_{26}&2_{27}&2_{28}&2_{29}
&2_{30}&2_{31}&2_{32}&2_{33}&2_{34}&3a&4_{1}&4_{2}&4_{3}&4_{4}&4_{5}
&4_{6}&4_{7}&4_{8}&4_{9}\\
\hline
2P&1a&1a&1a&1a&1a&1a&1a&1a&1a&1a&1a&1a&1a&1a&3a&2_{2}&2_{2}&2_{1}
&2_{1}&2_{10}&2_{10}&2_{10}&2_{11}&2_{10}\\
3P&2_{21}&2_{22}&2_{23}&2_{24}&2_{25}&2_{26}&2_{27}&2_{28}&2_{29}
&2_{30}&2_{31}&2_{32}&2_{33}&2_{34}&1a&4_{1}&4_{2}&4_{3}&4_{4}&4_{5}
&4_{6}&4_{7}&4_{8}&4_{9}\\
5P&2_{21}&2_{22}&2_{23}&2_{24}&2_{25}&2_{26}&2_{27}&2_{28}&2_{29}
&2_{30}&2_{31}&2_{32}&2_{33}&2_{34}&3a&4_{1}&4_{2}&4_{3}&4_{4}&4_{5}
&4_{6}&4_{7}&4_{8}&4_{9}\\
\hline
X.94&-12&-12&8&.&6&6&-4&-4&-4&-4&.&4&4&2&.&-6&-6&.&.&-4&4&-6&6&-6\\
\end{array}
$$
} }

\newpage
Character table of $D(2\Fi_{22})$ (continued){
\renewcommand{\baselinestretch}{0.5}

{\tiny
\setlength{\arraycolsep}{0,3mm}
$$
\begin{array}{r|rrrrrrrrrrrrrrrrrrrrrrrrrrr}
2&11&11&11&9&9&9&9&10&10&10&10&10&10&10&10&10&10&10&10&9&9&9&9&9
&9&9\\
3&.&.&.&1&1&1&1&.&.&.&.&.&.&.&.&.&.&.&.&.&.&.&.&.&.&.\\
5&.&.&.&.&.&.&.&.&.&.&.&.&.&.&.&.&.&.&.&.&.&.&.&.&.&.\\
\hline
&4_{10}&4_{11}&4_{12}&4_{13}&4_{14}&4_{15}&4_{16}&4_{17}&4_{18}
&4_{19}&4_{20}&4_{21}&4_{22}&4_{23}&4_{24}&4_{25}&4_{26}&4_{27}&4_{28}
&4_{29}&4_{30}&4_{31}&4_{32}&4_{33}&4_{34}&4_{35}\\
\hline
2P&2_{11}&2_{11}&2_{11}&2_{15}&2_{16}&2_{15}&2_{16}&2_{6}&2_{10}
&2_{14}&2_{8}&2_{8}&2_{14}&2_{10}&2_{11}&2_{11}&2_{10}&2_{2}&2_{6}
&2_{19}&2_{19}&2_{16}&2_{16}&2_{15}&2_{11}&2_{16}\\
3P&4_{10}&4_{11}&4_{12}&4_{13}&4_{14}&4_{15}&4_{16}&4_{17}&4_{18}
&4_{19}&4_{20}&4_{21}&4_{22}&4_{23}&4_{24}&4_{25}&4_{26}&4_{27}&4_{28}
&4_{29}&4_{30}&4_{31}&4_{32}&4_{33}&4_{34}&4_{35}\\
5P&4_{10}&4_{11}&4_{12}&4_{13}&4_{14}&4_{15}&4_{16}&4_{17}&4_{18}
&4_{19}&4_{20}&4_{21}&4_{22}&4_{23}&4_{24}&4_{25}&4_{26}&4_{27}&4_{28}
&4_{29}&4_{30}&4_{31}&4_{32}&4_{33}&4_{34}&4_{35}\\
\hline
X.94&6&-4&-4&.&.&.&.&.&4&.&.&.&.&-4&2&4&-2&-2&.&-4&4&-2&-2&4&.&-4\\
\end{array}
$$
} }

\newpage
Character table of $D(2\Fi_{22})$ (continued){
\renewcommand{\baselinestretch}{0.5}

{\tiny
\setlength{\arraycolsep}{0,3mm}
$$
\begin{array}{r|rrrrrrrrrrrrrrrrrrrrrrrrrrrr}
2&9&9&9&9&9&9&9&9&9&9&8&8&8&8&8&8&8&8&8&7&7&7&7&7&3&6&6\\
3&.&.&.&.&.&.&.&.&.&.&.&.&.&.&.&.&.&.&.&.&.&.&.&.&.&1&1\\
5&.&.&.&.&.&.&.&.&.&.&.&.&.&.&.&.&.&.&.&.&.&.&.&.&1&.&.\\
\hline
&4_{36}&4_{37}&4_{38}&4_{39}&4_{40}&4_{41}&4_{42}&4_{43}&4_{44}
&4_{45}&4_{46}&4_{47}&4_{48}&4_{49}&4_{50}&4_{51}&4_{52}&4_{53}&4_{54}
&4_{55}&4_{56}&4_{57}&4_{58}&4_{59}&5a&6a&6b\\
\hline
2P&2_{14}&2_{16}&2_{13}&2_{16}&2_{15}&2_{14}&2_{11}&2_{13}&2_{16}
&2_{10}&2_{16}&2_{20}&2_{19}&2_{20}&2_{20}&2_{23}&2_{16}&2_{19}&2_{23}
&2_{27}&2_{28}&2_{27}&2_{20}&2_{28}&5a&3a&3a\\
3P&4_{36}&4_{37}&4_{38}&4_{39}&4_{40}&4_{41}&4_{42}&4_{43}&4_{44}
&4_{45}&4_{46}&4_{47}&4_{48}&4_{49}&4_{50}&4_{51}&4_{52}&4_{53}&4_{54}
&4_{55}&4_{56}&4_{57}&4_{58}&4_{59}&5a&2_{15}&2_{16}\\
5P&4_{36}&4_{37}&4_{38}&4_{39}&4_{40}&4_{41}&4_{42}&4_{43}&4_{44}
&4_{45}&4_{46}&4_{47}&4_{48}&4_{49}&4_{50}&4_{51}&4_{52}&4_{53}&4_{54}
&4_{55}&4_{56}&4_{57}&4_{58}&4_{59}&1a&6a&6b\\
\hline
X.94&.&-4&.&-2&4&.&6&.&-2&-6&2&2&-2&.&-2&.&2&2&.&.&.&.&.&.&.&.&.\\
\end{array}
$$
} }

\newpage
Character table of $D(2\Fi_{22})$ (continued){
\renewcommand{\baselinestretch}{0.5}

{\tiny
\setlength{\arraycolsep}{0,3mm}
$$
\begin{array}{r|rrrrrrrrrrrrrrrrrrrrrrrrrrrrrrr}
2&6&6&6&6&6&5&5&5&5&5&5&5&5&5&5&5&5&4&4&7&7&7&7&7&7&7&6&6&6&6\\
3&1&1&1&1&1&1&1&1&1&1&1&1&1&1&1&1&1&1&1&.&.&.&.&.&.&.&.&.&.&.\\
5&.&.&.&.&.&.&.&.&.&.&.&.&.&.&.&.&.&.&.&.&.&.&.&.&.&.&.&.&.&.\\
\hline
&6c&6d&6e&6f&6g&6h&6i&6j&6k&6l&6m&6n&6o&6p&6q&6r&6s&6t&6u&8a&8b
&8c&8d&8e&8f&8g&8h&8i&8j&8k\\
\hline
2P&3a&3a&3a&3a&3a&3a&3a&3a&3a&3a&3a&3a&3a&3a&3a&3a&3a&3a&3a&4_{3}
&4_{5}&4_{5}&4_{4}&4_{4}&4_{5}&4_{5}&4_{17}&4_{3}&4_{18}&4_{17}\\
3P&2_{2}&2_{17}&2_{3}&2_{1}&2_{18}&2_{4}&2_{12}&2_{12}&2_{13}
&2_{7}&2_{6}&2_{21}&2_{5}&2_{22}&2_{9}&2_{8}&2_{13}&2_{25}&2_{26}&8a
&8b&8c&8d&8e&8f&8g&8h&8i&8p&8k\\
5P&6c&6d&6e&6f&6g&6h&6j&6i&6s&6l&6m&6n&6o&6p&6q&6r&6k&6t&6u&8a&8b
&8c&8d&8e&8f&8g&8h&8i&8j&8k\\
\hline
X.94&.&.&.&.&.&.&.&.&.&.&.&.&.&.&.&.&.&.&.&.&.&2&2&-2&-2&.&.&.&.&.\\
\end{array}
$$
} }

\newpage
Character table of $D(2\Fi_{22})$ (continued){
\renewcommand{\baselinestretch}{0.5}

{\tiny
\setlength{\arraycolsep}{0,3mm}
$$
\begin{array}{r|rrrrrrrrrrrrrrrrrrrrrrrr}
2&6&6&6&6&6&6&5&5&3&3&3&3&3&3&3&4&4&4&4&4&4&5&5\\
3&.&.&.&.&.&.&.&.&.&.&.&.&.&.&.&1&1&1&1&1&1&.&.\\
5&.&.&.&.&.&.&.&.&1&1&1&1&1&1&1&.&.&.&.&.&.&.&.\\
\hline
&8l&8m&8n&8o&8p&8q&8r&8s&10a&10b&10c&10d&10e&10f&10g&12a&12b&12c
&12d&12e&12f&16a&16b\\
\hline
2P&4_{19}&4_{20}&4_{19}&4_{6}&4_{18}&4_{20}&4_{29}&4_{30}&5a&5a&5a
&5a&5a&5a&5a&6a&6b&6c&6a&6c&6b&8a&8a\\
3P&8l&8m&8n&8o&8j&8q&8r&8s&10a&10c&10b&10d&10e&10g&10f&4_{13}
&4_{14}&4_{1}&4_{15}&4_{2}&4_{16}&16a&16b\\
5P&8l&8m&8n&8o&8p&8q&8r&8s&2_{2}&2_{4}&2_{4}&2_{3}&2_{1}&2_{5}
&2_{5}&12a&12b&12c&12d&12e&12f&16b&16a\\
\hline
X.94&.&.&.&.&.&.&.&.&.&.&.&.&.&.&.&.&.&.&.&.&.&.&.\\
\end{array}
$$
} } where
$
A = -4\zeta(4)_4,
B = -2\zeta(5)_5^3 - 2\zeta(5)_5^2 - 1,
C = -4\zeta(3)_3 - 2,
D = 2\zeta(8)_8^3 + 2\zeta(8)_8
$.
\end{ctab}

\end{appendix}

\newpage


\end{document}